\documentclass[english,11pt, reqno, oneside]{amsart}
\usepackage{amsmath}
\usepackage{amssymb}
\usepackage{amsthm}
\usepackage{mathtools}
\usepackage{bbm}
\usepackage{stmaryrd}
\usepackage{xcolor}
\usepackage[colorlinks=true,linkcolor=blue!70!black, citecolor = green!55!black,bookmarksdepth=3]{hyperref}
\usepackage[margin=1in, marginparwidth=2.2cm]{geometry}
\usepackage{multicol}
\usepackage{etoolbox}
\usepackage{enumitem}
\usepackage{float}
\usepackage[mathscr]{euscript}
\usepackage{graphicx}
\usepackage{subcaption}
\usepackage{thmtools}

\usepackage{thm-restate}

\usepackage{tikz}

\usetikzlibrary{intersections}
\usetikzlibrary{calc}
\usetikzlibrary{decorations.markings}
\usetikzlibrary{decorations.pathreplacing}
\usetikzlibrary{arrows}
\usetikzlibrary{snakes}
\usetikzlibrary{overlay-beamer-styles}
\usetikzlibrary{external}
\usetikzlibrary{matrix}
\usetikzlibrary{shapes.misc,calc, arrows.meta, fit, hobby}
\usetikzlibrary{decorations.pathmorphing}

\usepackage[normalem]{ulem}

\usepackage{microtype}

\usepackage{crossreftools}

\setlist[enumerate,1]{label = (\roman*)}

\allowdisplaybreaks

\newcommand{\hair}{\ifmmode\mskip1mu\else\kern0.08em\fi}
\renewcommand{\P}{\mathbb{P}}
\newcommand{\PF}{\mathbb{P}_{\mathcal F}}
\newcommand{\EF}{\mathbb{E}_{\mathcal F}}

\newcommand{\Var}{\mathrm{Var}}
\newcommand{\E}{\mathbb{E}}

\newcommand{\cP}{\mathcal{P}}
\newcommand{\F}{\mathcal{F}}
\newcommand{\R}{\mathbb{R}}

\newcommand{\N}{\mathbb{N}}
\newcommand{\Z}{\mathbb{Z}}

\newcommand{\one}{\mathbbm{1}}
\newcommand{\intint}[1]{\llbracket #1 \rrbracket}
\newcommand{\mc}{\mathcal}
\newcommand{\mf}{\mathfrak}

\newcommand{\A}{\mathcal A}
\renewcommand{\epsilon}{\varepsilon}

\renewcommand{\emptyset}{\varnothing}

\newcommand{\cL}{\mathcal{L}}
\newcommand{\tri}{\mathsf{Tri}_\theta}
\newcommand{\nonint}{\mathsf{NonInt}_{\theta, M}}
\newcommand{\xtan}{x^{\mathrm{tan}}}

\newcommand{\fav}{\mathsf{Fav}}
\newcommand{\Pfree}{\mathbb{P}_{\mathrm{free}}}
\newcommand{\Efree}{\mathbb{E}_{\mathrm{free}}}

\newcommand{\bpara}{B^{\mathrm{para}}}
\newcommand{\bparat}{B^{\mathrm{para},t}}
\newcommand{\ellt}{\ell^{\mathrm{tan}}}
\newcommand{\Ilin}{I_{\mathrm{lin}}}
\newcommand{\pin}{\mathsf{Pin}^\theta}

\newcommand{\linhull}{\mathsf{ConHull}_{a,b}}
\newcommand{\tent}{\mathsf{Tent}_{a,b}}
\newcommand{\dev}{\mathsf{Dev}}
\newcommand{\rmleft}{\mathrm{left}}
\newcommand{\rmcent}{\mathrm{cent}}
\newcommand{\rmright}{\mathrm{right}}

\newcommand{\h}{\mf h^t}
\newcommand{\hf}{\mf h^{t,f}}

\renewcommand{\tan}{\mathrm{tan}}

\newcommand{\Fext}{\mathcal F_\mathrm{ext}}

\newcommand{\msf}{\mathsf}
\newcommand{\mrm}{\mathrm}
\newcommand{\Hyp}{\mathrm{Hyp}}
\newcommand{\dif}{\mathrm{d}}
\newcommand{\midd}{\ \Big|\ }

\makeatletter
\newcommand*\bigcdot{\mathpalette\bigcdot@{.45}}
\newcommand*\bigcdot@[2]{\mathbin{\vcenter{\hbox{\scalebox{#2}{$\m@th#1\bullet$}}}}}
\makeatother

\newtheorem{maintheorem}{Theorem}

\newtheorem{theorem}{Theorem}[section]
\newtheorem*{theorem*}{Theorem}
\newtheorem*{proposition*}{Proposition}
\newtheorem{proposition}[theorem]{Proposition}
\newtheorem*{corollary*}{Corollary}
\newtheorem{corollary}[theorem]{Corollary}
\newtheorem{lemma}[theorem]{Lemma}

\theoremstyle{definition}
\newtheorem{definition}[theorem]{Definition}

\newtheorem*{notation}{Notation}
\newtheorem{remark}[theorem]{Remark}

\makeatletter
\renewenvironment{proof}[1][\proofname]{%
  \par\pushQED{\qed}%
  \normalfont
  \topsep6\p@\@plus6\p@\relax
  \trivlist
  \item[] % empty label – we typeset the head ourselves
    {\itshape #1.\ }% <-- the actual heading
  \ignorespaces
}{%
  \popQED\endtrivlist\@endpefalse
}
\makeatother

\title[Sharp upper tail behavior of line ensembles via the tangent method]{Sharp upper tail behavior of line ensembles\\ via the tangent method}

\author{Shirshendu Ganguly}
\address{Shirshendu Ganguly, Department of Statistics, U.C. Berkeley, Berkeley, CA, USA}
\email{sganguly@berkeley.edu}

\author{Milind Hegde}
\address{Milind Hegde, Division of Mathematical Sciences, School of Physical and Mathematical Sciences, Nanyang Technological University, Singapore.}
\email{milind.hegde@ntu.edu.sg}

\begin{document}

\begin{abstract}
We develop a new probabilistic and geometric method to obtain several sharp results pertaining to the upper tail behavior of continuum Gibbs measures on infinite ensembles of random continuous curves, also known as line ensembles, satisfying some natural assumptions. The arguments make crucial use of Brownian resampling invariance properties and correlation inequalities admitted by such Gibbs measures. We obtain sharp one-point upper tail estimates showing that the probability of the value at zero being larger than $\theta$ is $\exp(-\frac{4}{3}\theta^{3/2}(1+o(1)))$. A key intermediate step is developing a precise understanding of the profile when conditioned on the value at zero equaling $\theta$. Our method further allows one to obtain multi-point asymptotics which were out of reach of previous approaches. As an example, we prove sharp explicit two-point upper tail estimates. This framework is then used to establish the corresponding results for the KPZ equation, which are all new. Even for the zero-temperature case of the Airy$_2$ process, our arguments yield new proofs for one-point estimates previously known due to its connections to random matrix theory, as well as new two-point asymptotics. To showcase the reach of the method, we obtain the same results in a purely non-integrable setting under only assumptions of stationarity and extremality in the class of Gibbs measures. 
 Our method bears resemblance to the tangent method introduced by Colomo-Sportiello and mathematically realized by Aggarwal in the context of the six-vertex model.
\end{abstract}

\maketitle
\thispagestyle{empty}

\vspace{-0.3cm}

\begin{center}
\includegraphics[width=0.95\textwidth]{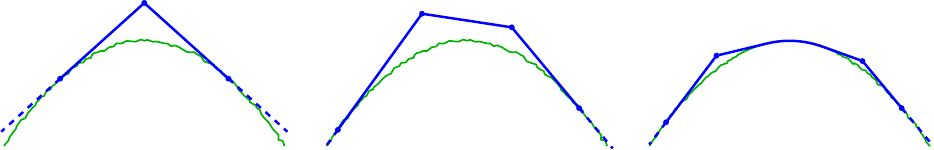}
\end{center}

\vspace*{-0.4cm}

\setcounter{tocdepth}{1}
\tableofcontents

\section{Introduction, results, and proof ideas}\label{s.intro}

Line ensembles are random objects which have been found to play central roles in a number of areas of probability theory over the last decade, including the Kardar-Parisi-Zhang (KPZ) universality class that this paper focuses on. They are collections of random curves with an explicit interaction captured via a Gibbs or spatial Markov property. In the KPZ universality class, the line ensembles that arise have the property that the lowest indexed curve is an observable, often called a height function, which is of independent interest. Two prominent examples of continuum line ensembles in KPZ, which will also be the main cases addressed in this paper, are the parabolic Airy line ensemble and the KPZ line ensemble; their lowest indexed curves are respectively the parabolic Airy$_2$ process and the KPZ equation started from the so-called narrow-wedge initial condition.

Of much interest is to understand the correlation structure of line ensembles. This paper provides a robust framework to understand the effect of raising one or multiple points of the lowest indexed curve to high values on the rest of the process, assuming certain reasonable hypotheses on the line ensemble, which are expected to hold for all natural examples (we will later give a detailed discussion on the validity of the assumptions for the concrete examples of the parabolic Airy and KPZ line ensembles). Such events of raised point values are often referred to as upper tail events and their probabilities have been a topic of active research in recent years. As a consequence, the framework also provides sharp probability estimates for these events.

The framework and the ideas underlying it have found a number of applications already since the first posting of this article, including the limiting behavior of geodesics and polymer measures (in the random metric-like object, the directed landscape, associated to the parabolic Airy line ensemble and the continuum directed random polymer, associated to the KPZ line ensemble, respectively) in the upper tail \cite{ganguly2023brownian} and models of area-tilted line ensembles coming from idealizations of the level lines of the low temperature 3D Ising model with a floor \cite{caputo2023uniqueness,chowdhury2023characterizing}. Related ideas have also played a role in recent work characterizing the Radon-Nikodym derivative of the increment of the Airy$_2$ process with respect to Brownian motion \cite{dauvergne2023wiener}. Broadly, we expect that the techniques would be applicable to many questions which concern the upper tail of the lowest-indexed curve of line ensembles with natural Gibbs properties. This robustness can be interpreted as a facet of universality.

Before reviewing the literature and the relevant background as well as introducing the objects of study formally, we start with a quick glimpse of our main results and techniques.

Most previous work obtaining quantitative probability bounds for processes such as the narrow-wedge-solution to the KPZ equation or the parabolic Airy$_2$ process has relied on exactly solvable structure, manifested in the form of explicit formulas for quantities like finite dimensional distributions or Laplace transforms, which are amenable to analytic techniques to obtain asymptotics. Still, it has proven difficult to obtain the sharp behaviour in many cases (including for multi-point upper tails). We emphasize that the methods here are probabilistic and geometric, thus allowing a unified treatment for all line ensembles satisfying our assumptions.

We start with matching (up to first order in the exponent) upper and lower bounds on the one-point upper tail of the first curve at depth $\theta$, obtaining the decay of $\exp(-\frac{4}{3}\theta^{3/2}(1+o(1)))$, with explicit lower order error bounds. 
Our arguments also yield sharp density estimates. 
 Further, the methods allow access to multi-point upper tails, which had been out of reach of previous approaches even for the example of the parabolic Airy$_2$ process. To showcase this, we obtain sharp estimates on two-point upper tails, and the method generalizes in a straightforward way to multiple points.

As mentioned, our arguments apply if the line ensemble in question satisfies certain assumptions. Informally, we assume the following (given in their precise forms in Section~\ref{s.assumptions}): (i) the line ensemble is stationary and possesses a certain explicit resampling or Gibbs property in terms of Brownian bridges; (ii)(a) the line ensemble is positively associated (satisfies the FKG inequality) and (b) satisfies an inequality we refer to as the van den Berg-Kesten (BK) inequality, which allows one to control the second curve conditional on the first curve; (iii) a certain conditional stochastic monotonicity property, given a finite number of values of the lowest indexed curve; and (iv) certain weak one-point tail bounds on the lowest indexed curve.

The objective of this article is to develop a method which applies to an axiomatic framework, but we pause here to discuss its applicability to particular examples. The assumptions all hold for the parabolic Airy line ensemble, either by earlier results or arguments given here.  For the KPZ line ensemble, Assumptions (i), (iii), and (iv) hold, as well as the positive association in (ii)(a). However, we believe the form of the BK inequality (Assumption~(ii)(b)) which holds for the parabolic Airy line ensemble is in fact too strong to hold in the KPZ case.\footnote{An earlier version of this article gave an incorrect proof that the KPZ line ensemble satisfies the BK inequality, which was brought to our attention by Xuan Wu.} For the purposes of the results here, a weaker form of the BK inequality actually suffices, which is proved for the KPZ line ensemble as a straightforward consequence of a form of the inequality established in the recent work \cite{ganguly2025van}. A more detailed discussion of the exact form of the inequality and the difficulties surrounding the strong form's proof is given in Section~\ref{s.weaker bk}.

The techniques bear resemblance to the tangent method proposed by Colomo-Sportiello \cite{colomo2016arctic} and mathematically realized by Aggarwal \cite{aggarwal2020arctic} to determine limit shapes in the context of the six-vertex model at the ice point.

To showcase the power and reach of our method, we also obtain the same sharp asymptotics in a purely non-integrable zero-temperature setting where we work with Brownian Gibbs ensembles whose laws are extremal in the space of such Gibbs measures and which enjoy a certain stationarity property. In this setting no integrable formulas (which would be needed to verify some of the assumptions for the parabolic Airy and KPZ line ensembles) are available and the arguments rely on purely qualitative assumptions; in particular, the quantitative assumption of a priori weak tail bounds mentioned above is derived as a consequence of the remaining assumptions. In this way, solely probabilistic considerations lend evidence to a conjecture suggested by Scott Sheffield (and formulated in \cite{corwin2014brownian}) which characterizes all ensembles with the mentioned properties;  this conjecture (indeed, a stronger form) was proven after the original arXiv posting of this article by Aggarwal-Huang \cite{aggarwal2023strong}.

\subsection{The KPZ and parabolic Airy line ensembles}\label{s.intro.kpz airy}
The KPZ universality class refers to a broad class of models of one-dimensional stochastic growth. These models include those of last passage percolation, exclusion processes, and polymer models, among others. Models in the class feature an observable called a \emph{height} function. It is often possible, through mappings such as the Robinson-Schensted-Knuth (RSK) correspondence (e.g., \cite{prahofer2002PNG}) and its generalizations (e.g., \cite{o2014geometric}) or the Yang-Baxter equation \cite{aggarwal2024colored}, to embed these height functions as the lowest indexed curve in a larger family of random curves, termed a \emph{line ensemble},  which possesses some sort of resampling or spatial Markov property, often called a Gibbs property. The laws of these structures can be viewed as infinite volume Gibbs measures on the space of collections of continuous curves.

For the cases of the KPZ equation (with the narrow wedge initial condition) and the parabolic Airy$_2$ processes, these line ensembles are the KPZ and parabolic Airy line ensembles, respectively. First constructed in \cite{corwin2016kpz}, the KPZ line ensemble (associated to $t>0$) is a collection of $\N$-indexed random continuous curves which interact with one another through an explicit resampling property known as the \emph{$H_t$-Brownian Gibbs property}. The narrow-wedge solution to the KPZ equation is then the lowest indexed curve in this collection. 

The parabolic Airy$_2$ process embeds as the lowest indexed curve in an analogous collection of random continuous non-intersecting curves known as the parabolic Airy line ensemble, first constructed in \cite{corwin2014brownian}, which enjoys an explicit resampling property known as simply the Brownian Gibbs property. An alternate perspective on the Airy line ensemble is that it is the scaling limit of the edge of Dyson Brownian motion. We will often refer to the KPZ equation as the \emph{positive temperature} case and the parabolic Airy$_2$ process as the \emph{zero temperature} case.

The interaction between curves can be understood as a hard or soft form of nonintersection. Very roughly speaking, this causes lower indexed curves to be pushed up by higher indexed curves. In the case of the parabolic Airy line ensemble, the parabolic Airy$_2$ process must avoid the second curve, while in the KPZ line ensemble, the narrow wedge solution suffers an exponential energetic penalty for staying below the second curve on an interval. These are both encoded via the respective Gibbs properties, which are more precisely introduced in Section~\ref{s.gibbs and line ensembles}.

We will denote the parabolic Airy$_2$ line ensemble by $\cP = (\cP_1,\cP_2, \ldots)$ and the (scaled) KPZ line ensemble at time $t>0$ by $\h=(\h_1,\h_2, \ldots)$. For the latter, the first curve $\h_1$ is a scaled version of the narrow wedge solution $\mc H$ to the KPZ equation $\partial_t \mc{H} = \frac{1}{4}\partial_x^2 \mc{H} + \frac{1}{4}(\partial_x \mc{H})^2 + \xi$:
\begin{align}\label{e.h_1 definition}
\h_1(x) = \frac{\mc H(t, t^{2/3}x) + \frac{t}{12}}{t^{1/3}};
\end{align}
see Section~\ref{s.kpz background} for a brief discussion of the solution theory for the KPZ equation and the definition of the narrow-wedge initial condition. 

Next we turn to giving our precise assumptions.

\subsection{Assumptions on the line ensembles} \label{s.assumptions}

Before stating our assumptions formally, we give some more detailed motivation for their content. 
The first assumption is that of stationarity and the possession of a Gibbs resampling property (see Section~\ref{s.gibbs and line ensembles} for precise definitions). The latter is fundamental to our arguments and perspective, and the former is a well-known property of our primary examples of interest, $\cP$ and $\h$. The remaining assumptions are all aimed towards proving qualitative or a priori control on the curves of the line ensembles. For instance, the first part of the second assumption is that the ensembles are positively associated, so that conditioning on increasing events (such as upper tail events) stochastically raises the entire ensemble. 

One of the key difficulties in working with non-intersecting curves and their upper tails is that lower curves can push up the top curve. As such, an important ingredient will be some control on the second curve of the ensemble conditional on the first curve; in particular, that it does not rise up too much even if the first curve is raised. There are a number of ways to encode this, and perhaps the cleanest (but a somewhat strong form) is given below in the second half of the second assumption; we will shortly thereafter, in Section~\ref{s.weaker bk}, introduce a weaker but more technical form which will suffice for our arguments. In spirit, the strong form is the same as the statement that if one conditions the largest eigenvalue of a random matrix to be large (e.g., in the large deviation regime), then the second eigenvalue behaves approximately like the unconditioned first eigenvalue (e.g., it lies at the unconditioned macroscopic location). An example of this in the cases of GOE and GUE can be found in \cite[Theorem 2.1]{biroli2020large}.

Stochastic monotonicity statements for line ensembles, which are closely related to the positive association property just mentioned, have played a crucial role in many previous studies, and will do so in our arguments as well. These statements usually concern the increasing nature of the law of curves on an interval with given boundary conditions under a Gibbs property as a function of the boundary data. Here, we need something slightly stronger. Namely, we assume that monotonicity holds also as a function of the values of the top curve at a finite number of intermediate points that have been conditioned upon (as in the case of one- or two-point upper tail conditionings), and this forms the third assumption.

Finally, an important ingredient in our proof is a finite range-of-effect phenomenon of the upper tail; more precisely, that conditional on an upper tail event of the first curve, there exists a point, perhaps quite far away, where the first curve has not gone up too far with high probability. By the parabolic decay of the first curve and the stationarity after shifting by the parabola, this is implied by weak one-point tail bounds at the origin, which form the fourth and final assumption.

Now we may turn to introducing some terminology and then stating the assumptions. We denote the line ensemble by $\cL = (\cL_1, \cL_2, \ldots)$. By the $t=\infty$ case of the $H_t$-Brownian Gibbs property we mean the usual (non-intersecting) Brownian Gibbs property; see Section~\ref{s.gibbs and line ensembles} for precise definitions. A subset $A\subseteq \mc C([a,b], \R)$ is said to be increasing if $f\in A$ and $g\in \mc C([a,b], \R)$ is such that $g(x)\geq f(x)$ for all $x\in[a,b]$, then $g\in A$. By stochastic domination, we mean that for any  finite interval $[a,b]$ and increasing function $F:\mc C([a,b],\R)\to\R$ (where the order on $\mc C([a,b],\R)$ is the usual point-wise order of functions), the expectation of $F$ under the dominating law is lower bounded by the same expectation under the dominated law.

\begin{enumerate}[itemsep=5pt]
	\item \label{as.bg} \textbf{Stationarity and Brownian Gibbs:} For some $t\in (0,\infty]$, $\cL$ possesses the $H_t$-Brownian Gibbs property, and $x\mapsto \cL_1(x) + x^2$ is stationary.

	\item \label{as.corr} \textbf{Correlation inequalities:} The first curve $\cL_1$ is positively associated and $\cL$ satisfies the van den Berg-Kesten (BK) inequality,\footnote{The original BK inequality comes from percolation theory, see for example \cite{grimmett1999percolation}, to bound the probability of two events occurring ``disjointly'' by the product of the probabilities of the events. We use the same terminology here because, by the RSK bijection, the top two curves of the parabolic Airy line ensemble can be related to weights of pairs of disjoint paths in limiting last passage percolation models, and in this context many cases of the inequality we describe are applications of the classical BK inequality.} i.e., for increasing events $A$ and $B$, and any event $C$ (all Borel subsets of $\mc C([a,b],\R)$ for some finite interval $[a,b]$),
	\begin{align*}
	&\text{(a)}\quad \P\left(\cL_1 \in A, \cL_1\in B\right) \geq \P\left(\cL_1 \in A\right)\cdot\P\left(\cL_1 \in B\right) \text{ and }\\
	&\text{(b)}\quad \P\left(\cL_2\in A, \cL_1\in C\right) \leq \P(\cL_1 \in A)\cdot\P(\cL_1\in C).
	\end{align*}

	\item \label{as.mono in cond} \textbf{Monotonicity in conditioning:} Let $m\in\N$, $x_1, \ldots, x_m\in\R$, and  $y^{(i)}_1, \ldots, y^{(i)}_m \in\R$ for $i=1,2$. If $\smash{y_j^{(1)} \geq y_j^{(2)}}$ for $j=1, \ldots, m$, then the conditional law of $\cL$ given $\cL_1(x_j) = \smash{y^{(1)}_j}$ for $j=1, \ldots, m$ stochastically dominates that of the same given $\cL_1(x_j) = \smash{y^{(2)}_j}$.

	\item \label{as.tails}\textbf{Uniform bounds on the one-point upper tail, and one-point tightness:} 
	There exist $\alpha, \beta>0$, $\theta_0$, and $c_1,c_2>0$ such that, for $\theta>\theta_0$,
	$$\exp(-c_1\theta^{\alpha}) \leq \P\left(\cL_1(0) > \theta\right) \leq \exp(-c_2\theta^{\beta}).$$

\end{enumerate}

These assumptions bear a thematic resemblance to those in an earlier paper of the authors \cite{ganguly2020optimal}, where upper and lower tail bounds with the correct exponents of $3/2$ and $3$ are derived in general last passage percolation models which satisfy the assumptions. Similar to here, the main things assumed or used in \cite{ganguly2020optimal} are parabolic curvature of the profile (Assumption~\ref{as.bg} here), the FKG and BK inequalities, and a priori tail bounds.

It turns out that Assumption~\ref{as.bg} implies that $\cL_1$ is absolutely continuous to Brownian motion on compact intervals (as shown in \cite{corwin2014brownian,corwin2016kpz}), and so the law of $\cL_1(0)$ is absolutely continuous with respect to Lebesgue measure, i.e., has a density. We also mention that in the case of $t=\infty$, Assumption~\ref{as.tails} is implied by the other three, as will be shown in Sections~\ref{s.one point.lower bound} and \ref{s.extremal} (this is also true in the case of finite $t$ if one does not demand the uniformity over $t>t_0$ for any fixed $t_0>0$ that is known for the KPZ equation \cite{corwin2018kpz}).

As mentioned in Section~\ref{s.intro.kpz airy}, it is shown that $\cP$ satisfies the assumptions (see Theorem~\ref{mt.extremal} and the discussion following) and $\h$ satisfies all but Assumption~\ref{as.corr}(b). As also already indicated, we in fact expect Assumption~\ref{as.corr}(b) to be too strong to hold for $\h$, and as a consequence of the recent work \cite{ganguly2025van} we establish a weaker version, recorded in Section~\ref{s.weaker bk} along with further discussion, which suffices for our argument. Thus the below results apply to both $\cP$ and $\h$.

\subsection{Main results on tail asymptotics}\label{s.intro.results.tails}
Here we give our main results on tail bounds. Results on limit shapes are given in Section~\ref{s.intro.results.shapes} and on one-point tail bounds for general initial data in Section~\ref{s.intro.general data}. 

\begin{remark}\label{r.uniformity}
We point out that all the constants in our results can in principle depend on $t$ from Assumption~\ref{as.bg}, as well as $c_1$, $c_2$, $\theta_0$, $\alpha$, and $\beta$ from Assumption~\ref{as.tails}. However, though we will not mention this explicitly in the statements, there is no dependence on $t$: if a class of line ensembles satisfy Assumption~\ref{as.tails} with the same values of the constants therein with possibly varying values of $t$, then our results hold uniformly over that class.
\end{remark}

In all of the below results, unless mentioned otherwise, Assumptions~\ref{as.bg}--\ref{as.tails} will be in force. As mentioned, Assumption~\ref{as.corr}(b) can also be replaced by a weaker form, which will be the assumption actually used in the arguments. This is formally recorded in Proposition~\ref{p.weak bk suffices}. Also as mentioned, the below results all hold for both the parabolic and KPZ line ensembles, as recorded in Theorem~\ref{t.assumptions hold} ahead.

\subsubsection{One-point tail and density asymptotics}

Our first result concerns the asymptotics of the density of $\cL_1(0)$ as the argument goes to $+\infty$.
We denote the density of $\cL_1(0)$ at $\theta$ by $\frac{1}{\dif \theta}\P(\cL_1(0)\in[\theta,\theta+\dif\theta])$. 
As noted above, Assumption~\ref{as.bg} guarantees that $\cL_1(0)$ has a density.

\begin{maintheorem}[One-point density asymptotics]\label{mt.one point density asymptotics}

There exist constants $C$ and $\theta_0$ such that, for $\theta>\theta_0$,
$$\exp\left(-\frac{4}{3}\theta^{3/2}-C\theta^{3/4}\right) \leq \frac{1}{\dif \theta}\P\Bigl(\cL_1(0)\in[\theta,\theta+\dif\theta]\Bigr) \leq \exp\left(-\frac{4}{3}\theta^{3/2}+C\theta^{3/4}\right).$$
\end{maintheorem}

Next we move to one-point tail asymptotics, which the previous result's proof also relies on.
Note that for the lower bound on the tail we have a better error term than in the density bound.

\begin{maintheorem}[One-point upper tail bounds]\label{mt.one point tail asymptotics}

There exist $\theta_0 > 0$ and $C<\infty$ such that, for $\theta > \theta_0$,
$$\exp\left(-\frac{4}{3}\theta^{3/2} - \theta^{1/2}\log \theta\right) \leq \P\Bigl(\cL_1(0) \geq \theta\Bigr) \leq \exp\left(-\frac{4}{3}\theta^{3/2} + C\theta^{3/4}\right).$$
\end{maintheorem}

Further, the lower bound in Theorem~\ref{mt.one point tail asymptotics} requires only Assumptions~\ref{as.bg} and \ref{as.corr}(a); see Theorem~\ref{t.upper tail lower bound}.

\subsubsection{Two-point tail asymptotics}\label{s.intro.two-point asymptotics}
We will consider the probability that $\cL_1$ is greater than $a\theta$ at $-\theta^{1/2}$ and greater than $b\theta$ at $\theta^{1/2}$ for $a\geq b>-1$ (so that $a\theta, b\theta > -\theta$, the value of the parabola $-x^2$ at $x=\pm\theta^{1/2}$). 

We note that by stationarity of $x\mapsto \cL_1(x)+x^2$ and if it holds that $x \mapsto \cL_1(x)$ is equal in distribution to $x\mapsto \cL_1(-x)$ (which holds for the narrow wedge KPZ equation and the parabolic Airy$_2$ process), the probability we analyze is equivalent to  the probability of any two-point event of the form $\{\cL_1(x_1) > -x_1^2 + t_1, \cL_1(x_2) > -x_2^2 + t_2\}$ with $x_1,x_2\in\R$ and $t_1,t_2$ large by an appropriate horizontal translation and choice of $a,b$. We consider $a\theta$ and $b\theta$ as this simplifies some expressions which, nonetheless, are still somewhat technical to look at. However, the main point that should be taken is that they are rather explicit.

 Let $\linhull:\R\to\R$ be the convex hull of $x\mapsto -x^2$ and the points $(-\theta^{1/2}, a\theta)$ and $(\theta^{1/2}, b\theta)$. The two-point asymptotics depends on the number of extreme points $\linhull$ has inside $\smash{[-\theta^{1/2},\theta^{1/2}]}$; see Figure~\ref{f.three cases of two-point}. Moreover, we will assert later in Theorem~\ref{mt.two-point limit shape} that $\linhull$ (or a minor variant) is the approximate shape that $\cL_1$ adopts under the conditioning of the two-point tail event we are considering here.

\begin{figure}[h!]
%!TEX root=../kpz-upper-tail-estimates.tex

\begin{tikzpicture}[scale=0.75]

% LEFT PANEL

\draw[green!70!black, semithick]  plot[smooth, domain=-2.5:2.5] (\x, -0.3*\x * \x);

\newcommand{\ltan}{2.3}
\newcommand{\rtan}{2}
\newcommand{\loc}{0.95} %was 0.8

\draw[blue, thick]  (-\ltan, -0.3*\ltan*\ltan) -- (-\loc, -0.3*\ltan*\ltan + 2*0.3*\ltan*\ltan - 2*0.3*\ltan*\loc) -- (\loc, -0.3*\rtan*\rtan + 2*0.3*\rtan*\rtan - 2*0.3*\rtan*\loc) -- (\rtan, -0.3*\rtan*\rtan);

%left corner
\node[circle, fill, blue, inner sep = 1pt] at (-\loc, -0.3*\ltan*\ltan + 2*0.3*\ltan*\ltan - 2*0.3*\ltan*\loc) {};
\node[anchor = south, scale=0.8] at (-\loc-0.05, -0.3*\ltan*\ltan + 2*0.3*\ltan*\ltan - 2*0.3*\ltan*\loc) {$(-\theta^{1/2}, a\theta)$};

%right corner
\node[circle, fill, blue, inner sep = 1pt] at (\loc, -0.3*\rtan*\rtan + 2*0.3*\rtan*\rtan - 2*0.3*\rtan*\loc) {};
\node[anchor = south, scale=0.8] at (\loc+0.05, -0.3*\rtan*\rtan + 2*0.3*\rtan*\rtan - 2*0.3*\rtan*\loc) {$(\theta^{1/2}, b\theta)$};

%left tangent
\node[circle, fill, blue, inner sep = 1pt] at (-\ltan, -0.3*\ltan*\ltan) {};
%\node[anchor = east, scale=0.9] at (-\ltan-0.1, -0.3*\ltan*\ltan) {$(-x^{\mrm{tan}}_\ell, -(x^{\mrm{tan}}_\ell)^2)$};

%right tangent
\node[circle, fill, blue, inner sep = 1pt] at (\rtan, -0.3*\rtan*\rtan) {};
%\node[anchor = west, scale=0.9] at (\rtan+0.1, -0.3*\rtan*\rtan) {$(x^{\mrm{tan}}_r, -(x^{\mrm{tan}}_r)^2)$};

% MIDDLE PANEL

\begin{scope}[shift={(6.5,0)}]
\renewcommand{\ltan}{2.2}
\renewcommand{\rtan}{2}
\renewcommand{\loc}{1.3}

\draw[green!70!black, semithick]  plot[smooth, domain=-2.5:(\ltan-2*\loc)] (\x, -0.3*\x * \x);
\draw[blue, thick]  plot[smooth, domain=(\ltan-2*\loc):(2*\loc-\rtan)] (\x, -0.3*\x * \x);
\draw[green!70!black, semithick]  plot[smooth, domain=(2*\loc-\rtan):2.5] (\x, -0.3*\x * \x);

% left tang to left corner
\draw[blue, thick]  (-\ltan, -0.3*\ltan*\ltan) -- (-\loc, -0.3*\ltan*\ltan + 2*0.3*\ltan*\ltan - 2*0.3*\ltan*\loc);

%right corner to right tang
\draw[blue,thick] (\loc, -0.3*\rtan*\rtan + 2*0.3*\rtan*\rtan - 2*0.3*\rtan*\loc) -- (\rtan, -0.3*\rtan*\rtan);

% left corner to middle left tang
\draw[blue,thick] (-\loc, -0.3*\ltan*\ltan + 2*0.3*\ltan*\ltan - 2*0.3*\ltan*\loc) -- (${(\ltan-2*\loc)}*(1, 0) - {0.3*(\ltan-2*\loc)^2}*(0, 1)$);

%right corner to middle right tang
\draw[blue,thick] (\loc, -0.3*\rtan*\rtan + 2*0.3*\rtan*\rtan - 2*0.3*\rtan*\loc) -- (${(2*\loc-\rtan)}*(1, 0) - {0.3*(2*\loc-\rtan)^2}*(0, 1)$);

%left corner
\node[circle, fill, blue, inner sep = 1pt] at (-\loc, -0.3*\ltan*\ltan + 2*0.3*\ltan*\ltan - 2*0.3*\ltan*\loc) {};
\node[anchor = south, scale=0.8] at (-\loc, -0.3*\ltan*\ltan + 2*0.3*\ltan*\ltan - 2*0.3*\ltan*\loc+0.1) {$(-\theta^{1/2}, a\theta)$};

%right corner
\node[circle, fill, blue, inner sep = 1pt] at (\loc, -0.3*\rtan*\rtan + 2*0.3*\rtan*\rtan - 2*0.3*\rtan*\loc) {};
\node[anchor = south, scale=0.8] at (\loc, -0.3*\rtan*\rtan + 2*0.3*\rtan*\rtan - 2*0.3*\rtan*\loc+0.1) {$(\theta^{1/2}, b\theta)$};

%left tangent
\node[circle, fill, blue, inner sep = 1pt] at (-\ltan, -0.3*\ltan*\ltan) {};
%\node[anchor = east, scale=0.9] at (-\ltan-0.1, -0.3*\ltan*\ltan) {$(-x^{\mrm{tan}}_\ell, -(x^{\mrm{tan}}_\ell)^2)$};

%right tangent
\node[circle, fill, blue, inner sep = 1pt] at (\rtan, -0.3*\rtan*\rtan) {};
%\node[anchor = west, scale=0.9] at (\rtan+0.1, -0.3*\rtan*\rtan) {$(x^{\mrm{tan}}_r, -(x^{\mrm{tan}}_r)^2)$};

\end{scope}

% RIGHT panel

\begin{scope}[shift={(13,0)}]
\renewcommand{\ltan}{3.2}
\renewcommand{\rtan}{1.4}
\renewcommand{\loc}{0.7}

%left corner
\node[circle, fill, blue, inner sep = 1pt] at (-\loc, -0.3*\ltan*\ltan + 2*0.3*\ltan*\ltan - 2*0.3*\ltan*\loc) {};
\node[anchor = east, scale=0.8] at (-\loc-0.05, -0.3*\ltan*\ltan + 2*0.3*\ltan*\ltan - 2*0.3*\ltan*\loc) {$(-\theta^{1/2}, a\theta)$};

\clip (-2.55, -0.3*2.5*2.5) rectangle (2.5, -0.3*\ltan*\ltan + 2*0.3*\ltan*\ltan - 2*0.3*\ltan*\loc+0.2);

\draw[green!70!black, semithick]  plot[smooth, domain=-2.5:2.5] (\x, -0.3*\x * \x);

% left tang to left corner
\draw[blue, thick]  (-\ltan, -0.3*\ltan*\ltan) -- (-\loc, -0.3*\ltan*\ltan + 2*0.3*\ltan*\ltan - 2*0.3*\ltan*\loc);

%right corner to right tang
% \draw[blue,thick] (\loc, -0.3*\rtan*\rtan + 2*0.3*\rtan*\rtan - 2*0.3*\rtan*\loc) -- (\rtan, -0.3*\rtan*\rtan);

% left corner to middle left tang
\draw[blue,thick] (-\loc, -0.3*\ltan*\ltan + 2*0.3*\ltan*\ltan - 2*0.3*\ltan*\loc) -- (${(\ltan-2*\loc)}*(1, 0) - {0.3*(\ltan-2*\loc)^2}*(0, 1)$);

% %right corner to middle right tang
% \draw[blue,thick] (\loc, -0.3*\rtan*\rtan + 2*0.3*\rtan*\rtan - 2*0.3*\rtan*\loc) -- (${(2*\loc-\rtan)}*(1, 0) - {0.3*(2*\loc-\rtan)^2}*(0, 1)$);

%right corner
\node[circle, fill, blue, inner sep = 1pt] at (\loc, -0.3*\rtan*\rtan + 2*0.3*\rtan*\rtan - 2*0.3*\rtan*\loc) {};
\node[anchor = south west, scale=0.8] at (\loc+0.05, -0.3*\rtan*\rtan + 2*0.3*\rtan*\rtan - 2*0.3*\rtan*\loc) {$(\theta^{1/2}, b\theta)$};

%left tangent
\node[circle, fill, blue, inner sep = 1pt] at (-\ltan, -0.3*\ltan*\ltan) {};
%\node[anchor = east, scale=0.9] at (-\ltan-0.1, -0.3*\ltan*\ltan) {$(-x^{\mrm{tan}}_\ell, -(x^{\mrm{tan}}_\ell)^2)$};

%right tangent
%\node[circle, fill, blue, inner sep = 1pt] at (\rtan, -0.3*\rtan*\rtan) {};
%\node[anchor = west, scale=0.9] at (\rtan+0.1, -0.3*\rtan*\rtan) {$(x^{\mrm{tan}}_r, -(x^{\mrm{tan}}_r)^2)$};

\end{scope}
\end{tikzpicture}
\caption{The three cases of Theorem~\ref{mt.two point tail}: from left to right, $\linhull$ (in blue) has two, infinitely many, and one extreme point inside $\smash{[-\theta^{1/2},\theta^{1/2}]}$. (The distance of $\pm\theta^{1/2}$ from the center of the parabola at 0 have been made to differ in the three figures just to visually better emphasize the geometric features of the three cases.)}
\label{f.three cases of two-point}
\end{figure}
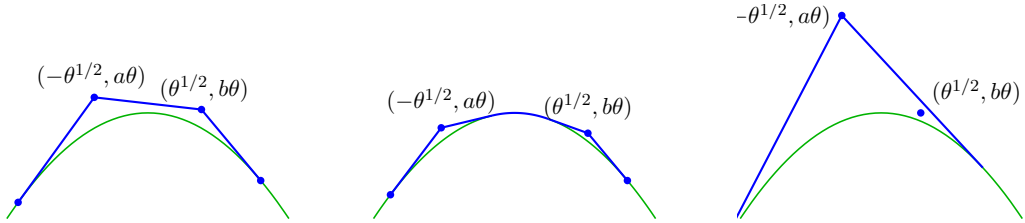

\begin{maintheorem}[Two-point upper tail bounds]\label{mt.two point tail}
There exist constants $\theta_0$ and $a_0=b_0$ such that the following holds. If (i) $\theta > \theta_0$ and $a\geq b>-1$ or (ii) $\theta>0$ and $a\geq a_0$, $b\geq b_0$, $a\geq b$, then, if $\linhull$ has two extreme points inside $[-\theta^{1/2}, \theta^{1/2}]$ (see Figure~\ref{f.three cases of two-point}),
\begin{align*}
\MoveEqLeft[4]
\P\Bigl(\cL_1(-\theta^{1/2}) \geq a \theta, \cL_1(\theta^{1/2}) \geq b \theta\Bigr)\\
&= \exp\left(-\frac{\theta^{3/2}}{24}\left[3(a-b)^2 + 24(a+b) + 16\left((1+a)^{3/2} + (1+b)^{3/2}\right) + 32\right] +\mathrm{error}\right);
\end{align*}
while if $\linhull$ has infinitely many extreme points inside $[-\theta^{1/2}, \theta^{1/2}]$,
\begin{align*}
\P\Bigl(\cL_1(-\theta^{1/2}) \geq a \theta, \cL_1(\theta^{1/2}) \geq b \theta\Bigr)
&= \exp\left(-\frac{4}{3}\theta^{3/2}\left[(1+a)^{3/2} + (1+b)^{3/2}\right] +\mathrm{error}\right);
\end{align*}
and finally if $\linhull$ has one extreme point inside $[-\theta^{1/2}, \theta^{1/2}]$,
\begin{align*}
\P\Bigl(\cL_1(-\theta^{1/2}) \geq a \theta, \cL_1(\theta^{1/2}) \geq b \theta\Bigr)
&= \exp\left(-\frac{4}{3}\theta^{3/2}(1+a)^{3/2} +\mathrm{error}\right).
\end{align*}
Further, the error terms have explicit bounds.
\end{maintheorem}

The mentioned explicit error bound is written out in the more technical version of the theorem stated in Section~\ref{s.two point asymptotics} as Theorem~\ref{t.two point asymptotics}.

It will be clear from the geometric picture developed in the course of the proof of Theorem~\ref{mt.two point tail} that a procedure for obtaining $k$-point asymptotics is also available. See Remark~\ref{r.k-point}.

Note that the result is not stated for two-point densities. In fact, the expression in the exponent for the density of $(\cL_1(-\theta^{1/2}), \cL_1(\theta^{1/2}))$ at $(a\theta, b\theta)$ will be different in the final case of Theorem~\ref{mt.two point tail}; it will be the same expression as the first case. Our methods should also yield density bounds and we indicate this briefly in Remark~\ref{r.density bounds}, though we do not pursue this.

We next point out that the last two cases of the theorem have nice geometric interpretations, which we explain now.

In the final case where there is a single extreme point in the interval, as is apparent by looking at Figure~\ref{f.three cases of two-point} and our earlier remark that $\linhull$ is the shape of $\cL_1$ on the two-point tail event (see Theorem~\ref{mt.two-point limit shape} ahead), the event is essentially caused by $\smash{\cL_1(-\theta^{1/2}) \geq a\theta}$ alone; thus the probability bound is the same as that of the latter event from Theorem~\ref{mt.one point tail asymptotics}.

The second case's geometric interpretation is more interesting. Essentially, the interaction or dependence of the events $\{\cL_1(-\theta^{1/2}) \geq a\theta\}$ and $\{\cL_1(\theta^{1/2})\geq b\theta\}$ is through the line connecting the points $(-\smash{\theta^{1/2}},a\theta)$ and $(\smash{\theta^{1/2}}, b\theta)$. In the second case where $\linhull$ has infinitely many extreme points in $\smash{[-\theta^{1/2},\theta^{1/2}]}$, the parabola can be thought of as being a \emph{barrier} to this line, thus preventing the interaction of the two events, and so the probability is the product of the individual probabilities (at least up to first order in the exponent).

This is closely related to the positive association of these line ensembles from Assumption~\ref{as.corr}(a) (also known as the Fortuin-Kastelyn-Ginibre (FKG) inequality  \cite{fortuin1971correlation}, and we will refer to it by this name) and it is of interest to understand when it is approximately sharp. 
As already alluded to above, the second case of Theorem~\ref{mt.two point tail} provides one condition and interpretation for approximate sharpness. As can be seen from Figure~\ref{f.three cases of two-point}, this condition is the same as when the line joining $(-\theta^{1/2}, a\theta)$ and $(\theta^{1/2}, b\theta)$ intersects the parabola at two points inside $\smash{[-\theta^{1/2}, \theta^{1/2}]}$ (as then $\linhull$ will not be piecewise-linear inside $\smash{[-\theta^{1/2}, \theta^{1/2}]}$, and hence have infinitely many extreme points there). Our next theorem extends this observation by also addressing the case where this line is \emph{tangent} to the parabola. Thus we obtain a criterion for the sharpness of the FKG inequality.

\begin{maintheorem}[Sharpness of the FKG inequality]\label{mt.fkg sharpness}

Let $a, b> -1$. Suppose the line connecting $(-\theta^{1/2}, a\theta)$ and $(\theta^{1/2}, b\theta)$ intersects the parabola $-x^2$ inside $[-\theta^{1/2}, \theta^{1/2}]$. Then, as $\theta\to\infty$ or as $a,b\to\infty$,
$$\P\left(\cL_1(-\theta^{1/2}) \geq a\theta, \cL_1(\theta^{1/2}) \geq b\theta\right) = \exp\left(-\frac{4}{3}\theta^{3/2}\left[(1+a)^{3/2} + (1+b)^{3/2}\right](1+o(1))\right),$$
i.e., the probability on the LHS is equal to $\P(\cL_1(-\theta^{1/2}) > a\theta)\cdot\P(\cL_1(\theta^{1/2}) > b\theta)$ up to first order in the exponent.
\end{maintheorem}

Theorem~\ref{mt.fkg sharpness} essentially follows by formulating the tangency condition in an algebraic form and applying the first case of Theorem~\ref{mt.two point tail}, where $\linhull$ has two extreme points inside $[-\theta^{1/2}, \theta^{1/2}]$. Simplifying the expression from that case of Theorem~\ref{mt.two point tail} will yield Theorem~\ref{mt.fkg sharpness}.

\subsubsection{Extremal ensembles}

Recall that our examples of interest like the parabolic Airy$_2$ process and KPZ equation embed as the first curve in a line ensemble. While $\cP$ and $\h$ are examples of such line ensembles in the continuum, there are many prelimiting discrete models which also possess an associated line ensemble (often with finitely many curves) that enjoys an explicit Gibbs property, and we anticipate that the ideas underlying our method will also be applicable in such contexts. Now, recall from Assumption~\ref{as.tails} that we a priori require upper and lower bounds on the one-point upper tail of the first curve. These are available for $\cP$ and $\h$, but may not be in prelimiting models of interest. Nevertheless, to showcase the reach of our method, we show that the sharp tail bounds can also be obtained in a context where the input tail bounds are not a priori available.

To introduce the context, recall that the laws of these infinite ensembles are examples from a class of Gibbs measures on the space of infinite collections of continuous curves. 
In such settings it is an important and natural objective to gain an understanding of the structure of the set of Gibbs measures, which is a convex set. Indeed, this has been an important field of research in areas such as tiling and dimer models \cite{sheffield2005random,kenyon2006dimers,aggarwal2019universality}. For such classifications, the \emph{extremal} Gibbs measures play a special role; a Gibbs measure is extremal if it cannot be written as a non-trivial convex combination of two other Gibbs measures. $\cP$ is an extremal ensemble in this sense; see the discussion in Section~\ref{s.notions of extremality}.

The following result regarding the structure of the extremal Gibbs measures first appeared as \cite[Conjecture 3.2]{corwin2014brownian}, following a suggestion of Scott Sheffield, and was very recently established by Aggarwal-Huang \cite{aggarwal2023strong}: 

\begin{theorem}[Corollary 2.12 of \cite{aggarwal2023strong}]\label{conj.extremality}
Let $\cL = (\cL_1, \cL_2, \ldots)$ be a line ensemble such that (i) $\cL(x)+x^2$ is  stationary under deterministic horizontal shifts, (ii) $\cL$ has the Brownian Gibbs property, and (iii) the law of $\cL$ is extremal in the set of such Gibbs measures. Then $\cL$ is the parabolic Airy line ensemble, up to a trivial deterministic vertical shift of the entire ensemble.
\end{theorem}

We will call an ensemble which satisfies the hypotheses of the conjecture an \emph{extremal stationary ensemble}. 

Note that despite the hypotheses on the ensemble in Theorem~\ref{conj.extremality} being purely qualitative---in particular, integrable inputs are not available---the conclusion is essentially that the ensemble must be the parabolic Airy line ensemble, which has very precise tail decay and integrable structure. 
Thus, while the next result follows as an immediate consequence of Theorem~\ref{conj.extremality}, we include its proof as it may be applicable to other models where a priori tail bounds are not known. We also mention that our arguments may be regarded as softer in comparison to those of \cite{aggarwal2023strong}.

\begin{maintheorem}[Tail asymptotics for extremal stationary ensembles]\label{mt.extremal}
Extremal stationary line ensembles satisfy Assumptions~\ref{as.bg}--\ref{as.tails}.

In particular, the probability bounds of Theorems~\ref{mt.one point density asymptotics}, \ref{mt.one point tail asymptotics}, \ref{mt.two point tail}, and \ref{mt.fkg sharpness} (as well as Theorems~\ref{mt.one point limit shape}--\ref{mt.two-point limit shape} ahead) all hold for the top curve of an extremal stationary ensemble and, for instance, if $\cL$ is such an ensemble,
\begin{align*}
\frac{1}{\dif\theta}\P\Bigl(\cL_1(0)\in [\theta,\theta+\dif\theta]\Bigr) = \exp\left(-\frac{4}{3}\theta^{3/2}(1+o(1))\right).
\end{align*}
(Note that explicit and implicit constants such as $\theta_0$ or $C$ in Theorems~\ref{mt.one point density asymptotics}--\ref{mt.fkg sharpness} and $o(1)$ above will depend on the extremal stationary ensemble measure in question.)

Thus, since $\cP$ is extremal stationary, Theorems~\ref{mt.one point density asymptotics}--\ref{mt.fkg sharpness} and \ref{mt.one point limit shape}--\ref{mt.two-point limit shape} hold unconditionally for $\cP$.
\end{maintheorem}

In essence, the main task in proving Theorem~\ref{mt.extremal} is establishing that Assumption~\ref{as.tails} can be derived from Assumptions~\ref{as.bg}--\ref{as.mono in cond}; proving that extremal ensembles satisfy the latter three assumptions is more straightforward and done in Appendix~\ref{app.monotonicity proofs}.

In fact, one does not need to verify that $\cP$ is an extremal stationary ensemble to conclude that Theorems~\ref{mt.one point density asymptotics}--\ref{mt.fkg sharpness} and \ref{mt.one point limit shape}--\ref{mt.two-point limit shape} apply to it unconditionally; one can verify that Assumptions~\ref{as.bg}--\ref{as.tails} hold for it directly, as we also do in Appendix~\ref{app.monotonicity proofs}.

\subsection{Main results on limit shapes}\label{s.intro.results.shapes}
We now move on to more geometric statements, i.e., about the form of $\cL_1$ on the respective upper tail events, and which, as it turns out, serve as crucial ingredients in the proofs of the results on asymptotics of upper tail event probabilities stated in the previous subsection.

\subsubsection{One-point limit shape}

Define $\tri:[-\theta^{1/2}, \theta^{1/2}]\to \R$ ($\tri$ is short for triangle) by
$$\tri(x) = -2\theta^{1/2}|x|+\theta.$$
This is the function obtained by considering the two tangents to the parabola $-x^2$ which pass through $(0,\theta)$; see Figure~\ref{f.one-point limit shape}. It is an easy computation that the tangents touch the parabola at $(\pm \theta^{1/2}, -\theta)$, which gives $\tri$.

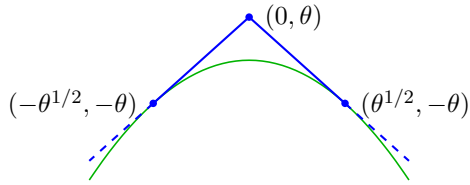
\begin{figure}[h!]
\begin{tikzpicture}[scale=0.85]
\draw[green!70!black, semithick]  plot[smooth, domain=-2.5:2.5] (\x, -0.3*\x * \x);

\draw[blue, thick]  (-1.5, -0.3*1.5*1.5) -- (0, -0.3*1.5*1.5 + 2*0.3*1.5*1.5) -- (1.5, -0.3*1.5*1.5);

\draw[blue, thick, dashed] (-1.5, -0.3*1.5*1.5) -- ++(-1, -2*0.3*1.5*1);
\draw[blue, thick, dashed] (1.5, -0.3*1.5*1.5) -- ++(1, -2*0.3*1.5*1);

\node[circle, fill, blue, inner sep = 1pt] at (0, -0.3*1.5*1.5 + 2*0.3*1.5*1.5) {};
\node[anchor = west, scale=0.9] at (0.1, -0.3*1.5*1.5 + 2*0.3*1.5*1.5) {$(0, \theta)$};

\node[circle, fill, blue, inner sep = 1pt] at (1.5, -0.3*1.5*1.5) {};
\node[anchor = west, scale=0.9] at (1.6, -0.3*1.5*1.5) {$(\theta^{1/2}, -\theta)$};

\node[circle, fill, blue, inner sep = 1pt] at (-1.5, -0.3*1.5*1.5) {};
\node[anchor = east, scale=0.9] at (-1.6, -0.3*1.5*1.5) {$(-\theta^{1/2}, -\theta)$};

\end{tikzpicture}
\caption{The blue solid curve is $\tri:[-\theta^{1/2}, \theta^{1/2}]\to\R$.}\label{f.one-point limit shape}
\end{figure}

\begin{maintheorem}[One-point limit shape]\label{mt.one point limit shape}
There exist $C<\infty$, $c>0$, and $\theta_0$ such that, for $\theta>\theta_0$ and $C< M < C^{-1}\theta^{3/4}$, 
\begin{align*}
\P\left(\sup_{|x|\leq\theta^{1/2}} \left(\cL_1(x) - \tri(x)\right) \geq M\theta^{1/4} \ \Big|\  \cL_1(0) = \theta\right) &\leq \exp(-cM^{2}) \qquad\text{and}\\
\P\left(\inf_{|x| \leq \theta^{1/2}}\left(\cL_1(x) - \tri(x)\right) \leq -M\theta^{1/4} \ \Big|\  \cL_1(0) = \theta\right) &\leq \exp(-cM^{2}) + 4\cdot\P(\cL_1(0)\leq -\tfrac{1}{2}M\theta^{1/4}).
\end{align*}
\end{maintheorem}

The second bound has a term involving the one-point lower tail of $\cL_1$ as we have made no assumption on such a quantity in Section~\ref{s.assumptions}. For examples like the parabolic Airy$_2$ process and the KPZ equation, lower tail bounds are available \cite{tracy1994level,corwin2020lower} which show that $\exp(-cM^2)$ is the dominant term.

While it may seem slightly odd that we impose an upper bound on $M$, this is sufficient for our purposes. Moreover, for all $M$ large enough, both probabilities should behave as $\exp(-c(M\theta^{1/4})^{3/2})$, and it is easy to check that the transition from the Gaussian tail to the former happens when $M$ is of order $O(\theta^{3/4})$. We will in fact impose similar upper bounds on $M$ in many places in the paper so as to ensure that we work only with Gaussian bounds.

The scale of fluctuations, $\theta^{1/4}$, is sharp, at least in the bulk of $[-\theta^{1/2}, 0]$ and $[0, \theta^{1/2}]$ owing to the diffusive nature of Brownian fluctuations. However, the fluctuations are expected to decrease significantly at $\pm\theta^{1/2}$ and beyond, where they should be $O(1)$. While more work using  bootstrapping ideas relying on the asymptotics from Theorem~\ref{mt.one point tail asymptotics} should yield some improvement, we don't pursue them to keep things somewhat less technical.

A natural question is what the profile looks like beyond $[-\theta^{1/2}, \theta^{1/2}]$. From the heuristic described later in Section \ref{s.intro.proof ideas}  one expects it to be close to the parabola $-x^2$, and this is indeed the case; see Proposition~\ref{p.para fluctuation h_t} for details.

Our last result concerns the limit shape under two-point upper tail events.

\subsubsection{Two-point limit shape}

Recall the definition and illustration of $\linhull$ from Section~\ref{s.intro.two-point asymptotics} and Figure~\ref{f.three cases of two-point}. 
Define $\tent:\R\to\R$ by the following. If $\linhull$ has two or infinitely many extreme points in $[-\theta^{1/2}, \theta^{1/2}]$, $\tent=\linhull$. If $\linhull$ has only one extreme point in $[-\theta^{1/2}, \theta^{1/2}]$, define $\tent$ to be the function which equals $\linhull$ on $(-\infty, -\theta^{1/2}]\cup [\theta^{1/2}, \infty)$ and which equals the line connecting $(-\theta^{1/2}, a\theta)$ and $(\theta^{1/2}, b\theta)$ on $[-\theta^{1/2}, \theta^{1/2}]$. Note that in the case that $\mrm{Tent}_{a,b}$ is convex, it equals $\linhull$. 

Also define $\Ilin$ to be the biggest closed set on which $\tent$ is piecewise linear; since this set is the same even if we replace $\tent$ by $\linhull$, looking at Figure~\ref{f.three cases of two-point} shows that $\Ilin$ is either one or the union of two closed intervals. Our shape theorem will be restricted to $\Ilin$. While it is possible to replace $\Ilin$ by the entire domain of $\tent$, it is technically slightly more cumbersome and not needed to obtain the two-point estimates, and so we do not do it. See Remark~\ref{r.why not two-point limit shape on whole domain} for a more detailed discussion.

\begin{restatable}[Two-point limit shape]{maintheorem}{twopointlimit}\label{mt.two-point limit shape}
Let $\theta > 0$ and $a \geq b > -1$. For $M>0$, let $M_{a,b} = \smash{M[(1+a)^{1/4}+(1+b)^{1/4}]}$. There exist $c>0$, $C<\infty$, $\theta_0$, and $a_0=b_0$ such that, if $\theta > \theta_0$ or $a,b\geq a_0, b_0$, and for $C<M\leq C^{-1}[(1+a)^{3/4}+(1+b)^{3/4}]\theta^{3/4}$,
\begin{align*}
\P\left(\sup_{x \in \Ilin} \left(\cL_1(x) - \tent(x)\right) \geq M_{a,b}\theta^{1/4} \ \Big|\  \cL_1(-\theta^{1/2}) = a\theta, \cL_1(\theta^{1/2}) = b\theta \right)
\leq \exp(-cM^2)
\end{align*}
and
\begin{align*}
\MoveEqLeft[20]
\P\left(\inf_{x \in \Ilin} \left(\cL_1(x) - \tent(x)\right) \leq -M_{a,b}\theta^{1/4} \ \Big|\  \cL_1(-\theta^{1/2}) = a\theta, \cL_1(\theta^{1/2}) = b\theta \right)\\
&\leq \exp(-cM^2) + 8\cdot\P\left(\cL_1(0)\leq -\tfrac{1}{2}M_{a,b}\theta^{1/4}\right).
\end{align*}
\end{restatable}

This concludes our main results on limit shapes and tail asymptotics of the top curve of line ensembles. In the next two sections we give some brief background on the KPZ equation and the parabolic Airy$_2$ process and previous results and predictions on their upper tail behavior. Then in Section~\ref{s.weaker bk} we present a weakened form of Assumption~\ref{as.corr}(b) (the BK inequality) which will suffice in place of Assumption~\ref{as.corr}(b) to derive all the previous results, and which can be verified for the KPZ line ensemble.

\subsection{Background on the KPZ equation}\label{s.kpz background}

The KPZ equation is the stochastic PDE given by
\begin{equation}\label{e.KPZ}
\partial_t \mc{H} = \frac{1}{4}\partial_x^2 \mc{H} + \frac{1}{4}(\partial_x \mc{H})^2 + \xi,
\end{equation}
where $\mc H:(0,\infty)\times\R\to\R$ and $\xi:(0,\infty)\times\R\to\R$ is a space-time white noise. 
(The coefficients chosen in the equation above are slightly different from the usual ones; all choices are equivalent up to a scaling, and above choice has certain convenient features which we will point out shortly.)

The solution theory for \eqref{e.KPZ} is highly non-trivial because of the combination of the non-linear term $(\partial_x \mc H)^2$ and the space-time white noise $\xi$: the latter suggests that $\mc H$ will be rough, i.e., differentiable only in the sense of distributions, which renders $(\partial_x \mc H)^2$ ill-defined. There has been much recent work on handling these issues, including regularity structures \cite{hairer2013solving}, paracontrolled distributions \cite{gubinelli2015paracontrolled}, and energy methods \cite{gonccalves2014nonlinear,gubinelli2017kpz}. 
All these solution theories agree with a physically relevant notion known as the Cole-Hopf solution, which is the one most used in studies of probabilistic properties of \eqref{e.KPZ}. To describe it we need the multiplicative stochastic heat equation (SHE), given by
$$\partial_t \mc Z = \frac{1}{4}\partial_x^2 \mc Z + \xi\mc Z,$$
where $\mc Z:(0,\infty)\times\R\to\R$ and $\xi:(0,\infty)\times\R\to\R$ is a space-time white noise. The solution theory of the SHE is more straightforward as it is a linear SPDE. The Cole-Hopf solution is \emph{defined} by
\begin{align*}
\mc H(t,x) = \log \mc Z(t,x)
\end{align*}
for all $t>0$ and $x\in\R$. It can be checked via a purely formal change of variables computation that the SHE becomes the KPZ equation under this substitution.

That $\mc H$ is well-defined by the above relation requires that, almost surely, $\mc Z(t,x) > 0$ for all $t>0$ and $x\in\R$ simultaneously; this was established for the Dirac delta initial condition in \cite{flores2014strict} (see also \cite{mueller1991support,bertini1995stochastic} for earlier related work).
The case of the initial condition for the SHE being the Dirac mass at the origin and defining $\mc H = \log \mc Z$ as above is called the \emph{narrow-wedge} solution to the KPZ equation. We discuss it more next.

Recall the definition \eqref{e.h_1 definition} of the scaled KPZ equation, i.e., the first curve $\h_1$ of the KPZ line ensemble $\h$. Here we briefly explain the scaling. The spatial scaling of $t^{2/3}$ scaling of fluctuations around $-t/12$ by $t^{1/3}$ reflects the well-known KPZ 1:2:3 scaling relation; indeed the family $\{\h\}_{t>t_0}$ is tight for any $t_0>0$ \cite{amir2011probability,quastel2022convergence,virag2020heat}.

In view of Assumption~\ref{as.bg}, $\h_1(x) + x^2$ is a stationary process in $x$ for every $t>0$ (proved in \cite{amir2011probability}, see also \cite{nica2021intermediate}). The choice of coefficients of the KPZ equation \eqref{e.KPZ} is to have the coefficient of $x^2$ here be 1. The same stationarity property holds true for the parabolic Airy$_2$ process $\cP_1$.

The narrow-wedge solution to the KPZ equation as well as the parabolic Airy$_2$ process have been extensively studied in the literature, often making use of their exactly solvable or integrable structure, i.e., that there are explicit exact formulas for quantities such as finite dimensional distributions and Laplace transforms. We discuss some of the work on upper tail behaviour next.

\subsection{Upper tail behaviour predictions and previous work}
Recall that $\cP_1(x)+x^2$ has the same distribution for each $x\in\R$ by stationarity. It is well-known that this common distribution is the GUE Tracy-Widom distribution, first discovered in random matrix theory as the scaling limit of the largest eigenvalue of the Gaussian Unitary Ensemble \cite{tracy1994level}. The upper tail asymptotics of the GUE Tracy-Widom are also known to be
\begin{equation}\label{e.airy asymptotic}
\P\left(\cP_1(0) > \theta\right) = \exp\left(-\frac{4}{3}\theta^{3/2}(1+o(1))\right)  \text{ as } \theta\to\infty;
\end{equation}
see \cite{ramirez2011beta} for an explicit statement, but this was widely known much earlier and to greater precision using the determinantal structure of the Airy point process (of which $\cP_1(0)$ is the largest point), whose kernel can be written in terms of the Airy function $\mrm{Ai}$. Indeed, the above asymptotic essentially follows from the classical fact that $\mrm{Ai}(\theta) = \exp(-\frac{2}{3}\theta^{3/2}(1+o(1)))$ as $\theta\to\infty$.

It has been known since 2011 that $\h_1(0)$ converges in distribution to $\cP_1(0)$ as $t\to\infty$ \cite{amir2011probability} (some non-rigorous derivations from the mathematical physics \cite{sasamoto2010crossover,sasamoto2010exact,spohn2010one} and physics \cite{dotsenko2010bethe,calabrese2010free} communities were also given at around the same time; see \cite{corwin2012kardar} for a discussion of the history). This has recently been strengthened to process-level convergence independently in \cite{quastel2022convergence,virag2020heat}.

We emphasize that while the asymptotic \eqref{e.airy asymptotic} is not difficult to derive using the determinantal structure of $\cP$, the same structure does not quite hold for the KPZ equation.

\subsubsection{Previous work on upper tail asymptotics}

There has been a substantial amount of work on one-point upper tail asymptotics for $\h_1$ which we briefly summarize here. A number of works \cite{conus2013chaotic,chen2015moments,khoshnevisan2017intermittency} have studied upper tails of the SHE directly with various types of non-linearities; all the bounds are non-uniform in $t$, however.

\cite{corwin2013crossover} obtains a uniform-in-$t$ upper bound, but does not show that the expected $\smash{\exp(-c\theta^{3/2})}$ behaviour holds in the large deviation regime of $\theta$. 

The first uniform-in-$t$ upper tail estimate capturing the correct behaviour on the level of the $\frac{3}{2}$ tail exponent in the entire upper tail is \cite{corwin2018kpz}, who obtain estimates of the form
\begin{equation}\label{e.cg bounds}
\exp(-c_1\theta^{3/2}) \leq \P(\h_1(0) > \theta) \leq \exp(-c_2\theta^{3/2})
\end{equation}
for $t>t_0$, $\theta>\theta_0$, and some constants $c_1, c_2>0$, where $\theta_0$, $c_1$, $c_2$ depend only on $t_0>0$. They also obtain bounds on $c_1$ and $c_2$ depending on the regime of $\theta$ in question: for instance, for any $\varepsilon>0$, $c_1=\smash{\frac{4}{3}}(1+\varepsilon)$ and $c_2 = \smash{\frac{4}{3}}(1-\varepsilon)$ can be taken when $\theta < O(\varepsilon^2 t^{2/3})$, i.e., $\theta$ up till the start of the large deviation regime. Actually, they obtain that $c_2$ (though not $c_1$) can be taken similarly nearly optimal far in the large deviation regime as well, but the values of both $c_1$ and $c_2$ deteriorate in all other ranges of $\theta$.
\cite{corwin2018kpz} also obtain upper tail estimates for general initial data which capture the $\frac{3}{2}$ tail exponent uniformly in $t$ (though they do not obtain the expected coefficient of $\frac{4}{3}$ in the exponent).

Apart from this, sharp behavior has been obtained in the large deviation, $t\to\infty$ limit, with results of the form $\lim_{t\to\infty} t^{-1}\log \P(\h_1(0)> yt^{2/3}) = -\frac{4}{3}y^{3/2}$ ($y>0$) \cite{das2021fractional} and similar results for general initial data \cite{ghosal2020lyapunov}. Note however that these types of results do not yield finite-$t$ tail bounds like~\eqref{e.cg bounds}.

Our results establish these sharp finite tail coefficients for the KPZ line ensemble as well as what the profile of $\h_1$ looks like at finite $t$ when conditioned on the upper tail events,  in the full tail regime. This is recorded ahead in Theorem~\ref{t.assumptions hold}. The one-point upper tail has been investigated in the $t\to 0$ limit in \cite{lamarre2021kpz}, verifying predictions from the physics literature \cite{kamenev2016short}; in this regime of $t$ the KPZ equation lies in the Gaussian or Edwards-Wilkinson universality class, and does not exhibit the non-linear behaviour characteristic of the KPZ class. The physics literature does not seem to predict the behavior for $t>0$, i.e., when the KPZ equation lies in the KPZ class. Similarly, the probability literature does not seem to have studied the zero temperature case (i.e., of $\cP_1$) of the profile's shape conditioned on the one-point or multi-point upper tail. 

However, there have been some recent studies in zero-temperature prelimiting models (such as last passage percolation and TASEP) investigating such questions about limit shapes and related geometric observables in the large deviation regime \cite{olla2019exceedingly,quastel2021hydrodynamic,basu2019connecting,basu2019delocalization}.

We also remark that while our focus has been on upper tail asymptotics and limit shapes conditioned on upper tail events, there has been a considerable amount of fruitful research done in a number of other directions regarding the KPZ equation. These include integrable formulas \cite{borodin2016moments,ghosal2018moments}, lower tail estimates and large deviations \cite{corwin2020lower,tsai2018exact,charlier2021uniform,corwin2018kpz,cafasso2022riemann}, studies of correlation and fractal structure of $t\mapsto \h(0)$ \cite{corwin2021kpz,das2021law}, connections to the Kadomtsev–Petviashvili equation \cite{quastel2019kp,le2020large}, and investigations of Brownian regularity \cite{wu2021tightness,wu2021brownian}, to give a taste.

\subsection{A weaker form of Assumption~\ref{as.corr}(b)} \label{s.weaker bk}

As mentioned above in Section~\ref{s.intro}, Assumption~\ref{as.corr}(b) (the BK inequality) holds in the zero temperature case, including the parabolic Airy line ensemble and prelimiting models such as the line ensemble associated to exponential last passage percolation, but its validity is not clear and seems delicate to establish in the positive temperature case of the KPZ line ensemble. %

The exact form of Assumption~\ref{as.corr}(b) in fact seems too strong a statement to hold for the KPZ line ensemble, due to certain features of the object which seem to ultimately arise from entropy considerations underlying the polymer model whose free energies the KPZ line ensemble encodes. This is discussed in more detail in Section~\ref{s.bk difficulty for kpz}. %

Our arguments do not actually need the precise form of the BK inequality in Assumption~\ref{as.corr}(b), and here we formulate a weaker version which will suffice for our purposes, and which can be obtained using the results from \cite{ganguly2025van} (see Section~\ref{s.weaker bk.validity of weak bk}). This weaker form will in fact be the assumption used in the rest of the paper and suffices, as captured ahead in Proposition~\ref{p.weak bk suffices}. It is the following:

\begin{enumerate}[start=2]
	\item (b\ensuremath{'}) There exist $C$ and $\theta_0$ such that, for any $x_0\in\R$, $\theta>\theta_0$, $M>C(\theta+x_0^2)^{3/4}$, and interval $I$ satisfying $\log |I| \leq (\log M)^{2}$, almost surely, \label{as.weak bk}
	\begin{align}\label{e.weak bk}
	\P\left(\sup_{x\in I} \left(\cL_2(x)+x^2\right) > (\log M)^C \midd \cL_1(x_0) \geq \theta \right) \leq \frac{1}{2}.
	\end{align}
	
\end{enumerate}

In fact, as the proofs will indicate, it also suffices if one replaces $\frac{1}{2}$ with any fixed number smaller than $1$. 

To see that this is indeed a weakening of Assumption~\ref{as.corr}(b), note that Assumption~\ref{as.weak bk}(b\ensuremath{'}) has three main restrictions in comparison: first, instead of conditioning on the entirety of $\cL_1$ on an interval, we only condition on its value at a point being lower bounded; second, we only consider the event of the supremum of the second curve being large rather than an arbitrary increasing event; and third, we are content with having a probability bound of $\frac{1}{2}$ rather than comparing to the corresponding unconditional probability for the first curve.

To handle the weakening of the conditioning mentioned in the second point, it will be necessary to have a slightly stronger form of monotonicity in conditioning (Assumption~\ref{as.mono in cond}) which provides some additional uniformity. This is captured in the mildly stronger form of that assumption that we state next, where we allow ourselves to additionally condition on $\cL_1$ being larger than a given function $f$ at some or all points.

\begin{enumerate}
\item[{\crtcrossreflabel{(iii\ensuremath{'})}[as.mono in cond stronger]}] Let $m\in\N$, $x_1, \ldots, x_m\in\R$,  $y^{(i)}_1, \ldots, y^{(i)}_m \in\R$ for $i=1$ and $2$. Let $a<b\in\R$ and $f:[a,b]\to\R\cup\{-\infty\}$ be upper semicontinuous. Suppose $\smash{y_j^{(1)} \geq y_j^{(2)}}$ for $j=1, \ldots, m$. Then the conditional law of $\cL$ given $\cL_1(x_j) = \smash{y^{(1)}_j}$ for $j=1, \ldots, m$ and $\inf_{[a,b]} (\cL_1 -f) \geq 0$ stochastically dominates that of the same given $\cL_1(x_j) = \smash{y^{(2)}_j}$ and $\sup_{[a,b]} (\cL_1 -f) \geq 0$.
\end{enumerate}

We prove Assumption~\ref{as.mono in cond stronger} for the parabolic Airy, extremal, and KPZ line ensembles in Appendix~\ref{app.monotonicity proofs}.

The sufficiency of Assumptions~\ref{as.weak bk}(b\ensuremath{'}) and \ref{as.mono in cond stronger} is recorded next.

\begin{proposition}\label{p.weak bk suffices}
Theorems~\ref{mt.one point density asymptotics}--\ref{mt.fkg sharpness} and \ref{mt.one point limit shape}--\ref{mt.two-point limit shape} all hold on replacing Assumptions~\ref{as.corr}(b) and \ref{as.mono in cond} by Assumptions~\ref{as.weak bk}(b\ensuremath{'}) and \ref{as.mono in cond stronger}.
\end{proposition}

\subsubsection{Validity of Assumption~\ref{as.weak bk}(b\ensuremath{'}) for parabolic Airy and KPZ line ensembles}\label{s.weaker bk.validity of weak bk}

Assumptions~\ref{as.bg}, \ref{as.corr}(b), and \ref{as.tails} imply Assumption~\ref{as.weak bk}(b\ensuremath{'}), which will be presented in the next lemma.

\begin{lemma}\label{l.derive weak bk}
Suppose $\cL$ satisfies Assumptions~\ref{as.bg}, \ref{as.corr}(b), and \ref{as.tails}. Then $\cL$ satisfies Assumption~\ref{as.weak bk}(b\ensuremath{'}).
\end{lemma}

\begin{proof}
First, by Assumption~\ref{as.corr}(b), for any interval $I$,
\begin{align*}
\P\left(\sup_{x\in I} \left(\cL_2(x)+x^2\right) > (\log M)^C \midd \cL_1(0) \geq \theta \right) \leq \P\left(\sup_{x\in I} \left(\cL_1(x)+x^2\right) > (\log M)^C\right).
\end{align*}
It is not hard to show that a consequence of Assumptions~\ref{as.bg} and \ref{as.tails} is that the righthand side with the supremum over $[-1,1]$ instead of $I$ is upper bounded by $\exp(-c(\log M)^{C\beta})$, by, e.g., Proposition~\ref{p.sharp sup over interval tail} ahead (similar statements have been shown in the literature a number of times, e.g., \cite[Proposition 2.27]{hammond2016brownian}).
 Picking $C>\beta^{-1}$ large enough, using the stationarity from Assumption~\ref{as.bg} and doing a union bound over $O(|I|)$ many intervals of size 2 implies \eqref{e.weak bk} (using the bound we have assumed on $|I|$). 
\end{proof}

For the positive temperature KPZ line ensemble, the journey to establish the validity of Assumption~\ref{as.corr}(b) or \ref{as.weak bk}(b\ensuremath{'}) has had some stumbles, through which the subtle role entropy plays in this setting (as opposed to in zero temperature) became increasingly apparent. 
While (as also mentioned earlier) Assumption~\ref{as.corr}(b) is expected to be too strong to hold, in fact, a slightly stronger version of Assumption~\ref{as.weak bk}(b\ensuremath{'}) was conjectured to hold in an earlier arXiv version of this article, though at this point we also expect that to be too strong. However, a formulation of an inequality similar to Assumption~\ref{as.corr}(b) for the KPZ line ensemble established in the recent work \cite{ganguly2025van} allows us to deduce Assumption~\ref{as.weak bk}(b\ensuremath{'}) for it fairly straightforwardly, which we do as a part of the proof of Theorem~\ref{t.assumptions hold}.

\subsubsection{The difficulty in establishing Assumption~\ref{as.corr}(b) in positive temperature}\label{s.bk difficulty for kpz}

As promised, we conclude this section by briefly discussing the issues in establishing the BK inequality for the KPZ line ensemble, compared to the parabolic Airy line ensemble $\cP$ (for which it will be proved in Appendix~\ref{app.monotonicity proofs}). In the latter, the strategy is to show that a form of the inequality holds for a prelimiting model (Dyson Brownian motion), and then take an appropriate scaling limit to arrive at $\cP$. The argument in the prelimit relies on the curves being ordered at at least one point (indeed, the curves are ordered throughout), for instance, the origin, where the curves coincide.

The analogous prelimiting model for the KPZ line ensemble is the O'Connell-Yor diffusion \cite{OY01,o2012directed}, which consists of $N$ interacting random continuous functions on $(0,\infty)$. The issue is that, unlike Dyson Brownian motion, these curves are not guaranteed to be ordered or to coincide even at the origin. Instead, the entrance law can be described as the curves coming up from ``$-\infty$'' (in a certain precise sense), where they are in \emph{reverse} order \cite[discussion following Proposition 8.3]{o2012directed}. The reversal at the origin is a genuine distinction between the zero and positive temperature cases, arising from entropy considerations in the associated polymer model (namely the reduced entropy of a pair of polymers with adjacent starting points that are forced to be disjoint, whose free energy is what the second curve encodes, in comparison to the same without any disjointness condition).

Another manifestation of the distinction coming from entropy is the following. In the case of Dyson Brownian motion (DBM), the argument in the prelimit establishes that the second curve of $N$-DBM, conditional on the first curve of the same, is stochastically dominated by the first curve of $(N-1)$-DBM. Both $N$- and $(N-1)$-DBM converge to $\cP$ under essentially the same centering and scaling, which yields the BK inequality for $\cP$. For the KPZ line ensemble $\h$, it is possible to show an analogous statement for a different prelimiting ensemble, the log-gamma polymer; namely that the second curve in the $N$ ensemble, conditional on the first curve, is stochastically dominated by the first curve of the $(N-1)$ ensemble. However, the centering required to take the $N$-ensemble to $\h$ differs from that required to take the $(N-1)$-ensemble to the same by a term of $\log N$, and there is no scaling. Thus in the large $N$ limit, the stochastic domination becomes trivial and does not lead to the BK inequality for $\h$. In essence, \cite{ganguly2025van} avoids such issues by proving an analogous inequality in the log-gamma model for the case of disjoint polymers with \emph{far away} starting and ending points, for which a limit to the continuum can be taken, at which stage the points can be brought close enough together for the free energy to approximate the KPZ line ensemble.

\subsection{One-point tail asymptotics for general initial data}\label{s.intro.general data}

Our last result on asymptotics concerns fairly straightforward consequences of our upper tail asymptotics results to general initial data for the cases of the KPZ equation and the KPZ fixed point.

Here we exploit the well-known fact that the one-point distribution of the KPZ equation at a given time with general initial data can be expressed via a convolution formula involving the entire spatial process of the narrow-wedge solution at the same time. 
Using this it is possible to use asymptotics for the narrow-wedge solution  to obtain one-point asymptotics for general initial data.

To formulate this result, we need to define the scaled solution to the KPZ equation under general data, as well as a class of initial conditions under which the solution is well-defined and non-trivial. 

Let $\mc H(t,x)$ be the solution to the KPZ equation \eqref{e.KPZ} started from general initial data $\mc H(0,\cdot)$. To respect the KPZ scaling, we will allow the initial data to vary with the value of $t$ being considered, but, for brevity, we will still denote it by $\mc H(0,\cdot)$, omitting the $t$-dependence in the notation.

More precisely, for a family of functions $f^{(t)}:\R\to\R\cup\{-\infty\}$, we set $\mc H(0, \cdot)$ by
\begin{equation}\label{e.form of initial condition}
\mc H(0,y)=t^{1/3} f^{(t)}(t^{-2/3}y) \iff t^{-1/3} \mc H(0,t^{2/3}y) = f^{(t)}(y).
\end{equation}
Let $\mc H(t,x)$ be the solution to the KPZ equation at time $t$ with this $t$-dependent initial data, and define the scaled solution $\hf$ by
\begin{equation}\label{e.hf definition}
\hf(x) =\frac{\mc H(t,t^{2/3}x) + \frac{t}{12} - \frac{2}{3}\log t}{t^{1/3}}.
\end{equation}
This scaling of the solution, as well as the initial condition in \eqref{e.form of initial condition}, is convenient since, for example, if $f^{(t)} = f$ for a given function $f$ for all $t$, then $\hf$ converges in distribution as $t\to\infty$ to the KPZ fixed point at time 1 started from initial condition $f$ by \cite{quastel2022convergence,virag2020heat} (see the following page for a brief discussion of this object).

We next list the conditions we impose on the initial data. Essentially the conditions ensure that the solution does not grow too quickly and there is at least some amount of the domain where it is not too negative; otherwise, one runs into pathological settings such as the solution exploding immediately.
The form of the assumptions are somewhat standard in the literature by now, and, for example, are similar to but slightly weaker than those adopted in \cite{corwin2016kpz}. 

The first thing we require is that $f$ essentially grows at most like $x^2$. This comes from matching the decay of $\h_1(x)$, i.e., $-x^2$, since, heuristically, if $\mc H(0,x)$ grows quadratically in $x$ as $x\to\infty$, then $\mc H(t,\cdot)$ will blow up at some finite $t>0$, while a growth faster than quadratic will lead to immediate blow-up. However, because we vary $\mc H(0,\cdot)$ with $t$, we in fact need to also limit the coefficient of the quadratic growth to ensure blow up does not occur immediately. More precisely, the time of blow-up $T_0$ can be seen to be the smallest $t$ such that $\sup_{y\in\R} (\mc H(0,y) -y^2/t) = \infty$; using the relation \eqref{e.form of initial condition}, the latter condition is equivalent to $\smash{\sup_{y\in\R} (f^{(t)}(y) -y^2)= \infty}$.

The following is the precise form of the initial data conditions.

\begin{definition}\label{d.gen initial condition hypotheses}
For $K$, $L$, $M$ $\delta> 0$ we say that a function $f : \R \to \R \cup \{-\infty\}$ satisfies hypothesis $\Hyp(K, L, M, \delta)$ if
\begin{itemize}
	\item $f(x) \leq x^2 - L|x| + K$ for all $x \in \R$;
	\item $\mrm{Leb}\{x \in [-M, M] : f(x) \geq -K\} \geq \delta$ where $\mrm{Leb}$ denotes Lebesgue measure.
\end{itemize}
\end{definition}

Since the form of our initial condition is \eqref{e.form of initial condition}, $f^{(t)} \in \Hyp(K,L,M,\delta)$ for all $t<T_0$ for some $T_0>0$ will imply that $\hf$ is defined for all $t<T_0$, since the extra linear term in the first bullet point is enough to handle the random fluctuations beyond the parabola.

\begin{maintheorem}[One-point upper tail bounds for general initial data]\label{mt.general initial data}
Let $t_0>0$. Let $T_0 \in (t_0,\infty]$ and $f^{(t)}\in \Hyp(K, L, M, \delta)$ for some fixed $K$, $L$, $M$, $\delta>0$ for all $t\in [t_0, T_0)$. There exist $\theta_0>0$ and $C<\infty$ such that, for $t\in [t_0, T_0)$ and $\theta>\theta_0$,
$$\exp\left(-\frac{4}{3}\theta^{3/2} - C\theta^{3/4}\right)\leq \P\Bigl(\hf(0) \geq \theta\Bigr) \leq \exp\left(-\frac{4}{3}\theta^{3/2} + C\theta^{3/4}\right).$$
\end{maintheorem}

In fact, as the proofs will show, the upper and lower bounds in Theorem~\ref{mt.general initial data} require different parts of the hypotheses in Definition~\ref{d.gen initial condition hypotheses}. For this reason, Theorems~\ref{t.general initial data upper bound} and \ref{t.general initial data lower bound}, appearing later in Section \ref{s.general data}, separate the two bounds and are more precise about which hypotheses are needed for each.

Unlike in the narrow-wedge case, we have not stated a bound on the one-point densities for general initial data. Such a result would require a more complicated argument, though it is possible that a suitable extension of the narrow-wedge argument would yield it; we discuss this more in Remark~\ref{r.general density}.

As we mentioned, Theorem~\ref{mt.general initial data} is proved using a convolution formula which relates $\hf$ with $\h_1$. This formula was also used in \cite{corwin2018kpz} to show uniform-in-$t$ upper tail bounds of the form $\exp(-c\theta^{3/2})$ for general initial data---i.e., the correct tail exponent, but not the sharp coefficient. The lack of sharpness is mainly because the procedure they adopt to go from narrow-wedge solution to the general solution was lossy.
We take a different and somewhat simpler approach which allows us to obtain sharp asymptotics.

Our arguments extend easily to the zero temperature analogue for general initial data mentioned above, the KPZ fixed point $\mf h^{\mrm{FP}, f}(t,\cdot)$ (where $f$ is the initial condition), which we spend a little time introducing now. This is a Markov process in $t$, in a certain function space, constructed in \cite{matetski2016kpz} and given a variational description in \cite{nica2020one} in terms of the directed landscape constructed in \cite{dauvergne2018directed}. It has recently been proven to be the $t\to\infty$ limit of the KPZ equation under general initial data \cite{quastel2022convergence,virag2020heat}; for example, the parabolic Airy$_2$ process is the KPZ fixed point started from narrow-wedge initial data, and so is the limit of the KPZ equation from the same data.

For our purposes, it will be easier to define the KPZ fixed point via the variational description than the mentioned limit; this is because the limits have not been proven for initial conditions that grow quadratically that we also want to include. Though the variational formula uses the directed landscape, we will only be interested in $\mf h^{\mrm{FP}, f}$ for a fixed time and space location, for which the formula can be given in terms of the parabolic Airy$_2$ process: for each fixed $x\in\R$, in distribution,
\begin{align}\label{e.KPZ fixed point variational formula}
\mf h^{\mrm{FP}, f}(1, x) = \sup_{y\in\R} \Bigl(\cP_1(y) + f(y+x)\Bigr).
\end{align}
One can get a formula for $\mf h^{\mrm{FP}, f}(t, x)$ from the above by a rescaling invariance (see \cite[Theorem~4.5]{matetski2016kpz}): $\smash{\mf h^{\mrm{FP}, f}(t, x) \stackrel{d}{=} t^{1/3}\mf h^{\mrm{FP}, f^{\{t\}}}(1, t^{-2/3}x)}$, where $\smash{f^{\{t\}}(x) = t^{-1/3}f(t^{2/3}x)}$.

The following result records the same upper tail bounds for $\mf h^{\mrm{FP}, f}$ as for the KPZ equation under general initial data; a similar upper bound without quantitative error terms was also given in \cite[Proposition~4.7]{matetski2016kpz} using integrable methods.

\begin{maintheorem}[One-point upper tail bounds for KPZ fixed point]\label{mt.fixed point}
Let $K$, $L$, $M$, $\delta>0$ be fixed, and suppose that $f^{\{t\}}\in \Hyp(K, L, M, \delta)$ for some $t>0$. There exist $\theta_0>0$ and $C<\infty$ (depending only on the fixed constants and not $t$) such that, for $\theta>\theta_0$,
$$\exp\left(-\frac{4}{3}\theta^{3/2} - C\theta^{3/4}\right)\leq \P\Bigl(\mf h^{\mrm{FP}, f}(t,0) \geq \theta t^{1/3}\Bigr) \leq \exp\left(-\frac{4}{3}\theta^{3/2} + C\theta^{3/4}\right).$$
\end{maintheorem}

\begin{remark}
We do not explicitly prove Theorem~\ref{mt.fixed point}, but the proof can be done quickly: one just mimics the proof of Theorem~\ref{mt.general initial data} presented in Section~\ref{s.general data} but with the zero-temperature version of the convolution formula mentioned above (i.e., \eqref{e.KPZ fixed point variational formula}) with the parabolic Airy$_2$ process in place of the narrow-wedge KPZ equation solution. The result for $\mf h^{\mrm{FP},f}(0,t)$ for $t\neq 1$ is obtained by the mentioned scaling properties of the KPZ fixed point.
\end{remark}

This completes the statements of our main results.

\subsection{Proof outline}\label{s.intro.proof ideas}
The proofs involve several new ideas. Since the formal implementation of them takes a while, to facilitate readability, we include a reasonably detailed description of the key ideas in this section. The subsections can more or less be read independently.

As indicated earlier, the central tool driving our technique is the Brownian resampling properties we have assumed is enjoyed by line ensemble $\cL$. We start by discussing these properties in more detail.

\subsubsection{The Brownian Gibbs property}
\cite{corwin2016kpz} constructed the KPZ line ensemble  into which $\h$ embeds (one ensemble for each fixed $t>0$). The zero temperature case of $t=\infty$ is known as the parabolic Airy line ensemble and was constructed earlier in \cite{corwin2014brownian}.

We next explain the nature of the \emph{Brownian Gibbs properties} or simply, Gibbs properties, they enjoy. Let us first describe this in the zero-temperature case ($t=\infty$). We fix $k\in\N$ and an interval $[a,b]$, and condition upon everything in the line ensemble outside of the top $k$ curves on $[a,b]$; in other words, on $\{\cL_i(x): 1\leq i\leq k, x\not\in[a,b] \text{ or } i\geq k+1, x\in\R\}$. The Brownian Gibbs property says that the conditional distribution of $\cL_1, \ldots, \cL_k$ on $[a,b]$ is given by $k$ independent Brownian bridges of rate \emph{two}, the $i$\textsuperscript{th} one from $(a, \cL_i(a))$ to $(b, \cL_i(b))$, conditioned on non-intersection (with each other, and the lower boundary curve $\cL_{k+1}$). 
That the Brownian bridges are rate two is just due to a convention established when the parabolic Airy$_2$ process was defined in \cite{prahofer2002PNG}, but must be kept track of to obtain the correct coefficients in tail asymptotics.

In the positive temperature case ($t<\infty$), the conditional distribution after conditioning on the same data can be described similarly in terms of independent rate two Brownian bridges; however here, instead of the hard non-intersection constraint, intersections are energetically penalized in terms of an explicit Radon-Nikodym derivative (which, for brevity, will be termed henceforth as the soft constraint). The precise expression is somewhat technical and will not be particularly illuminating at this stage, so we defer it to Section~\ref{s.gibbs and line ensembles}.

To convey the main ideas, in the remainder of this section we will describe our arguments for a line ensemble $\cL$ which enjoys the zero temperature ($t=\infty$) Gibbs property, as its non-intersection resampling is easier to reason about; we will address the differences encountered when working with the positive temperature Gibbs property after outlining the one-point upper tail argument. 

We will primarily focus on the arguments for the one-point case, i.e., tail asymptotics and the shape of the corresponding conditioned profile of the top curve, and say a few words about obtaining the density estimates, as these already contain most of the new ideas in our work; the two-point analogue of these results serve to illustrate the power of the method but do not introduce any substantially new ingredients.

\subsubsection{The assumptions: monotonicity, positive association, and control on a conditioned second curve.} We start by briefly discussing Assumptions~\ref{as.corr} and \ref{as.mono in cond} on the ensembles from Section~\ref{s.assumptions}. 

Assumption~\ref{as.corr}(a) is that of \emph{positive association} and the FKG inequality for the first curve. That $\cP$, $\h$, and zero temperature extremal ensembles satisfy the FKG inequality is a consequence of the fact that the same holds for certain  Brownian bridge and Brownian motion ensembles under the hard or soft non-intersection constraints (which are prelimiting models for $\cP$ and $\h$). 

Let us say a few words on how to go from the positive association property for the Brownian bridge ensembles to the same for the extremal stationary ensembles. Essentially we apply the Gibbs property to a sequence of increasing domains, e.g., $[-k,k]\times\{1, \ldots, k\}$ (resampling the top $k$ curves on $[-k,k]$) and use that we know the required positive association for the resulting conditional laws. Extremality is equivalent to the corresponding tail $\sigma$-algebras being trivial (see Section~\ref{s.extremal}), and so we obtain the unconditioned statement we desire by taking $k\to\infty$. The extremality assumption is invoked essentially only in carrying out the {above and related correlation and monotonicity-in-conditioning arguments described next}, which are the final properties needed for our arguments. 

A crucial ingredient in the proof of positive association for the Brownian bridge ensemble is a certain monotonicity in conditioning property (a version of Assumption~\ref{as.mono in cond} for more than two points), which is a refinement of a monotonicity property that we describe next.

Recall that the Brownian bridge ensembles (under hard or soft constraints) are defined given some boundary data: an lower and/or upper boundary curve, and the values of the Brownian bridges at the boundaries of the interval. The monotonicity property is simply that the law of these Brownian ensembles stochastically increases if the boundary data increases; for example, if the boundary values are kept the same but the lower curve $f_1$ is increased to $f_2$ (i.e., $f_1(\cdot)\leq f_2(\cdot)$), there is a coupling of the bridge ensembles such that the bridges associated to the lower curve $f_1$ are lower than those associated to $f_2$. A similar statement holds for ordered boundary values. These monotonicity properties were first proven in \cite{corwin2014brownian} and \cite{corwin2016kpz} (see also \cite{dimitrov2021characterization,dimitrov2021characterizationH} for more detailed proofs) and have proven indispensable in studies of line ensembles (e.g., in \cite{hammond2016brownian,hammond2017modulus,brownianLPPtransversal,hammond2017patchwork,wu2021brownian,calvert2019brownian,dauvergne2018basic,dauvergne2018directed}); similar monotonicity properties in other models with nice resampling properties have likewise proved important (e.g., \cite{aggarwal2020arctic}).

Next, Assumption~\ref{as.corr}(b) is the (strong) BK inequality, which, recall, holds for the zero temperature cases but is unclear in positive temperature. A heuristic to understand why it should hold in zero temperature is that, under the Brownian Gibbs property, $\cL_1$ is essentially an upper boundary for $\cL_2$. By monotonicity properties of the Gibbs property, as the upper boundary increases (e.g., if $\theta$ gets larger in the one-point upper tail conditioning event), $\cL_2$ gets stochastically larger; so the law of $\cL_2$ conditional on $\{\cL_1(0)\geq \theta\}$ is stochastically increasing in $\theta$. But in the $\theta\to\infty$ limit, the upper boundary goes to $\infty$ and disappears, and in that case $\cL_2$ can be thought of as becoming an unconditioned $\cL_1$ as the latter also has no upper boundary. It is this basic reasoning that allows us to control the lower curve $\cL_2$ in all our arguments. As we saw in Section~\ref{s.weaker bk}, Assumption~\ref{as.weak bk}(b\ensuremath{'}) is a weaker version which suffices in the arguments, and which holds for the KPZ line ensemble.

Finally we turn to Assumption~\ref{as.mono in cond}. This will be needed, for example, to compare the conditional law given $\h_1(0) = \theta$ to the same given $\h_1(0)\geq \theta$; the usefulness of such comparisons is that the latter event is increasing, and so is amenable to applying the positive association (FKG) inequality or the above monotonicity properties. The proofs that the assumption holds in our examples again go via proving the same for Brownian bridge ensembles and taking appropriate limits. 

It is worth emphasizing that these conditional monotonicity statements are \emph{not} implied by the monotonicity statements mentioned above such as the FKG inequality; indeed, the FKG inequality can be derived from monotonicity in conditioning. 
We expect that, similar to the monotonicity in boundary data statements, these monotonicity in conditioning statements may prove to be useful in many contexts.

We now move on to the key ideas in the paper.

\subsubsection{Obtaining sharp one-point asymptotics}\label{s.intro.proof ideas.one-point}
We first sketch how to obtain the upper-tail one-point asymptotics of $\exp(-\frac{4}{3}\theta^{3/2}(1+o(1)))$. The key ideas appear mostly in the proof of  Theorem~\ref{mt.one point limit shape} on the shape of the top curve under the conditioning that $\cL_1(0)=\theta$, which we assume for the moment.

First we observe why the exponent of $\theta$ is $\smash{\frac{3}{2}}$. Indeed, by  Theorem~\ref{mt.one point limit shape}, conditionally on $\{\cL_1(0)=\theta\}$, $\smash{\cL_1(\pm\theta^{1/2})}$ is with high probability close to $-\theta$ and hence, on a horizontal interval of scale $\smash{\theta^{1/2}}$, $\cL_1$ oscillates upwards by order $\theta$. If we believe, based on the Gibbs property, that $\cL_1$ behaves roughly like a Brownian bridge on the interval $\smash{[-\theta^{1/2}, \theta^{1/2}]}$ (i.e., assuming the upward push it receives from  the lower boundary of the second curve is not too substantial), then the tail is predicted to have a Gaussian form: $\smash{\exp(-O(\frac{\theta^2}{\theta^{1/2}})) = \exp(-O(\theta^{3/2}))}$.

Following this logic more precisely and taking into account the effect of the second curve gives the correct coefficient of $\frac{4}{3}$ for $\theta^{3/2}$, which we explain next. We also mention an alternate equivalent perspective (which we however do not explicitly use) in terms of identifying the function maximizing the Dirichlet energy or Cameron-Martin functional subject to staying above a parabola in order to predict the limit shape and tail asymptotics; a further brief discussion is in Section~\ref{s.cameron-martin heuristic}.

\medskip
\emph{The upper bound for the one-point tail:} Let us pretend that we know that the shape from Theorem~\ref{mt.one point limit shape} holds even if we condition on $\{\cL_1(0)\geq \theta\}$ instead of $\{\cL_1(0) = \theta\}$ (this is indeed the case, but we do not explicitly prove it); this implies that $\smash{\cL_1(\pm\theta^{1/2}) \leq -\theta+M\theta^{1/4}}$ with probability at least $\smash{\frac{1}{2}}$, conditionally on $\{\cL_1(0)\geq\theta\}$, for large enough $M$ (independent of $\theta$). To simplify the quantities, in the rest of the discussion we will assume that $\cL_1(\pm\theta^{1/2})\leq-\theta$.

As already indicated, ultimately, we wish to apply the Brownian Gibbs property on $[-\theta^{1/2},\theta^{1/2}]$. This says that, conditional on $\cL_1$ on $[-\theta^{1/2}, \theta^{1/2}]^c$ and on $\cL_2,\cL_3, \ldots$ on $\R$, the distribution of $\cL_1$ on $[-\theta^{1/2}, \theta^{1/2}]$ is a Brownian bridge from $(-\theta^{1/2}, \cL_1(-\theta^{1/2}))$ to $(\theta^{1/2}, \cL_1(\theta^{1/2}))$, conditioned on it lying above $\cL_2$. So there are three pieces of boundary data: $\cL_1(-\theta^{1/2})$, $\cL_1(\theta^{1/2})$, and $\cL_2$. The previous paragraph controls the side values $\smash{\cL_1(\pm\theta^{1/2})}$. What remains is the lower boundary curve~$\cL_2$, and this control is provided by Assumption~\ref{as.weak bk}(b\ensuremath{'}) (or by Assumption~\ref{as.corr}(b) by the argument deriving \ref{as.weak bk}(b\ensuremath{'}) from it as in Section~\ref{s.weaker bk}).

With this setup, we may apply the Brownian Gibbs property on $[-\theta^{1/2},\theta^{1/2}]$ for the top curve. Now, we are trying to upper bound the probability of the increasing event that $\cL_1(0)\geq\theta$, and intuitively this probability increases if we raise the boundary data. Thus we may take the boundary data to be as large as possible, i.e., $(\pm\theta^{-1/2}, -\theta)$ at the sides and the function $x\mapsto-x^2$ below.

This yields, with $B$ a rate two Brownian bridge from $(-\theta^{1/2},-\theta)$ to $(\theta^{1/2}, -\theta)$, that
\begin{align}
\P\left(\cL_1(0)\geq \theta\right)
&\leq \P\left(B(0)\geq \theta \ \Big|\  B(x)\geq -x^2 \ \ \forall x\in[-\theta^{1/2},\theta^{1/2}]\right)\nonumber\\
&\leq \frac{\P\left(B(0)\geq \theta\right)}{\P\left(B(x)\geq -x^2 \ \ \forall x\in[-\theta^{1/2},\theta^{1/2}]\right)}.\label{e.intro brownian ratio}
\end{align}
Now $B(0)$ is a normal random variable with mean $-\theta$ and variance $2\times\frac{\theta^{1/2}\times\theta^{1/2}}{2\theta^{1/2}} = \theta^{1/2}$. So, using the form of the Gaussian tail, the numerator is at most 
$$\exp(-(\theta+\theta)^2/2\theta^{1/2}) = \exp(-2\theta^{3/2}).$$
The denominator turns out to be lower bounded by $\smash{\exp(-\frac{2}{3}\theta^{3/2}(1+o(1)))}$ which follows from a straightforward but tedious Brownian computation involving controlling the Brownian bridge at a fine mesh of points in Proposition~\ref{p.parabola avoidance probability}. Plugging the above estimates into \eqref{e.intro brownian ratio} yields the upper bound of $\exp(-\frac{4}{3}\theta^{3/2})$.
Interestingly, a matching upper bound on the denominator can be proved more straightforwardly using tail bounds of a Gaussian observable; see Remark~\ref{r.avoidance prob upper bound}.

The actual argument has a few more complications than the sketch above, mainly arising from working carefully with the monotonicity properties available to us, which we will refrain from explaining further at this point. 
To obtain the density estimates of Theorem~\ref{mt.one point density asymptotics} from these tail bounds, it is enough to prove a degree of regularity for the density. This is done in Proposition~\ref{p.density control} using resampling arguments whose detailed explanation we also omit here.

\medskip

\emph{Moving from zero to positive temperature:} As mentioned above, in the positive temperature case ($t<\infty$), the Gibbs property involves a softer constraint of Brownian bridges subject to a Radon-Nikodym derivative which energetically penalizes but does not prohibit intersection. 

While we postpone giving the precise form of the Radon-Nikodym derivative to Section~\ref{s.gibbs and line ensembles}, it can be written as $W_{H_t}/Z_{H_t}$, where $W_{H_t}\in[0,1]$ and $Z_{H_t}$ is a normalization constant so that the ratio is a probability density. $Z_{H_t}$ is called the \emph{partition function} and it is a deterministic function of the boundary data (i.e., as above, side values and lower curve); $H_{t}$ is called the Hamiltonian and is needed to specify the precise form of $W_{H_t}$ and $Z_{H_t}$. The analogue of $Z_{H_t}$ in the zero-temperature, $t=\infty$ case above is the probability of the Brownian bridge staying above the second curve, and $W_{H_t}$ is the corresponding indicator function.

We saw above that this probability appeared in the denominator when we were trying to estimate the probability that $\cL_1(0)\geq \theta$. Essentially the only difference when working with the positive temperature case is that we need to replace a non-intersection probability by the analogous partition function in all the estimates and this involves lower bounding $Z_{H_t}$ with a parabolic lower boundary.  For this, we have a simple lemma (Lemma~\ref{l.pos temp Z  lower bound via non-avoid prob}) that allows us to transfer lower bounds in the zero temperature case (such as the mentioned Proposition~\ref{p.parabola avoidance probability}) to the positive temperature case.

\medskip
\emph{The lower bound for the one-point tail:}
A lower bound could be argued on similar lines to the upper bound, but for one issue: this time, we would want to say that the lower boundary condition (i.e., $\cL_2$) does not go too low since, as already indicated, to get the sharp $\exp(-\frac{4}{3}\theta^{3/2})$ bound one has to consider the probability of avoiding the lower boundary. This in turn would require an estimate on the \emph{lower} tail of $\cL_2(0)$ and $\smash{\inf_{[-\theta^{1/2},\theta^{1/2}]}\cL_2(x)}$. While lower tail estimates are available for $\h_1(0)$ \cite{corwin2020lower,charlier2021uniform} and $\cP_1$ \cite{tracy1994level} via integrable arguments, and may be upgraded to lower tails for the infimum of the respective second curves over an interval by resampling arguments (see \cite[Proposition~A.2]{hammond2016brownian} and \cite[Proposition~B.1]{wu2021tightness}), we have not included lower tail estimates for $\cL_1(0)$ in our assumptions; further, these estimates would not be available for the extremal stationary ensembles that we also wish to make statements about. 

Observe, however, that the \emph{tightness} of $\cL_1(0)$ on the lower tail side is available. While on first glance this appears quite weak, our actual argument makes use of this one-point tightness along with a bootstrapping procedure. In fact, the argument initially appears to completely ignore the lower boundary of the second curve and requires no explicit control on it. For this reason it directly applies to both positive and zero temperature ensembles with no need to transfer non-intersection probability bounds to partition functions. Let us sketch the argument.

We consider the following method of obtaining that $\cL_1(0)\geq\theta$: we demand that $\cL_1(\pm\frac{1}{2}\theta^{1/2}) \geq 0$ and resample $\smash{\cL_1}$ on $\smash{[-\frac{1}{2}\theta^{1/2},\frac{1}{2}\theta^{1/2}]}$, i.e., we apply the Brownian Gibbs property to $\smash{[-\frac{1}{2}\theta^{1/2},\frac{1}{2}\theta^{1/2}]}$. Because we are lower bounding the probability of an increasing event, our bound is only worsened by lowering the boundary data: so we may take $\cL_1(\pm\frac{1}{2}\theta^{1/2}) = 0$ and take the second curve to $-\infty$ on $\smash{[-\frac{1}{2}\theta^{1/2},\frac{1}{2}\theta^{1/2}]}$, i.e., have no lower boundary. 

Observe that $\P(\cL_1(\frac{1}{2}\theta^{1/2})\geq0) = \P(\cL_1(0)\geq \theta/4)$ by stationarity, and similarly for $\cL_1(-\frac{1}{2}\theta^{1/2})$. Letting $P(\theta) = \P(\cL_1(0)\geq\theta)$, we see that, by positive association, $\P(\cL_1(\pm\frac{1}{2}\theta^{1/2})\geq 0) \geq P(\theta/4)^2$. Letting $B$ be a rate two Brownian bridge from $(-\frac{1}{2}\theta^{1/2},0)$ to $(\frac{1}{2}\theta^{1/2},0)$ and using the form of the Gaussian tail, the above discussion shows that
\begin{align*}
P(\theta)\geq P(\theta/4)^2\cdot\P(B(0)\geq\theta) = P(\theta/4)^2\cdot\exp\left(-\frac{\theta^2}{2\times\frac{1}{2}\theta^{1/2}}\right) = P(\theta/4)^2\cdot\exp\left(-\theta^{3/2}\right),
\end{align*}
the last factor since $B(0)$ has mean 0 and variance $2\times\frac{\frac{1}{2}\theta^{1/2}\times\frac{1}{2}\theta^{1/2}}{\theta^{1/2}} = \frac{1}{2}\theta^{1/2}$. 

Somewhat surprisingly, on iteration, this yields a lower bound of $\smash{\exp(-\frac{4}{3}\theta^{3/2}(1+o(1)))}$ (it is easy to check $\smash{P(\theta) = \exp(-\frac{4}{3}\theta^{3/2})}$ satisfies the recurrence) provided that the recurrence is non-trivial, i.e.,  if one knows that $P(\theta)>0$ for all $\theta>0$. This can shown by a very similar argument. 

\begin{figure}
\begin{tikzpicture}[scale=0.7]
\draw[green!70!black, semithick]  plot[smooth, domain=-2.5:2.5] (\x, -0.3*\x * \x);

\newcommand{\ltan}{1.9}
\newcommand{\rtan}{1.9}
\newcommand{\loc}{0.95} %was 0.8

\draw[blue, thick, dashed]  (-\loc, -0.3*\ltan*\ltan + 2*0.3*\ltan*\ltan - 2*0.3*\ltan*\loc) -- (\loc, -0.3*\rtan*\rtan + 2*0.3*\rtan*\rtan - 2*0.3*\rtan*\loc);

%left corner
\node[circle, fill, blue, inner sep = 1pt] at (-\loc, -0.3*\ltan*\ltan + 2*0.3*\ltan*\ltan - 2*0.3*\ltan*\loc) {};
\node[anchor = south east, scale=0.8] at (-\loc-0.05, -0.3*\ltan*\ltan + 2*0.3*\ltan*\ltan - 2*0.3*\ltan*\loc) {$(-\tfrac{1}{2}\theta^{1/2}, 0)$};

%right corner
\node[circle, fill, blue, inner sep = 1pt] at (\loc, -0.3*\rtan*\rtan + 2*0.3*\rtan*\rtan - 2*0.3*\rtan*\loc) {};
\node[anchor = south west, scale=0.8] at (\loc+0.05, -0.3*\rtan*\rtan + 2*0.3*\rtan*\rtan - 2*0.3*\rtan*\loc) {$(\tfrac{1}{2}\theta^{1/2}, 0)$};

%left tangent
\node[circle, fill, blue, inner sep = 1pt] at (-\ltan, -0.3*\ltan*\ltan) {};
\node[anchor=east,scale=0.8] at (-\ltan, -0.3*\ltan*\ltan) {$(-\theta^{1/2}, -\theta)$};

%right tangent
\node[circle, fill, blue, inner sep = 1pt] at (\rtan, -0.3*\rtan*\rtan) {};
\node[anchor=west, scale=0.8] at (\rtan, -0.3*\rtan*\rtan) {$(\theta^{1/2}, -\theta)$};

\node[circle, fill, blue, inner sep = 1pt] at (0, 0.3*\ltan*\ltan) {};
\node[anchor=west, scale=0.8] at (0.1, 0.3*\ltan*\ltan) {$(0,\theta)$};

\draw[thick, blue] (-\ltan, -0.3*\ltan*\ltan) -- ++(\ltan, 2*0.3*\ltan*\ltan) -- (\ltan, -0.3*\ltan*\ltan);

\end{tikzpicture}
\caption{The setup for the argument for the lower bound on the upper tail. The interval on which we resample is now $[-\frac{1}{2}\theta^{1/2}, \frac{1}{2}\theta^{1/2}]$; note that the boundary points are such that $\tri$ equals zero at them, and so the line connecting $(x, \tri(x))$ and $(-x, \tri(-x))$ when $x=\frac{1}{2}\theta^{1/2}$ is tangent to $-x^2$. Thus the Brownian bridge defined between these points will avoid $-x^2$ with constant probability, and the FKG inequality will be essentially sharp in lower bounding $\P(\cL_1(-\frac{1}{2}\theta^{1/2}),\cL_1(\frac{1}{2}\theta^{1/2})\geq 0)$.}\label{f.fkg}
\end{figure}
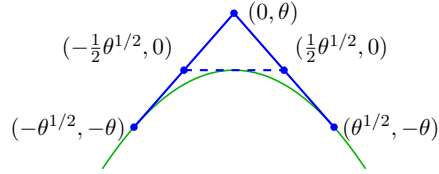

The reason this recursion works, and the choice of demanding that $\cL_1(x)\geq 0$ for $x=\pm\frac{1}{2}\theta^{1/2}$ may appear mysterious, but can be understood using the geometric picture being developed; see Figure~\ref{f.fkg}. Indeed, the choice is made by combining the information on the shape of the profile under the conditioning that $\cL_1(0)\geq\theta$ from Theorem~\ref{mt.one point limit shape} (note that $(\pm\smash{\frac{1}{2}\theta^{1/2}},0)$ lie on $\tri$), the information from Theorem~\ref{mt.fkg sharpness} on when the FKG inequality will be sharp (the line joining \smash{$(\pm\frac{1}{2}\theta^{1/2},0)$} is tangent to $-x^2$), and the intuition that a Brownian bridge starting from height zero will essentially not be affected by the lower boundary $x\mapsto -x^2$ and so will allow the boundary to be safely ignored. In other words, the choices made render the recursion essentially an equality up to first order in the exponent, explaining why the recursion converges to the sharp bound. Note however, as we saw above, that while we use these mentioned theorems to inform our choice of $x=\pm\frac{1}{2}\theta^{1/2}$, the argument does not actually formally apply these theorems (indeed, for instance, Theorem~\ref{mt.fkg sharpness} is not yet proven and will actually use the lower bound on the upper tail in its proof).

We next turn to sketching the proofs of the shapes of the top curve under the conditionings $\{\cL_1(0)=\theta\}$ (which, as already evident, is the main ingredient in the proof of the one point estimate) and $\{\cL_1(-\theta^{1/2})=a\theta,\cL_1(\theta^{1/2})=b\theta\}$. We will then outline a remaining step in the argument for the extremal stationary ensembles, which relies on an extension of such arguments.

\subsubsection{A heuristic for the shape of the top curve under conditionings}\label{s.convex heurstic}

The basic idea driving the proofs of the shapes of the profile under the two conditionings considered in Theorems~\ref{mt.one point limit shape} and \ref{mt.two-point limit shape} is the following: that Brownian bridges approximately follow the line connecting their endpoints implies that ensembles with a resampling property in terms of Brownian bridges must have shapes that are approximately convex. The argument is fleshed out in Figure~\ref{f.convexity from resampling}. 

\begin{figure}[h!]
\begin{center}
\begin{tikzpicture}[scale=0.65]

	% LEFT PANEL
	\draw[green!70!black, semithick] (1,1) -- (3,3) to [curve through={ (4,2) }] (5,3) -- (7,1);

	\begin{scope}[shift={(0,-0.75)}]
	\draw[green!70!black, semithick] (1,1) -- (3,3) to [curve through={ (4,2) }] (5,3) -- (7,1);
	\end{scope}

	% MIDDLE PANEL
  \begin{scope}[shift={(8,0)}]

  	\pgfmathsetseed{23655}
 	\draw[blue, semithick, decorate, decoration={random steps,segment length=2pt,amplitude=1pt}] (3,3) -- (5,3); 

   	\node[circle, fill, green!70!black, inner sep=1pt] at (3,3) {};
  	\node[circle, fill, green!70!black, inner sep=1pt] at (5,3) {};

  	\draw[green!70!black, semithick] (1,1) -- (3,3); 
 	\draw[green!70!black, semithick] (5,3) -- (7,1);

	\begin{scope}[shift={(0,-0.75)}]
		\draw[green!70!black, semithick] (1,1) -- (3,3) to [curve through={ (4,2) }] (5,3) -- (7,1);
	\end{scope}
  \end{scope}

  % RIGHT PANEL
    \begin{scope}[shift={(16,0)}]

   	\node[circle, fill, green!70!black, inner sep=1pt] at (3,3) {};
  	\node[circle, fill, green!70!black, inner sep=1pt] at (5,3) {};

  	\draw[green!70!black, semithick] (1,1) -- (3,3); 
 	\draw[green!70!black, semithick] (5,3) -- (7,1);

  	\draw[green!70!black, semithick] (3,3) -- (5,3); 

	\begin{scope}[shift={(0,-0.75)}]
		\draw[green!70!black, semithick] (1,1) -- (3,3) to [curve through={ (4,2) }] (5,3) -- (7,1);
	\end{scope}
  \end{scope}
\end{tikzpicture}
\end{center}
\caption{On the left panel we consider the situation where the limit shape for the top curve is non-convex on some interval; the non-intersection condition then pushes the second curve down on the same interval. In the second panel we resample the top curve on the interval of non-convexity. The Brownian bridge would typically approximately follow the straight line between its endpoints if it was unconditioned; here it is conditioned to avoid the second curve, but, since the second curve is already lower, the avoidance conditioning is met by the bridge's typical behaviour. Thus as we see in the third panel, the resampling removes the non-convexity from the top curve, thus contradicting the initial non-convexity, since resampling preserves the distribution of the ensemble.}
\label{f.convexity from resampling}
\end{figure}
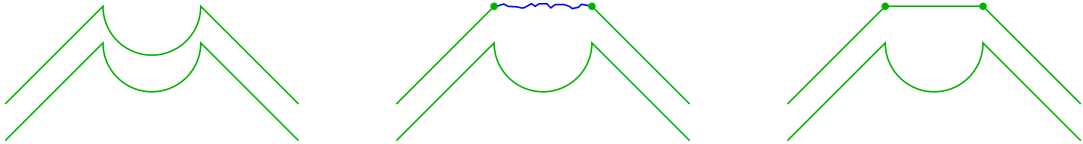

Similar reasoning leads to the statement that the shape of the top curve must be approximately the convex hull (i.e., minimal convex shape above) of the lower curve and constraints imposed by conditioning. This explains the shapes which arise in Theorems~\ref{mt.one point limit shape} and \ref{mt.two-point limit shape}, which are indeed respectively the convex hull of $x\mapsto -x^2$ (representing the lower curve) and $(0,\theta)$ (due to the conditioning that $\h_1(0)=\theta$) and the convex hull of $x\mapsto -x^2$, $(-\theta^{1/2},a\theta)$, and $(\theta^{1/2},b\theta)$.

Similar observations formed the basis of work  by Aggarwal in \cite{aggarwal2020arctic} to obtain the limit shape of the arctic boundary in the six-vertex model at the ice point which admits an encoding in terms of non-intersecting random walks, thus providing mathematical justification of the \emph{tangent method} heuristic introduced by Colomo-Sportiello \cite{colomo2016arctic}.

\subsubsection{An equivalent heuristic via Cameron-Martin}\label{s.cameron-martin heuristic}

Since $\cL$ satisfies a Brownian Gibbs property, and since the second curve is approximately the parabola $p(x):= -x^2$, one can guess that, conditional on $\cL_1(0)\geq \theta$, $\cL_1$ will approximate the maximizer of the energy functional (coming from the Cameron-Martin theorem)
\begin{align*}
\sup\left\{\int_{-z}^z \left(|f'(x)|^2 - |p'(x)|^2\right)\,\dif x : \parbox[c]{2.25in}{\centering $f\in\mc C([-z,z], \R), f(\pm z) = -z^2$, \\$f(x)\geq p(x) \ \forall x\in[-z,z], f(0)\geq \theta$}\right\},
\end{align*}
where $z$ is sufficiently large, and that the negative of the logarithm of the probability of $\cL_1(0)\geq \theta$ will be approximately the above supremum value. (Here the cost of avoiding $p$ is essentially the same as equaling it, which accounts for the second term in the integrand.) Indeed, the solution to this variational problem is precisely the limit shape from Theorem~\ref{mt.one point limit shape} (and the one from Theorem~\ref{mt.two-point limit shape} for the analogous variational problem with $f(-\theta^{1/2})\geq a\theta$, $f(\theta^{1/2})\geq b\theta$), though this is not immediate; perhaps the quickest way to see that the solution should be the convex hull is by the argument outlined in Section~\ref{s.convex heurstic}.

\subsubsection{Moving from heuristic to the proof of Theorem~\ref{mt.one point limit shape}}\label{s.intro.proof sketch.limit shape}

Now we sketch the actual arguments we develop to establish Theorem~\ref{mt.one point limit shape} (similar arguments, though with some technical complications, suffice for Theorem~\ref{mt.two-point limit shape} on the two-point limit shape as well but we do not discuss this). Again we stick to the zero temperature case.

Let us focus on proving the shape of the top curve is $\tri$ on the \emph{left} side of 0 only, i.e., in $[-\theta^{1/2},0]$; the argument for the right side will clearly be symmetric.

Consider an extension of $\tri$ to $(-\infty, 0]$ given by the same line, i.e., by $x\mapsto \theta+2\theta^{1/2}x$; see the dotted lines in Figure~\ref{f.one-point limit shape}. The crucial point is the following:  if we can find two $x$-coordinates at which $\cL_1$ is on the extended version of $\tri$, then the Brownian Gibbs property gives an approximate linear resampling showing that $\cL_1$ will be approximately equal to $\tri$ on the entire interval in between these two points. Then the resampling will be approximately linear in spite of the lower boundary condition (the second curve which must be avoided) because the latter is approximately the inverted parabola $x\mapsto-x^2$. So, as a consequence of the convex hull property of $\tri$, it is \emph{tangent} to $x\mapsto-x^2$ at the point $x=-\theta^{1/2}$; thus, intuitively, the unconditioned Brownian bridge between two points on $\tri$ should avoid $x\mapsto-x^2$ with uniformly positive probability (as can be checked), and so the conditioned and unconditioned processes can be treated as essentially the same.

So we have to find two points at which $\cL_1$ is equal to $\tri$. We will refer to these as \emph{pinning} points. We already have one pinning point, the one to the right, at $0$, since we have conditioned on $\cL_1(0)= \theta=\tri(0)$. Indeed, this is why we consider this conditioning instead of $\{\cL_1(0)\geq \theta\}$, as for the latter we a priori have very little control over the value of $\cL_1(0)$.

However, for convenience let us pretend we are actually conditioning on $\{\cL_1(0)\geq \theta\}$ (which as an increasing event is easier to work with) while we look for the left pinning point. The pinning point we find will differ based on whether we are trying to lower bound or upper bound the shape of $\cL_1$.

\smallskip

\paragraph*{\emph{Lower bounding the profile:}} Since we are conditioning on the increasing event $\{\cL_1(0)\geq\theta\}$, the FKG inequality says that the first curve under this conditioning is stochastically higher than the unconditioned first curve. The unconditioned first curve is with high probability close to $-\theta$ at $\pm\theta^{1/2}$ (using one-point tightness of $\cL_1(x)+x^2$ at $x=\pm\theta^{1/2}$), so $\cL_1(-\theta^{1/2})$ is at least $-\theta$ (ignoring any lower order terms for simplicity) under the conditioning as well. This suffices because, again by monotonicity, we can lower $\smash{\cL_1(-\theta^{1/2})}$ from its value to $-\theta$, and this can only lower the profile---not an issue when proving a lower bound. This is a form of pinning, as $\smash{(\pm\theta^{1/2}, -\theta)}$ lies on $\tri$.

Next, we apply the Brownian Gibbs property on $[-\theta^{1/2}, 0]$. This tells us that $\smash{\cL_1}$ is a Brownian bridge from $(-\theta^{1/2},-\theta)$ to $(0,\theta)$, conditioned to avoid $\cL_2$. But, again by monotonicity, $\cL_1$ must be larger than the \emph{unconditional} (i.e., with $\cL_2\equiv -\infty$) Brownian bridge between the two mentioned pinning points. This unconditioned Brownian bridge approximately follows $\tri$, up to Brownian fluctuations, which occur on scale $(\theta^{1/2})^{1/2}=\theta^{1/4}$. This establishes the lower bound on the profile.

\smallskip

\paragraph*{\emph{Upper bounding the profile:}} For simplicity, we assume in this discussion that Assumption~\ref{as.tails} holds with $\alpha=\beta=\frac{3}{2}$, i.e., $\exp(-c_1\theta^{3/2}) \leq \P(\cL_1(0)>\theta)\leq \exp(-c_2\theta^{3/2})$ for some $c_1,c_2>0$. This side is more delicate than the lower bound, as we cannot ignore the lower boundary $\cL_2$.

We first observe that, again, we do not actually need $\cL_1$ at our pinning point to be close to $\tri$; it is sufficient if the point lies \emph{below} $\tri$ by monotonicity, as we are proving a profile upper bound.

So we need to find a point $x_\theta$ such that $\P(\cL_1(x_\theta) \geq \tri(x_\theta) \mid \cL_1(0)\geq \theta)$ is small. It will be convenient to set $x_\theta = -\theta^{1/2}z$. Recalling that $\tri(x) = \theta+2\theta^{1/2}x$, we see using stationarity that the mentioned probability is upper bounded by
\begin{align*}
\frac{\P(\cL_1(x_\theta) \geq \tri(x_\theta))}{\P\left(\cL_1(0)\geq \theta\right)} = \frac{\P(\cL_1(0) \geq \theta (z^2 - 2 z+1))}{\P\left(\cL_1(0)\geq \theta\right)} \leq \frac{\exp\left(-c_2\theta^{3/2}(z-1)^{3}\right)}{\exp\left(-c_1\theta^{3/2}\right)}.
\end{align*} 
Clearly, there exists a large enough $z$ independent of $\theta$ such that $c_1(z-1)^3> c_2$.

Thus we have found a left pinning point at $x_\theta=-\theta^{1/2}z$. To obtain the upper bound on the profile, we wish to apply the Brownian Gibbs property on $[x_\theta, 0]$. However, to implement this we will need some control on the second curve on this interval which forms the lower boundary; if, for example, it has peaks rising much above $\tri$ inside $[x_\theta,0]$, then there is no way to upper bound $\cL_1$ by $\tri$.

It is here that it is crucial that $x_\theta$ is polynomial in $\theta$. This is because we can break up $[x_\theta,0]$ into polynomially many unit intervals, and, since one can upgrade one-point upper tail bounds to upper tail bounds on $\sup (\cL_2(x) + x^2)$ (where the supremum is over a unit interval) which decay like $\smash{\exp(-c\theta^{3/2})}$, a union bound suffices. A subtle point is that this bound has to hold conditional on $\cL_1(0)\geq \theta$, which we accomplish by the BK inequality. With its weakened form Assumption~\ref{as.weak bk}(b\ensuremath{'}), we do not need the fast decaying tail of $\sup (\cL_1(x)+x^2)$ or the same with $\cL_2$; but the polynomial size of $x_\theta$ is still important.

At this stage we have two pinning points at which $\cL_1$ is equal to $\tri$ and a lower boundary condition which is essentially a parabola on the interval between the pinning points. We know that $\cL_1$ is lower than a Brownian bridge conditioned on avoiding the parabola; but, since $\tri$ is tangent to $x\mapsto -x^2$, this conditioning has uniformly positive probability. Thus $\cL_1$ lies below essentially an unconditioned Brownian bridge on this interval; this yields the desired upper bound on the profile with a Brownian fluctuation scale of $\theta^{1/4}$, since $[x_\theta,0]$ is an interval of length $O(\theta^{1/2})$.

Similar arguments also yield bounds on the profile outside of $[-\theta^{1/2},\theta^{1/2}]$: this is captured in Proposition~\ref{p.para fluctuation h_t}, which says that the profile is close to $-x^2$ on an interval of scale $\theta^{1/2}$.

\subsubsection{The arguments for the extremal ensembles}\label{s.intro.proof ideas.extremal}

We will show that a line ensemble $\cL$ satisfying Assumptions~\ref{as.bg}--\ref{as.mono in cond} (with $t=\infty$) also satisfies Assumption~\ref{as.tails}, which suffices since extremal stationary ensembles will be shown to satisfy \ref{as.bg}--\ref{as.mono in cond} in Appendix~\ref{app.monotonicity proofs}. In fact, similar arguments also hold for $t<\infty$, but would not yield tail estimates uniform in $t$. Note that here we will work with Assumption~\ref{as.corr}(b) and not \ref{as.weak bk}(b\ensuremath{'}).
We focus on the upper bound on the one-point upper tail.

Let us highlight two parts of the proof of the upper bound on the tail which has been already sketched for $\cL_1$. Part 1 was the upper bound on the top curve's profile when conditioned on $\cL_1(0)=\theta$, and Part 2 was control on the fluctuations of the second curve on an interval on which we resampled the top curve. (In fact, Part 2 was also needed to prove Part 1, and we will return to this point.)

As we saw in the just concluded sketch, we obtained Part 1 by first finding point $\pm x_\theta$ at which $\cL_1$ was with high probability (conditioned on $\cL_1(0)=\theta$) below the extended version of $\tri$ (by monotonicity, we could then raise the point to lie on $\tri$). Then by resampling the top curve on $[-x_\theta, 0]$ and $[0,x_\theta]$, we could show that the top curve remained close to $\tri$, assuming control on the second curve on the same intervals (it is to obtain this control that Part 2 is needed).

The existence of pinning points was done using a priori bounds on the upper tail of $\cL_1(0)$ from Assumption~\ref{as.tails}, proved for the parabolic Airy and KPZ equation examples via integrable inputs. 
Integrable inputs were used for pinning also in the argument given by Aggarwal in \cite{aggarwal2020arctic}, where, as one part of the larger argument, he obtains a pinning of a curve in the six-vertex model (under a particular choice of parameters) using explicit combinatorial formulas for a certain correlation function, which in turn gives asymptotics for a pinning probability.

\medskip
\emph{A weak form of pinning:} We start by observing that it is possible, using only one-point tightness, stationarity, and parabolic decay of $\cL_1$, to obtain a weak form of pinning. By ``weak'' we mean that we do not attempt to find a point at which $\cL_1$ is below $\tri$, which has slope $-\theta^{1/2}$, but instead where it is below a line of $\theta$-independent slope $-1$. Indeed, for given $\varepsilon>0$, we can show that 
\begin{equation}\label{e.intro weak pinning}
\P\left(\cL_1(x) > -|x|\right) \leq \varepsilon\cdot\P(\cL_1(0)\geq\theta)
\end{equation}
(i.e., the RHS is $\varepsilon$ times the probability of the conditioning event) for some $x$:
\begin{align*}
 \P\left(\cL_1(x) > -|x|\right)  = \P\left(\cL_1(0) > x^2-|x|\right) \to 0
\end{align*}
as $x\to\infty$ since $x^2-|x|\to\infty$. But of course, the $x=x_\theta>0$ which achieves \eqref{e.intro weak pinning} depends on $\theta$ in some completely uncontrolled way, another sense in which the pinning is weak. (While the existence of such an $x$ is guaranteed even if the slope of the line was $-\theta^{1/2}$, this would cause a problem later, as we will soon see.)

This is in contrast to the situation when we had a priori inputs, as that yielded that $x_\theta=O(\theta^{1/2})$. This was very important because the exponentially decaying upper tail bound we had on the one-point distribution meant that we could easily control the supremum of $\cL_2(x) + x^2$ on an interval whose size was polynomial in $\theta$. (This is where Part 2 was used in Part 1, as mentioned above.) But even with the a priori estimates we cannot control $\cL_2$ on arbitrarily large intervals, and here we have mere one-point tightness. 

\begin{figure}[t]
\includegraphics{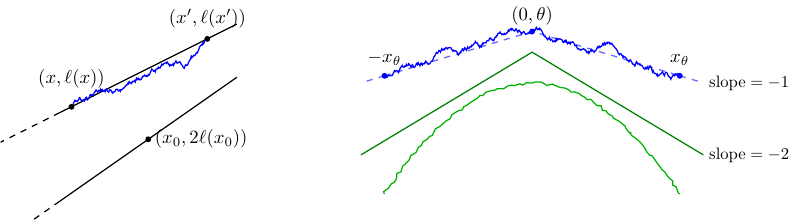}
\caption{\emph{Left panel}: A Brownian bridge whose endpoints lie on a line of given slope is unlikely to intersect another line of different slope if the endpoints of the bridge are well-separated from the second line. \emph{Right panel}: The lower green curve is what $\cL_2$ should look like, but which we cannot yet prove; the darker green tent-shaped function above it is the upper bound we are able to prove on $\cL_2$, lines of slope $\pm2$; and the blue curve is the upper bound we are able to prove on $\cL_1$ using resampling. In more detail, the blue curve is a Brownian bridge from $(-x_\theta, \theta-x_\theta)$ to $(0,\theta)$ to $(x_\theta,\theta-x_\theta)$ (which are endpoints of lines of slope $\pm 1$, drawn in dashed blue) conditioned to stay above the dark green tent; this latter conditioning has probability lower bounded by a uniformly positive constant by the argument from the left panel. Thus, at $\pm \theta$, $\cL_1$ is with high probability below $\theta/2$, as the Brownian bridge that bounds it has mean $0$ and standard deviation $O(\theta^{1/2})$ at those locations.}\label{f.stay below line}
\end{figure}

\medskip

\emph{Control on the second curve over all of $\R$:} To handle this issue, we weaken what we ask from the second curve: we have it remain below $-m|x| + K$ for a constant $m>0$ (that we choose) and a random $K$ almost surely. (One might try asking for $\cL_2$ to stay below $-\alpha x^2$ plus a random constant for any $0<\alpha<1$, as would hold for $\cP_2$. Our proof is not able to furnish this demand, but we are able to establish the linear bound. The difference is that a linear bound plays well with Brownian bridge's linear trajectory, as we will see.) 

We ask for the above conditionally on $\cL_1(0)\geq \theta$, so it may seem natural that $K$ would depend on $\theta$: but recall the BK inequality, which gives that the second curve, conditionally on $\cL_1(0)\geq \theta$, is stochastically smaller than an \emph{unconditional} $\cL_1$. The latter's law has no $\theta$ dependence, so $K$ can be taken to have no $\theta$-dependence. (It is here that we would not have been able to have $K$ be $\theta$-independent if we had taken the slope $m$ to depend on $\theta$.)

So we need to establish that there exists an almost surely finite $K$ such that $\cL_1(x) \leq -m|x| + K$ for all $x\in\R$; we will show that for each $m>0$ there exists such a random $K$. Establishing this is the main new argument in the extremal ensemble case compared to the arguments sketched above, and we outline it now.

The main tool we have available is the Brownian Gibbs property. The basic difficulty is that we have no control over the lower boundary $\cL_2$: in particular, if it has large peaks at an infinite sequence of random points $\tau_n$, that may force $\cL_1(\tau_n)$ to be larger than $-m|\tau_n|+K_n$ for a sequence $K_n\to\infty$.

If the above scenario happens, however, then $\cL_1$ is likely to remain high (in particular, higher than $-2m|x|$ say) in between these peaks as well (see the left panel of Figure~\ref{f.stay below line}): this is because unconditioned Brownian bridge approximately follows the linear path between its endpoints and so will with positive probability not hit a steeper line, and the lower boundary condition imposed by $\cL_2$ only pushes it further up (we need no control on $\cL_2$ for this statement!). Importantly, this holds true in a uniform way no matter how large the intervals $[\tau_n, \tau_{n+1}]$ are. But this leads to a contradiction because, by one-point tightness, stationarity, and parabolic decay, we can find a sequence of deterministic points $x_n$ such that $\cL_1(x_n) < -2|x|_n$ for all but finitely many $n$ almost surely; this is argued in a similar way to \eqref{e.intro weak pinning} and by invoking the Borel-Cantelli lemma.

\medskip

\emph{Combining the two parts:} With this control on the second curve in hand, combined with the weak pinning at $\pm x_\theta$, we can argue with similar ideas as the earlier cases that, conditionally on $\cL_1(0)\geq \theta$, the top curve is not too high at a pair of points which are not too far from zero with probability at least $\frac{1}{2}$; we do not take the points to be $\pm\theta^{1/2}$ as earlier since we do not have such strong control yet. Instead, we show that $\cL_1(\pm\theta)$ is at most $\theta/2$ with high probability. The argument for this involves resampling on the interval $[-x_\theta, 0]$ and $[0,x_\theta]$,  using that the second curve is below $K-2|x|$, and that a Brownian bridge $B$ is unlikely to hit a line of steeper slope than the slope between its endpoints (which here is $-1$); 
 see the right panel of Figure~\ref{f.stay below line}.
 Also note the important point that while $x_\theta$ can be arbitrarily large, the fluctuations of $B$ at $\pm\theta$ will only be of order $\theta^{1/2}$.

 With this control on $\cL_1(\pm\theta)$, we can do a resampling on $[-\theta,\theta]$ to say that $\P(\cL_1(0) \geq \theta)$ is at most the probability a Brownian bridge $B'$ from $(-\theta,\theta/2)$ to $(\theta,\theta/2)$ which is conditioned to stay above $K-2|x|$ satisfies $B'(0)\geq \theta$, which is clearly $\exp(-c\theta)$, as again the probability of the conditioning event is uniformly positive.

Now with an exponential tail on the one-point distribution available, the earlier arguments can be applied to yield all the main results for the extremal ensembles as well.

\subsection*{Organization of the paper}
In Section~\ref{s.tools} we recall the precise definitions of the Brownian Gibbs properties and collect various monotonicity and technical tools we will need in the main arguments, though we defer most of their proofs to the appendices as they do not require many new ideas and would obstruct the flow of the main arguments. In Sections~\ref{s.one point limit shape} and \ref{s.one point asymptotics} respectively we prove the one-point limit shape and the one-point asymptotics. We do the two-point versions of the same in Sections~\ref{s.two point limit shape} and \ref{s.two point asymptotics}. In Section~\ref{s.extremal} we obtain a preliminary one-point upper tail decay for extremal stationary ensembles, thus verifying that extremal stationary ensembles satisfy the upper bound of Assumption~\ref{as.tails} and making the earlier arguments for the sharp asymptotics applicable. Finally in Section~\ref{s.general data} we give the arguments to obtain the one-point asymptotics for general initial data.

To streamline the presentation, there are three appendices. Appendix~\ref{app.monotonicity proofs} proves that the assumptions from Section~\ref{s.assumptions} hold in the ensembles of interest and gives proofs for the implicated monotonicity properties. Appendix~\ref{app.hamiltonian} collects some calculations involving $\h$ such as of its Hamiltonian for its Gibbs property. The final Appendix~\ref{app.brownian estimates} provides the proofs of various Brownian estimates from Section~\ref{s.monotonicity tools}.

\begin{notation}\label{n.conditional prob notation}
$\mc C([a,b],\R)$ will denote the space of real-valued continuous functions defined on $[a,b]$.

We will often consider conditional probability distributions on conditioning on a $\sigma$-algebra $\F$. For these objects we will use the shorthand notation
$$\PF(\cdot) = \P(\cdot \mid \F).$$

The existence of the regular conditional probability measures we will need is ensured by the fact that we will always take $\F$ to be generated by random variables taking values in a Borel space, and then invoking well-known abstract results such as \cite[Theorem~8.5]{Kallenberg}. Conditional probabilities of the form $\P(\cdot \mid \h_1(0)=\theta)$ and $\P(\cdot \mid \h_1(-\theta^{1/2})=a\theta, \h_1(\theta^{1/2})=b\theta)$ are also defined via these regular conditional distributions and their associated probability kernels.
\end{notation}

\subsection*{Acknowledgements}
MH thanks Patrik Ferrari for posing a question a couple years ago which ultimately led to this investigation. We thank Adam Jaffe for pointing us to references for characterizations of extremal Gibbs measures and extremality of determinantal point processes, and Ivan Corwin for helpful suggestions. We would also like to thank the anonymous referees for their careful comments, and Xuan Wu for bringing to our attention an issue in the verification of Assumption~\ref{as.corr}(b) for the KPZ line ensemble in an earlier version of this article.

This material is based upon work partially supported by the National Science Foundation under Grant No.\ 1440140, while the authors were in residence at the Mathematical Sciences Research Institute in Berkeley, California, during the Fall 2021 program ``Universality and Integrability in Random Matrix Theory and Interacting Particle Systems.'' SG is also partially supported by NSF grant DMS-1855688, NSF CAREER Award DMS-1945172 and a Sloan Research Fellowship. MH is also partially supported by NSF grant DMS-1937254.

%%%%%%%%%%%%%%%%%%%%%%%%%%%%%%%%%%%%%%%%%%%%%%%%%%

\section{Line ensembles, monotonicity \& Brownian estimates}\label{s.tools}
 
In this section we introduce the formal definition of the Brownian Gibbs properties we work with as well as of the parabolic Airy line ensemble $\cP$ and collect come monotonicity tools and Brownian estimates that we will be making extensive use of in the main arguments in upcoming sections.

We introduce the Gibbs properties and $\cP$ in Section~\ref{s.gibbs and line ensembles}. We cover monotonicity tools in Section~\ref{s.monotonicity tools} and the Brownian ones in Sections~\ref{s.tools.brownian}--\ref{s.tools.brownian above para}. 
In Section~\ref{s.tools.kpz analogues} we obtain a useful consequence of Assumption~\ref{as.weak bk}(b\ensuremath{'}) that will be repeatedly used in subsequent sections.

The proofs of most of the tools are straightforward but tedious, or follow the same lines as arguments already existing in the literature. For this reason such proofs have been deferred to the appendices. %

\subsection{Line ensembles \& Brownian Gibbs}\label{s.gibbs and line ensembles}

We start by stating formally the definitions of the spaces and objects we will be working with.

\begin{definition}
A line ensemble is a random continuous function defined from $\R\times\N$ to $\R$, where the space of continuous functions $\R\times\N\to \R$ is endowed with the topology of uniform convergence on compact sets and the corresponding Borel $\sigma$-algebra.
\end{definition}

\begin{definition}[$H$-Brownian Gibbs and Brownian Gibbs properties]\label{d.bg}
Let $\cL = (\cL_1, \cL_2, \ldots)$ be a line ensemble. For $\intint{j,k}\subseteq \N$ and $[\ell,r]\subseteq \R$ a finite interval, define $\Fext(\intint{j,k},[\ell,r])$ to be the $\sigma$-algebra generated by $\{\cL_i(x) : (i,x)\not\in\intint{j,k}\times[\ell,r]\}$, i.e., all the data external to $[\ell,r]$ of the $j$\textsuperscript{th} to $k$\textsuperscript{th} curves; when $j=1$, we write $\Fext(k,[\ell,r]) := \Fext(\intint{1,k},[\ell,r])$. 

Let $H:\R\to[0,\infty)$ be a continuous function (called a \emph{Hamiltonian}).
We say a line ensemble $\cL$ has the \emph{$H$-Brownian Gibbs} property with respect to rate $\sigma^2$ Brownian bridge if, for any $\intint{j,k}\subseteq \N$ and $[a,b]\subseteq \R$, conditionally on $\F:=\Fext(\intint{j,k},[a,b])$, the following holds. Let $\vec x, \vec y \in \R^{k-j+1}$ be defined by $x_i = \cL_i(a)$ and $y_i = \cL_i(b)$ for $i=j, \ldots, k$. The Radon-Nikodym derivative of the $\F$-conditional law $\P_{\cL}^{j,k,a,b}$ of $(\cL_j, \ldots, \cL_k)$ on $[a,b]$ with respect to the law $\Pfree^{j-k+1, a,b, \vec x, \vec y}$ of $k-j+1$ independent rate $\sigma^2$ Brownian bridges on $[a,b]$ with endpoint values $\vec x $ and $\vec y$ is given by
$$\frac{\mathrm d \P_{\cL}^{j,k,a,b}}{\mathrm d \Pfree^{k-j+1,a,b,\vec x,\vec y}}(\cL_j, \ldots, \cL_k) = \frac{W_H^{k-j+1,a,b, \cL_{j-1}, \cL_{k+1}}(\cL_j, \ldots, \cL_k)}{Z_H^{k-j+1,a,b,\vec x, \vec y, \cL_{j-1}, \cL_{k+1}}},$$
where $\cL_0 = \infty$ and, for any $k\in\N$, $[a,b]\subseteq \R$, $\vec x, \vec y\in\R^k$, and $f, g, B_1, \ldots, B_k:[a,b]\to\R$,
\begin{equation}\label{e.gibbs rn derivative}
W_H^{k,a,b, f, g}(B_1, \ldots, B_k) = \exp\left\{-\sum_{i=0}^k \int_a^b H\left(B_{i+1}(u) - B_i(u)\right)\, \mathrm du\right\}
\end{equation}
(with $B_0=f$ and $B_{k+1}=g$) and
$$Z_H^{k,a,b,\vec x, \vec y, f, g} = \Efree^{k,a,b,\vec x, \vec y}\left[W_H^{k,a,b, f, g}(B_1, \ldots, B_k)\right].$$
When $f=\infty$, we will often drop it from the superscript of both $W$ and $Z$.

The \emph{Brownian Gibbs} property is specified by the above with $H(x) = \infty\cdot\one_{x>0}$; in plain words, the Radon-Nikodym derivative corresponds to conditioning on non-intersection of the curves $\cL_1, \ldots, \cL_k$ between themselves as well as the lower curve.
\end{definition}
In the remainder of the paper we will set $\sigma^2 = 2$.

 At a few points in the argument we will need the notion of a stopping domain, which is a random interval analogous to stopping times for Markov processes; the important point is that the Brownian Gibbs property can be applied to stopping domains, as we recall in Lemma~\ref{l.strong bg}.

\begin{definition}[Stopping domain and strong Brownian Gibbs]\label{d.strong bg}
Let $\cL:\N\to\R$ be an $\N$-indexed collection of continuous curves and $H$ a Hamiltonian.
A pair of random variables $\mf l, \mf r$ is a \emph{stopping domain} for $\cL_1, \ldots, \cL_k$ if, for all $\ell < r$,
$$\{\mf l \leq \ell, \mf r \geq r\} \in \F_{\mathrm{ext}}(k,\ell, r).$$
Define the $\sigma$-algebra $\F_{\mathrm{ext}}(k,\mf l, \mf r)$ to be the one generated by events $A$ such that $A\cap \{\mf l\leq \ell, \mf r\geq r\} \in \F_{\mathrm{ext}}(k,\ell, r)$ for all $\ell < r$. Also define, for $k\in\N$, the set $\mc C^k= \{(\ell,r, f_1, \ldots, f_k): \ell<r, f_1, \ldots, f_k\in \mc C([\ell,r])\}$.

An infinite collection of random continuous curves $\cL$ has the \emph{strong $H$-Brownian Gibbs} property if, for any $k\in\N$, stopping domain $[\mf l, \mf r]$ with respect to $\cL_1, \ldots, \cL_k$, and $F:\mc C^k\to \R$
$$\E[F(\mf l, \mf r, \cL_1|_{[\mf l,\mf r]}, \ldots, \cL_k|_{[\mf l,\mf r]}) \mid \Fext(k,\mf l, \mf r)] = \E_H^{k,a,b,\vec x, \vec y, \cL_{k+1}}[F(a,b, B_1, \ldots, B_k)],$$
where $\vec x = (\cL_1(\mf l), \ldots, \cL_k(\mf l))$, $\vec y = (\cL_1(\mf r), \ldots, \cL_k(\mf r))$, and $B_1, \ldots, B_k$ are distributed according to $\E_H^{k,a,b,\vec x, \vec y, \cL_{k+1}}$, i.e., Brownian bridges tilted by the Radon-Nikodym factor from Definition~\ref{d.bg}.

In words, the distribution of $\cL_1, \ldots, \cL_k$ on a stopping domain is still given by Brownian bridges with the appropriate endpoints reweighted as in the usual $H$-Brownian Gibbs property, except on the random interval $[\mf l, \mf r]$.
\end{definition}

\begin{lemma}[Lemma~2.5 of \cite{corwin2016kpz} and Lemma~2.5 of \cite{corwin2014brownian}]\label{l.strong bg}
If a line ensemble $\cL$ satisfies the $H$-Brownian Gibbs property, it also satisfies the strong $H$-Brownian Gibbs property, and similarly for the usual Brownian Gibbs property.
\end{lemma}

\begin{definition}[Parabolic Airy line ensemble]\label{d.parabolic Airy line ensemble}
The parabolic Airy line ensemble $\cP: \N\times\R\to \R$ is the line ensemble such that the finite dimensional distributions of the ensemble $\A$ given by $\A_i(x) = \cP_i(x) +x^2$ can be described as follows: for every $m\in\N$ and real $t_1 < \ldots< t_m$, the point process $\{(\A_i(t_j),t_j) : i\in\N, j\in\intint{1,m}\}$ is determinantal with correlation kernel given by the extended Airy kernel $K_\mathrm{Ai}^{\mathrm{ext}}:(\R\times(0,\infty))^2\to\R$, where
\begin{equation}\label{e.extended airy kernel}
K_\mathrm{Ai}^{\mathrm{ext}}\bigl((x,t); (y,s)\bigr) = \begin{cases}
\int_0^\infty e^{-\lambda(t-s)}\mathrm{Ai}(x+\lambda)\mathrm{Ai}(y+\lambda)\, \mathrm d\lambda & t\geq s\\
-\int^0_{-\infty} e^{-\lambda(t-s)}\mathrm{Ai}(x+\lambda)\mathrm{Ai}(y+\lambda)\, \mathrm d\lambda & t< s
\end{cases}
\end{equation}
with $\mrm{Ai}$ the classical Airy function. (The reader is referred to \cite{manjunath} for background on determinantal point processes.)
\end{definition}

Next are the Gibbs properties for $\h$ and $\cP$.

\begin{proposition}[Gibbs properties of $\h$ and $\cP$]\label{p.scaled gibbs}
(i) For each $0<t<\infty$, there exists a line ensemble $\h$ such that the (scaled as in \eqref{e.h_1 definition}) narrow wedge solution to the KPZ equation $\h_1$ is the lowest indexed curve of $\h$. Further, the latter has the $H_t$-Brownian Gibbs property with respect to rate two Brownian bridge, where the Hamiltonian $H_t$ is given by $H_t(x) = 2t^{2/3}\exp(t^{1/3}x)$.

(ii) $\cP$ has the Brownian Gibbs property with respect to rate two Brownian bridge.
\end{proposition}

The existence of the line ensemble $\h$ and its Brownian Gibbs property is proven in \cite{corwin2016kpz} (see also \cite{nica2021intermediate}); in fact, by later results characterizing the law of line ensembles by the distribution of the top curve \cite{dimitrov2021characterizationH}, this line ensemble is also unique in law. The Gibbs property of $\cP$ is established in \cite[Theorem 3.1]{corwin2014brownian}. 

We note that the above specification of the Hamiltonian for $\h$ differs from \cite[Theorem~2.15 (iii)]{corwin2016kpz}, which does not include the $\smash{2t^{2/3}}$ pre-factor in the Hamiltonian. This is due to a neglecting of a Jacobian factor which should be present, and so the correct Hamiltonian for $\h$ is indeed as stated in Proposition~\ref{p.scaled gibbs}. To clarify this point we write out the proof of Proposition~\ref{p.scaled gibbs} (i) in Appendix~\ref{app.hamiltonian}. A different fix involving a certain shift of the curves in the line ensemble to obtain the Hamiltonian of $\exp(t^{1/3}x)$ was adopted in \cite{wu2021tightness}.

Below we record the result of verifying the relevant assumptions for the KPZ, parabolic Airy, and extremal stationary line ensembles for easy reference; the proof will be given in Appendices~\ref{app.monotonicity proofs} and \ref{app.hamiltonian.verifying assumptions}. Recall the definition of an extremal stationary line ensemble from after Conjecture~\ref{conj.extremality}.

\begin{theorem}\label{t.assumptions hold}
The parabolic Airy and extremal stationary line ensembles satisfy Assumptions~\ref{as.bg}--\ref{as.tails} with $t=\infty$, as well as Assumptions~\ref{as.corr}(b\ensuremath{'}) and \ref{as.mono in cond stronger}. The KPZ line ensemble $\h$ satisfies Assumptions~\ref{as.bg}, \ref{as.corr}(a), \ref{as.weak bk}(b\ensuremath{'}), \ref{as.mono in cond stronger}, and \ref{as.tails} for any $t>0$, and, for any fixed $t_0>0$, the constants in Assumptions~\ref{as.weak bk}(b\ensuremath{'}) and \ref{as.tails} may be taken uniform over all $t>t_0$.

As a result, Theorems~\ref{mt.one point density asymptotics}--\ref{mt.fkg sharpness} and \ref{mt.one point limit shape}--\ref{mt.two-point limit shape} all hold for the parabolic Airy, extremal (in particular, implying Theorem~\ref{mt.extremal}), and KPZ line ensembles.
\end{theorem}

\subsection{Monotonicity tools}\label{s.monotonicity tools}

There are two monotonicity tools. The first is monotonicity in the boundary data of the Brownian bridges under both hard and soft constraints, Lemma~\ref{l.monotonicity}. The second, Lemma~\ref{l.bound point conditioning by tail conditioning}, is a straightforward consequence of Assumption~\ref{as.mono in cond} on monotonicity in conditioning, and relates conditioning on point values like $\cL_1(0)=\theta$ with conditioning on positive probability events like $\{\cL_1(0)\geq \theta\}$ (the latter has the benefit of being an increasing event).

\begin{lemma}[Monotonicity in boundary data]\label{l.monotonicity}
Fix $t>0$, $k_1\leq k_2\in \Z$, $a<b$, two pairs of vectors $\smash{w^{(i)}, z^{(i)}}\in \R^{k_2-k_1+1}$ and two pairs of measurable functions $\smash{(f^{(i)},g^{(i)})}$ for $i\in \{1,2\}$ such that $\smash{w^{(1)}_{j}\leq w^{(2)}_{j}}$ and $\smash{z^{(1)}_{j}\leq z^{(2)}_{j}}$ for all $j=k_1,\ldots,k_2$ and $\smash{f^{(i)}}:(a,b)\rightarrow \R\cup\{\infty\}$, $\smash{g^{(i)}}:(a,b)\rightarrow \R\cup\{-\infty\}$ and for all $s\in (a,b)$, $\smash{f^{(1)}(s)\leq f^{(2)}(s)}$ and $\smash{g^{(1)}(s)\leq g^{(2)}(s)}$. 

For $i\in \{1,2\}$, let $\smash{\mathcal{Q}^{(i)}=\{\mathcal{Q}^{(i)}_j\}_{j=k_1}^{k_2}}$ be a $\{k_1,\ldots,k_2\}\times (a,b)$-indexed line ensemble such that $\smash{\mathcal{Q}^{(i)}}$ has the $H_t$-Brownian Gibbs property with entrance data $\smash{w^{(i)}}$, exit data $\smash{z^{(i)}}$ and boundary data $\smash{(f^{(i)},g^{(i)})}$).

There exists a coupling of the laws of $\smash{\{\mathcal{Q}^{(1)}_j\}}$ and $\smash{\{\mathcal{Q}^{(2)}_j\}}$ such that almost surely $\smash{\mathcal{Q}^{(1)}_j(s)}\leq \smash{\mathcal{Q}^{(2)}_j(s)}$ for all $j\in \{k_1,\ldots, k_2\}$ and all $s\in (a,b)$.

The same is true in the $t=\infty$ (zero-temperature case) if additionally $\smash{w_j^{(i)} > w_{j+1}^{(i)}}$ and $\smash{z_j^{(i)} > z_{j+1}^{(i)}}$ for $j=k_1, \ldots, k_2-1$ and $i=1,2$, and $\smash{f^{(i)}(a) > w_{k_1}^{(i)}}$, $\smash{f^{(i)}(b) > z_{k_1}^{(i)}}$, $\smash{g^{(i)}(a) < w_{k_2}}$, $\smash{g^{(i)}(b) < z_{k_2}}$ for $i=1,2$.
\end{lemma}

\begin{proof}
The positive temperature ($t<\infty$) statements are Lemmas~2.6 and 2.7 of \cite{corwin2016kpz}. The zero temperature ($t=\infty$) statements are Lemmas~2.6 and 2.7 of \cite{corwin2014brownian}. See also \cite{dimitrov2021characterization} and \cite{dimitrov2021characterizationH} for more detailed proofs of the respective cases.
\end{proof}

\begin{lemma}[Monotonicity in conditioning]\label{l.bound point conditioning by tail conditioning}
Suppose $\cL$ satisfies Assumption \ref{as.mono in cond}. Let $F$ be an increasing function of $(\cL_1, \cL_2, \ldots)$. Let $y_1, y_2\in \R$ and $E_1, E_2\subseteq \R$ be (possibly infinite) intervals such that $\inf E_i = y_i$ and $y_i\in E_i$ for $i=1,2$. Then, for any $\theta>0$,
$$\E\left[F \ \big|\  \cL_1(-\theta^{1/2}) = y_1, \cL_1(\theta^{1/2}) = y_2\right] \leq \E\left[F \ \big|\  \cL_1(-\theta^{1/2}) \in E_1, \cL_1(\theta^{1/2}) \in E_2\right].$$%
Similarly, if $F$ is a decreasing function and $E_i$ are intervals such that $\sup E_i = y_i$ and $y_i\in E_i$ for $i=1$ and 2,
$$\P\left[F \ \big|\  \cL_1(-\theta^{1/2}) = y_1, \cL_1(\theta^{1/2}) = y_2\right] \leq \E\left[F \ \big|\  \cL_1(-\theta^{1/2}) \in E_1, \cL_1(\theta^{1/2}) \in E_2\right].$$
The same holds if the conditioning is at a single point, such as $0$, instead of two points $\pm\theta^{1/2}$ as stated above.
\end{lemma}

As mentioned, the proof of this is a straightforward consequence of Assumption~\ref{as.mono in cond} and is deferred to Appendix~	\ref{app.mono.proof of mono in cond lemma}.

\subsection{A bound on the tail of $\sup_{|x|\leq 1}(\cL_1(x)+x^2)$}\label{s.tools.sup tail}

Next we record a result which bounds the tail of the supremum of $\cL_1$ over an interval in terms of the one-point tail. This will be useful in Section~\ref{s.general data} on tail bounds for general initial data.

\begin{proposition}\label{p.sharp sup over interval tail}
Let $K$ be such that $\P(\cL_1(0) \leq -K) \leq \frac{1}{4}$. Under Assumptions~\ref{as.bg} and \ref{as.corr}(a), there exists $\theta_0 = \theta_0(K)$ such that, for all $\theta>\theta_0$,
$$\P\left(\sup_{x\in [-1,1]} \left(\cL_1(x)+x^2\right)\geq \theta\right) \leq 4\theta\cdot \P\left(\cL_1(0)\geq \theta-2\right).$$
\end{proposition}

Observe that if we know that the one-point tail decays at least stretched exponentially, as we assume in Assumption~\ref{as.tails}, then the tail decay at the level of the exponent is preserved for the supremum. And indeed, if we know that the one-point tail is $\exp(-\frac{4}{3}\theta^{3/2}(1+o(1)))$, the same gets transferred to the above supremum's tail.

The proof of Proposition~\ref{p.sharp sup over interval tail} is a refinement of the ``no big max'' argument given in \cite[Proposition~2.27]{hammond2016brownian} and \cite[Proposition~4.4]{corwin2014brownian}, and will be given in Appendix~\ref{s.gen data.no big max}.

\subsection{Gaussian and Brownian bridge estimates}\label{s.tools.brownian}
Here we recall well-known bounds for Gaussian random variables and Brownian bridges.
We start with a standard bound on the tail of centered Gaussian random variables. Here and in the rest of the paper, $\mc N(\mu, \sigma^2)$ represents a random variable distributed according to the normal distribution with mean $\mu$ and variance $\sigma^2$.

\begin{lemma}\label{l.normal bounds}
For $x\geq (4/3)^{1/2}\sigma$,
$$\frac{1}{\sqrt{2\pi}}\cdot\frac{\sigma}{4x}\exp\left(-\frac{x^2}{2\sigma^2}\right) \leq \P\left(\mc N(0,\sigma^2) \geq x\right) \leq \exp\left(-\frac{x^2}{2\sigma^2}\right).$$
\end{lemma}

\begin{proof}
We may set $\sigma=1$ without loss of generality. The upper bound is simply the Chernoff bound. For the lower bound it is well-known via integration by parts that $\P(\mc N(0,1) \geq x) \geq (2\pi)^{-1/2} (x^{-1}-x^{-3}) \exp(-x^2/2)$. Using that $x\geq (4/3)^{1/2}$ gives the result.
\end{proof}

Next we recall the well-known distribution of the supremum of a standard Brownian bridge $B$ defined on an interval $I$. Then we use it to prove Lemma~\ref{l.brownian bridge restricted sup tail}, which states a bound on the tail of $\sup_{J} B$ for an interval $J\subseteq I$ such that the bound adapts to the variance of $B$ on $J$.

\begin{lemma}[Equation (3.40) in chapter 4 of \cite{revuz2013continuous}]\label{l.brownian bridge sup tail exact}
Let $I=[a,b]\subseteq \R$ be an interval and $B:I\to\R$ a Brownian bridge with $B(a) = B(b) = 0$. Let $\sigma_I^2 = \max_{x\in I} \Var(B(x)) = |I|/4$. For any $M>0$, 
$$\P\left(\sup_{x\in I} B(x) \geq M\sigma_I\right) = \exp\left(-\frac{1}{2}M^2\right).$$
\end{lemma}

The following is the mentioned tail bound on $\sup_J B$. It will be useful to use in obtaining estimates on probabilities of events such as that a Brownian bridge stays above a parabola or a line. It is proved in Appendix~\ref{app.brownian estimates}.

\begin{lemma}\label{l.brownian bridge restricted sup tail}
Let $I=[a,b], J\subseteq \R$ be intervals with $J\subseteq I$, and $B:I\to\R$ be a Brownian bridge with $B(a) = B(b)=0$. Let $\sigma_J^2 = \max_{x\in J} \Var(B(x))$. Then, for all $M>0$,
$$\P\left(\sup_{x\in J} B(x) \geq M\sigma_J\right) \leq 3\exp\left(-\frac{1}{8}M^2\right).$$
\end{lemma}

We note that a somewhat analogous statement giving a tail bound in terms of $\sigma_J$ as defined above is an immediate consequence of the famous Borell-TIS inequality for Gaussian processes (see e.g., \cite[Theorem~2.1.1]{adler2007random}); however, the bound coming from the Borell-TIS inequality would only hold for $M$ larger than $\E[\sup_{J}B]$, which grows with $\sigma_J$. This can be a technical nuisance compared to the above statement which only needs $M>0$.

\subsection{Estimates on parabolic avoidance probabilities}

The Brownian Gibbs property imposes a lower boundary condition of $\cL_2$; as is exactly true in the $t=\infty$ case, the lower boundary can be thought of as a curve which the top curve must avoid. Further we know that $\cL_2$ decays like a parabola $-x^2$ and typically has unit order separation from $\cL_1$.

Given these facts, it will be crucial to have a good understanding of the probability that a Brownian bridge stays above a parabola, where the starting and ending points of the bridge are at least order one away from the parabola. Such a bound is the next statement.

\begin{proposition}\label{p.parabola avoidance probability}
Let $z_1<z_2$ and $B$ be a rate two Brownian bridge on $[z_1, z_2]$ with endpoints at least as high as $(z_1,-z_1^2+1)$ and $(z_2,-z_2^2+1)$. Then, for all $z_2-z_1$ large enough,
$$\P\Bigl(B(x) > -x^2 \ \text{ for all } x\in[z_1, z_2]\Bigr) \geq \exp\left(-\frac{1}{12}(z_2-z_1)^3 - 2(z_2-z_1)\log(z_2-z_1)\right)$$
\end{proposition}

\begin{remark}
The first order term in the exponent looks very similar to the lower tail asymptotics of $\cL_1(0)$ (when not too deep in the tail) and the parabolic Airy$_2$ process, namely $\P(\cL_1(0) \leq -\theta) \approx \exp(-\frac{1}{12}\theta^3)$ (see for example \cite[Theorem~1.1]{corwin2020lower} and \cite[Theorem~1.3]{ramirez2011beta}). We do not know if this is a coincidence or if some deeper phenomenon is present.
\end{remark}

\begin{remark}\label{r.avoidance prob upper bound}
Perhaps somewhat surprisingly, an upper bound which matches Proposition~\ref{p.parabola avoidance probability} to first order can be obtained rather quickly. We illustrate this in the simpler symmetric case that $z_1 = -z_2 = -z$ for convenience. We assume that $B$ is a Brownian bridge on $[-z,z]$ with endpoints equal to $-z^2+1$. In that case the lower bound from Proposition~\ref{p.parabola avoidance probability} is, to first order in the exponent, $\exp(-2z^3/3)$.

Let $\tilde B$ be a rate two Brownian bridge from $(-z,0)$ to $(z,0)$, i.e., $B$ shifted to have mean zero. Observe that
\begin{align*}
\P\left(B(x) > -x^2\ \  \forall x\in[-z,z]\right)
&= \P\left(\tilde B(x) - z^2+1 > -x^2\ \ \forall x\in[-z,z]\right)\\
&\leq \P\left(\int_{-z}^z (\tilde B(x) - z^2 + x^2 + 1) \, \mathrm dx > 0\right).
\end{align*}
Now $\int_{-z}^z \tilde B(x)\,\dif x$ is a normal random variable with mean zero and variance given by
\begin{align*}
\E\left[\left(\int_{-z}^z \tilde B(s)\, \mathrm ds\right)^2\right] &= \int_{-z}^z\int_{-z}^z \mathrm{Cov}(\tilde B(s), \tilde B(t))\,\mathrm ds\,\mathrm dt\\
&= 2\iint\limits_{-z<s<t<z} 2\cdot\frac{(s+z)(z-t)}{2z}\,\mathrm ds\,\mathrm dt
= \frac{4}{3}z^3.
\end{align*}
Inputting this into the earlier display and evaluating the integrals of the constant terms there gives that
\begin{align*}
\P\left(B(x) > -x^2\ \  \forall x\in[-z,z]\right) \leq \P\left(\mc N(0, 4z^3/3) > 4z^3/3-2z\right).
\end{align*}
The last probability is at most $\exp(-2z^3/3)$ by Lemma~\ref{l.normal bounds}, up to first order in the exponent.

The fact that a matching upper bound could be obtained by considering the integral may initially appear surprising, but can be explained by observing that the parabola is the function which minimizes the Brownian energy functional under the constraint that the function's integral be larger than a given level (i.e., the function encloses at least a certain area).
\end{remark}

The proof of Proposition~\ref{p.parabola avoidance probability} is straightforward but slightly tedious, and relies on first demanding that the Brownian bridge remain above $-x^2$ on a fine mesh, and then controlling the fluctuation of the bridge on the intervals in between the mesh points. It is deferred to Appendix~\ref{app.brownian estimates}.

As was indicated in the proof outline in Section~\ref{s.intro.proof ideas}, we will need to have lower bounds also for partition functions with similar parabolic boundary data. The following is a general tool to convert lower bounds on non-avoidance probabilities to lower bounds on analogous partition functions. Though simple to prove, as we will see, this already yields sharp estimates needed to transfer our arguments from zero temperature to positive temperature.

\begin{lemma}\label{l.pos temp Z  lower bound via non-avoid prob}
Suppose $t>0$ and $z_1<z_2$. Let $p$ be a given lower boundary curve. Let $x > -z_1^2$ and $y > -z_2^2$, and let $B$ be a rate two Brownian motion from $(z_1,x)$ to $(z_2, y)$. Then, for any measurable function $g:[z_1,z_2]\to \R$, $Z_{H_t}^{1,z_1, z_2, x, y, p}$ is lower bounded by
$$\exp\left(-2t^{2/3}e^{-t^{1/6}}\int_{z_1}^{z_2}\exp(-t^{1/3}g(u))\,\mathrm du)\right)\cdot\P\left(B(u) > p(u) + g(u) + t^{-1/6} \ \forall u\in[z_1,z_2]\right).$$
\end{lemma}

Here, the fact that we give a buffer of $t^{-1/6}$ on the right-hand side, and in particular the constant $\smash{\frac{1}{6}}$, has no significance; it is merely to ensure that we get a decay factor, in this case $\exp(-t^{1/6})$, which goes to zero faster than $t^{2/3}$ as $t\to\infty$. 

\begin{proof}
The proof essentially amounts to a form of Markov's inequality: for any event $A$, $Z_{H_t}^{1,z_1, z_2, x, y, p} = \Efree^{1,z_1,z_2,x,y}[W^{1,z_1,z_2, p}_{H_t}(B)] \geq \E[W^{1,z_1,z_2, p}_{H_t}(B)\one_{B\in A}] \geq \inf_{f\in A}W^{1,z_1,z_2, p}_{H_t}(f)\Pfree^{1,z_1,z_2,x,y}(B\in A)$. Indeed, taking $A = \{B(u) > p(u) + g(u) + t^{-1/6}\ \forall u\in[z_1,z_2]$\} yields
\begin{align*}
\MoveEqLeft[1]
Z_{H_t}^{1,z_1, z_2, x, y, p}\\
&\geq \exp\left(-2t^{2/3}e^{-t^{1/6}}\int_{z_1}^{z_2}\exp(-t^{1/3}g(u))\,\mathrm du\right)\cdot\P\bigl(B(u) > p(u) + g(u) + t^{-1/6} \ \forall u\in[z_1,z_2]\bigr).\qedhere
\end{align*}
\end{proof}

The following proposition lower bounding the normalizing constant in positive temperature is an easy consequence of the zero-temperature calculation we performed in Proposition~\ref{p.parabola avoidance probability} and the just-proved Lemma~\ref{l.pos temp Z  lower bound via non-avoid prob}.

\begin{corollary}\label{c.pos temp parabolic avoidance lower bound}
Suppose $t>0$ and $z_1<z_2$. Let $p$ be given by $p(u) = -u^2$. Let $x \geq -z_1^2 + 1 + t^{-1/6}$, $y \geq -z_2^2+1+ t^{-1/6}$. Then, for all $z_2-z_1$ larger than an absolute constant, 
$$Z_{H_t}^{1,z_1, z_2, x, y, p} \geq \exp\left(-\frac{1}{12}(z_2-z_1)^3 - 3(z_2-z_1)\log(z_2-z_1)\right).$$
\end{corollary}

\begin{proof}
We apply Proposition~\ref{p.parabola avoidance probability} and Lemma~\ref{l.pos temp Z  lower bound via non-avoid prob} with $p(u) = -u^2$ and $g(u) \equiv 0$, and then observe by direct calculation that $t^{2/3}e^{-t^{1/6}} \leq 5$ for all $t>0$. Finally we use that $10 \leq\log(z_2-z_1)$ for sufficiently large $z_2-z_1$.
\end{proof}

\begin{remark}\label{r.upper bound on Z for parabola}
As in the zero-temperature case discussed in Remark~\ref{r.avoidance prob upper bound}, we can prove a matching upper bound on $Z_{H_t}^{1,z_1,z_2,x,y,p}$ when $t\geq t_0$ (with an error term depending on $t_0$) and $p(u) = -u^2$ . We show it in the simpler case that $z_2=-z_1 = z$ and $x=y=-z^2+1+t^{-1/6}$ for convenience. Indeed,
\begin{align*}
Z_{H_t}^{1,-z, z, x, y, p} 
&= \Efree^{1,-z,z,x,y}\left[\exp\left\{-\int_{-z}^{z} 2t^{2/3}\exp\left(t^{1/3}(p(u) - B(u))\right)\,\mathrm du\right\}\right]\\
&\leq \Efree^{1,-z,z,x,y}\left[\exp\left\{-\int_{-z}^{z} 2t^{2/3}\exp\left(t^{1/3}(p(u) - B(u))\right)\one_{B(u) < p(u)}\,\mathrm du\right\}\right]\\
&\leq \Efree^{1,-z,z,x,y}\left[\exp\left\{-\int_{-z}^{z} 2t^{2/3}(-C+(2t_0)^{-1}t^{1/3}(p(u) - B(u)))\one_{B(u) < p(u)}\,\mathrm du\right\}\right],
\end{align*}
using that $\exp(w) \geq -C+(2t_0)^{-1}w$ for all $w\in\R$ for some $C=C(t_0)$; this can be seen to be true for some $C$ immediately using the convexity of $w\mapsto\exp(w)$, and simple calculus shows that $C=(2t_0)^{-1}(\log((2t_0)^{-1})-1)$ works. Using the monotonicity in $t$ of the integrand and that $t\geq t_0$, we see that the previous expression is at most
\begin{align*}
\MoveEqLeft[10]
e^{4Czt^{2/3}}\cdot\Efree^{1,-z,z,x,y}\left[\exp\left\{-\int_{-z}^{z} t_0^{-1}t(p(u) - B(u))\one_{B(u) < p(u)}\,\mathrm du\right\}\right]\\
&\leq e^{4Czt^{2/3}}\cdot\Efree^{1,-z,z,x,y}\left[\exp\left\{\int_{-z}^{z} (B(u) - p(u))\one_{B(u) < p(u)}\,\mathrm du\right\}\right]\\
&\leq e^{4Czt^{2/3}}\cdot\Efree^{1,-z,z,x,y}\left[\exp\left\{\int_{-z}^{z} (B(u) - p(u))\,\mathrm du\right\}\right].
\end{align*}
Now, $\int_{-z}^{z} -p(u)\,\mathrm du = 2z^3/3$, while $\int_{-z}^{z} B(u)\,\mathrm du$ is a normal random variable with mean $-z^2+1+t^{-1/6}$ and variance $4z^3/3$ (by the calculation in Remark~\ref{r.avoidance prob upper bound}).
So, integrating the mean of the Gaussian separately, the previous upper bound on the normalizing constant is equal to
\begin{align*}
\MoveEqLeft[16]
\exp\left(-2z^3+2z+2zt^{-1/6} + \frac{2}{3}z^3 + 4Czt^{2/3}\right)\E\left[\exp(\mc N(0,4z^3/3))\right]\\
&= \exp\left(-\frac{4}{3}z^3+2z+2zt^{-1/6}+\frac{2z^3}{3}+4Czt^{2/3}\right)\\
&\leq \exp\left(-\frac{2}{3}z^3+6z+4Czt^{2/3}\right).
\end{align*}
Observe from our expression for $C(t_0)$ that we may take it to be 0 if $t_0\geq (2e)^{-1}$.
\end{remark}

\subsection{Estimates on Brownian fluctuations above parabolas}\label{s.tools.brownian above para}

While Proposition~\ref{p.parabola avoidance probability} stated a lower bound on the probability of a Brownian bridge staying above a parabola when its endpoints were close to the parabola, we will often be considering cases where the bridge's endpoints are much higher than the parabola, such that the line connecting the endpoints is tangent to the parabola. In such cases we will need to know that the probability of avoidance is basically of constant order.

In the following proposition we make essentially that statement. See Figure~\ref{f.tangent avoidance} for a depiction of the setting of Proposition~\ref{p.brownian versatile tangent estimate}.

\begin{figure}[t]
\includegraphics[scale=1.3]{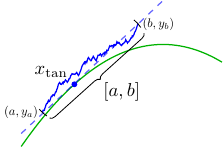}
\caption{A depiction of the setup of Proposition~\ref{p.brownian versatile tangent estimate}. There is a Brownian bridge $B$ (in blue) defined on an interval $[a,b]$, and we consider the probability that it avoids the parabola (in green) plus a quantity $\varepsilon M \sigma_\tan$, where $\sigma_\tan^2$ is the variance of $B$ at $x_\tan$; because $x_\tan$ is the point at which the line joining $B$'s endpoints is closest to the parabola, the probability of $B$ avoiding $-x^2$ in its neighbourhood is the governing contribution to the overall non-intersection probability, and it is for this reason that we scale the demanded separation between $B$ and $-x^2$ by $\sigma_\tan$.}\label{f.tangent avoidance}
\end{figure}

\begin{proposition}\label{p.brownian versatile tangent estimate}
Let $[a,b] \subseteq \R$ be a finite interval. Suppose $x_\tan\in [a + 1, b-1]$, and let $(a, y_{a})$ and $(b, y_{b})$ be points on the line $\ellt: [a,b]\to\R$ tangent to the curve $-x^2$ at the point $(x_\tan, -x_\tan^2)$. Let $B$ be a rate two Brownian bridge from $(a, y_{a})$ to $(b, y_{b})$ and $\sigma^2_\tan = \Var(B(x_\tan))$. Then there exist constants $C >0$, $\varepsilon_0 >0$, and $M_0$ (all universal) such that, for any $M_0 < M < C^{-1}\min(x_\tan-a, b -x_\tan)^{3/2}$,
\begin{equation}\label{e.brownian versatile tangent to bound}
\P\left(B(x) > -x^2 + \varepsilon_0 M\sigma_\tan\ \forall x\in [a,b]\right) \geq C^{-1}\exp(-M^2).
\end{equation}
There exists a constant $\varepsilon_0>0$ (independent of $t$) such that the same bound holds with the left-hand side equal to $Z_{H_t}$ for any $t>0$, where $Z_{H_t}$ is the partition function on $[a,b]$ with boundary data $(a , y_{a})$ and $(b, y_{b})$ and lower boundary curve $-x^2 + \varepsilon_0 M\sigma_\tan$.
\end{proposition}

Proposition~\ref{p.brownian versatile tangent estimate} is also proved in Appendix~\ref{app.brownian estimates}.

We remind the reader of our remark after Theorem~\ref{mt.one point limit shape} that the upper bound on $M$ in many of our estimates arises for simplicity, as beyond that range a KPZ type tail instead of a Gaussian one will be present. Here, however, the upper bound on $M$ has a different reason: without it, the probability would be zero, as $B$ would not satisfy the event at $x=a$ or $b$ where its value is deterministic.

We prove Proposition~\ref{p.brownian versatile tangent estimate} by controlling point values and fluctuations of $B$ on a sequence of dyadic scales. This multi-scale argument is necessary because we allow the tangency point $x_\tan$ to be very close to the boundaries of $[a,b]$, independent of the size of the interval; allowing this closeness is necessary for our arguments, presented in Section~\ref{s.extremal}, verifying that extremal ensembles satisfy Assumption~\ref{as.tails} (tail probability bounds). One can observe some of the delicacy of the estimate by noting that if the tangency location \emph{was} the boundary, then the left-hand side of \eqref{e.brownian versatile tangent to bound} would be zero (similar to why we impose the upper bound on $M$). It is also to ensure that the probability is uniformly bounded below that we consider fluctuations on the scale of $\sigma_\tan$.

The situation where Proposition~\ref{p.brownian versatile tangent estimate} will often be needed is to control the probability of a Brownian bridge, reweighted by the Radon-Nikodym derivative associated to the $H_t$-Brownian Gibbs property, deviating from the line joining its endpoints. This can be accomplished by controlling the ratio of a Brownian bridge deviation probability and the non-intersection probability or partition function, since the numerator $W_t$ of the reweighting factor is at most $1$ always, and we record such a bound as an immediate corollary of Lemma~\ref{l.brownian bridge restricted sup tail} and Proposition~\ref{p.brownian versatile tangent estimate} next.

\begin{corollary}\label{c.ratio of deviation prob and part func}
Let $[a,b]$, $x_\tan$, and $B$ be as in Proposition~\ref{p.brownian versatile tangent estimate} and satisfy the same assumptions. Let $J\subseteq [a,b]$ be an interval and $\sigma_J^2 = \sup_{x\in J}\Var(B(x))$. Then there exist constants $C >0$, $c>0$, $\varepsilon_0 >0$, and $M_0$ (all universal) such that, for any $M_0 < M < C^{-1}\min(x_\tan-a, b -x_\tan)^{3/2}$,
$$\frac{\P\left(\sup_{x\in J} B(x) - \ellt(x) \geq M\sigma_J\right)}{\P\left(B(x) > -x^2 + \varepsilon_0 M\sigma_\tan\ \forall x\in [a,b]\right)} \leq \exp(-cM^2).$$
The same bound holds with the denominator equal to $Z_{H_t}$ (with $\varepsilon_0$ and $c$ independent of $t$) for any $t>0$, where $Z_{H_t}$ is the partition function on $[a,b]$ with boundary data $(a, y_{a})$ and $(b, y_{b})$ and lower boundary curve $-x^2 + \varepsilon_0 M\sigma_\tan$.
\end{corollary}

Here $\varepsilon_0$ will be slightly different from its value in Proposition~\ref{p.brownian versatile tangent estimate} so as to successfully play off the constants in the exponent coming from the latter and Lemma~\ref{l.brownian bridge restricted sup tail}.

The scales $\sigma_J$ and $\sigma_\tan$ respectively of the fluctuations of the bridge and the height above the parabola that we demand the bridge stay above are chosen such that the ratio of the probabilities is small and independent of $J$ or the location of $x_{\tan}$.

The next result concerns the fluctuations of a Brownian bridge which is conditioned to stay above a parabola. It will allow us to control the shape of the profile (conditional on $\cL_1(0)  = \theta$) beyond $[-\theta^{1/2},\theta^{1/2}]$ and say that it is close to the parabola $-x^2$, as asserted in Proposition~\ref{p.para fluctuation h_t}.

\begin{proposition}\label{p.para fluctuation}
Let $I=[a,b]$ be an interval. For $0< H\leq |I|^{1/2}\log|I|$, let $\bpara:I\to\R$ be a Brownian bridge from $(a, -a^2 + H)$ to $(b, -b^2+ H)$ conditioned on $\bpara(x) > -x^2$ for all $x\in I$. Then there exist absolute constants $C,c>0$, such that for $|I|>C$ and $M>\log |I|$,
$$\P\left(\sup_{x\in I} \left(\bpara(x) + x^2\right) > 2M|I|^{1/2}\right) \leq \exp(-cM^2).$$
The same statement holds uniformly in $t>0$ if we replace $\bpara$ by a Brownian bridge $\bparat$ between the same endpoints as $\bpara$ but tilted by the Radon-Nikodym derivative given by $W_{H_t}/Z_{H_t}$, where the latter is associated to the same boundary data and lower boundary curve $-x^2$.
\end{proposition}

Note that we impose that $M>\log |I|$, unlike previous estimates where $M$ was lower bounded just by an absolute constant. This is to allow us to perform a union bound; essentially, we approximate the parabola by $O(|I|)$ many line segments which are close to the parabola for a unit order length, control fluctuations of Brownian bridges defined on these line segments, and take a union bound over all of them.

\begin{proof}[Proof of Proposition~\ref{p.para fluctuation}]
We will adopt the notation that $\bpara = \bparat$ with $t=\infty$, and prove the proposition for $\bparat$ for all $t > 0$; to get the $t=\infty$ case, one needs to just replace applications of the $H_t$-Brownian Gibbs property by the usual Brownian Gibbs property.

We will show that there exists $c>0$ such that, for any $k\in[a,b-1]$,
\begin{equation}\label{e.para fluc to prove}
\P\left(\sup_{x\in[k,k+1]} \left(\bparat(x) + x^2\right) > 2M|I|^{1/2}\right) \leq \exp(-cM^2).
\end{equation}
This suffices to prove the proposition by a union bound over $O(|I|)$ many values of $k$, since we have set $M>\log|I|$ and $|I|$ large enough.

Now \eqref{e.para fluc to prove} is easy to prove by considering the line $\ellt_k$ which is tangent to $-x^2$ at $k$. Indeed, let $B_k$ be a Brownian bridge from $(a, \ellt_k(a)+|I|^{1/2}\log|I|)$ to $(b, \ellt_k(b)+|I|^{1/2}\log|I|)$. By concavity of the function $-x^2$ and the assumed upper bound on $H$, these points are above $(a, -a^2+H)$ and $(b, -b^2+H)$ respectively. Let $\tilde W_{H_t}$ and $\tilde Z_{H_t}$ be the Boltzmann weight and partition function associated to the boundary values of $B_k$ and the lower boundary curve $-x^2 + |I|^{1/2}\log|I|$; note that $\ellt_k(x) + |I|^{1/2}\log|I|$ is tangent to this lower boundary curve at $x=k$. So, by the $H_t$-Brownian Gibbs property and monotonicity (Lemma~\ref{l.monotonicity}), $B_k$, tilted by $\tilde W_{H_t}/\tilde Z_{H_t}$, is stochastically larger than $\bparat$. Thus,
\begin{align*}
\MoveEqLeft[6]
\P\left(\sup_{x\in[k,k+1]} \left(\bparat(x) + x^2\right) > 2M|I|^{1/2}\right)\\
&\leq \tilde Z_{H_t}^{-1}\E\left[\one_{\sup_{x\in[k,k+1]} \left(B_k(x) + x^2\right) > 2M|I|^{1/2}} \tilde W_{H_t}\right].
\end{align*}

Next we observe that on $[k,k+1]$, $+x^2$ differs from $-\ellt_k(x)$ by a unit order. Combining this with $\tilde W_{H_t}\leq 1$, the right-hand side is upper bounded by
$$\tilde Z_{H_t}^{-1}\P\left(\sup_{x\in[k,k+1]} \left(B_k(x) - \ellt_k(x)\right) > \tfrac{3}{2}M|I|^{1/2}\right).$$
By Corollary~\ref{c.ratio of deviation prob and part func},
 this is upper bounded by $\exp(-cM^2)$, so we are done.
\end{proof}

\subsection{Control on the second curve conditional on the first}\label{s.tools.kpz analogues}

As discussed in the proof ideas Section~\ref{s.intro.proof ideas}, we will need some control on the second curve conditional on the first one. The source of this control is Assumption~\ref{as.weak bk}(b\ensuremath{'}), but here we derive a more convenient, albeit more technical looking formulation. Instead of attempting to justify the details of its form, let us simply point out that it allows us to condition on multiple points of $\cL_1$ rather than a single point as in Assumption~\ref{as.weak bk}(b\ensuremath{'}) and also allows us to simultaneously handle events of the first curve being larger than a given function, such as those which appear in the limit shape results Theorems~\ref{mt.one point limit shape}--\ref{mt.two-point limit shape}. As a result, this statement will be frequently used in the upcoming arguments. The proof is a fairly straightforward consequence of Assumptions~\ref{as.corr}(b\ensuremath{'}) and \ref{as.mono in cond stronger}. This is the only location where Assumption~\ref{as.mono in cond stronger} is required rather than \ref{as.mono in cond}.

\begin{lemma}\label{l.lower curve control new}
Suppose $\cL$ satisfies Assumptions~\ref{as.weak bk}(b\ensuremath{'}) and \ref{as.mono in cond stronger}. Then there exists $C>0$ such that the following holds. For an interval $I$ and $R>0$, let $E_{I,R} = \{\sup_{x\in I} (\cL_2(x)+x^2) > (\log R)^C\}$ and $f:I\to\R\cup\{-\infty\}$ be upper semicontinuous. 

Fix any $m\in\N$, real numbers $x_1, \ldots, x_m\in I$, and bounded Borel sets $E_1, \ldots, E_m\subseteq \R$, and let $A = \cap_{i=1}^m\{g \in \mc C(I, \R): g(x_i) \in E_i\}$ if $m>1$ and $\mc C(I, \R)$ if $m=1$. 
Next, fix $j\in\intint{1,m}$ and let $\theta_j = \sup E_j$. Let $A^\shortuparrow = \cap_{i\in\intint{1,m}\setminus\{j\}}\{g \in \mc C(I, \R): g(x_i) \geq \sup E_i\}$ if $m>1$ and $\mc C(I, \R)$ if $m=1$.

 If $I$ and $R,K>0$ satisfy $R>C(\theta_j+K + x_j^2)^{3/4}$ and $\log |I| \leq (\log R)^{2}$, then
\begin{align*}
\P\left(E_{I,R}, \sup_{I}(\cL_1 -f )\geq 0 \midd \cL_1|_{I} \in A\right) \leq \frac{\P\left(\sup_{I}(\cL_1-f)\geq 0 \mid \cL_1|_{I} \in A\right)}{2\cdot \P\left(\inf_I (\cL_1 - f) \geq 0, \cL_1|_I \in A^{\shortuparrow} \mid \cL_1(x_j) = \theta_j + K\right)}.
\end{align*}
\end{lemma}

\begin{proof}%
First, trivially,
\begin{align*}
\MoveEqLeft[6]
\P\left(E_{I,R}, \sup_{I}(\cL_1-f)\geq 0 \midd \cL_1|_{I} \in A\right)\\
&= \P\left(\sup_{I}(\cL_1-f)\geq 0 \midd \cL_1|_{I} \in A\right)\cdot \P\left(E_{I,R} \midd \sup_{I}(\cL_1-f)\geq 0, \cL_1|_{I} \in A\right).
\end{align*}
We now apply monotonicity in conditioning (Assumption~\ref{as.mono in cond stronger}) to modify the conditioning in the second factor. Indeed, by an averaging argument as in the proof of Lemma~\ref{l.bound point conditioning by tail conditioning}, since every member of $A^\shortuparrow\cap\{f: f(x_j)=\theta_j+K\}$ is larger than every member of $A$ on $\{x_1, \ldots, x_m\}$, and $E_{I,R}$ is an increasing event, we obtain from Assumption~\ref{as.mono in cond stronger} that 
\begin{align}\label{e.weak bk after mono in cond}
\P\left(E_{I,R} \midd \sup_I(\cL_1 - f) \geq 0, \cL_1|_{I} \in A\right) \leq \P\left(E_{I,R} \midd \inf_I(\cL_1 - f) \geq 0, \cL_1|_{I} \in A^{\shortuparrow}, \cL_1(x_j) = \theta_j+K\right).
\end{align}
Then by the trivial bound, we obtain that
\begin{align*}
\P\left(E_{I,R} \midd \sup_I(\cL_1 - f) \geq 0, \cL_1|_{I} \in A\right)
&\leq \frac{\P\left(E_{I,R} \midd \cL_1(x_j) = \theta_j + K\right)}{\P\left(\inf_I(\cL_1-f)\geq 0, \cL_1|_I \in A^{\shortuparrow} \midd \cL_1(x_j) = \theta_j+ K\right)}.
\end{align*}
By Assumption~\ref{as.mono in cond} we may replace the conditioning in the numerator by $\cL_1(x_j)\geq \theta+j+K$. Then by Assumption~\ref{as.weak bk}(b\ensuremath{'}), there exists $C$ independent of $I, x_j, \theta_j$, $K$, and $R$ such that the numerator is at most $\frac{1}{2}$, as long as $K$ satisfies $R > C(\theta_j + K + x_j^2)^{3/4}$ and $\log |I| \leq (\log R)^{2}$, as assumed. This completes the proof.
\end{proof}

\section{One-point limit shape}\label{s.one point limit shape}

In this section we prove a result on the shape of the profile of $\cL_1$ on $[-\theta^{1/2}, \theta^{1/2}]$ under the conditioning that $\cL_1(0) = \theta$. It has a slightly weaker tail bound than Theorem~\ref{mt.one point limit shape}, but will be bootstrapped later.

Recall the role of $\beta$ from Assumption~\ref{as.tails}, namely that $\P(\cL_1(0) > \theta)\leq \exp(-c_2\theta^\beta)$.

\begin{theorem}\label{t.limit shape}
Let $\cL$ satisfy Assumptions~\ref{as.bg}--\ref{as.tails}. There exist $C<\infty$, $c>0$, and $\theta_0$ such that, for all $\theta>\theta_0$ and $C< M < C^{-1}\theta^{3/4}$, 
$$\P\left(\sup_{x\in[-\theta^{1/2},\theta^{1/2}]} \left(\cL_1(x) - \tri(x)\right) \geq M\theta^{1/4} \ \Big|\  \cL_1(0) = \theta\right) \leq \exp(-cM^{4\beta/3}) + \exp(-cM^2)$$
and
\begin{align*}
\MoveEqLeft[20]
\P\left(\inf_{x\in[-\theta^{1/2},\theta^{1/2}]} \left(\cL_1(x) - \tri(x)\right) \leq -M\theta^{1/4} \ \Big|\  \cL_1(0) = \theta\right)\\
&\leq \exp(-cM^{2}) + 4\cdot\P\left(\cL_1(0) \leq -\tfrac{1}{2}M\theta^{1/4}\right).
\end{align*}
\end{theorem}

To move from Theorem~\ref{t.limit shape} to Theorem~\ref{mt.one point limit shape}, one needs to replace the first term in the first bound above by $\exp(-cM^2)$, i.e., show we can take $\beta = 3/2$. This is accomplished in Section~\ref{s.one point asymptotics} with the proof of Theorem~\ref{mt.one point tail asymptotics}, and the proof of Theorem~\ref{mt.one point limit shape} is given there.

\begin{proof}[Proof of Theorem~\ref{t.limit shape}: $\cL_1$ is above $\tri - O(\theta^{1/4})$ with high probability]
We first show the second inequality above.
In fact, by a union bound and symmetry, it is enough to prove the inequality when the supremum inside the probability is over only $[-\theta^{1/2}, 0]$. Now,
\begin{equation}\label{e.lower bound first split}
\begin{split}
\MoveEqLeft[2]
\P\left(\inf_{x\in[-\theta^{1/2}, 0]} \left(\cL_1(x) - \tri(x)\right) \leq - M\theta^{1/4} \ \big|\ \cL_1(0)=\theta\right)\\
& \leq \P\left(\inf_{x\in[-\theta^{1/2}, 0]} \left(\cL_1(x) - \tri(x)\right) \leq - M\theta^{1/4}, \cL_1(-\theta^{1/2}) > -\theta - \tfrac{1}{2}M\theta^{1/4} \ \big|\ \cL_1(0)=\theta\right)\\
&\qquad+ \P\left(\cL_1(-\theta^{1/2}) \leq -\theta - \tfrac{1}{2}M\theta^{1/4} \ \big|\ \cL_1(0)=\theta\right).
\end{split}
\end{equation}
Let us bound the second term. We see from Lemma~\ref{l.bound point conditioning by tail conditioning} that 
$$\P\left(\cL_1(-\theta^{1/2}) \leq -\theta - \tfrac{1}{2}M\theta^{1/4} \ \big|\ \cL_1(0)=\theta\right) \leq \P\left(\cL_1(-\theta^{1/2}) \leq -\theta - \tfrac{1}{2}M\theta^{1/4} \ \big|\ \cL_1(0)\leq \theta\right).$$
We may assume $\theta$ is large enough using Assumption~\ref{as.tails} that $\P(\cL_1(0) \leq \theta)\geq \frac{1}{2}$. Then we obtain that, for all large enough $\theta$,
\begin{align*}
\P\left(\cL_1(-\theta^{1/2}) \leq -\theta - \tfrac{1}{2}M\theta^{1/4} \ \big|\ \cL_1(0)\leq \theta\right)
&\leq 2\cdot\P\left(\cL_1(-\theta^{1/2}) \leq -\theta - \tfrac{1}{2}M\theta^{1/4}\right)\\
&= 2\cdot\P\left(\cL_1(0) \leq - \tfrac{1}{2}M\theta^{1/4}\right),
\end{align*}
the last line by stationarity of $\cL_1(x) + x^2$.

Now we return to the first term of \eqref{e.lower bound first split}. 
We want to bound the numerator using the $H_t$-Brownian Gibbs property. Let $\F=\Fext(1,[-\theta^{1/2}, \theta^{1/2}])$. Recall the notation $\PF(\cdot) = \P(\cdot\mid \F)$ introduced on page~\pageref{n.conditional prob notation}. Then the first term of \eqref{e.lower bound first split} is
$$\E\left[\PF\left(\inf_{x\in[-\theta^{1/2}, 0]}\left(\cL_1(x) - \tri(x)\right) \leq - M\theta^{1/4}\ \Big|\  \cL_1(0) = \theta\right)\one_{\cL_1(-\theta^{1/2}) > -\theta - \frac{1}{2}M\theta^{1/4}}\right].$$
By the $H_t$-Brownian Gibbs property, under the conditional measure $\PF$, $\cL_1$ is a rate two Brownian bridge from $(-\theta^{1/2}, \cL_1(-\theta^{1/2}))$ to $(0, \cL_1(0))$ subject to the Radon-Nikodym derivative given by $W_{H_t}/Z_{H_t}$ associated to these boundary values and lower boundary data $\cL_2$ (see Definition~\ref{d.bg}). Now, we are restricted to the situation that $\cL_1(0) = \theta$ and $\cL_1(-\theta^{1/2}) \geq -\theta - \frac{1}{2}M\theta^{1/4}$. By monotonicity (Lemma~\ref{l.monotonicity}), the probability that we are considering, that $\cL_1$ is below $\tri$, is increased if we lower the endpoints of the Brownian bridge as much as possible (i.e., so that they are $-\theta-\frac{1}{2}M\theta^{1/4}$ and $\theta$ at the corresponding points) and take the lower boundary curve to $-\infty$, i.e., remove the lower boundary conditioning from the Brownian bridge. Thus, we see
\begin{align*}
\MoveEqLeft[6]
\PF\left(\inf_{x\in[-\theta^{1/2}, 0]}\left(\cL_1(x) - \tri(x)\right) \leq - M\theta^{1/4}\ \Big|\  \cL_1(0) = \theta\right)\one_{\cL_1(-\theta^{1/2}) > -\theta - \frac{1}{2}M\theta^{1/4}}\\
&\leq \P\left(\inf_{x\in[-\theta^{1/2}, 0]} \left(B(x) - \tri(x)\right) \leq - M\theta^{1/4}\right),
\end{align*}
where $B$ is a rate two Brownian bridge from $(-\theta^{1/2}, -\theta-\frac{1}{2}M\theta^{1/4})$ to $(0, \theta)$. Observe that $\E[B(x)] = \tri(x) + \frac{1}{2}M\theta^{-1/4}x$. So,
\begin{align*}
\MoveEqLeft[6]
\PF\left(\inf_{x\in[-\theta^{1/2}, 0]}\left(\cL_1(x) - \tri(x)\right) \leq - M\theta^{1/4}\ \Big|\  \cL_1(0) = \theta\right)\one_{\cL_1(-\theta^{1/2}) > -\theta - \tfrac{1}{2}M\theta^{1/4}}\\
&\leq \P\left(\inf_{x\in[-\theta^{1/2}, 0]} \left(B(x) - \E[B(x)]\right) \leq - \tfrac{1}{2}M\theta^{1/4}\right)\\
&\leq \exp(-cM^2),
\end{align*}
using tail bounds on the infimum of Brownian bridge from Lemma~\ref{l.brownian bridge sup tail exact} (since the Brownian bridge law is symmetric under negation).
Tracing the steps back to \eqref{e.lower bound first split}, we obtain
\begin{align*}
\P\left(\inf_{x\in[-\theta^{1/2}, 0]}\left(\cL_1(x) - \tri(x)\right) \leq - M\theta^{1/4}, \cL_1(-\theta^{1/2}) > -\theta - \tfrac{1}{2}M\theta^{1/4} \ \Big|\ \cL_1(0)=\theta\right) \leq \exp(-cM^2).
\end{align*}
Returning to \eqref{e.lower bound first split} and using the earlier bound for its second term, we see that
\begin{align*}
\P\left(\inf_{x\in[-\theta^{1/2}, 0]} \cL_1(x) - \tri(x) \leq - M\theta^{1/4} \ \Big|\ \cL_1(0)=\theta\right) \leq \exp(-cM^2) + 2\cdot\P\left(\cL_1(0) \leq -\tfrac{1}{2}M\theta^{1/4}\right),
\end{align*}
which completes the proof of one side of the theorem.
\end{proof}

\bigskip

Now we return to the second half of the proof of Theorem~\ref{t.limit shape}.

\begin{proof}[Proof of Theorem~\ref{t.limit shape}: $\cL_1$ is below $\tri + O(\theta^{1/4})$ with high probability]
As in the previous bound, we will only bound 
$$\P\left(\sup_{x\in[-\theta^{1/2}, 0]} \left(\cL_1(x) - \tri(x)\right) \geq  M\theta^{1/4} \ \Big|\ \cL_1(0)=\theta\right),$$
i.e., we look at the curve to the left of $0$ only.

Recall from Section~\ref{s.intro.proof sketch.limit shape} and the figure there that we want to find a point \emph{on} the tangency line such that, even under the large deviation conditioning, $\cL_1(x)$ is below that point with high probability. We take $x_0=2\theta^\gamma$, where $\gamma = (2\beta)^{-1}(\alpha+\frac{3}{2})$,\footnote{This choice of $\gamma$ is made purely to ensure that $\P(\cL_1(0) > \theta^{2\gamma})/\P(\cL_1(0)\geq \theta) \leq \exp(-\theta^{3/2})$; the expression for $\gamma$ then follows from the upper and lower bounds in terms of $\alpha$ and $\beta$ on the one-point probability from Assumption~\ref{as.tails}.} with $\alpha$ and $\beta$ as in Assumption~\ref{as.tails}. Now, the value of the tangent (i.e., $\tri$ extended in the obvious linear way to $\R$) at $-x_0$ is $\theta -2\smash{\theta^{1/2}}x_0 = \theta - \smash{4\theta^{\gamma+1/2}}$ and the value of the parabola is $-x_0^2 = -\smash{4\theta^{2\gamma}}$; clearly the height separating the tangent and the parabola is $\theta - 2\theta^{1/2}x_0 + x_0^2 = (x_0-\theta^{1/2})^2 = (2\theta^\gamma - \theta^{1/2})^2$. Consider the probability
$$\P\left(\cL_1(-x_0) >  \theta-2\theta^{1/2}x_0\mid \cL_1(0) > \theta\right) = \P\bigl(\cL_1(-x_0) +x_0^2 >  (x_0-\theta^{1/2})^2\mid \cL_1(0) > \theta\bigr).$$
Using the tail bounds from Assumption~\ref{as.tails} and the stationarity from Assumption~\ref{as.bg}, this probability can be bounded, for $\theta>\theta_0(t_0)$, as
\begin{equation}\label{e.tangent point choice}
\P\Bigl(\cL_1(-x_0)+x_0^2 > (2\theta^{\gamma}-\theta^{1/2})^2 \mid \cL_1(0) > \theta\Bigr) \leq \exp\left(-c_2(2\theta^\gamma - \theta^{1/2})^{2\beta} + c_1\theta^\alpha \right).
\end{equation}
Since $\alpha\geq\beta$, it follows that $\gamma>\frac{1}{2}$, so the previous right-hand side is at most $\exp(-c_2\theta^{2\gamma\beta} + c_1\theta^\alpha) \leq \exp(-\theta^{3/2})$ for all large enough $\theta$ since our choice of $\gamma$ satisfies $2\gamma\beta= \alpha+\frac{3}{2}$.
Also observe that, since $M < C^{-1}\theta^{3/4}$, $\exp(-\theta^{3/2}) \leq \exp(-c M^2)$.

Now, similar to the proof lower bounding, we include some auxiliary events to aid us in the analysis. Apart from control over $\cL_1(-x_0\theta^{1/2})$, we will need an upper bound on the second curve's deviation above $-x^2$. We adopt the shorthand 
$$\Bigl\{\cL_1\in A_{\theta,M}\Bigr\} = \left\{\sup_{x\in[-\theta^{1/2},0]} \cL_1(x) - \tri(x) \geq  M\theta^{1/4}\right\}$$
for notational convenience. Now,
\begin{align}
\MoveEqLeft[0]
\P\left(\cL_1\in A_{\theta,M} \ \big|\ \cL_1(0)=\theta\right)\nonumber\\
&\leq \P\left(\cL_1\in A_{\theta,M},\, \cL_1(-x_0) \leq \theta-2\theta^{1/2}x_0,\, \sup_{x\in[-x_0, 0]}\left(\cL_2(x)+x^2\right) \leq \varepsilon_0 M\theta^{1/4} \ \Big|\ \cL_1(0)=\theta\right)\nonumber\\
&\quad + \P\Bigl(\cL_1(-x_0) > \theta-2\theta^{1/2}x_0 \ \Big|\  \cL_1(0) = \theta\Bigr) \label{e.upper bound first split}\\
&\quad + \P\left(\cL_1\in A_{\theta,M}, \sup_{x\in[-x_0, 0]} \left(\cL_2(x)+x^2\right) > \varepsilon_0 M \theta^{1/4} \ \Big|\  \cL_1(0) = \theta\right),\nonumber
\end{align}
where $\varepsilon_0>0$ is the $t$-independent constant from Proposition~\ref{p.brownian versatile tangent estimate}.
We already saw that the second term is bounded above by $\exp(-cM^2)$, with $c$ independent of $t$ for $\theta>\theta_0(t_0)$ coming from Assumption~\ref{as.tails}. 
By Lemma~\ref{l.lower curve control new} with $f(x)=\tri(x)- M\theta^{1/4}$ and $I=[-x_0,0]$, the  third term is bounded by
$$\tfrac{1}{2}\cdot\frac{\P(\cL_1\in A_{\theta,M}\mid \cL_1(0)=\theta)}{\P(\inf_{[-\theta^{1/2},0]} \cL_1(x)-\tri(x) \geq M\theta^{1/4} \mid \cL_1(0)=\theta+K)}$$
for any $K$ such that $\varepsilon_0 M \theta^{1/4} \geq \frac{3}{4}\log (C(\theta+K))$ as long as $\log x_0 \leq (\log \theta)^{2}$. The latter clearly holds since $x_0$ is polynomial in $\theta$. We take $K=\theta$, in which case the former condition also easily holds. Now, for this choice of $K$ the denominator of the previous display is lower bounded by $3/4$ for all large enough $\theta$ by the first part of Theorem~\ref{t.limit shape} already proved. So the third term of \eqref{e.upper bound first split} is at most $\frac{2}{3} \P(\cL_1\in A_{\theta,M} \mid \cL_1(0) = \theta)$, and it may be taken to the lefthand side.

So we may focus on the first term of \eqref{e.upper bound first split}. The argument we now present is similar to that of Proposition~\ref{p.versatile Airy tangent estimate} and essentially extends that proposition to the case where $\beta<3/2$. 

As in the lower bound, we apply the Brownian Gibbs property. Similarly to before, let $\F=\Fext(1,[-x_0,0])$. Then we see that the first term in the last display is equal to
$$\E\left[\PF(\cL_1\in A_{\theta,M}\mid \cL_1(0) = \theta) \one_{\cL_1(-x_0) \leq \theta - 2\theta^{1/2}x_0,\, \sup_{x\in[-x_0, 0]}\left(\cL_2(x)+x^2\right) \leq \varepsilon_0 M\theta^{1/4}}\right]$$
By the Brownian Gibbs property, under $\PF$ and $\{\cL_1(0) = \theta\}$, $\cP_1$ is a rate two Brownian bridge from $(-x_0, \cL_1(-x_0))$ to $(0, \cL_1(0)) = (0,\theta)$ subject to the Radon-Nikodym derivative given by $\tilde W_{H_t}/\tilde Z_{H_t}$ associated with the mentioned boundary values and lower boundary data $\cL_2$. By monotonicity, the probability of this Brownian bridge lying in $A_{\theta,M}$ increases by raising its endpoints and raising the lower boundary condition. Thus the previous display is upper bounded by (using that indicators are bounded by $1$)
\begin{align*}
\E\left[\one_{B\in A_{\theta,M}} \frac{W_{H_t}}{Z_{H_t}}\right] \leq \P\left(B\in A_{\theta,M}\right)\cdot Z_{H_t}^{-1},
\end{align*}
where $B$ is a rate two Brownian bridge from $(-x_0, \theta-2\theta^{1/2}x_0)$ to $(0,\theta)$ and $W_{H_t}$ and $Z_{H_t}$ are associated with the same boundary values along with the lower boundary curve $-x^2 + \varepsilon M\theta^{1/4}$.

Now, since $\tri$ is the line tangent to $-x^2$ at $(-\theta^{1/2}, -\theta)$ we may apply Corollary~\ref{c.ratio of deviation prob and part func} with $I = [-x_0, 0]$, $J=[-\theta^{1/2},0]$, and $x_\tan = -\theta^{1/2}$ to see that the previous display is upper bounded by $\exp(-cM^2)$ when $\theta>\theta_0(t_0)$.
\end{proof}

As promised before, while Theorem~\ref{mt.one point limit shape} is focused on the profile inside $[-\theta^{1/2}, \theta^{1/2}]$, we do have a result for the shape outside the interval.

\begin{proposition}\label{p.para fluctuation h_t}
Let $\cL$ satisfy Assumptions~\ref{as.bg}--\ref{as.tails}. For any $L>1$, there exist $c = c(L)>0$ and $\theta_0=\theta_0(L)$ such that, for $\theta>\theta_0$,
$$\P\left(\sup_{x : |x|\in[\theta^{1/2}, L\theta^{1/2}]}\left(\cL_1(x) + x^2\right) > \theta^{1/4}\log\theta \  \Big|\  \cL_1(0) = \theta\right) \leq \exp(-c(\log\theta)^2)$$
and
\begin{align*}
\MoveEqLeft[14]
\P\left(\inf_{x : |x|\in[\theta^{1/2}, L\theta^{1/2}]}\left(\cL_1(x) + x^2\right) < -\theta^{1/4}\log\theta \  \Big|\  \cL_1(0) = \theta\right)\\
&\leq 2\cdot \P\left(\inf_{x : |x|\in[\theta^{1/2}, L\theta^{1/2}]}\left(\cL_1(x)+x^2\right)  < - \theta^{1/4}\log\theta\right).
\end{align*}

\end{proposition}

The second bound is stated without a quantitative dependence on $\theta$ as we have not assumed lower tails for the infimum of $\cL_1$ over an interval. These are however known for the KPZ and parabolic Airy line ensembles (see \cite[Proposition~B.1]{wu2021tightness} (for $t<\infty$) and \cite[Proposition~A.2]{hammond2016brownian} respectively). 

Next note that we need the deviation to be of order $\theta^{1/4}\log\theta$ instead of $\theta^{1/4}$; the reason is the same as in Proposition~\ref{p.para fluctuation}.

As the reader might already realize, while the true fluctuation scale inside $[-\theta^{1/2}, \theta^{1/2}]$ is $\theta^{1/4}$ (at least in the bulk of the intervals $[-\theta^{1/2}, 0]$ and $[0, \theta^{1/2}]$), the true fluctuation scale on intervals of the form $[-L\theta^{1/2}, \theta^{1/2}]$ and $[\theta^{1/2}, L\theta^{1/2}]$ for $L>1$ is expected to be $(\log \theta)^{2/3}$ (rather, $O(1)$ at any given point in such intervals, with the logarithmic factor coming when taking a supremum of fluctuations over the whole interval). While a refinement of our approach is expected to yield improvements over the stated $\theta^{1/4}\log\theta$  bound we do not pursue this.

Proposition~\ref{p.para fluctuation h_t} is not used for any argument in the paper; in particular, the proof of Theorem~\ref{mt.one point tail asymptotics} is independent of it, and we will invoke Theorem~\ref{mt.one point tail asymptotics} in the proof of Proposition~\ref{p.para fluctuation h_t}.

\begin{proof}[Proof of Proposition~\ref{p.para fluctuation h_t}, the lower bound on the limit shape:]
By Assumption~\ref{as.mono in cond} (monotonicity in conditioning),
\begin{align*}
\MoveEqLeft[12]
\P\left(\inf_{x : |x|\in[\theta^{1/2}, L\theta^{1/2}]}\left(\cL_1(x)+x^2\right)  < - \theta^{1/4}\log\theta \  \Big|\  \cL_1(0) = \theta\right)\\
&\leq \P\left(\inf_{x : |x|\in[\theta^{1/2}, L\theta^{1/2}]}\left(\cL_1(x)+x^2\right)  < - \theta^{1/4}\log\theta \  \Big|\  \cL_1(0) \leq \theta\right)\\
&\leq 2\cdot \P\left(\inf_{x : |x|\in[\theta^{1/2}, L\theta^{1/2}]}\left(\cL_1(x)+x^2\right)  < - \theta^{1/4}\log\theta\right),
\end{align*}
the last inequality by picking $\theta_0$ large enough that $\P(\cL_1(0) \leq \theta)\geq \frac{1}{2}$ for all $\theta>\theta_0$.

\medskip
\emph{Proof of the upper bound on the limit shape:}
By symmetry, it is enough to prove the bound when the supremum inside the probability is over $x\in[-L\theta^{1/2}, \theta^{1/2}]$, i.e., only on the left side of zero.

Similar to the argument in Theorem~\ref{t.limit shape}, the idea is to show that, conditionally on $\{\cL_1(0)=\theta\}$,  $\cL_1$ is ``pinned'' close to the parabola $-x^2$ at $-L\theta^{1/2}$ and at $-\theta^{1/2}$ and then apply Proposition~\ref{p.para fluctuation}. From Theorem~\ref{t.limit shape} we have this pinning at $-\theta^{1/2}$. We next prove that we also have it at $-L\theta^{1/2}$.

Consider the line $\ellt_L$ tangent to $-x^2$ at $(-L\theta^{1/2}, -L^2\theta)$. Let $K>L$ be such that, for some $c>0$,
\begin{equation}\label{e.alpha requirement}
\P\left(\cL_1(-K\theta^{1/2}) > \ellt_L(-K\theta^{1/2}) \mid \cL_1(0) = \theta\right) \leq \exp(-c\theta^{3/2}).
\end{equation}
Such a $K$ exists because, by stationarity of $\cL_1(x)+x^2$ and monotonicity in conditioning (Lemma~\ref{l.bound point conditioning by tail conditioning}), we can upper bound the previous probability by
\begin{align*}
\P\left(\cL_1(-K\theta^{1/2}) > \ellt_L(-K\theta^{1/2}) \midd \cL_1(0) \geq \theta)\right)
&\leq \frac{\P(\cL_1(-K\theta^{1/2}) > \ellt_L(-K\theta^{1/2}))}{\P(\cL_1(0) \geq \theta)}\\
&= \frac{\P(\cL_1(0) > \ellt_L(-K\theta^{1/2})+K^2\theta)}{\P(\cL_1(0) \geq \theta)},
\end{align*}
and $\ellt_L(-K\theta^{1/2}) + K^2\theta$ is such that, for any $C<\infty$, there is a $K$ such that $\ellt_L(-K\theta^{1/2}) + K^2\theta > C\theta$. By choosing a $C$ large enough and using the upper and lower bounds on $\P(\cL_1(0)\geq \theta)$ from Theorem~\ref{mt.one point tail asymptotics}, we obtain \eqref{e.alpha requirement}.

We will now show that we have the pinning at $-L\theta^{1/2}$ with high probability, i.e., letting $\pin_{L,M} = \smash{\{\cL_1(-L\theta^{1/2}) \leq -L^2\theta + M\theta^{1/4}\}}$, where $M$ will ultimately be taken to be $\log\theta$, we want to show that
\begin{equation}\label{e.non-immediate pinning}
\P\left((\pin_{L,M})^c\midd \cL_1(0) = \theta\right) \leq \exp(-cM^2).
\end{equation}
To prove \eqref{e.non-immediate pinning}, we may assume from \eqref{e.alpha requirement} and Theorem~\ref{t.limit shape} with high probability that, at $-K\theta^{1/2}$ and $-\theta^{1/2}$, $\cL_1$ is lower than the tangent line $\ellt_L$. More precisely, we see from \eqref{e.alpha requirement} and Theorem~\ref{t.limit shape} that
\begin{align*}
\MoveEqLeft[4]
\P\left((\pin_{L,M})^c\midd \cL_1(0) = \theta\right)\\
&\leq \P\left((\pin_{L,M})^c, \cL_1(-K\theta^{1/2}) \leq \ellt_L(-K\theta^{1/2}), \cL_1(-\theta^{1/2}) \leq -\theta + M\theta^{1/4} \ \Big|\  \cL_1(0) = \theta\right)\\
&\qquad + \exp(-c\theta^{3/2}) + \exp(-cM^2).
\end{align*}
The first term is bounded by $\exp(-cM^2)$ by the first inequality of Proposition~\ref{p.versatile Airy tangent estimate} by taking $I = [-K\theta^{1/2}, -\theta^{1/2}]$ and $x_\tan = -L\theta^{1/2}$. We need $M\leq \theta^{3/8}$ to apply this bound, but this is satisfied as we have taken $M=\log\theta$. So, overall, we obtain \eqref{e.non-immediate pinning} for such~$M$.

With the pinning of $\cL_1$ at $-L\theta^{1/2}$ and $-\theta^{1/2}$ established, we may move on to the fluctuations of $\cL_1$ on $J = [-L\theta^{1/2}, -\theta^{1/2}]$. Let $A_{L, \theta}(\cL_1) = \{\sup_{x\in J} \left(\cL_1(x)+x^2\right) > \theta^{1/4} \log \theta\}$, where we have now made $M = \log \theta$ explicit.

First, Lemma~\ref{l.lower curve control new} with $I=J$, $R=\theta_j=\theta$, and $f(x) = -x^2-\theta^{1/4}\log\theta$ yields that, if $\theta > (\theta+K)^{3/4}$ and $\log((L+1)\theta^{1/2}) \leq (\log \theta)^2$ (which clearly holds for $\theta>\theta_0(L)$), then
\begin{align*}
\MoveEqLeft[10]
\P\left(A_{L,\theta}(\cL_1), \sup_{x\in J}(\cL_2(x) + x^2) \geq (\log \theta)^C \mid \cL_1(0) = \theta\right)\\
&\leq \frac{\P\left(A_{L,\theta}(\cL_1)\mid \cL_1(0)=\theta\right)}{2\cdot \P\left(\inf_{x\in J} (\cL_1(x) + x^2) \geq \theta^{1/4}\log\theta \mid \cL_1(0)=\theta+K\right)}.
\end{align*}
By taking $K=4L^2\theta$ and applying the first part of Theorem~\ref{t.limit shape} with $M$ a large enough constant, we see that the probability in the denominator can be made at least $3/4$. Decomposing $\P(A_{L,\theta}(\cL_1)\mid \cL_1(0)=\theta)$ based on whether $\sup_{x\in J}(\cL_2(x) + x^2) \geq (\log \theta)^C$ occurs or not, the previous display implies
\begin{align*}
\P\left(A_{L,\theta}(\cL_1)\mid \cL_1(0)=\theta\right) 
&\leq 3\cdot \P\left(A_{L,\theta}(\cL_1), \sup_{x\in J}(\cL_2(x) + x^2) \leq (\log \theta)^C\mid \cL_1(0)=\theta\right).
\end{align*}

Then by a union bound we see that
\begin{equation}\label{e.parabolic fluctuation break up}
\begin{split}
\MoveEqLeft[2]
\P(A_{L,\theta}(\cL_1) \mid \cL_1(0) = \theta)\\
&\leq 3\cdot \P\left(A_{L,\theta}(\cL_1), \pin_{L, \log \theta}, \pin_{1, \log \theta}, \sup_{x\in J}\left(\cL_2(x)+x^2\right) \leq (\log \theta)^{C} \midd \cL_1(0) = \theta\right)\\
&\quad + 3\cdot \P\left((\pin_{L,\log\theta})^c  \cup (\pin_{1,\log\theta})^c  \midd \cL_1(0) = \theta\right).%
\end{split}
\end{equation}
By Theorem~\ref{mt.one point limit shape} (with $M=\log\theta$) and \eqref{e.non-immediate pinning}, we know that
$$\P\left((\pin_{L,\log\theta})^c\cup (\pin_{1,\log\theta})^c \ \Big|\ \cL_1(0) = \theta\right) \leq \exp(-c(\log\theta)^2).$$
So it only remains to bound the first term of \eqref{e.parabolic fluctuation break up}.

We will do this by massaging things into the form of Proposition~\ref{p.para fluctuation}. The argument bears similarities to that of Proposition~\ref{p.versatile Airy tangent estimate} as derived as a consequence of Proposition~\ref{p.brownian versatile tangent estimate}, with Proposition~\ref{p.para fluctuation} in place of the latter.
We again make use of the $H_t$-Brownian Gibbs property. Let $\F$ be the $\sigma$-algebra generated by the lower curves and the top curve outside of $J$ (which recall is $[-L\theta^{1/2}, -\theta^{1/2}]$). Then
\begin{align*}
\MoveEqLeft[8]
\P\left(A_{L,\theta}(\cL_1), \pin_{L, \log \theta}, \pin_{1, \log \theta}, \sup_{x\in J}\left(\cL_2(x)+x^2\right) \leq (\log \theta)^{C}\ \Big|\  \cL_1(0) = \theta\right)\\
&= \E\left[\PF\left(A_{L,\theta}(\cL_1)\right)\one_{\pin_{L, \log \theta},\, \pin_{1, \log \theta},\, \sup_{x\in J} \left(\cL_2(x)+x^2\right) \leq (\log \theta)^{C}}  \midd \cL_1(0) = \theta\right].
\end{align*}
Under $\PF$, $\cL_1$ on $J$ is a Brownian bridge from $(-L\theta^{1/2}, \cL_1(-L\theta^{1/2}))$ to $(-\theta^{1/2}, \cL_1(-\theta^{1/2}))$, tilted by the Radon-Nikodym derivative given by $\smash{\tilde W_{H_t}/\tilde Z_{H_t}}$ with the same boundary values and lower boundary curve $\cL_2$. On the events mentioned, by monotonicity (Lemma~\ref{l.monotonicity}), this Brownian bridge is stochastically dominated by a Brownian bridge $B$ from $(-L\theta^{1/2}, -L^2\theta + \theta^{1/4}\log\theta)$ to $(-\theta^{1/2}, -\theta + \theta^{1/4}\log\theta)$ which is reweighted by $W_{H_t}/Z_{H_t}$ associated with the same boundary data and lower curve $-x^2 + (\log \theta)^{C}$. Since $A_{L,\theta}$ is an increasing event, we see that the expectation in the immediately previous display is bounded above by
$$Z_{H_t}^{-1}\E\left[\one_{B\in A_{L,\theta}} W_{H_t}\right].$$
We want to bound this by applying Proposition~\ref{p.para fluctuation} after using that $W_{H_t}\leq 1$. To have the expression fit into the framework of that proposition, we need to shift $B$ and the event $A_{L,\theta}$ down by $\smash{(\log\theta)^{C}}$. Doing so and applying Proposition~\ref{p.para fluctuation} with $H=\smash{\theta^{1/4}\log\theta - (\log\theta)^{C}}$ yields that the previous display is upper bounded by $\exp(-c(\log\theta)^2)$.
Tracing this bound back completes the proof of the upper bound on the limit shape in Proposition~\ref{p.para fluctuation h_t}.
\end{proof}

\section{One-point estimates}\label{s.one point asymptotics}

In this section we prove the one-point estimates under Assumptions~\ref{as.bg}--\ref{as.tails}, thereby proving Theorems~\ref{mt.one point density asymptotics} and \ref{mt.one point tail asymptotics}. Here and in the rest of the paper, unless explicitly indicated otherwise, Assumptions~\ref{as.bg}--\ref{as.tails} refers to Assumptions~\ref{as.bg}, \ref{as.corr}(a), \ref{as.weak bk}(b\ensuremath{'}), \ref{as.mono in cond stronger}, and \ref{as.tails}.

We prove the upper bound on the tail in Section~\ref{s.one point.upper bound}, the lower bound in Section~\ref{s.one point.lower bound}, and both the density bounds in Section~\ref{s.one point.density}.

\subsection{Upper bound on the tail}\label{s.one point.upper bound}

For the reader's convenience, here we restate the upper bound in Theorem~\ref{mt.one point tail asymptotics}. The lower bound is stated as Theorem~\ref{t.upper tail lower bound} (using fewer assumptions).

\begin{theorem}\label{t.upper tail upper bound}
Let $\cL$ satisfy Assumptions~\ref{as.bg}--\ref{as.tails}. There exist $\theta_0 > 0$ and $C<\infty$ such that, for $\theta > \theta_0$,
$$\P\left(\cL_1(0) \geq \theta\right) \leq \exp\left(-\frac{4}{3}\theta^{3/2} + C\theta^{3/4}\right).$$
\end{theorem}

\begin{remark}
In the zero temperature case, the above bound trivially also applies to $\cL_k(0)$ for any $k$ by ordering of the curves (in the positive temperature case they are not ordered, but it is also not clear if they are stochastically ordered, which would have sufficed). 
However observe that, in the zero-temperature case, Assumption~\ref{as.corr}(b) and Theorem~\ref{t.upper tail upper bound} imply an improved better tail bound for $\cL_2(0)$:
\begin{equation}\label{e.upper tail with bk}
\begin{split}
\P(\cL_2(0)\geq \theta)
= \P(\cL_2(0)\geq \theta, \cL_1(0)\geq \theta)
&= \P(\cL_2(0)\geq \theta\mid \cL_1(0)\geq \theta)\cdot\P(\cL_1(0)\geq \theta)\\
&\leq \P(\cL_1(0)\geq \theta)^2 \leq \exp\left(-\frac{8}{3}\theta^{3/2}+2C\theta^{3/4}\right).
\end{split}
\end{equation}
In the first equality we used that, since $t=\infty$, $\cL$ is a non-intersecting ensemble and so $\cL_2(0)\geq \theta \implies \cL_1(0)\geq \theta$.

Now, if one has the BK inequality in Assumption~\ref{as.corr}(b) extended in a natural fashion to the $k$\textsuperscript{th} curve (bounding an upper tail probability of the $k$\textsuperscript{th} curve by the first curve's probability to the $k$\textsuperscript{th} power), which indeed can be established for the parabolic Airy line ensemble, then in conjunction with a similar argument as above one obtains that
\begin{equation}\label{e.bk upper tail for general k}
\P(\cL_k(0)\geq \theta) \leq \exp\left(-\frac{4}{3}k\theta^{3/2} + Ck\theta^{3/4}\right).
\end{equation}
We expect the first order term of $\frac{4}{3}k\theta^{3/2}$ would be sharp; however, have not included an argument for the lower bound here, and we have not found an estimate of this explicit form in the literature.
\end{remark}

Given Theorems~\ref{t.limit shape} and \ref{t.upper tail upper bound}, we can prove Theorem~\ref{mt.one point limit shape}.

\begin{proof}[Proof of Theorem~\ref{mt.one point limit shape}] 
Theorem~\ref{t.upper tail upper bound} yields that Assumption~\ref{as.tails} holds with $\beta=3/2$. Using this in the first bound of Theorem~\ref{t.limit shape} yields the first bound of Theorem~\ref{mt.one point limit shape}; the second bound is the same in both.
\end{proof}

We next prove Theorem~\ref{t.upper tail upper bound}; recall the proof outline given in Section~\ref{s.intro.proof ideas.one-point}.

\begin{proof}[Proof of Theorem~\ref{t.upper tail upper bound}]
We will in fact show that $\P(\cL_1(0)\in[\theta-1,\theta]) \leq \exp(-\frac{4}{3}\theta^{3/2}+C\theta^{3/4})$, which will clearly suffice since $\P(\cL_1(0)\geq \theta) = \sum_{s=\theta}^\infty\P(\cL_1(0)\in[s,s+1])$. We consider the event $\cL_1(0)\in[\theta-1,\theta]$ instead of the entire upper tail as under the former event we have an upper bound on $\cL_1(0)$ which converts to a high probability upper bound on $\smash{\cL_1(\pm \theta^{1/2})}$ from Theorem~\ref{t.limit shape}.

Indeed,
recall that Theorem~\ref{t.limit shape} implies that $\P(\cL_1(\pm\theta^{1/2})+\theta \geq M\theta^{1/4}\mid \cL_1(0)=\theta)$ is small for large enough $M$ independent of $\theta>\theta_0$. By Lemma~\ref{l.bound point conditioning by tail conditioning} (monotonicity in conditioning), we can convert Theorem~\ref{t.limit shape}'s statement to the same statement conditional on $\cL_1(0)\in[\theta-1,\theta]$. Next, Lemma~\ref{l.lower curve control new} with $f\equiv -\infty$ gives that $\P(\sup_{[-\theta^{1/2},\theta^{1/2}]}\left(\cL_2(x)+x^2\right)\geq \smash{(\log \theta)^{C}}\mid \cL_1(0)\in[\theta-1,\theta])\leq\frac{1}{2}$ for some $C>0$. Combining all this by a union bound, there exists a large enough constant $M$ (independent of $\theta$) such that
\begin{align*}
\MoveEqLeft[2]
\tfrac{1}{4}\P\left(\cL_1(0)\in [\theta-1,\theta]\right)\\
&\leq \P\left(\cL_1(0)\in [\theta-1,\theta]\right)\\
&\qquad\times \P\left(\cL_1(\pm\theta^{1/2}) +\theta \leq M\theta^{1/4}, \sup_{[-\theta^{1/2}, \theta^{1/2}]} \left(\cL_2(x)+x^2\right) \leq (\log \theta)^{C}\ \Big|\  \cL_1(0)\in [\theta-1,\theta]\right)\\
&= \P\left(\cL_1(0)\in [\theta-1,\theta], \cL_1(\pm \theta^{1/2}) +\theta \leq M\theta^{1/4}, \sup_{[-\theta^{1/2}, \theta^{1/2}]} \left(\cL_2(x)+x^2\right) \leq (\log \theta)^{C}\right).
\end{align*}
Let $\F$ be the $\sigma$-algebra generated by everything outside the top curve on $[-\theta^{1/2}, \theta^{1/2}]$. The probability in the last display is
\begin{align*}
\MoveEqLeft[8]
\E\left[\PF\bigl(\cL_1(0)\in [\theta-1,\theta]\bigr)\one_{\cL_1(\pm \theta^{1/2})+\theta \leq M\theta^{1/4},\, \sup_{[-\theta^{1/2}, \theta^{1/2}]} \left(\cL_2(x)+x^2\right) \leq (\log \theta)^{C}}\right]\\
&\leq \E\left[\PF\bigl(\cL_1(0)\geq \theta-1\bigr)\one_{\cL_1(\pm \theta^{1/2})+\theta \leq M\theta^{1/4},\, \sup_{[-\theta^{1/2}, \theta^{1/2}]} \left(\cL_2(x)+x^2\right) \leq (\log \theta)^{C}}\right].
\end{align*}

Now conditionally on $\F$, $\cL_1$ is a Brownian bridge from $(-\theta^{1/2}, \cL_1(-\theta^{1/2}))$ to $(\theta^{1/2}, \cL_1(\theta^{1/2}))$ which is reweighted by the Radon-Nikodym factor $\smash{\tilde W_{H_t}/\tilde Z_{H_t}}$ associated to the boundary data $\cL_2$. By monotonicity Lemma~\ref{l.monotonicity} and on the $\F$-measurable event in the indicator, this bridge is stochastically dominated by the Brownian bridge $B$ from $(-\theta^{1/2}, -\theta+M\theta^{1/4})$ to $(\theta^{1/2}, -\theta+M\theta^{1/4})$ reweighted by the factor $W_{H_t}/Z_{H_t}$ associated to the boundary data $-x^2 + (\log \theta)^{C}$. Further $W_{H_t}\leq 1$. Thus we see that the previous display is upper bounded (since the event there is increasing) by
\begin{align}\label{e.ratio of brownian probabilities}
\MoveEqLeft[8]
Z_{H_t}^{-1}\P\left(B(0)\geq \theta-1\right)
\end{align}
Since $B(0)$ is a normal random variable of mean $-\theta+M\theta^{1/4}$ and variance $2\times\frac{\theta^{1/2}\times\theta^{1/2}}{2\theta^{1/2}}=\theta^{1/2}$, using the normal tail bounds from Lemma~\ref{l.normal bounds}, the numerator is upper bounded by
\begin{align*}
\exp\left(-\frac{(2\theta-M\theta^{1/4}-1)^2}{2\theta^{1/2}}\right) \leq \exp\left(-2\theta^{3/2} + C\theta^{3/4}\right).
\end{align*}
Next we turn to the denominator of \eqref{e.ratio of brownian probabilities}. By Corollary~\ref{c.pos temp parabolic avoidance lower bound}, for large enough $\theta$ such that $\smash{M\theta^{1/4} > t_0^{-1/6}+1 > t^{-1/6}+1}$ so as to satisfy the corollary's hypotheses,
\begin{align*}
Z_{H_t}
\geq \exp\left(-\frac{1}{12}(2\theta^{1/2})^3 - 6\theta^{1/2}\log(2\theta^{1/2})\right)
&\geq \exp\left(-\frac{2}{3}\theta^{3/2} - C\theta^{3/4}\right).
\end{align*}
Together, this implies that 
\begin{align*}
\P\left(\cL_1(0) \in [\theta-1,\theta]\right) \leq \exp\left(-\frac{4}{3}\theta^{3/2} + C\theta^{3/4}\right)
\end{align*}
for all large $\theta$, completing the proof of Theorem~\ref{t.upper tail upper bound}.
\end{proof}

\subsection{Lower bound on the tail}\label{s.one point.lower bound}

The following is the lower bound half of Theorem~\ref{mt.one point tail asymptotics}.

\begin{theorem}\label{t.upper tail lower bound}
Let $\cL$ satisfy Assumptions~\ref{as.bg} and \ref{as.corr}(a). There exist $C>0$ and $\theta_0$ such that, for $\theta>\theta_0$,
$$\P\left(\cL_1(0) \geq \theta\right) \geq \exp\left(-\frac{4}{3}\theta^{3/2} - C\theta^{1/2}\log \theta\right).$$
\end{theorem}

Observe that the above statement is made without using Assumption~\ref{as.tails} which gives an a priori lower bound on the upper tail. Thus the above actually verifies the lower bound part of the upper tail in Theorem~\ref{mt.extremal}, as well as the lower bound of Assumption~\ref{as.tails} with $\alpha=\frac{3}{2}$.

\begin{proof}[Proof of Theorem~\ref{mt.one point tail asymptotics}]
This is an immediate consequence of Theorems~\ref{t.upper tail upper bound} and \ref{t.upper tail lower bound}.
\end{proof}

We start by proving a lower bound with the right exponent of $3/2$ but a suboptimal constant using only Assumptions~\ref{as.bg} and \ref{as.corr}(a).

\begin{lemma}\label{l.first suboptimal constant}
Let $\cL$ satisfy Assumptions~\ref{as.bg} and \ref{as.corr}(a). There exists $\theta_0$ such that, for $\theta>\theta_0$,
$$\P\Big(\cL_1(0) \geq \theta\Big) \geq \exp\big(-5\theta^{3/2}\big).$$
\end{lemma}

\begin{proof}
For an $M$ to be chosen, consider the favourable event 
$$\fav = \left\{\cL_1(-\theta^{1/2})\geq -\theta- M\right\}\cap \left\{\cL_1(\theta^{1/2})\geq -\theta - M\right\}.$$
By stationarity and positive association of $\cL_1(x) + x^2$, and almost sure finiteness of $\cL_1(0)$, it follows that $\P(\fav)\geq 1/2$ for $M>M_0$ sufficiently large.

Consider the $\sigma$-algebra $\F = \Fext(1,[-\theta^{1/2}, \theta^{1/2}])$. 
The Brownian Gibbs property says that the distribution of $\cL$ on $[-\theta^{1/2}, \theta^{1/2}]$, conditionally on $\F$, is that of a rate two Brownian bridge tilted by the Radon-Nikodym derivative $W_{H_t}/Z_{H_t}$ associated to the conditioned boundary data. By monotonicity Lemma~\ref{l.monotonicity}, on $\fav$, this Brownian bridge stochastically dominates the rate two Brownian bridge $B$ from $(-\theta^{1/2}, -\theta-M)$ to $(\theta^{1/2}, -\theta-M)$ with no lower boundary condition.

Thus we see
\begin{align*}
\P\big(\cL_1(0) \geq \theta\big) \geq \E\Big[\PF\big(\cL_1(0) \geq \theta\big)\cdot\one_\fav\Big]  \geq \frac{1}{2}\cdot \P\left(B(0) \geq \theta\right).
\end{align*}
Now $B(0)$ is a normal random variable with mean $-\theta-M$ and variance $\theta^{1/2}.$ Thus using the standard lower bound on the normal probability from Lemma~\ref{l.normal bounds}, we see that, on $\fav$,
\begin{align*}
\P\big(\cL_1(0) \geq \theta\big) \geq c\theta^{-3/4} \cdot\exp\left(-\frac{(\theta+\theta+M)^2}{2\theta^{1/2}}\right)%
&\geq \exp(-5\theta^{3/2}),
\end{align*}
the last inequality for $\theta>M$. This completes the proof.
\end{proof}

Next we prove Theorem~\ref{t.upper tail lower bound}, i.e., obtain the optimal constant in the exponent. We remind the reader of the proof sketch given in Section~\ref{s.intro.proof ideas.one-point}.

\begin{proof}[Proof of Theorem~\ref{t.upper tail lower bound}]
Suppose we know that there exist $C_n>0$, $\gamma_n$, and $\theta_n$ such that for all $\theta>\theta_n$,
\begin{equation}\label{e.current lower bound}
\P\left(\cL_1(0) \geq \theta\right) \geq \exp(-C_n \theta^{3/2}-\gamma_n\log\theta).
\end{equation}
We will show that then there exist $C_{n+1}$, $\gamma_{n+1}$, and $\theta_{n+1}$ such that, for $\theta>\theta_{n+1}$,
$$\P\left(\cL_1(0) \geq \theta\right) \geq \exp(-C_{n+1} \theta^{3/2}-\gamma_{n+1}\log\theta),$$
with the property that $\lim_{n\to\infty} C_n = \frac{4}{3}$; $\gamma_n$ and $\theta_n$ will go to infinity, but at a rate which contributes to the error term of $-\theta^{1/2}\log\theta$. To start the iterations, recall that we have \eqref{e.current lower bound} for $n=0$ for some constant $\theta_0$, $C_0 = 5$, and $\gamma_0=0$ by Lemma~\ref{l.first suboptimal constant}.

We define the $\sigma$-algebra $\F = \Fext(1,[-\frac{1}{2}\theta^{1/2}, \frac{1}{2}\theta^{1/2}])$. Consider the $\F$-measurable favourable event $\fav$ defined by
$$\fav := \left\{\cL_1(-\tfrac{1}{2}\theta^{1/2}) \geq 0\right\} \cap \left\{\cL_1(\tfrac{1}{2} \theta^{1/2}) \geq 0\right\}.$$
From \eqref{e.current lower bound}, the stationarity of $\cL_1(x) + x^2$, and the positive association of $\cL_1(-\tfrac{1}{2}\theta^{1/2})$ and $\cL_1(\tfrac{1}{2}\theta^{1/2})$ from Assumption~\ref{as.corr}(a), we see that
\begin{equation}\label{e.fav lower bound}
\begin{split}
\P(\fav) \geq \P(\cL_1(0) > \tfrac{1}{4}\theta)^2
&\geq \exp\left(-2C_n 4^{-3/2}\theta^{3/2} - 2\gamma_n\log(\theta/4)\right)\\
&\geq \exp\left(-\tfrac{1}{4}C_n\theta^{3/2} - 2\gamma_n\log\theta\right)
\end{split}
\end{equation}
for $\theta> 4\theta_n$. Now,
\begin{align*}
\P\big(\cL_1(0) \geq \theta\big) \geq \E\Big[\PF(\cL_1(0) \geq \theta) \cdot \one_{\fav}\Big]
\end{align*}
The Brownian Gibbs property says that $\cL_1$, on $[-\frac{1}{2}\theta^{1/2}, \frac{1}{2}\theta^{1/2}]$ and conditionally on $\F$, is distributed as a Brownian bridge with the appropriate endpoints and tilted by $\smash{W_{H_t}/Z_{H_t}}$ associated to the boundary conditions. As usual, on $\fav$, this Brownian bridge stochastically dominates the Brownian bridge $B$ with between $\smash{(-\frac{1}{2}\theta^{1/2}, 0)}$ and $\smash{(\frac{1}{2}\theta^{1/2},0)}$ and no lower boundary condition. Now, $B(0)$ is a normal random variable with mean zero and variance $\smash{\sigma^2 = \frac{1}{2} \theta^{1/2}}$. Thus on $\fav$ and using \eqref{e.fav lower bound} and a lower bound on normal tails from Lemma~\ref{l.normal bounds}, the previous display is lower bounded by
\begin{align*}
\exp\left(-\theta^{3/2}\left[1+ \tfrac{1}{4}C_n\right]   -(2\gamma_n + 1)\log\theta\right)
\end{align*}
(lower bounding the constant factor in the normal bound by $\theta^{-1/4}$ for convenience).

Thus we have shown that, if $\P(\cL_1(0) > \theta) \geq \exp(-C_n\theta^{3/2}-\gamma_n\log\theta)$ for $\theta > \theta_n$, then the same holds true with $n$ replaced by $n+1$, with $C_{n+1} = 1+C_n/4$, $\gamma_{n+1} = 2\gamma_n + 1$, and $\smash{\theta_{n+1} = 4\theta_n}$. It is easy to solve these recurrences to get $\smash{\theta_n = 4^n\theta_0}$, $\smash{\gamma_n = 2^n-1}$ and (using $C_0 = 5$)
$$C_n = \frac{4}{3}\left(1+11\cdot4^{-n-1}\right).$$
This establishes that, for $n\in\N$ and $4^n\theta_0 < \theta \leq 4^{n+1}\theta_0$,
$$\P\left(\cL_1(0) > \theta\right) \geq \exp\left(-\frac{4}{3}(1+11\cdot 4^{-n-1})\theta^{3/2} - 2^n \log \theta\right).$$
Given $\theta>\theta_0$, we choose $n$ such that the condition for the previous display holds. Now observing that $\smash{4^{-n-1} \leq \theta_0 \theta^{-1}}$ and $\smash{2^n \leq (\theta/\theta_0)^{1/2}}$ yields the claim.
\end{proof}

\subsection{Density bounds}\label{s.one point.density}

In this subsection we restate and prove Theorem~\ref{mt.one point density asymptotics} on the upper and lower bounds on the density of $\cL_1(0)$, which, recall, we denote by $\frac{1}{\dif\theta}\P(\cL_1(0)\in[\theta,\theta+\dif\theta])$:

\setcounter{maintheorem}{0}
\begin{maintheorem}\label{t.density}
Let $\cL$ satisfy Assumptions~\ref{as.bg}--\ref{as.tails}. There exist constants $C$ and $\theta_0$ such that, for $\theta>\theta_0$,
$$\exp\left(-\frac{4}{3}\theta^{3/2}-C\theta^{3/4}\right) \leq \frac{1}{\dif \theta}\P\Bigl(\cL(0)\in[\theta,\theta+\dif\theta]\Bigr) \leq \exp\left(-\frac{4}{3}\theta^{3/2}+C\theta^{3/4}\right).$$
\end{maintheorem}

Starting with the basic fact that 
$$\int_{\theta}^{\theta+1}\frac{1}{\dif\theta}\P(\cL_1(0)\in[\theta,\theta+\dif\theta])\,\dif\theta = \P(\cL_1(0)\in[\theta,\theta+1]) \leq \P(\cL_1(0) \geq \theta),$$
and given the upper bound we have on the right-hand side via Theorem~\ref{mt.one point tail asymptotics}, to obtain a bound on the density at a particular value instead of its integral, we will seek to establish that the density satisfies some regularity property. For instance, it would be sufficient to know that the density is decreasing at least for all large arguments. However, as plausible as it might sound, proving this does not seem straightforward. Instead, we establish that it doesn't decay too fast which also turns out to suffice allowing us to approximate the integral by the density's value at a point.

\begin{proposition}\label{p.density control}
Let $\cL$ satisfy Assumptions~\ref{as.bg}--\ref{as.tails}. There exist $M>0$ and $\theta_0$ such that, for $\theta>\theta_0$ and $s>0$,
\begin{align*}
\MoveEqLeft[10]
\frac{1}{\dif\theta}\P\left(\cL_1(0) \in [\theta+s,\theta+s+\dif\theta]\right)\\
&\geq \frac{1}{2}\cdot\frac{1}{\dif\theta}\P\left(\cL_1(0) \in [\theta,\theta+\dif\theta]\right)\cdot\exp\left(-2s\theta^{1/2}-Ms\theta^{-1/4}-s^2\theta^{-1/2}\right).
\end{align*}
In particular there exists (a slightly larger) $M>0$ such that, if $0<s\leq\theta^{1/4}$,
\begin{equation}\label{e.density regularity}
\frac{1}{\dif\theta}\P\left(\cL_1(0) \in [\theta+s,\theta+s+\dif\theta]\right)
\geq \frac{1}{\dif\theta}\P\left(\cL_1(0) \in [\theta,\theta+\dif\theta]\right)\cdot\exp\left(-2s\theta^{1/2}-M\right).
\end{equation}
\end{proposition}

While our eventual goal of obtaining an upper bound on the density does allow a lot of room in the comparison estimate, note that the coefficient of $-2$ in front of $s\theta^{1/2}$ in the exponential is in fact sharp: anticipating that the density at $\theta$ is approximately $\exp(-\frac{4}{3}\theta^{3/2})$, we see that
\begin{align*}
\exp\left(-\frac{4}{3}(\theta+s)^{3/2}\right) = \exp\left(-\frac{4}{3}\theta^{3/2}(1+s\theta^{-1})^{3/2}\right)
&\approx \exp\left(-\frac{4}{3}\theta^{3/2} - \frac{4}{3}\theta^{3/2}\cdot \frac{3}{2}s\theta^{-1}\right)\\
&= \exp\left(-\frac{4}{3}\theta^{3/2} - 2s\theta^{1/2}\right).
\end{align*}
The sharpness of the estimate in Proposition~\ref{p.density control} turns out to be crucial in obtaining the tight lower bound on the density in Theorem~\ref{mt.one point density asymptotics}.

Before turning to the proof of Proposition~\ref{p.density control}, we give the proof of Theorem~\ref{t.density}.
To prove the lower bound in Theorem~\ref{t.density}, we will need a preliminary uniform lower bound on the density. This is a consequence of Proposition~\ref{p.density control} as well as both the upper and lower bounds on the one-point upper tail from Theorems~\ref{t.upper tail upper bound} and \ref{t.upper tail lower bound}.

\begin{lemma}\label{l.crude uniform density lower bound}
Let $\cL$ satisfy Assumptions~\ref{as.bg}--\ref{as.tails}. There exists $c>0$ and $\theta_0>0$ such that, %
$$\tfrac{1}{\dif\theta}\P\left(\cL_1(0)\in[\theta,\theta+\dif\theta]\right)\Bigr|_{\theta=\theta_0}\geq c.$$
\end{lemma}

\begin{proof}
Let $g(\theta)$ be the density of $\cL_1(0)$. First we observe that, for any $\theta_0>0$
$$\sup_{x\in[\frac{1}{2}\theta_0,\theta_0]} g(x)\geq 2\theta_0^{-1}\P(\cL_1(0) \in[\tfrac{1}{2}\theta_0,\theta_0]).$$
It is easy to check that a uniform lower bound on the right-hand side follows from the upper and lower bounds on the one-point upper tail from Theorem~\ref{mt.one point tail asymptotics} if we take $\theta_0$ large enough.
Then to see that the previous display implies a lower bound on $g(\theta_0)$, we apply Proposition~\ref{p.density control}'s statement that, for $s>0$,
$$g(\theta_0)\geq \frac{1}{2}g(\theta_0-s)\cdot\exp\left(-2s(\theta_0-s)^{1/2}-Ms(\theta_0-s)^{-1/4}-s^2(\theta_0-s)^{-1/2}\right),$$
with $s\in[0, \frac{1}{2}\theta_0]$ such that that $g(\theta_0-s)\geq \frac{1}{2}\sup_{x\in[\frac{1}{2}\theta_0,\theta_0]} g(x)$.
\end{proof}

\begin{proof}[Proof of Theorem~\ref{t.density}]
Let $f(\theta) = \log \frac{1}{\dif\theta}\P(\cL_1(0)\in[\theta,\theta+\dif\theta])$. We start with the lower bound. 

Let $\theta_0$ be the value from Lemma~\ref{l.crude uniform density lower bound}, which is uniform in $t>t_0$, and we assume that $\theta$ is such that $\theta^{1/4}>\theta_0$. In terms of $f$ we observe that \eqref{e.density regularity} of Proposition~\ref{p.density control} (substituting $\theta-s$ for $\theta$) implies that, for $0< s\leq \smash{\frac{1}{2}\theta^{1/4}}$ (which implies that $0< s\leq \smash{(\theta-s)^{1/4}}$ as needed for \eqref{e.density regularity}),
$$f(\theta) \geq f(\theta-s) - 2s(\theta-s)^{1/2}-M.$$
We iterate this inequality $2\theta^{3/4}-1$ times with $s=\frac{1}{2}\theta^{1/4}$, and one last time with $s\in[0,\smash{\frac{1}{2}\theta^{1/4}}]$ such that $\theta-\frac{1}{2}\theta^{1/4}(2\theta^{3/4}-1)-s= \theta_0$. This yields
\begin{align}\label{e.density lower bound recursion}
f(\theta) \geq -2\cdot\tfrac{1}{2}\theta^{1/4}\sum_{i=1}^{2\theta^{3/4}}\bigl(\theta-i\cdot\tfrac{1}{2}\theta^{1/4}\bigr)^{1/2} - 3M\theta^{3/4} + f(\theta_0);
\end{align}
in the first term the final summand should be $\theta_0^{1/2}$ instead of 0, but we absorb this discrepancy into the $M\theta^{3/4}$ term (which is why its coefficient is $3$ instead of $2$). Since we know by our choice of $\theta_0$ that $f(\theta_0)$ is bounded from below, we may absorb $f(\theta_0)$ also into the $M\theta^{3/4}$ term for large enough $\theta$ by raising its coefficient to $4$.
Now, 
$$\theta^{1/4}\sum_{i=1}^{2\theta^{3/4}}\bigl(\theta-i\cdot\tfrac{1}{2}\theta^{1/4}\bigr)^{1/2} = \theta^{3/2}\cdot \theta^{-3/4}\sum_{i=1}^{2\theta^{3/4}}\bigl(1-i\cdot\tfrac{1}{2}\theta^{-3/4}\bigr)^{1/2}$$
and so equals $\theta^{3/2}$ times the (right) Riemann sum of $\int_0^2(1-\tfrac{1}{2}x)^{1/2}\,\dif x = \frac{4}{3}$. Since $(1-\tfrac{1}{2}x)^{1/2}$ is decreasing and we are considering the right Riemann sum, it follows that
$$\theta^{1/4}\sum_{i=1}^{2\theta^{3/4}}\bigl(\theta-i\cdot\tfrac{1}{2}\theta^{1/4}\bigr)^{1/2} \leq \frac{4}{3}\theta^{3/2}.$$
Substituting this into \eqref{e.density lower bound recursion} yields that $f(\theta) \geq -\frac{4}{3}\theta^{3/2} - M\theta^{3/4}$ for some $M$ (relabeling from its previous value).

For the upper bound, it follows from \eqref{e.density regularity} and the fact that $\P(\cL_1(0) \in [\theta,\theta+1]) \leq \exp(-\frac{4}{3}\theta^{3/2})$. In more detail, using \eqref{e.density regularity} for the second line,
\begin{align*}
\P\left(\cL_1(0) \in[\theta,\theta+1]\right) 
&= \int_{0}^{1} \frac{1}{\dif\theta}\P(\cL_1(0)\in[\theta+s,\theta+s+\dif \theta])\,\dif s\\
&\geq  \frac{1}{\dif\theta}\P(\cL_1(0)\in[\theta,\theta+\dif \theta])\int_{0}^{1}\exp(-2s\theta^{1/2}-M)\,\dif s\\
&\geq  \frac{1}{\dif\theta}\P(\cL_1(0)\in[\theta,\theta+\dif \theta])\cdot\exp(-2\theta^{1/2}-M).
\end{align*}
Since $\P(\cL_1(0)\in[\theta,\theta+1]) \leq \P(\cL_1(0)\geq \theta) \leq \exp\left(-\frac{4}{3}\theta^{3/2} + C\theta^{3/4}\right)$ by Theorem~\ref{t.upper tail upper bound}, we are done.
\end{proof}

Now we turn to the proof of Proposition~\ref{p.density control}. The basic argument relies on a resampling trick which allows us to take a configuration with $\cL_{1}(0)\in [\theta,\theta+\dif\theta]$ and construct one where $\cL_{1}(0)\in [\theta+s,\theta+s+\dif\theta]$ and in the process compare their ``probabilities."  However to carry this out efficiently, we will need more information about the distribution of $\cL_1$ than is available from just $H_t$-Brownian Gibbs and monotonicity statements, which have been the main ingredients of previous arguments. The idea here is to condition on more information, in such a way that the conditional distribution of $\cL_1(0)$ is more explicitly Gaussian; in the vanilla conditioning of the $H_t$-Brownian Gibbs property, the distribution of $\cL_1(0)$ is Gaussian, but reweighted by a Radon-Nikodym factor which is affected by $\cL_1$'s values at other points as well. This makes things hard to control, and the conditioning here will avoid these extra dependencies in the Radon-Nikodym derivative.

In more detail, here, in addition to conditioning on $\Fext(1,[a,b])$, where $[a,b]$ will be taken as $[-\theta^{1/2},\theta^{1/2}]$, we further include the \emph{bridges} of $\cL_1$ on $[-\theta^{1/2}, 0]$ and $[0,\theta^{1/2}]$ and call the resulting $\sigma$-algebra $\F$. The bridge $f^{[a,b]}$ of a function $f:I\to\R$ on an interval $[a,b]\subseteq I$ is the function obtained by affinely shifting $f$ to equal zero at $a$ and $b$; more explicitly, it is given by
$$x\mapsto f(x) - \frac{b-x}{b-a}f(a)-\frac{x-a}{b-a}f(b).$$
The effect of including this data in the $\sigma$-algebra we condition on is that the only information that remains random is the value of $\cL_1(0)$. By the $H_t$-Brownian Gibbs property, the conditional distribution of $\cL_1(0)$ is a suitably reweighted normal random variable (where the reweighting is only a function of $\cL_1(0)$), and the proof of Proposition~\ref{p.density control} relies crucially on this fairly explicit representation.  An important ingredient for this representation is that, conditionally on the richer $\sigma$-algebra, $\cL_1$ is given inside of $[-\theta^{1/2}, \theta^{1/2}]$ (recall it has been conditioned on outside of $[-\theta^{1/2}, \theta^{1/2}]$) by $\smash{\cL^{X}_1:[-\theta^{1/2}, \theta^{1/2}]}\to\R$, where
\begin{equation}\label{e.reconstructed narrow-wedge}
\cL^{x}_1(u) = \begin{cases}
\cL^{[-\theta^{1/2}, 0]}_1(u) + \frac{u+\theta^{1/2}}{\theta^{1/2}}\cdot x + \frac{-u}{\theta^{1/2}}\cdot\cL_1(-\theta^{1/2}) & u\in[-\theta^{1/2},0]\\
\cL^{[0,\theta^{1/2}]}_1(u) + \frac{\theta^{1/2}-u}{\theta^{1/2}}\cdot x + \frac{u}{\theta^{1/2}}\cdot\cL_1(\theta^{1/2}) & u\in[0,\theta^{1/2}],
\end{cases}
\end{equation}
 and $X$ is distributed according to the $\F$-conditional distribution of $\cL_1(0)$. 

Using \eqref{e.reconstructed narrow-wedge} we can describe the $\F$-conditional distribution of $\cL_1(0)$ more precisely: it is a normal distribution with $\F$-measurable mean $\mu = \smash{\frac{1}{2}(\cL_1(-\theta^{1/2}) + \cL_1(\theta^{1/2}))}$ and variance $\sigma^2 = 2\times\smash{\frac{\theta^{1/2}\times\theta^{1/2}}{2\theta^{1/2}}} = \theta^{1/2}$, which is reweighted by a Radon-Nikodym factor $W_t^{\mrm{pt}}(\cL_1(0))/Z_t^{\mrm{pt}}$, where $W_t^{\mrm{pt}}$ and $Z_t^{\mrm{pt}}$ are given by
\begin{align*}
W_t^{\mrm{pt}}(x) &= \exp\left(-\int_{-\theta^{1/2}}^{\theta^{1/2}} H_t\bigl(\cL_2(u) - \cL_1^{x}(u)\bigr)\, \dif u\right)\\
Z_t^{\mrm{pt}} &= \E_\F\left[W^{\mrm{pt}}(\cL_1(0))\right].
\end{align*}
Observe that $W^{\mrm{pt}}(x)$ is increasing in $x$. 

The idea of including the bridge data in Brownian Gibbs resamplings was introduced in \cite{hammond2016brownian} and has been used several times in subsequent studies \cite{calvert2019brownian,KPZfixedptHD}. The correctness of the above description of the conditional distribution is straightforward to verify using the $H_t$-Brownian Gibbs property; see for example proofs of similar statements in \cite[Section~4.1.3]{calvert2019brownian}.

With this description in hand we can turn to the proof of Proposition~\ref{p.density control}.

\begin{proof}[Proof of Proposition~\ref{p.density control}]
As above, let $\F$ be the $\sigma$-algebra generated by $\Fext(1,[a,b])$ and the bridges of $\cL_1$ on $[-\theta^{1/2},0]$ and $[0,\theta^{1/2}]$. Conditional on $\F$, the distribution of $\cL_1(0)$ is that of a normal random variable with an $\F$-measurable mean $\mu$ (whose formula is given above) and $\sigma^2=\theta^{1/2}$ reweighted by $W_t^{\mrm{pt}}(\cL(0))/Z_t^{\mrm{pt}}$.

Condition on $\F$, let $X$ be a normal random variable of mean $\mu$ and variance $\sigma^2$, and set $\fav=\smash{\{\mu\geq -\theta-M\theta^{1/4}\}}$ for an $M$ to be chosen (so that, from Theorem~\ref{t.limit shape} on the conditional limit shape, $\fav$ has uniformly positive probability when $M$ is large). Then,
\begin{align*}
\frac{1}{\dif \theta}\P\left(\cL_1(0) \in [\theta+s,\theta+s+\dif\theta]\right) 
&= \frac{1}{\dif \theta}\E\left[\PF\left(\cL_1(0) \in [\theta+s,\theta+s+\dif\theta]\right)\right]\\
&= \frac{1}{\dif \theta}\E\left[\EF\left[\one_{X\in [\theta+s,\theta+s+\dif \theta]}W^{\mrm{pt}}(X)\right](Z^{\mrm{pt}})^{-1}\right]\\
&=\frac{1}{\dif \theta}\E\left[\PF\left(X\in [\theta+s,\theta+s+\dif \theta]\right)W^{\mrm{pt}}(\theta+s)(Z^{\mrm{pt}})^{-1}\right]\\
&\geq \E\left[(2\pi\theta^{1/2})^{-1/2}\exp\left(-\frac{(\theta+s-\mu)^2}{2\theta^{1/2}}\right)W^{\mrm{pt}}(\theta+s)(Z^{\mrm{pt}})^{-1}\one_{\msf{Fav}}\right]\\
&\geq \E\biggl[(2\pi \theta^{1/2})^{-1/2}\exp\left(-\frac{(\theta-\mu)^2 +s^2+2s(\theta-\mu)}{2\theta^{1/2}}\right)\\
&\hspace{6cm}\times 
W^{\mrm{pt}}(\theta)(Z^{\mrm{pt}})^{-1}\one_{\msf{Fav}}\biggr];
\end{align*}
the last inequality using that $W^{\mrm{pt}}(x)$ is increasing in $x$. Next we reinterpret $(2\pi\theta^{1/2})^{-1/2}\exp(-(\theta-\mu)^2/(2\theta^{1/2}))W^{\mrm{pt}}(\theta)(Z^{\mrm{pt}})^{-1}$ as $\EF[\one_{X\in[\theta,\theta+\dif\theta]}W^{\mrm{pt}}(\theta)(Z^{\mrm{pt}})^{-1}]$ and use that $\theta -\mu\leq 2\theta+M\theta^{1/4}$ on $\fav$ to see that, for $s>0$,
\begin{align*}
\MoveEqLeft[6]
\frac{1}{\dif \theta}\P\left(\cL_1(0) \in [\theta+s,\theta+s+\dif\theta]\right)\\
&\geq \frac{1}{\dif \theta}\E\left[\EF\left[\one_{X\in[\theta,\theta+\dif\theta]}W^{\mrm{pt}}(\theta)(Z^{\mrm{pt}})^{-1}\right]\exp\left(-\frac{s^2+2s(\theta-\mu)}{2\theta^{1/2}}\right)\one_{\msf{Fav}}\right]\\
&=\frac{1}{\dif \theta}\E\left[\EF\left[\one_{X\in[\theta,\theta+\dif\theta]}W^{\mrm{pt}}(X)(Z^{\mrm{pt}})^{-1}\right]\exp\left(-\frac{s^2+2s(\theta-\mu)}{2\theta^{1/2}}\right)\one_{\msf{Fav}}\right]\\
&\geq\frac{1}{\dif \theta}\E\Bigl[\PF\left(\cL_1(0)\in[\theta,\theta+\dif\theta]\right)\\
&\qquad\times\exp\left(-2s\theta^{1/2}-Ms\theta^{-1/4}-\tfrac{1}{2}s^2\theta^{-1/2}\right)\one_{\msf{Fav}}\Bigr]\\
&\geq \frac{1}{\dif \theta}\P\left(\cL_1(0)\in[\theta,\theta+\dif\theta], \fav\right)\cdot\exp\Bigl(-2s\theta^{1/2}-Ms\theta^{-1/4}-s^2\theta^{-1/2}\Bigr),
\end{align*}
the last line using that $\fav$ is $\F$-measurable and the tower property of conditional expectations. Next, by Theorem~\ref{t.limit shape}, there exists $M$ large enough independent of $t$ and $\theta$ such that $\P(\cL_1(-\theta^{1/2}), \cL_1(\theta^{1/2})\geq -\theta-M\theta^{1/4}\mid \cL_1(0)=\theta) \geq \frac{1}{2}$. So, setting $M$ to be such a value,
\begin{align*}
\frac{1}{\dif \theta}\P\left(\cL_1(0)\in[\theta,\theta+\dif\theta], \fav\right)
&= \frac{1}{\dif \theta}\P\left(\cL_1(0)\in[\theta,\theta+\dif\theta]\right)\cdot\P\left(\fav \mid \cL_1(0) = \theta\right)\\
&= \frac{1}{\dif \theta}\P\left(\cL_1(0)\in[\theta,\theta+\dif\theta]\right)\cdot\P\left(\mu \geq -\theta - M\theta^{1/4} \mid \cL_1(0) = \theta\right)\\
&\geq \frac{1}{2}\cdot \frac{1}{\dif \theta}\P\left(\cL_1(0)\in[\theta,\theta+\dif\theta]\right).
\end{align*}
Putting it together, we have shown that, for some $M$ and all $s>0$,
\begin{equation*}
\frac{1}{\dif \theta}\P\left(\cL_1(0) \in[\theta+s,\theta+s+\dif\theta]\right) \geq \frac{1}{2}\frac{1}{\dif \theta}\P\left(\cL_1(0)\in[\theta,\theta+\dif\theta]\right)\cdot\exp(-2s\theta^{1/2}-Ms\theta^{-1/4}-s^2\theta^{-1/2}).\qedhere
\end{equation*}
\end{proof}

\section{Two-point limit shape}\label{s.two point limit shape}

We now prove Theorems~\ref{mt.two-point limit shape} and \ref{mt.two point tail} on two-point upper tail limit shapes and asymptotics, in this and the next sections respectively. As already indicated in Section~\ref{s.intro.proof ideas}, while conceptually all the ideas were present in the arguments for the one-point results, the arguments in these two sections are technically more complex and may be somewhat harder to parse, owing mostly to the more complicated formulas that arise in the two-point case.

As said before, the proofs will make clear that similar arguments would also yield sharp asymptotic expressions for $k$-point upper tails and limit shapes; see Remark~\ref{r.k-point}.

\begin{figure}[t]
\includegraphics[width=0.9\textwidth]{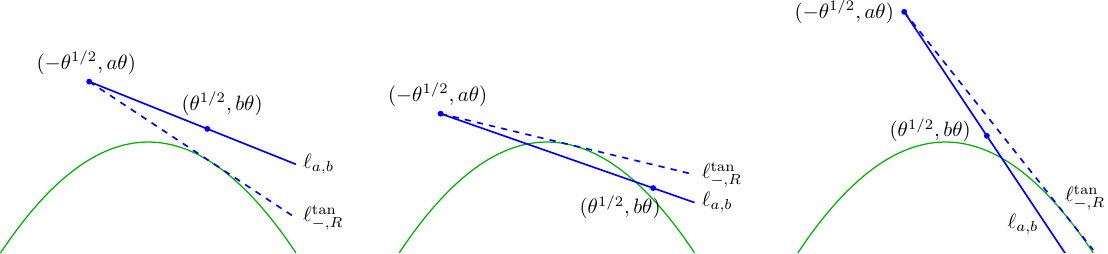}
\caption{A depiction of Lemma~\ref{l.convexity algebra}. Assuming $a\geq b>-1$, if $a\leq (\sqrt{1+b}+1)(\sqrt{1+b}+3) \iff (a-b)^2\leq 8(a+b)$, then $\ell_{a,b}$ stays above $\ellt_{-,R}$ to the right of $-\theta^{1/2}$, as depicted in the left panel. Thus $\ell_{a,b}$ intersects the parabola at at most one point. In the middle and right panels two geometrically distinct subcases of $(a-b)^2>8(a+b)$ (so that $\ell_{a,b}$ lies below $\ellt_{-,R}$ and intersects the parabola at two distinct points) are shown: in the middle panel, the intersection points of $\ell_{a,b}$ with $-x^2$ lie inside $[-\theta^{1/2}, \theta^{1/2}]$, while in the right panel they lie in $(\theta^{1/2},\infty)$.}\label{f.two-point algebra}
\end{figure}

The following technical lemma relates certain geometric conditions that will be relevant in the analysis with algebraic relations, and will be used many times to work with expressions that arise in the proofs. The geometric situations are depicted in Figure~\ref{f.two-point algebra}.

\begin{lemma}\label{l.convexity algebra}
Let $a\geq b> -1$. Let $\ell_{a,b}$ be the line through $\smash{(-\theta^{1/2}, a\theta)}$ and $\smash{(\theta^{1/2},b\theta)}$ and $\ellt_{-,R}$ be the line through $\smash{(-\theta^{1/2},a\theta)}$ which is tangent to $\smash{-x^2}$ at a point in $\smash{[-\theta^{1/2}, \infty)}$, i.e., right of $\smash{-\theta^{1/2}}$.

Suppose also $(a-b)^2 \leq 8(a+b)$, which, with $a\geq b>-1$, is equivalent to $b\geq (\sqrt{1+a}-1)(\sqrt{1+a}-3)$ and implies $a\leq (\sqrt{1+b}+1)(\sqrt{1+b}+3)$. Then, for $x\in[-\theta^{1/2}, \infty)$, $\ell_{a,b}(x)\geq \ellt_{-,R}(x)$ ; see Figure~\ref{f.two-point algebra}. If $(a-b)^2 > 8(a+b)$, then $\ell_{a,b}(x) < \ellt_{-,R}(x)$ for all $x\in [-\theta^{1/2}, \infty)$.

As a consequence, under the condition $(a-b)^2\leq 8(a+b)$, the convex hull of the function $x\mapsto -x^2$ and the points $(-\theta^{1/2}, a\theta)$ and $(\theta^{1/2}, b\theta)$ has both of the points as extreme points, and $\ell_{a,b}$ intersects the parabola at at most one point.

Finally, if $\ell_{a,b}$ and $\ellt_{-,R}$ coincide, then $(a-b)^2 = 8(a+b)$.
\end{lemma}

\begin{proof}
The equation of the line $\ellt_{-,R}$ that passes through $(-\theta^{1/2}, a\theta)$ and is tangent to $-x^2$ at a point in $[-\theta^{1/2}, \infty)$ is $-2x_0(x+\theta^{1/2})+a\theta$, where $x_0 = \theta^{1/2}(\sqrt{1+a}-1)$, while the equation of $\ell_{a,b}(x)$ is $(b-a)\theta^{1/2}(x+\theta^{1/2})/2 + a\theta$. Thus the condition that the latter equation lies above the former on $[-\theta^{1/2},\infty)$ is equivalent to
\begin{equation}\label{e.convexity algebra main ienquality}
\frac{(b-a)\theta^{1/2}}{2} \geq -2x_0 \iff b\geq a-4\sqrt{1+a} + 4 = (\sqrt{1+a}-1)(\sqrt{1+a}-3).
\end{equation}
Viewing the latter inequality as a quadratic in $\sqrt{1+a}$, it is easy to show that it implies that 
$$2-\sqrt{1+b} \leq \sqrt{1+a}\leq 2+\sqrt{1+b};$$
squaring the second inequality yields $a\leq(\sqrt{1+b}+1)(\sqrt{1+b}+3)$.

It can also be checked that \eqref{e.convexity algebra main ienquality} is implied by $(a-b)^2\leq 8(a+b)$, by solving the latter as a quadratic in $b$, which yields the following equivalent condition on $b$ when $a\geq -1$:
$$(\sqrt{1+a}-1)(\sqrt{1+a}-3)\leq b\leq (\sqrt{1+a}+1)(\sqrt{1+a}+3).$$
When $b\leq a$, the second inequality is automatically satisfied, so \eqref{e.convexity algebra main ienquality} along with $-1< b\leq a$ implies $(a-b)^2\leq 8(a+b)$. It is clear that $\ell_{a,b}$ being a tangent line corresponds to equality in all of the preceding inequalities.

A line lying below a tangent line on an infinite ray does not imply the line intersects the parabola at two points. So to see that $(a-b)^2 > 8(a+b)$ is equivalent to $\ell_{a,b}$ having two intersection points (counted with multiplicity) with $-x^2$, we solve the quadratic $-x^2 = \ell_{a,b}(x)$ and note that its discriminant (which needs to be non-negative) is $(b-a)^2/4-2(b+a)$. 
\end{proof}

Recall $\tent$ from before Theorem~\ref{mt.two-point limit shape}, and recall that $\Ilin$ is the smallest closed set outside of which $\tent(x) = -x^2$. On $\Ilin$, $\tent$ is piecewise linear. Observe from Lemma~\ref{l.convexity algebra} (see also Figure~\ref{f.three cases of two-point}) that $\Ilin$ is either an interval or a union of two disjoint intervals, depending on whether $\ell_{a,b}$ intersects $-x^2$ inside $[-\theta^{1/2}, \theta^{1/2}]$ or not.

For the reader's convenience we restate here Theorem~\ref{mt.two-point limit shape} before giving its proof in Section~\ref{s.two point limit shape proofs}.

\setcounter{maintheorem}{6}

\begin{maintheorem}\label{t.two-point limit shape}
Let $\cL$ satisfy Assumptions~\ref{as.bg}--\ref{as.tails}. Let $\theta > 0$ and $a \geq b > -1$. For $M>0$, let $M_{a,b} = M[(1+a)^{1/4}+(1+b)^{1/4}]$. There exist $c>0$, $C<\infty$, $\theta_0$, and $a_0=b_0$ such that, if $\theta > \theta_0$ or $a,b\geq a_0, b_0$, and for $0<M\leq C^{-1}[(1+a)^{3/4}+(1+b)^{3/4}]\theta^{3/4}$,
\begin{align}\label{e.two point limit prob bound}
\MoveEqLeft[30]
\P\left(\sup_{x \in \Ilin} (\cL_1(x) - \tent(x)) \geq M_{a,b}\theta^{1/4} \ \Big|\  \cL_1(-\theta^{1/2}) = a\theta, \cL_1(\theta^{1/2}) = b\theta \right)
\leq \exp(-cM^2)
\end{align}
and
\begin{equation}\label{e.two point limit prob lower bound}
\begin{split}
\MoveEqLeft[18]
\P\left(\inf_{x \in \Ilin} (\cL_1(x) - \tent(x)) \leq -M_{a,b}\theta^{1/4} \ \Big|\  \cL_1(-\theta^{1/2}) = a\theta, \cL_1(\theta^{1/2}) = b\theta \right) \\
&\leq \exp(-cM^2) + 8\cdot \P\left(\cL_1(0) \leq -\tfrac{1}{2}M_{a,b}\theta^{1/4}\right) .
\end{split}
\end{equation}
\end{maintheorem}

\begin{remark}\label{r.why not two-point limit shape on whole domain}
Note that we have restricted the statement of Theorem~\ref{t.two-point limit shape} to the set $\Ilin$ where $\tent$ is piecewise linear; in the case that $\Ilin$ is two disjoint intervals (see the middle panel of Figure~\ref{f.three cases of two-point}), this excludes, for example, the portion between the two intervals, where $\tent(x) = -x^2$. In fact we could prove a statement similar to Theorem~\ref{t.two-point limit shape} including this interval as well by making use of Proposition~\ref{p.para fluctuation}; we have not done so simply because it would complicate the statement (e.g., we would need $M>\log \theta$ just for this interval in order to apply Proposition~\ref{p.para fluctuation}), and we do not need this extra information to prove the two-point asymptotic Theorem~\ref{t.two point asymptotics}.
\end{remark}

\subsection{A useful estimate}

In the proof of the first half of Theorem~\ref{t.two-point limit shape}, i.e., \eqref{e.two point limit prob bound}, we will need to make a similar argument a number of times, namely that the second curve can be controlled and that that can then be used to show that the first curve follows a linear path with sub-Gaussian upper tails. We package that argument in the proof of the following  statement that covers all the cases that will arise next, after which we will turn to the proof of Theorem~\ref{t.two-point limit shape} in Section~\ref{s.two point limit shape proofs}.

\begin{proposition}\label{p.versatile Airy tangent estimate}
Suppose Assumptions~\ref{as.bg}, \ref{as.weak bk}(b\ensuremath{'}), \ref{as.mono in cond stronger}, and Assumption~\ref{as.tails}, with $\beta =\frac{3}{2}$,  hold. Let $I=[c,d]\subseteq \R$ be an interval and let $(c, y_{c})$ and $(d, y_{d})$ be points on the line $\ellt_{x_{\tan}}$ tangent to the curve $-x^2$ at the point $(x_\tan, -x_\tan^2)$ with $x_\tan\in[c+|I|^{1/2}, d - |I|^{1/2}]$. Let $\msf{BT}$ be the event $\{\cL_1(c) \leq y_{c}, \cL_1(d)\leq y_{d}\}$ of being below the tangent. 

Suppose $-\theta^{1/2},\theta^{1/2}\not\in(c,d)$, that $a\geq b > -1$. Then there exist absolute constants  $c>0$, $M_0$, and $C>0$ such that, if $|I|\geq \max(C,|b^{1/4}|)$ and $\max_{u\in\{c,d\}} |\theta^{1/2}- u| \leq C|I|$, then, for all $M_0< M < C^{-1}|I|^{3/4}$,
\begin{align*}
\MoveEqLeft[16]
\P\left(\sup_{x\in I}\left(\cL_1(x)-\ellt_{x_\tan}(x)\right) > M|I|^{1/2} \ \Big|\  \cL_1(-\theta^{1/2}) = a\theta, \cL_1(\theta^{1/2})=b\theta\right)\\
&\leq \exp(-cM^2)  + 3\cdot \P\left(\msf{BT}^c \mid \cL_1(-\theta^{1/2}) = a\theta, \cL_1(\theta^{1/2})=b\theta\right).
\end{align*}
\end{proposition}

We do not strictly need $\beta=\frac{3}{2}$ for Proposition~\ref{p.versatile Airy tangent estimate} to hold, but we do so as such a bound has already been established in Section~\ref{s.one point asymptotics}. 

\begin{proof}[Proof of Proposition~\ref{p.versatile Airy tangent estimate}]
Let $\varepsilon_0>0$ be the constant from Proposition~\ref{p.brownian versatile tangent estimate} and $\sigma_\tan^2 = (x_\tan-a)(b- x_\tan)/|I|$. Note that $\smash{\sigma_\tan^2 \geq |I|^{1/2}(1-|I|^{-1/2}) \geq \frac{1}{2}|I|^{1/2}}$ by our assumption that $x_\tan\in[a+|I|^{1/2}, b-|I|^{1/2}]$, since the minimum is attained at one of the boundaries, and since $|I|>4\implies |I|^{-1/2}\leq 2^{-1}$. To bound the probability appearing in the statement, we first break up the probability based on the occurrence of a favourable event concerning the second curve. Let $E = \{\sup_{x\in I} \left(\cL_2(x)+x^2\right) \leq \varepsilon_0 M\sigma_\tan\}$. Then,
\begin{align}
\MoveEqLeft[3]
\P\left(\sup_{x\in I}\left(\cL_1(x)-\ellt_{x_\tan}(x)\right)\geq M|I|^{1/2} \ \Big|\ \cL_1(-\theta^{1/2}) = a\theta, \cL_1(\theta^{1/2})=b\theta\right)\nonumber\\
&\leq \P\left(\sup_{x\in I}\left(\cL_1(x)-\ellt_{x_\tan}(x)\right) \geq M|I|^{1/2},\, \msf{BT},\, E \ \Big|\ \cL_1(-\theta^{1/2}) = a\theta, \cL_1(\theta^{1/2})=b\theta\right)\nonumber\\
&\quad + \P\left(\sup_{x\in I}\left(\cL_1(x)-\ellt_{x_\tan}(x)\right)\geq M|I|^{1/2}, E^c  \ \Big|\ \cL_1(-\theta^{1/2}) = a\theta, \cL_1(\theta^{1/2})=b\theta\right) \label{e.tangent probability for h^t breakup}\\
&\quad + \P\left(\msf{BT}^c \mid \cL_1(-\theta^{1/2}) = a\theta, \cL_1(\theta^{1/2})=b\theta\right).
\end{align}
By Lemma~\ref{l.lower curve control new} (with $f(x) = \ellt_{x_\tan}(x) - M|I|^{1/2}$ and $R$ such that $(\log R)^C = \varepsilon_0 M \sigma_{\tan}$) and since $\sigma_\tan \geq 2^{-1/2}|I|^{1/4}$, the second term of \eqref{e.tangent probability for h^t breakup} is upper bounded by
\begin{equation}\label{e.two point limit bk step}
\frac{1}{2}\cdot\frac{\P\left(\sup_{x\in I}\left(\cL_1(x)-\ellt_{x_\tan}(x)\right)\geq M|I|^{1/2} \midd \cL_1(-\theta^{1/2}) = a\theta, \cL_1(\theta^{1/2})=b\theta\right)}{\P\left(\inf_{x\in I}\left(\cL_1(x)-\ellt_{x_\tan}(x)\right)\geq M|I|^{1/2}, \cL_1(\theta^{1/2}) \geq b\theta \midd \cL_1(-\theta^{1/2}) = a\theta +  K\right)}
\end{equation}
for any $K$ such that $\varepsilon_0 M \sigma_{\tan} \geq \log (C((a+1)\theta+K)^{3/4})$. 
Take 
$$K= 2\max_{u\in\{c,d\}} |u-\theta^{1/2}|^2 + 2(b+D)\theta + 2MI^{1/2}$$
for a sufficiently large constant $D$. The value of $K$ is not special apart from the consideration that raising the value of $\cL_1$ at a given point by a level $K$ compared to the parabola $-x^2$ (ignoring $a\theta$ as it may be essentially at the parabola level when $a\approx -1$) only affects $\cL_1$ at a distance of $K^{1/2}$ from that point (from Theorem~\ref{mt.one point limit shape}), so we take a value of $K$ which is a multiple of the square of the distances that need to be affected and then add a multiple of how large we need to raise the points by.  It is easy to verify from the first part of Theorem~\ref{t.limit shape} that then the probability in the denominator in \eqref{e.two point limit bk step} is at least $3/4$ if $M$, $\theta$ are larger than some large enough constants. Since $\sigma_\tan$ is polynomial in $|I|$, $\max_{u\in\{c,d\}}|u-\theta^{1/2}|\leq C|I|$ by assumption, and $|I|\geq \max(C,|b|^{1/4})$ it is also immediate that the required condition on $K$ is satisfied. Thus \eqref{e.two point limit bk step} is at most 
$$\P\left(\sup_{x\in I}\left(\cL_1(x)-\ellt_{x_\tan}(x)\right)\geq M|I|^{1/2} \midd \cL_1(-\theta^{1/2}) = a\theta, \cL_1(\theta^{1/2})=b\theta\right),$$
which, substituting into \eqref{e.tangent probability for h^t breakup}, yields that
\begin{align*}
\MoveEqLeft[6]
\P\left(\sup_{x\in I}\left(\cL_1(x)-\ellt_{x_\tan}(x)\right)\geq M|I|^{1/2} \ \Big|\ \cL_1(-\theta^{1/2}) = a\theta, \cL_1(\theta^{1/2})=b\theta\right)\\
&\leq 3\cdot \P\left(\sup_{x\in I}\left(\cL_1(x)-\ellt_{x_\tan}(x)\right) \geq M|I|^{1/2},\, \msf{BT},\, E \ \Big|\ \cL_1(-\theta^{1/2}) = a\theta, \cL_1(\theta^{1/2})=b\theta\right)\\
& \qquad +3\cdot \P\left(\msf{BT}^c \mid \cL_1(-\theta^{1/2}) = a\theta, \cL_1(\theta^{1/2})=b\theta\right).
\end{align*}

Next we analyze the first term on the righthand side of the previous display. Let $\F = \Fext(1,I)$. Observe that $E = \{\sup_{x\in I}\left(\cL_2(x)+x^2\right) \leq \varepsilon_0 M \sigma_\tan\}\in \F$. Then we see that the first term equals
\begin{align*}
3\cdot\E\left[\PF\left(\sup_{x\in I}\left(\cL_1(x)-\ellt_{x_\tan}(x)\right) \geq M|I|^{1/2}\ \Big|\ \cL_1(-\theta^{1/2}) = a\theta, \cL_1(\theta^{1/2})=b\theta\right)\one_{\msf{BT},\, E}\right].
\end{align*}
Under $\PF$, $\cL_1$ is distributed as a Brownian bridge from $(c, \cL_1(c))$ to $(d, \cL_1(d))$ tilted by the Radon-Nikodym derivative $W_{H_t}/Z_{H_t}$ associated with lower boundary curve $\cL_2$ on $I$. 

On the event $\msf{BT} \cap \{\sup_{x\in I}\left(\cL_2(x)+x^2\right) \leq \varepsilon_0 M \sigma_{\tan}\}$, Lemma~\ref{l.monotonicity} (monotonicity) says that $\cL_1$ under $\PF$ is stochastically dominated by a Brownian bridge $B$ from $(c, y_{c})$ to $(d, y_{d})$, tilted by $\smash{\tilde W_{H_t}/\tilde Z_{H_t}}$ associated to the lower boundary curve $-x^2+ \varepsilon_0 M \sigma_\tan$ on $I$. 

Since $B$ does not depend on any data in $\F$ and $\tilde W_{H_t}\leq 1$, the previous display is bounded above by
$$\tilde Z_{H_t}^{-1}\cdot\P\left(\sup_{x\in J_{k}} \left(B(x) -\ellt_{x_\tan}(x)\right) \geq M|I|^{1/2}\right).$$
By Corollary~\ref{c.ratio of deviation prob and part func} and our choice of $\varepsilon$, the latter probability is bounded by $\exp(-cM^2)$, so we have proved the first inequality of Proposition~\ref{p.versatile Airy tangent estimate}; the second inequality is proved in the same way by considering the event $\{\sup_{x\in I}(\cL_1(x) -\ellt_\beta(x)) > M|I|^{1/2}\}$ in place of $\{\sup_{x\in J_{k}} (\cL_1(x) + x^2) \geq M|I|^{1/2}\}$.
\end{proof}

\subsection{Proof of the two-point limit shapes}\label{s.two point limit shape proofs}

Here we prove Theorem~\ref{t.two-point limit shape}. We start with the first half, the upper bound.

\begin{proof}[Proof of Theorem~\ref{t.two-point limit shape}, upper bound]
Here we prove \eqref{e.two point limit prob bound}. 
As noted above, there are two cases: (1) $\Ilin$ is a union of two disjoint intervals (2) $\Ilin$ is a single interval (and so $\tent$ is piecewise linear inside $\Ilin$).

\medskip

\emph{Case 1: $\Ilin$ is a union of two disjoint intervals.} Let $\xtan_{\pm,\ell}$ and $\xtan_{\pm, r}$ be defined by $\Ilin = [\xtan_{-,\ell}, \xtan_{-,r}] \cup [\xtan_{+,\ell}, \xtan_{+,r}]$. Note that (as can be seen by looking at the middle panel of Figure~\ref{f.three cases of two-point}) $-\theta^{1/2}\in [\xtan_{-,\ell}, \xtan_{-,r}]$ and $\theta^{1/2}\in[\xtan_{+,\ell}, \xtan_{+,r}]$. 

We argue \eqref{e.two point limit prob bound} with $[\xtan_{-,\ell}, \xtan_{-,r}]$ in place of $\Ilin$; the same argument will apply to $[\xtan_{+,\ell}, \xtan_{+,r}]$. Define the event
$$\dev^{\rmleft}_{-,M} = \left\{\sup_{x\in[\xtan_{-,\ell}, -\theta^{1/2}]} \left(\cL_1(x) - \tent(x)\right) \geq M_{a,b}\theta^{1/4}\right\}$$
and similarly $\smash{\dev^{\rmright}_{-,M}}$ by replacing $\smash{[\xtan_{-,\ell}, -\theta^{1/2}]}$ with $[-\theta^{1/2}, \xtan_{-,r}]$. It is enough to show that the conditional probability given the values of $\smash{\cL_1(\pm \theta^{1/2})}$ of each of these events is at most $\exp(-cM^2)$.

Let $\ellt_{-,L}$ and $\ellt_{-,R}$ be the tangent lines passing through $(-\theta^{1/2}, a\theta)$ on the left and right sides respectively.  We now find pinning points on either side of $[\xtan_{-,\ell}, \xtan_{-,r}]$, i.e., points on the left and right of the interval at which $\smash{\cL_1}$ is below $\smash{\ellt_{-,L}}$ and $\smash{\ellt_{-,R}}$ respectively with high probability, conditionally on $\smash{\cL_1(-\theta^{1/2}) = a\theta}$ and $\smash{\cL_1(\theta^{1/2}) = b\theta}$. 

To obtain a pinning point on the right side, observe (as can be seen from the middle panel of Figure~\ref{f.two-point algebra}) that the assumption that $\Ilin$ is two disjoint intervals implies that $(\theta^{1/2}, b\theta)$ is below $\smash{\ellt_{-,R}(\theta^{1/2})}$, and so $\smash{\theta^{1/2}}$ serves as the pinning point.  Recall that $[-\theta^{1/2}, \theta^{1/2}] \supseteq [-\theta^{1/2}, \xtan_{-,r}]$. Thus we see, by applying Proposition~\ref{p.versatile Airy tangent estimate} with $I = \smash{[-\theta^{1/2}, \theta^{1/2}]}$ and since the pinning points $\pm\theta^{1/2}$ we have chosen satisfy $\P(\msf{BT}\mid \cL_1(-\theta^{1/2}) = a\theta, \cL_1(\theta^{1/2})=b\theta) = 1$,
\begin{align*}
\P\left(\dev^{\rmright}_{-,M} \mid \cL_1(-\theta^{1/2}) = a\theta, \cL_1(\theta^{1/2})=b\theta\right) \leq \exp(-cM^2).
\end{align*}
To find a pinning point on the left we make use of stationarity and parabolic decay, as in the proof of Theorem~\ref{t.limit shape}. By Lemma~\ref{l.bound point conditioning by tail conditioning} (monotonicity in conditioning),
\begin{equation*}\label{e.two point left pinning}
\begin{split}
\MoveEqLeft[10]
\P\left(\cL_1(-x_L) > \ellt_L(-x_L) \ \big|\  \cL_1(-\theta^{1/2}) = a\theta, \cL_1(\theta^{1/2}) = b\theta\right)\\
&\leq \P\left(\cL_1(-x_L) > \ellt_L(-x_L) \ \big|\  \cL_1(-\theta^{1/2}) \geq a\theta, \cL_1(\theta^{1/2}) \geq b\theta\right)\\
&\leq \frac{\P\left(\cL_1(-x_L) > \ellt_L(-x_L)\right)}{\P\left(\cL_1(-\theta^{1/2}) \geq a\theta, \cL_1(\theta^{1/2}) \geq b\theta\right)}.
\end{split}
\end{equation*}
Now we can lower bound the denominator by the FKG inequality (Assumption~\ref{as.corr}(a)) and Theorem~\ref{t.upper tail lower bound}. Upper bounding the numerator by Theorem~\ref{t.upper tail upper bound}, and using stationarity and parabolic decay, we can find $-x_L < \xtan_{-,\ell}$ which is $O([(1+a)^{1/2}+(1+b)^{1/2}]\theta^{1/2})$ such that, for an absolute constant $c>0$,
\begin{align*}
\P\left(\cL_1(-x_L) > \ellt_L(-x_L) \ \big|\  \cL_1(-\theta^{1/2}) = a\theta, \cL_1(\theta^{1/2}) = b\theta\right) \leq \exp\left(-c[(1+a)^{3/2}+(1+b)^{3/2}]\theta^{3/2}\right),
\end{align*}
which is further bounded by $\exp(-cM^2)$ since $M\leq C^{-1}[(1+a)^{3/4}+(1+b)^{3/4}]\theta^{3/4}$.
Now using this and Proposition~\ref{p.versatile Airy tangent estimate} with $I=[-x_{L}, -\theta^{1/2}]$ we see that
\begin{equation}\label{e.devleft split}
\begin{split}
\P\left(\dev^{\rmleft}_{-,M} \ \big|\  \cL_1(-\theta^{1/2}) = a\theta, \cL_1(\theta^{1/2}) = b\theta \right)\leq 4\exp(-cM^2).
\end{split}
\end{equation}
This completes the proof by absorbing $4$ into the exponential by reducing $c$.

\bigskip
\emph{Case 2: $\Ilin$ is one interval.} 
We label $\Ilin$ as $[\xtan_\ell, \xtan_r]$. These are locations of tangency of the tangents to $-x^2$ which pass through $(\smash{-\theta^{1/2}}, a\theta)$ and $(\smash{\theta^{1/2}}, b\theta)$. For reference in a future proof we record the equations of the tangents as
\begin{equation}\label{e.two point tangency lines}
\begin{split}
y &= 2x_\ell^{\mathrm{tan}}(x_\ell^{\mathrm{tan}}-x)-(x_\ell^{\mathrm{tan}})^2,\\
y &= -2x_r^{\mathrm{tan}}(x-x_r^{\mathrm{tan}}) - (x_r^{\mathrm{tan}})^2,
\end{split}
\end{equation}
where
\begin{equation}\label{e.two-point tangency points}
\xtan_\ell = -(1+\sqrt{1+a})\theta^{1/2} \qquad\text{and}\qquad \xtan_r = (1+\sqrt{1+b})\theta^{1/2},
\end{equation}
as can be calculated by using the information that the lines in \eqref{e.two point tangency lines} pass through $(-\theta^{1/2},a\theta)$ and $(\theta^{1/2}, b\theta)$ respectively and solving.

Let $\Ilin^{\mathrm{left}} = [\xtan_\ell, -\theta^{1/2}]$, $\Ilin^{\mathrm{cent}} = [-\theta^{1/2}, \theta^{1/2}]$, and $\Ilin^{\mathrm{right}} = [\theta^{1/2}, \xtan_r]$. Define the event
$$\dev^\rmleft_M = \left\{\sup_{x \in \Ilin^\rmleft} \left(\cL_1(x) - \tent(x)\right) \geq M_{a,b}\theta^{1/4}\right\},$$
and similarly define $\dev^\rmcent_M$ and $\dev^\rmright_M$ with $\Ilin^\rmcent$ and $\Ilin^\rmright$ in place of $\Ilin^\rmleft$; thus we are trying to bound 
$$\P\left(\dev^\rmleft_M\cup\dev^\rmcent_M\cup\dev^\rmright_M \ \big|\  \cL_1(-\theta^{1/2})= a\theta,\cL_1(\theta^{1/2})= b \theta\right).$$
The conditional probabilities of $\dev^\rmleft_M$ and $\dev^\rmright_M$ are bounded by the same argument as was explained for bounding the conditional probability of $\dev^{\rmleft}_{-,M}$ in Case 1. So we turn to $\dev^{\rmcent}_M$.

By applying the Brownian Gibbs property, $\P(\dev^\rmcent_M \ \big|\  \cL_1(-\theta^{1/2})= a\theta,\cL_1(\theta^{1/2})= b \theta)$ is the probability that a Brownian bridge between $(-\theta^{1/2},a\theta)$ and $(\theta^{1/2}, b\theta)$ which is conditioned to stay above $\cL_2$ has a deviation greater than $M_{a,b}\theta^{1/4}$. As in arguments already presented, we use Lemma~\ref{l.lower curve control new} to control $\cL_2$. This yields
\begin{align*}
\MoveEqLeft[2]
\P\left(\dev^{\rmcent}_M, \sup_{x\in\Ilin^\rmcent}\left(\cL_2(x)+x^2\right) > \varepsilon M_{a,b}\theta^{1/4} \ \Big|\  \cL_1(-\theta^{1/2}) =a\theta, \cL_1(\theta^{1/2}) = b\theta\right)\\
&\leq \frac{\P\left(\dev^{\rmcent}_M \ \Big|\  \cL_1(-\theta^{1/2}) =a\theta, \cL_1(\theta^{1/2}) = b\theta\right)}{2\cdot \P\left(\inf_{x \in \Ilin^\rmleft} \left(\cL_1(x) - \tent(x)\right) \geq M_{a,b}\theta^{1/4}, \cL_1(\theta^{1/2}) \geq b\theta \midd \cL_1(-\theta^{1/2}) =a\theta + K\right)},
\end{align*}
as long as $K$ is such that $\varepsilon M_{a,b}\theta^{1/4} \geq \log(C((a+1)\theta+K)^{3/4} )$. As in the proof of Proposition~\ref{p.versatile Airy tangent estimate}, taking $K = 8\theta + 2(b+D)\theta + 2MI^{1/2} - (a+1)\theta$, for a large enough constant $D$, suffices to lower bound the denominator by $\frac{3}{4}$. This can be used to conclude that
\begin{align*}
\MoveEqLeft[5]
\P\left(\dev^{\rmcent}_M \ \Big|\  \cL_1(-\theta^{1/2}) =a\theta, \cL_1(\theta^{1/2}) = b\theta\right)\\
&\leq 3\cdot \P\left(\dev^{\rmcent}_M, \sup_{x\in\Ilin^\rmcent}\left(\cL_2(x)+x^2\right) \leq \varepsilon M\theta^{1/4} \ \Big|\  \cL_1(-\theta^{1/2}) =a\theta, \cL_1(\theta^{1/2}) = b\theta\right).
\end{align*}

By applying the Brownian Gibbs property and then using monotonicity (Lemma~\ref{l.monotonicity}) we can replace the lower boundary condition in the Gibbs measure by $-x^2+\varepsilon M_{a,b}\theta^{1/4}$. Since in this subcase the line joining the endpoints of the Brownian bridge is tangent to or lies above $-x^2$, we get a Gaussian bound on the previous display by Corollary~\ref{c.ratio of deviation prob and part func}.
\end{proof}

\begin{proof}[Proof sketch of Theorem~\ref{t.two-point limit shape}, lower bound]
Here we give a sketch of the proof of \eqref{e.two point limit prob lower bound}, as it is quite similar to the lower bound of Theorem~\ref{t.limit shape}. We start with the case that $\Ilin$ is a single interval $[x^{\tan}_{\ell}, x^{\tan}_r]$. We call the event we are bounding the conditional probability of by $\msf{LowDev}_M$, and define $\msf{BdyCtrl}_M=\{\min(\cL_1(\xtan_\ell)+(\xtan_\ell)^2, \cL_1(\xtan_r)+(\xtan_r)^2) >  - \tfrac{1}{2}M_{a,b}\theta^{1/4}\}$.  We write
\begin{equation}\label{e.two point lower bound first split}
\begin{split}
\MoveEqLeft[10]
\P\left(\msf{LowDev}_M \midd \cL_1(-\theta^{1/2}) =a\theta, \cL_1(\theta^{1/2}) = b\theta\right)\\
&\leq \P\left(\msf{LowDev}_M,  \msf{BdyCtrl}_M\midd \cL_1(-\theta^{1/2}) =a\theta, \cL_1(\theta^{1/2}) = b\theta\right)\\
&\qquad + \P\left(\msf{BdyCtrl}_M^c \midd \cL_1(-\theta^{1/2}) =a\theta, \cL_1(\theta^{1/2}) = b\theta\right).
\end{split}
\end{equation}
By a union bound, Lemma~\ref{l.bound point conditioning by tail conditioning}, Assumption~\ref{as.tails}, and stationarity, it follows that, for large enough $a\theta$ and $b\theta$ such that $\P(\cL_1(-\theta^{1/2}) \leq a\theta, \cL_1(\theta^{1/2}) \leq b\theta)\geq \frac{1}{2}$,
\begin{align*}
\P\left(\msf{BdyCtrl}_M^c \midd \cL_1(-\theta^{1/2}) =a\theta, \cL_1(\theta^{1/2}) = b\theta\right)
&\leq \P\left(\msf{BdyCtrl}_M^c \midd \cL_1(-\theta^{1/2}) \leq a\theta, \cL_1(\theta^{1/2}) \leq b\theta\right)\\
&\leq 4\cdot \P\left(\cL_1(0) \leq -\tfrac{1}{2}M_{a,b}\theta^{1/4}\right).
\end{align*}
Now we must bound the first term of the righthand side of \eqref{e.two point lower bound first split}. We apply the Gibbs property to $[\xtan_\ell, \xtan_r]$ and use monotonicity to drop the second curve. This yields that the probability is upper bounded by that of a Brownian bridge from $(\xtan_\ell, \cL_1(\xtan_\ell))$ to $(\xtan_r, \cL_1(\xtan_r))$  which is conditioned to pass through $(-\theta^{1/2}, a\theta)$ and $(\theta^{1/2}, b\theta)$ satisfying $\msf{LowDev}_M$. This can be bounded by breaking up the Brownian bridge into three Brownian bridges on the intervals $[\xtan_\ell, -\theta^{1/2}]$, $[-\theta^{1/2}, \theta^{1/2}]$, and $[\theta^{1/2}, \xtan_r]$ (with the prescribed endpoint values), and applying Lemma~\ref{l.brownian bridge sup tail exact}. See the second half of the proof of Theorem~\ref{t.limit shape}'s lower bound for a very similar argument.

The case where $\Ilin$ is two intervals is done in the same way by treating each interval separately.
\end{proof}

\subsection{Control on the marginal height when conditioned on the upper tail}

The following lemma will be needed for the upper bound on the two-point upper tail that we prove in the next section. Roughly, as in the proof of Theorem~\ref{t.upper tail upper bound}, we will first upper bound the probability that $\cL_1(-\theta^{1/2})\in[a\theta, a\theta+1], \cL_1(\theta^{1/2})\in[b\theta, b\theta+1]$ instead of the probability of the entire upper tail, and to do so we will consider the probability that a Brownian bridge $B$ with appropriate endpoints near the endpoints of $\linhull$ satisfies $B(-\theta^{1/2})\geq a\theta, B(\theta^{1/2})\geq b\theta$, i.e., we come back to the entire upper tail for $B$ (this is due to invoking monotonicity at an earlier point in the argument, as in the one-point case).

To estimate this probability we will break it up into $\P(B(-\theta^{1/2})\geq a\theta)\cdot\P(B(\theta^{1/2})\geq b\theta\mid B(-\theta^{1/2})\geq a\theta)$; to estimate the second probability we will need to know that $B(-\theta^{1/2})$ is not too high above $a\theta$. Observe that such considerations did not arise in the one-point case. The following implies exactly this control on the margin of $B(-\theta^{1/2})$ above its conditioned value; see also Figure~\ref{f.not too much margin} for a depiction of the situation.

\begin{lemma}\label{l.brownian bridge margin}
Suppose $a\geq b>-1$ and $(a-b)^2\leq8(a+b)$. For $M\in\R$, let $M_{a,b} = M((1+a)^{1/4}+(1+b)^{1/4})$ and let $B$ be a rate two Brownian bridge from $(\xtan_\ell, -(\xtan_\ell)^2+M_{a,b}\theta^{1/4})$ to $(\xtan_{r}, -(\xtan_r)^2+M_{a,b}\theta^{1/4})$, with $\xtan_\ell$ and $\xtan_r$ from \eqref{e.two-point tangency points}. Then there exists an absolute constant $M_0$ such that, for $\theta>0$ and $M>M_0$,
$$\P\left(B(-\theta^{1/2})\geq a\theta+2M_{a,b}\theta^{1/4}\text{ or } B(\theta^{1/2})\geq b\theta+2M_{a,b}\theta^{1/4} \ \Big|\  B(-\theta^{1/2})\geq a\theta, B(\theta^{1/2})\geq b\theta\right) \leq \frac{1}{2}.$$
\end{lemma}

\begin{proof}
Let $y,z$ be such that $(\xtan_\ell,y)$ and $(\xtan_r,z)$ lie on the line $\ell$ connecting $(-\theta^{1/2},a\theta+M_{a,b}\theta^{1/4})$ and $(\theta^{1/2},b\theta+M_{a,b}\theta^{1/4})$; see Figure~\ref{f.not too much margin}. The convexity hypothesis that $(a-b)^2\leq 8(a+b)$ implies that $y>-(\xtan_\ell)^2+M_{a,b}\theta^{1/4}$ and $z>-(\xtan_r)^2+M_{a,b}\theta^{1/4}$.

Let $\tilde B$ be a rate 2 Brownian bridge from $(\xtan_\ell,y)$ to $(\xtan_r,z)$. Let $I=[\xtan_\ell, \xtan_r]$. Note that $\ell(x)-M_{a,b}\theta^{1/4}$ passes through $(-\theta^{1/2},a\theta)$ and $(\theta^{1/2},b\theta)$. By monotonicity (Lemma~\ref{l.monotonicity}),
\begin{align*}
\MoveEqLeft[1]
\P\left(B(-\theta^{1/2})\geq a\theta+2M_{a,b}\theta^{1/4}\text{ or } B(\theta^{1/2})\geq b\theta+2M_{a,b}\theta^{1/4} \ \Big|\  B(-\theta^{1/2})\geq a\theta, B(\theta^{1/2})\geq b\theta\right)\\
&\leq \P\left(\tilde B(-\theta^{1/2})\geq a\theta+2M_{a,b}\theta^{1/4} \text{ or } \tilde B(\theta^{1/2})\geq b\theta+2M_{a,b}\theta^{1/4} \ \Big|\  \inf_{x\in I}(\tilde B(x) - \ell(x) )\geq -M_{a,b}\theta^{1/4}\right)\\
&\leq \P\left(\sup_{x\in I}(\tilde B(x)-\ell(x))\geq M_{a,b}\theta^{1/4} \ \Big|\  \inf_{x\in I}(\tilde B(x)- \ell(x))\geq -M_{a,b}\theta^{1/4}\right)
\leq 2\exp(-cM_{a,b}^2),
\end{align*}
using Lemma~\ref{l.brownian bridge sup tail exact} on the tail of the supremum of a Brownian bridge and for $c =2\theta^{1/2}(\xtan_r-\xtan_\ell)^{-1}=2(2+\sqrt{1+a}+\sqrt{1+b})^{-1}$ (the second equality by \eqref{e.two-point tangency points}). For large enough $M$ (independent of $a,b$) this probability is less than $1/2$, as required.
\end{proof}

\begin{figure}
\includegraphics[scale=1.1]{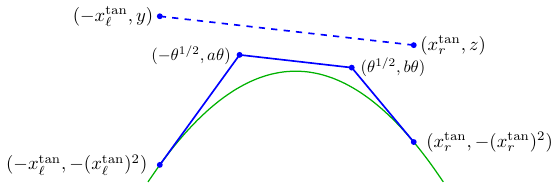}
\caption{The solid blue lines show $\linhull$ in the main case of Theorem~\ref{t.two-point limit shape}, when $(a-b)^2\le 8(a+b)$. The dotted line above is at a height of $M_{a,b}\theta^{1/4}$ above the line connecting $(-\theta^{1/2},a\theta)$ and $(\theta^{1/2}, b\theta)$ and appears in the proof of Lemma~\ref{l.brownian bridge margin}.}\label{f.not too much margin}
\end{figure}

\section{Two-point estimates}\label{s.two point asymptotics}

Here we prove Theorem~\ref{mt.two point tail}, after reformulating it slightly using Lemma~\ref{l.convexity algebra}. Recall that Theorem~\ref{mt.two point tail} has three cases depending on the number of extreme points of the convex hull $\linhull$ inside $[-\theta^{1/2}, \theta^{1/2}]$. Using Lemma~\ref{l.convexity algebra}, it can be seen that the number of extreme points of the convex hull can be characterized in terms of the intersection of $\ell_{a,b}$ (the line passing through $\smash{(-\theta^{1/2}, a\theta)}$ and $\smash{(\theta^{1/2}, b\theta)}$) with $-x^2$, along with an algebraic condition on $a$ and $b$. This turns out to be convenient for the proof. 

The reader can take a look at Figure~\ref{f.three cases of two-point} for a depiction of the three cases below (the panels are in the same order as the cases).

\begin{theorem}\label{t.two point asymptotics}
Let $\cL$ satisfy Assumptions~\ref{as.bg}--\ref{as.tails}. There exist constants $\theta_0$ and $a_0=b_0$ such that the following hold. If (i) $\theta > \theta_0$ and $a\geq b>-1$ or (ii) $\theta>0$ and $a\geq a_0$, $b\geq b_0$, $a\geq b$, then, if $(a-b)^2\leq 8(a+b)$,
\begin{align*}
\MoveEqLeft[4]
\P\Bigl(\cL_1(-\theta^{1/2}) \geq a \theta, \cL_1(\theta^{1/2}) \geq b \theta\Bigr)\\
&= \exp\left(-\frac{\theta^{3/2}}{24}\left[3(a-b)^2 + 24(a+b) + 16\left((1+a)^{3/2} + (1+b)^{3/2}\right) + 32\right] +\mathrm{error}\right).
\end{align*}
while if $(a-b)^2 > 8(a+b)$ and $\ell_{a,b}$ intersects $-x^2$ inside $[-\theta^{1/2}, \theta^{1/2}]$,
\begin{align*}
\P\Bigl(\cL_1(-\theta^{1/2}) \geq a \theta, \cL_1(\theta^{1/2}) \geq b \theta\Bigr)
&= \exp\left(-\frac{4}{3}\theta^{3/2}\left[(1+a)^{3/2} + (1+b)^{3/2}\right] +\mathrm{error}\right),
\end{align*}
and if $(a-b)^2 > 8(a+b)$ but $\ell_{a,b}$ intersects $-x^2$ outside $[-\theta^{1/2}, \theta^{1/2}]$,
\begin{align*}
\P\Bigl(\cL_1(-\theta^{1/2}) \geq a \theta, \cL_1(\theta^{1/2}) \geq b \theta\Bigr)
&= \exp\left(-\frac{4}{3}\theta^{3/2}(1+a)^{3/2} +\mathrm{error}\right),
\end{align*}
The error term may be lower bounded, up to a universal constant factor, by $-((1+a)^{1/2}+(1+b)^{1/2})\theta^{1/2}\log[(1+a)(1+b)\theta]$ in the first and second case, and $-(1+a)^{3/4}\theta^{3/4}$ in the third. 

It may be upper bounded, again up to a universal constant factor, by $((1+a)^{1/2}+(1+b)^{1/2})\theta^{1/2}\log[(1+a)(1+b)\theta]$ for the first case, $((1+a)^{3/4}+(1+b)^{3/4})\theta^{3/4}$ for the second case, and $(1+a)^{3/4}\theta^{3/4}$ in the third.
\end{theorem}

\begin{proof}[Proof of Theorem~\ref{mt.two point tail}]
This follows immediately by combining Theorem~\ref{t.two point asymptotics} and Lemma~\ref{l.convexity algebra}.
\end{proof}

\begin{proof}[Proof of lower bound of Theorem~\ref{t.two point asymptotics}]
We prove the bounds case-by-case, in increasing order of complexity of the argument. Note that this order is not the same as the order in which the estimates are stated in Theorem~\ref{t.two point asymptotics}.

\medskip

\emph{Case 1: $(a-b)^2 > 8(a+b)$ and $\ell_{a,b}$ intersects $-x^2$ inside $[-\theta^{1/2}, \theta^{1/2}]$.} Here the lower bound follows immediately from the FKG inequality (Assumption~\ref{as.corr}(a)) and the one-point lower bound Theorem~\ref{t.upper tail lower bound}. The error term is $-C[(1+a)^{1/2}+(1+b)^{1/2}]\theta^{1/2}\log[(1+a)(1+b)\theta]$.

\medskip

\emph{Case 2: $(a-b)^2>8(a+b)$ and $\ell_{a,b}$ intersects $-x^2$ outside $[-\theta^{1/2}, \theta^{1/2}]$.} Observe from Lemma~\ref{l.convexity algebra} that the tangent line from $(-\theta^{1/2},a\theta)$ to the right lies above $(\theta^{1/2}, b\theta)$. So by the one-point limit shape (Theorem~\ref{t.limit shape}) and monotonicity in conditioning (Lemma~\ref{l.bound point conditioning by tail conditioning}), for large enough constant $M$,
\begin{align*}
\P\left(\cL_1(\theta^{1/2})\geq b\theta \ \Big|\  \cL_1(-\theta^{1/2})\geq a\theta+M\theta^{1/4}\right) \geq \P\left(\cL_1(\theta^{1/2})\geq b\theta \ \Big|\  \cL_1(-\theta^{1/2})= a\theta+M\theta^{1/4}\right)
&\geq \frac{1}{2}.
\end{align*}
So we see that
\begin{align*}
\P\left(\cL_1(-\theta^{1/2})\geq a\theta,\, \cL_1(\theta^{1/2})\geq b\theta\right)
&\geq \P\left(\cL_1(-\theta^{1/2})\geq a\theta+M\theta^{1/4},\, \cL_1(\theta^{1/2})\geq b\theta\right)\\
&\geq \frac{1}{2}\P\left(\cL_1(-\theta^{1/2})\geq a\theta+M\theta^{1/4}\right)\\
&\geq \exp\left(-\frac{4}{3}\theta^{3/2}(1+a)^{3/2} - C(1+a)^{3/4}\theta^{3/4}\right),
\end{align*}
using Theorem~\ref{t.upper tail lower bound} for the last line. This completes the proof of the lower bound in this case.

\bigskip

\emph{Case 3: $(a-b)^2 \leq 8(a+b)$.} This is the main case. The idea is to consider the convex hull of the parabola $-x^2$ and the two points $(-\theta^{1/2}, a\theta)$ and $(\theta^{1/2}, b\theta)$. By assuming $a\geq b > -1$ and $(a-b)^2\leq 8(a+b)$, we have ensured (see Lemma~\ref{l.convexity algebra}) that the convex hull is the parabola outside an interval, and a piecewise linear function inside, with both $(-\theta^{1/2}, a\theta)$ and $(\theta^{1/2}, b\theta)$ as extreme points; see Figure~\ref{f.not too much margin}. In other words, $\ell_{a,b}$ lies above the parabola.

We want to get a lower bound on the two-point probability by considering the event where $\cL_1$ lies above a well-chosen point on each of the tangency lines, and then compute the probability that a Brownian bridge with these two boundary points passes above $(-\theta^{1/2}, a\theta)$ and $(\theta^{1/2}, b\theta)$. Notice that the Brownian bridge will ignore the lower boundary condition, which is allowed when proving lower bounds since, by monotonicity, the lower boundary would only push the bridge up. We will lower bound the probability of $\cL_1$ being above the two chosen boundary points using the FKG inequality. We anticipate that the FKG inequality will be sharp in lower bounding the probability of $\{\cL_1(x_1) > \theta_1, \cL_1(x_2) > \theta_2\}$ for any $x_1 < x_2$ and $\theta_1$, $\theta_2$ such that the line connecting $(x_1,\theta_1)$ and $(x_2,\theta_2)$ is tangent to $-x^2$, so this is the condition that guides the choice of the boundary points.

In fact, we will take one of the boundary points to be $\theta^{1/2}$ and the other to be the location where the left tangent line to $-x^2$ passing through $(\theta^{1/2}, b\theta)$ intersects the left tangent line passing through $(-\theta^{1/2}, a\theta)$. So the Brownian computation consists only of the probability of the Brownian bridge lying above $(-\theta^{1/2}, a\theta)$.
We label the left boundary point as $(-x_0\theta^{1/2}, -y_0\theta)$. It is the unique point of intersection of the lines $y=2(1+\sqrt{1+a})\theta^{1/2}x + a\theta + 2(1+\sqrt{1+a})\theta$ (the left tangency line passing through $(-\theta^{1/2}, a\theta)$) and $y=-2(1-\sqrt{1+b})\theta^{1/2}x + 2(1-\sqrt{1+b})\theta + b\theta$ (the left tangency line passing through $(\theta^{1/2}, b\theta)$). Solving the linear equations and some algebraic manipulations yield that
\begin{align}\label{e.x0 y0 definition}
x_0 := \tfrac{1}{2}(\sqrt{1+a} + \sqrt{1+b}) \quad\text{and}\quad y_0 := (1+\sqrt{1+a})(\sqrt{1+b}-1).
\end{align}
Observe that $-x_0\theta^{1/2} \leq -\theta^{1/2}$: indeed, this is equivalent to $\sqrt{1+a}+\sqrt{1+b} \geq 2$, which in turn is equivalent (since $1+a,1+b\geq 0$) by rearranging and squaring both side to $b \geq (\sqrt{1+a}-1)(\sqrt{1+a}-3)$; this is implied by Lemma~\ref{l.convexity algebra} by the condition $(a-b)^2\leq 8(a+b)$ in force.

Let $\F$ be the $\sigma$-algebra generated by the lower curves on $\R$ and the top curve outside of $[-x_0\theta^{1/2}, \theta^{1/2}]$. Now 
\begin{align}\label{e.main case two-point lower bound}
\PF\Bigl(\cL_1(-\theta^{1/2}) > a\theta, \cL_1(\theta^{1/2}) > b\theta\Bigr) \geq \PF\Bigl(\cL_1(-\theta^{1/2}) > a\theta\Bigr)\cdot\one_{\cL_1(-x_0\theta^{1/2})\geq -y_0\theta, \cL_1(\theta^{1/2})\geq b\theta}.
\end{align}
By the $H_t$-Brownian Gibbs property and monotonicity Lemma~\ref{l.monotonicity}, the conditional probability on the RHS on the specified events is lower bounded by the probability that a rate two Brownian bridge $B$ from $(-x_0\theta^{1/2}, -y_0\theta)$ to $(\theta^{1/2}, b\theta)$ is greater than $a\theta$ at $-\theta^{1/2}$. (Note that we have removed the lower boundary condition.) The latter probability is
\begin{align*}
\P\bigl(B(-\theta^{1/2}) \geq a\theta\bigr).
\end{align*}
Since this is a one-point tail probability of a rate two Brownian bridge, it is equal, up to a polynomial prefactor error, to the analogous one-point density of a rate two Brownian bridge. This in turn can be calculated as
\begin{align}\label{e.two point brownian lower bound}
\exp\left(-\theta^{3/2}\left[\frac{(a+y_0)^2}{2\cdot 2(x_0-1)} + \frac{(b-a)^2}{2\cdot 4} - \frac{(b+y_0)^2}{2\cdot 2(x_0+1)}\right]\right).
\end{align}
Next, by the FKG inequality and Theorem~\ref{t.upper tail lower bound},
\begin{align*}
\MoveEqLeft[6]
\P\left(\cL_1(-x_0\theta^{1/2})\geq -y_0\theta, \cL_1(\theta^{1/2})\geq b\theta\right)\\
&\geq \P\left(\cL_1(-x_0\theta^{1/2})\geq -y_0\theta\right)\cdot\P\left(\cL_1(\theta^{1/2})\geq b\theta\right)\\
&\geq \exp\Bigl(-\frac{4}{3}\theta^{3/2}\left[(1+b)^{3/2} + (x_0^2-y_0)^{3/2}\right]\\
&\quad +O\left((x_0^2-y_0)^{1/2}\theta^{1/2}\log((x_0^2-y_0)\theta) + (1+b)^{1/2}\theta^{1/2}\log((1+b)\theta)\right)\Bigr).
\end{align*}
By using that, for any $c,d\in\R$, $(c+d)^2-4cd = (c-d)^2$, it is an easy calculation that $x_0^2-y_0 = \frac{1}{4}(2+\sqrt{1+a} - \sqrt{1+b})^2$. So we may write
\begin{align}
\MoveEqLeft[6]
\P\left(\cL_1(-x_0\theta^{1/2})\geq -y_0\theta, \cL_1(\theta^{1/2})\geq b\theta\right)\nonumber\\
&\geq \exp\Bigl(-\frac{4}{3}\theta^{3/2}(1+b)^{3/2} - \frac{1}{6}\theta^{3/2} (2+\sqrt{1+a} - \sqrt{1+b})^{3} \label{e.two point fkg lower bound}\\
&\quad +O\left(\bigl((1+a)^{1/2}+(1+b)^{1/2}\bigr)\theta^{1/2}\log((1+a)(1+b)\theta)\right)\Bigr).\nonumber
\end{align}
We will soon substitute \eqref{e.two point brownian lower bound}  into \eqref{e.main case two-point lower bound}, take expectations, and apply \eqref{e.two point fkg lower bound} to obtain a lower bound on $\P(\cL_1(-\theta^{1/2})\geq a\theta, \cL_1(\theta^{1/2})\geq b\theta)$. To obtain our claim, we need to simplify the sum of the expressions in the exponents in \eqref{e.two point brownian lower bound} and \eqref{e.two point fkg lower bound}. We first observe that $(b-a)^2/8$ present in \eqref{e.two point brownian lower bound} is also present in the statement of Theorem~\ref{t.two point asymptotics} we are proving so we may leave it as is.

Next, writing $a=(\sqrt{1+a}+1)(\sqrt{1+a}-1)$ and $b$ similarly, we see from \eqref{e.x0 y0 definition} that $a+y_0 = 2(\sqrt{1+a}+1)(x_0-1)$ and $b+y_0 = 2(\sqrt{1+b}-1)(x_0+1)$. This yields that the sum of the first and last terms in the square brackets in \eqref{e.two point brownian lower bound} can be written as
\begin{align*}
\frac{(a+y_0)^2}{2\cdot 2(x_0-1)} - \frac{(b+y_0)^2}{2\cdot 2(x_0+1)} = (1+\sqrt{1+a})^2(x_0-1) - (\sqrt{1+b}-1)^2(x_0+1).
\end{align*}
Let $z_1 := 1+\sqrt{1+a}$ and $z_2 := \sqrt{1+b}-1$. Then we see from the previous display and the definition \eqref{e.x0 y0 definition} of $x_0$ that
\begin{align*}
\MoveEqLeft[6]
\frac{(a+y_0)^2}{2\cdot 2(x_0-1)} - \frac{(b+y_0)^2}{2\cdot 2(x_0+1)} + \frac{4}{3}(1+b)^{3/2} + \frac{1}{6} (2+\sqrt{1+a} - \sqrt{1+b})^{3}\\
&= z_1^2\left(\tfrac{1}{2}(z_1+z_2)-1\right) -z_2^2\left(\tfrac{1}{2}(z_1+z_2)+1\right) + \tfrac{4}{3}(1+z_2)^3 + \tfrac{1}{6}(z_1-z_2)^3\\
&= -z_1^2 - z_2^2 + \tfrac{1}{2}(z_1+z_2)^2(z_1-z_2) + \tfrac{4}{3}(1+z_2)^3 + \tfrac{1}{6}(z_1-z_2)^3\\%
&=\tfrac{1}{3}\left(z_1^2(2z_1-3) + (z_2+2)^2(2z_2+1)\right)\\
&=\tfrac{1}{3}\left((1+\sqrt{1+a})^2(2\sqrt{1+a}-1) + (1+\sqrt{1+b})^2(2\sqrt{1+b}-1)\right).
\end{align*}
Putting this together with \eqref{e.main case two-point lower bound}, \eqref{e.two point brownian lower bound}, and \eqref{e.two point fkg lower bound} yields that
\begin{align*}
\MoveEqLeft[1]
\P\left(\cL_1(-\theta^{1/2})\geq a\theta, \cL_1(\theta^{1/2})\geq b\theta\right)\\
&\geq \exp\left(-\frac{\theta^{3/2}}{24}\left[3(a-b)^2 + 8\left((1+\sqrt{1+a})^2(2\sqrt{1+a}-1) + (1+\sqrt{1+b})^2(2\sqrt{1+b}-1)\right)\right]\right).
\end{align*}
It is easy to check that this is the same as the expression claimed in Case 2 of Theorem~\ref{t.two point asymptotics}.

\bigskip

\emph{Proof of upper bound of Theorem~\ref{t.two point asymptotics}.}
We again prove the statement case-by-case in order of increasing complexity of the argument. (Cases 1 and 2 are switched compared to the lower bound arguments).

\medskip

\emph{Case 1: $(a-b)^2>8(a+b)$ and $\ell_{a,b}$ intersects $-x^2$ outside $[-\theta^{1/2}, \theta^{1/2}]$.} In this case the bound follows immediately from Theorem~\ref{t.upper tail upper bound} (one-point upper tail) and stationarity, since always $\P(A \cap B)\leq \P(A)$.

\bigskip

In the remaining two cases we will in fact prove that the claimed upper bounds are upper bounds for the probability of the event
$$A(\theta, a, b) = \left\{\cL_1(-\theta^{1/2})\in[a\theta,a\theta+1], \cL_1(\theta^{1/2})\in[b\theta,b\theta+1]\right\};$$
clearly this suffices, by summing over $a$ and $b$ (and applying the probability bound to each summand corresponding to the case in the theorem statement it falls under). Indeed, that the sum will have the same form claimed in the statement follows from the fact that the summands decay exponentially.

\bigskip

\emph{Case 2: $(a-b)^2>8(a+b)$ and $\ell_{a,b}$ intersects $-x^2$ inside $[-\theta^{1/2}, \theta^{1/2}]$.} Let $M_{a,b} =M((1+a)^{1/4}+(1+b)^{1/4})$ and let $\fav$ be given by
$$\fav = \left\{\max_{x\in\{\xtan_\ell, x^{\mrm{in}}_\ell, x^{\mrm{in}}_r, \xtan_r\}}\left(\cL_1(x)+x^2\right)\leq M_{a,b}\theta^{1/4}, \sup_{x\in[\xtan_\ell, \xtan_r]} \left(\cL_2(x)+x^2\right) \leq \tfrac{1}{4} M_{a,b}\theta^{1/4}\right\},$$
where $x^{\mrm{in}}_\ell=(-1+\sqrt{1+a})\theta^{1/2}, x^{\mrm{in}}_r = (1-\sqrt{1+b})\theta^{1/2}$ are the other tangency points (apart from $\xtan_{\ell}$ and $\xtan_r$) corresponding to the tangents to $-x^2$ which pass through $(-\theta^{1/2}, a\theta)$ and $(\theta^{1/2}, b\theta)$, and $C$ is a large constant to be set. It can be checked that the hypotheses in this case imply that $x^{\mrm{in}}_\ell < x^{\mrm{in}}_r$, which also implies that both lie inside $(-\theta^{1/2},\theta^{1/2})$; we label them ``in'' as they are the inner ones, i.e., lie inside $(-\theta^{1/2},\theta^{1/2})$.

  By a union bound, Theorem~\ref{t.two-point limit shape}, and Lemma~\ref{l.lower curve control new} (with $f\equiv -\infty$), we see that, for all large enough $M$,
  \begin{align*}
  \P\left(\fav^c \mid A(\theta,a,b)\right)
  &\leq \P\left(\max_{x\in\{\xtan_\ell, x^{\mrm{in}}_\ell, x^{\mrm{in}}_r, \xtan_r\}}\left(\cL_1(x)+x^2\right)\geq M_{a,b}\theta^{1/4} \midd A(\theta,a,b)\right)\\
  &\qquad + \P\left(\sup_{x\in[\xtan_\ell, \xtan_r]} \left(\cL_2(x)+x^2\right) \geq \tfrac{1}{4} M_{a,b}\theta^{1/4} \midd A(\theta,a,b)\right)\\
  &\leq \frac{1}{8} + \frac{1}{2\cdot \P(\cL_1(\theta^{1/2})\geq b\theta+1 \mid \cL_1(-\theta^{1/2}) = a\theta+1+K)},
  \end{align*}
  as long as $K$ is such that $\tfrac{1}{4} M_{a,b}\theta^{1/4} \geq \log(C( (a+1)\theta+1+K)^{3/4})$. By Theorem~\ref{t.limit shape}, the probability in the denominator is lower bounded by $\frac{4}{5}$ for all large enough $\theta$ or all large enough $a,b$ if $K=2(b+4)\theta$, which clearly satisfies the just mentioned condition on $K$. Thus, for all large enough $M$,
$$\P\left(\fav \mid A(\theta,a,b)\right)\geq \tfrac{1}{4},$$
so that
\begin{align*}
\tfrac{1}{4}\P\left(A(\theta,a,b)\right) \leq \P\left(A(\theta,a,b), \fav\right).
\end{align*}
Let $\F$ be the usual sigma algebra associated with $[\xtan_\ell, x^{\mrm{in}}_\ell]$. Then we see that the previous probability is equal to
\begin{align*}
\E\left[\PF\left(\cL_1(-\theta^{1/2})\in[a\theta,a\theta+1]\right)\one_{\cL_1(\theta^{1/2})\in[b\theta,b\theta+1],\, \fav}\right].
\end{align*}
Let $B$ be a Brownian bridge from $(\xtan_\ell, -(\xtan_\ell)^2+M_{a,b}\theta^{1/4})$ to $(x^{\mrm{in}}_\ell, -(x^{\mrm{in}}_\ell)^2+M_{a,b}\theta^{1/4})$ and $Z_{H_t}$ the partition function associated to the same boundary data and lower curve $-x^2+ \frac{1}{4}M_{a,b}\theta^{1/4}$. By monotonicity Lemma~\ref{l.monotonicity}, the previous display is upper bounded by
$$Z_{H_t}^{-1}\P\left(B(-\theta^{1/2})\geq a\theta\right)\cdot\P\left(\cL_1(\theta^{1/2})\in[b\theta,b\theta+1],\, \fav\right).$$
The first two terms can be bounded, as in the proof of Theorem~\ref{t.upper tail upper bound}, by using normal tail bounds Lemma~\ref{l.normal bounds} and the lower bound on the partition function associated to boundary curve $-x^2$ from Lemma~\ref{l.pos temp Z  lower bound via non-avoid prob}. Doing so yields an overall upper bound on the previous display of
$$\exp\left(-\frac{4}{3}\theta^{3/2}(1+a)^{3/2}+C[(1+a)^{3/4} + (1+a)^{1/2}(1+b)^{1/4}]\theta^{3/4}\right)\cdot\P\left(\cL_1(\theta^{1/2})\in[b\theta,b\theta+1],\, \fav\right),$$ 
where the $(1+b)^{1/4}$ factor is due to the error coming from the lower boundary being as high as $M( (1+a)^{1/4} + (1+b)^{1/4} )\theta^{1/4}$.
A similar argument applied to the second term gives the claim: in slightly more detail, we consider the $\sigma$-algebra $\F$ generated by $\cL_1$ outside $[x^{\mrm{in}}_r, \xtan_r]$ and  $\cL_2,\cL_3, \ldots$ on $\R$. Again $\fav$ is $\F$-measurable, and an argument using the $H_t$-Brownian Gibbs property and monotonicity as above applies to yield that
\begin{align*}
\P\left(\cL_1(\theta^{1/2})\in[b\theta,b\theta+1],\, \fav\right) \leq \exp\left(-\frac{4}{3}\theta^{3/2}(1+b)^{3/2} + C[(1+b)^{3/4} + (1+b)^{1/2}(1+a)^{1/4}]\theta^{3/4}\right).
\end{align*}
Multiplying the two bounds completes the proof.

\bigskip

\emph{Case 3: $(a-b)^2\leq 8(a+b)$.} This is the main case. Recall $M_{a,b}=M((1+a)^{1/4}+(1+b)^{1/4})$ and
\begin{align*}
\fav &= \left\{\max\Bigl\{\cL_1(\xtan_\ell) + (\xtan_\ell)^2, \cL_1(\xtan_r) + (\xtan_r)^2\Bigr\} \leq M_{a,b}\theta^{1/4}\right\}\\
&\qquad\cap \left\{\sup_{x\in[\xtan_\ell, \xtan_r]}\left(\cL_2(x)+x^2\right) \leq \tfrac{1}{4}M_{a,b}\theta^{1/4}\right\}.
\end{align*}
As in the previous case, Theorem~\ref{t.two-point limit shape} (two-point limit shape), Lemma~\ref{l.lower curve control new} (upper tail control of $\cL_2$ conditioned on $\cL_1$), and Lemma~\ref{l.bound point conditioning by tail conditioning} (monotonicity in conditioning variable) imply that there is a constant $M$ large enough such that $\P(\fav \mid A(\theta, a, b))\geq \frac{1}{4}$ for all large enough $\theta$ or large enough $a,b$. So,
\begin{align*}
\tfrac{1}{4}\P\left(A(\theta, a, b)\right)
&\leq \P\left(A(\theta, a, b)\right)\cdot \P\left(\fav \mid A(\theta, a, b)\right)\\
&= \P\left(A(\theta, a, b), \fav\right).
\end{align*}
Let $\F = \Fext(1, [\xtan_\ell, \xtan_r])$ be the $\sigma$-algebra generated by everything outside the top curve on $[\xtan_\ell, \xtan_r]$. Conditioning on $\F$, the probability in the last display is $\PF(A(\theta, a, b))\one_{\fav}$. By the $H_t$-Brownian Gibbs property and monotonicity (Lemma~\ref{l.monotonicity}), this conditional probability, on $\fav$, is upper bounded by
\begin{equation}
\begin{split}\label{e.two point upper bound ratio}
Z_{H_t}^{-1}\E\Bigl[\one_{B(-\theta^{1/2}) \geq a\theta, B(\theta^{1/2}) \geq b\theta} W_{H_t}\Bigr]
&\leq Z_{H_t}^{-1}\P\left(B(-\theta^{1/2}) \geq a\theta, B(\theta^{1/2}) \geq b\theta\right),
\end{split}
\end{equation}
where $B$ is a Brownian bridge from $(\xtan_\ell, -(\xtan_\ell)^2 + M_{a,b}\theta^{1/4})$ to $(\xtan_r, -(\xtan_r)^2 + M_{a,b}\theta^{1/4})$ and $W_{H_t}$ and $Z_{H_t}$ are the Boltzmann factor and partition function associated with the same boundary values and lower boundary curve $-x^2+ \frac{1}{4}M_{a,b}\theta^{1/4}$; note that we have replaced the intervals $[a\theta,a\theta+1]$ and $[b\theta,b\theta+1]$ by $[a\theta,\infty)$ and $[b\theta,\infty)$, which is what allows us to apply monotonicity.

We estimate the numerator of \eqref{e.two point upper bound ratio} first. We write, using Lemma~\ref{l.brownian bridge margin},
\begin{align}
\MoveEqLeft[6]
\frac{1}{2}\P\left(B(-\theta^{1/2}) \geq a\theta,\, B(\theta^{1/2}) \geq b\theta\right)\nonumber\\
&\leq \P\left(B(-\theta^{1/2}) \in [a\theta,a\theta+2M_{a,b}\theta^{1/4}],\, B(\theta^{1/2})\in[b\theta,b\theta+2M_{a,b}\theta^{1/4}]\right) \nonumber\\
&\leq\P\bigl(B(-\theta^{1/2})\geq a\theta\bigr)\cdot\P\left(B(\theta^{1/2}) \geq b\theta \ \Big|\  B(-\theta^{1/2})\in[a\theta,a\theta+2M_{a,b}\theta^{1/4}]\right) \label{e.two point brownian breakup}
\end{align}
Now $B(-\theta^{1/2})$ is a normal random variable with mean $\mu_{\ell,0,r}$ and variance $\sigma_{\ell,0,r}^2$ given by
\begin{align*}
\mu_{\ell,0,r} &= \frac{\xtan_\ell-\theta^{1/2}}{\xtan_r + \xtan_\ell}(-(\xtan_r)^2) + \frac{\xtan_r + \theta^{1/2}}{\xtan_r + \xtan_\ell}(-(\xtan_\ell)^2)+M_{a,b}\theta^{1/4}\\
&= -\left(1+2\sqrt{1+a}+\sqrt{(1+a)(1+b)}\right)\theta +M_{a,b}\theta^{1/4};\\
\sigma_{\ell,0,r}^2 &= 2\cdot\frac{(\xtan_\ell - \theta^{1/2})(\xtan_r + \theta^{1/2})}{\xtan_r + \xtan_\ell}\\
&= \frac{2\sqrt{1+a}(2+\sqrt{1+b})}{2+\sqrt{1+a}+\sqrt{1+b}}\cdot \theta^{1/2};\
\end{align*}
Observe that $\mu_{\ell,0,r}-M_{a,b}\theta^{1/4}\leq -\theta$ while $a>-1$, so $\mu_{\ell,0,r}-M_{a,b}\theta^{1/4}\leq a\theta$. So by standard normal bounds from Lemma~\ref{l.normal bounds}, the first factor in \eqref{e.two point brownian breakup} is upper bounded by
\begin{align}\label{e.a*theta bound}
\exp\left(-\frac{1}{2\sigma_{\ell, 0, r}^2}(a\theta - \mu_{\ell,0, r})^2 + \frac{1}{2\sigma_{\ell,0,r}^2}M_{a,b}^2\theta^{1/2}\right),
\end{align}
the final term included to handle the case that $\mu_{0,\ell,0,r} \geq a\theta$ in which case we cannot apply Lemma~\ref{l.normal bounds}; however then $|a\theta-\mu_{\ell,0,r}|\leq M_{a,b}\smash{\theta^{1/4}}$, and so in that case the above displayed expression is greater than 1, making the bound hold trivially.

The second term of \eqref{e.two point brownian breakup} is upper bounded by
$$\P\left(B(\theta^{1/2}) \geq b\theta \ \Big|\  B(-\theta^{1/2}) = a\theta+2M_{a,b}\theta^{1/4}\right).$$
Now $B(\theta^{1/2})$, when conditioned on $B(-\theta^{1/2})=a\theta+2M_{a,b}\theta^{1/4}$, is distributed as a normal random variable with mean $\mu_{0,0,r}$ and variance $\sigma^2_{0,0,r}$ given by
\begin{align*}
\mu_{0,0,r} &= \frac{2\theta^{1/2}}{\xtan_r + \theta^{1/2}}(-(\xtan_r)^2+M_{a,b}\theta^{1/4}) + \frac{\xtan_r - \theta^{1/2}}{\xtan_r + \theta^{1/2}}(a\theta+2M_{a,b}\theta^{1/4})\\
&\leq \frac{-2(\sqrt{1+b}+1)^2+a\sqrt{1+b}}{2+\sqrt{1+b}}\cdot \theta + 2M_{a,b}\theta^{1/4};\\
\sigma_{0,0,r}^2 &= 2\cdot\frac{2\theta^{1/2}(\xtan_r - \theta^{1/2})}{\xtan_r + \theta^{1/2}}
= \frac{4\sqrt{1+b}}{2+\sqrt{1+b}}\cdot \theta^{1/2}.
\end{align*}
Above the bound on $\mu_{0,0,r}$ is obtained by replacing the first term of $M_{a,b}\theta^{1/4}$ by $2M_{a,b}\theta^{1/4}$. 

Next, we can apply the convexity hypothesis $a\leq (\sqrt{1+b}+1)(\sqrt{1+b}+3)$ to see that $\mu_{0,0,r}-2M_{a,b}\leq b\theta$. So by the standard normal bounds from Lemma~\ref{l.normal bounds}, the second factor of \eqref{e.two point brownian breakup} is at most
\begin{align}\label{e.b*theta bound}
\exp\left(- \frac{1}{2\sigma_{0, 0, r}^2}(b\theta - \mu_{0,0, r})^2 + \frac{1}{2\sigma_{0,0,r}^{2}}4M_{a,b}^2\theta^{1/2}\right),
\end{align}
where the second term is again present to handle the situation where $\mu_{0,0,r}\geq b\theta$, in which case $|\mu_{0,0,r}-b\theta|\leq 2M_{a,b}\theta^{1/4}$.

By Corollary~\ref{c.pos temp parabolic avoidance lower bound}, the denominator $Z_{H_t}$ of \eqref{e.two point upper bound ratio} is lower bounded by
\begin{equation*}
\exp\left(-\frac{1}{12}\left(\xtan_\ell+\xtan_r\right)^3 - 3(\xtan_\ell + \xtan_r)\log(\xtan_\ell + \xtan_r)\right).
\end{equation*}
Combining the previous bound with \eqref{e.a*theta bound} and \eqref{e.b*theta bound}, substituting into \eqref{e.two point upper bound ratio} and \eqref{e.two point brownian breakup}, and simplifying the exponent, we obtain that \eqref{e.two point upper bound ratio} is upper bounded by
\begin{align*}
\MoveEqLeft[14]
\exp\Biggl(-\frac{\theta^{3/2}}{24}\left[3(a-b)^2 + 24(a+b) + 16\left((1+a)^{3/2} + (1+b)^{3/2}\right) + 32\right]\\
&\qquad + CM_{a,b}\theta^{3/4}+ 3(\xtan_\ell+\xtan_r)\log(\xtan_\ell+\xtan_r)+ CM_{a,b}^2\Biggr),
\end{align*}
which completes the proof of Theorem~\ref{t.two point asymptotics}.
\end{proof}

\begin{remark}[Two-point density bounds]\label{r.density bounds}
A similar argument as above can also produce two-point density bounds. Indeed, by Theorem~\ref{mt.two-point limit shape} it follows that, conditional on $\{\cL_1(-\theta^{1/2}) = a\theta, \cL_1(\theta^{1/2}) = b\theta\}$ and with probability at least $\frac{1}{2}$, $\cL_1(\xtan_{\ell})\in [-(\xtan_{\ell})^2-M, -(\xtan_{\ell})^2+M]$ and the analogue for $\cL_1(\xtan_{r})$. Given this the two-point density bound is a Brownian calculation as in the proof of Case 3 (i.e., $(a-b)^2\leq 8(a+b)$) of Theorem~\ref{t.two point asymptotics}. In particular, one obtains the first asymptotic in Theorem~\ref{t.two point asymptotics} as long as $\ell_{a,b}$ does not intersect the parabola $-x^2$ in $[-\theta^{1/2}, \theta^{1/2}]$, and the second asymptotic if it does.
\end{remark}

We can now prove Theorem~\ref{mt.fkg sharpness} on the characterization of the sharpness of the FKG inequality.

\begin{proof}[Proof of Theorem~\ref{mt.fkg sharpness}]
The case when $(a-b)^2 > 8(a+b)$ and $\ell_{a,b}$ intersects $-x^2$ inside $[-\theta^{1/2}, \theta^{1/2}]$ is asserted in Theorem~\ref{t.two point asymptotics}, so we may assume $(a-b)^2 = 8(a+b)$ and the same condition on $\ell_{a,b}$.

We remind the reader from the proof of Lemma~\ref{l.convexity algebra} that $\ell_{a,b}$ is given by
$$\ell_{a,b}(x) = \frac{(b-a)}{2}\cdot \theta^{1/2} (x+\theta^{1/2}) + a\theta,$$
and that the discriminant of the quadratic obtained by equating the right-hand side with $-x^2$ is equal to zero exactly when $(a-b)^2=8(a+b)$. This discriminant being zero is the condition for $\ell_{a,b}$ being a tangent to $-x^2$ somewhere on $\R$.

Now, for the tangency point to be inside $[-\theta^{1/2}, \theta^{1/2}]$, the slope $(b-a)\theta^{1/2}/2$ of $L$ must lie inside the range of slopes of tangency lines of points inside $[-\theta^{1/2}, \theta^{1/2}]$, i.e., inside $[-2\theta^{1/2}, 2\theta^{1/2}]$. Thus, we get the condition
$$|a-b| \leq 4.$$
Let us set $a-b = 4z$ for $z \in[0,1]$, as $a\geq b$. Then from $(a-b)^2=8(a+b)$, we obtain
$$a = z^2 +2z \quad\text{and}\quad b = z^2 - 2z.$$
Looking at the upper bound from Theorem~\ref{t.two point asymptotics} and using this parametrization of $a$ and $b$, we see that 
\begin{align*}
3(a-b)^2 + 24(a+b) + 32 = 48z^2 + 48z^2 + 32 = 96z^2+32,
\end{align*}
while
\begin{align*}
16\left[(1+a)^{3/2} + (1+b)^{3/2}\right] = 16[(1-z)^3 + (1+z)^3] = 16(6z^2+2) = 96z^2 + 32.
\end{align*}
Thus we see from Theorem~\ref{t.two point asymptotics} that
\begin{align*}
\P\left(\cL_1(-\theta^{1/2}) > a\theta, \cL_1(\theta^{1/2}) > b\theta\right) = \exp\left(-\frac{32}{24}\theta^{3/2}\left[(1+a)^{3/2} + (1+b)^{3/2}\right](1+o(1))\right),
\end{align*}
which simplifies to the claim.
\end{proof}

\begin{remark}[Procedure for $k$-point asymptotics]\label{r.k-point}
At this point the procedure for obtaining $k$-point asymptotics is clear: one finds pinning points (the number depends on the heights and locations of the individual point values, through their effect on the convex hull of the point values and the parabola, as in the three cases of Theorem~\ref{mt.two point tail}) and uses Brownian bridge resamplings between pinning points to obtain the limit shape. 

To obtain asymptotics, one essentially just has to calculate the Brownian bridge probabilities of achieving the desired heights when the pinning points are fixed, while taking into account the essentially parabolic lower boundary condition; for the lower bound, one handles the boundary condition by making use of the FKG inequality carefully (i.e., such that it satisfies the tangency conditions and so will be sharp) but otherwise ignoring it, while for the upper bound one makes use of Proposition~\ref{p.parabola avoidance probability} or Corollary~\ref{c.pos temp parabolic avoidance lower bound}.

Implementing this for higher values of $k$ in general is quite tedious, as the resulting expressions get more complicated and the number of cases grows quite quickly; already from $k=1$ to 2 it has increased from $1$ case to $3$. However, in practice one may be interested in only particular cases, in which case the procedure can be implemented for only those ones.
\end{remark}

\section{Extremal stationary ensembles}\label{s.extremal}

In this section we prove part of Theorem~\ref{t.assumptions hold} for extremal stationary ensembles. More precisely, we show that any line ensemble satisfying Assumptions~\ref{as.bg}--\ref{as.mono in cond} also satisfies the upper bound of Assumption~\ref{as.tails} with $\beta=1$. The validity of Assumptions~\ref{as.corr} and \ref{as.mono in cond} for extremal stationary line ensembles is proved in Appendix~\ref{app.monotonicity proofs}, and the lower bound of Assumption \ref{as.tails} with $\alpha = \frac{3}{2}$ follows from Theorem~\ref{t.upper tail lower bound}.

Throughout this section, $\cL$ satisfying Assumptions~\ref{as.bg} will mean with $t=\infty$, i.e., the zero temperature case. We also emphasize that in this section we work with the BK inequality (Assumption~\ref{as.corr}(b)) and not its weaker version Assumption~\ref{as.weak bk}(b\ensuremath{'}); see the beginning of Section~\ref{s.ingredients of extremal proof} for a brief discussion.

\begin{theorem}\label{t.extremal initial tail}
Suppose $\cL$ satisfies Assumptions~\ref{as.bg}--\ref{as.mono in cond}. Then there exist $\theta_0$ and $c>0$ such that, for $\theta>\theta_0$,
$$\P\left(\cL_1(0) > \theta\right) \leq \exp(-c\theta).$$
\end{theorem}

While we have stated the result for the zero temperature case,
 it is not hard to check that the argument also applies to the positive temperature case, if one changes applications of the Brownian Gibbs property to the $H_t$ analogue (again with Assumption~\ref{as.corr}(b) and not Assumption~\ref{as.weak bk}(b\ensuremath{'})). However, in the $t<\infty$ case, the argument may not yield uniform bounds for $t>t_0$. See also Remark~\ref{r.why no uniformity in extremal argument}.

Before turning to the arguments for Theorem~\ref{t.extremal initial tail}, we clarify a point of confusion concerning the meaning of extremal stationary ensembles that has arisen between two previous works which discuss them, namely \cite{corwin2014brownian} and \cite{corwin2014ergodicity}.

\subsection{Different notions of extremality}\label{s.notions of extremality}
The first work, \cite{corwin2014brownian}, introduced extremal stationary ensembles in the context of Conjecture~\ref{conj.extremality} that all such ensembles are the parabolic Airy line ensemble up to a deterministic vertical shift. The second, \cite{corwin2014ergodicity}, proved that the parabolic Airy line ensemble is ergodic, and stated that this implies that it is the only candidate for the conjecture in \cite{corwin2014brownian}. However, there is a discrepancy in the meaning of an extremal stationary ensemble in the two works.

The discrepancy lies in the choice of convex set for which the extreme points are considered. Let $\msf{Gibbs}$ be the collection of measures of line ensembles which possess the Brownian Gibbs property, and $\msf{Stat}$ that of line ensembles which are stationary under horizontal shifts after an addition of the parabola $x\mapsto x^2$; by standard theory, both are convex sets. For a convex set $A$, let $\mrm{Ex}(A)$ be the set of its extreme points. 

In \cite{corwin2014brownian}, the conjecture characterizes the elements of $\mrm{Ex}(\msf{Gibbs})\cap\msf{Stat}$ as being the parabolic Airy line ensemble up to a constant vertical shift, while \cite{corwin2014ergodicity} says the same about the elements of $\mrm{Ex}(\msf{Gibbs}\cap\msf{Stat})$.
Unfortunately, it is not obvious that these two sets coincide. However, it is easy to see that $\mrm{Ex}(\msf{Gibbs})\cap\msf{Stat} \subseteq \mrm{Ex}(\msf{Gibbs}\cap \msf{Stat})$ as a general property of extreme points of convex sets (and indeed, \cite{corwin2014ergodicity} uses the similar fact that $\mrm{Ex}(\msf{Stat})\cap\msf{Gibbs} \subseteq \mrm{Ex}(\msf{Gibbs}\cap\msf{Stat})$). 

While it is not clear that the two sets are the same, we can say that the parabolic Airy line ensemble with any given deterministic vertical shift also belongs to $\mrm{Ex}(\msf{Gibbs})\cap\msf{Stat}$, just as \cite{corwin2014ergodicity} establishes their membership in $\mrm{Ex}(\msf{Gibbs}\cap\msf{Stat})$. This follows from results on the triviality of the tail $\sigma$-algebra of determinantal point processes \cite{osada2018discrete,lyons2018note,bufetov2016kernels}\footnote{Technically, this is a statement about the triviality of the tail $\sigma$-algebra generated by finite collections of values of the line ensemble; further, the processes considered in \cite{osada2018discrete,lyons2018note,bufetov2016kernels} take values in the space of probability measures as they are regarded as point processes. To get the statement made here about the tail $\sigma$-algebra of the line ensemble as defined on the space of infinite collections of continuous functions, one needs to change the space and approximate the process by the finite collections of values. We do not do the former here as this is a technical and not significant point for our purposes, while the latter follows from the almost sure continuity of the ensemble.} and the fact that the finite dimensional distributions of the parabolic Airy line ensemble are determinantal (with kernel the extended Airy kernel), combined with the abstract characterization of extremal measures as those which have trivial tail $\sigma$-algebras \cite[Theorem~2.1]{preston2006random}. In particular, the parabolic Airy line ensemble is extremal stationary, and thus our arguments apply to it. 

Having noted this discrepancy between \cite{corwin2014brownian} and \cite{corwin2014ergodicity}, in the rest of this section, we will work with elements of $\mrm{Ex}(\msf{Gibbs})\cap\msf{Stat}$; in particular, we will need the triviality of the $\sigma$-algebra containing boundary data as the domain goes to infinity, an equivalent condition to lying in $\mrm{Ex}(\msf{Gibbs})$ as we just noted.

\subsection{The ingredients for the proof of Theorem~\ref{t.extremal initial tail}}\label{s.ingredients of extremal proof}
To prove Theorem~\ref{t.extremal initial tail} we have several ingredients which are qualitative counterparts to similar steps in the proof of Theorem~\ref{t.upper tail upper bound}, as we explained earlier in Section~\ref{s.intro.proof ideas}. At a basic level, we need to be able to control the second curve to be below some decaying deterministic function, and have control over the top curve at two boundary points at some location.

We need to control the second curve on the event that $\cL_1(0)$ is large; unlike in previous sections where we worked with Assumption~\ref{as.weak bk}(b\ensuremath{'}), here we will work with Assumption~\ref{as.corr}(b). This is because while Assumption~\ref{as.weak bk}(b\ensuremath{'}) says that the second curve lies $(\log M)^C$ above the parabola on $[-M,M]$ conditional on $\cL_1$, in the setting of extremal ensembles we will have no a priori control on how large we will need to make the interval $[-M,M]$ for the overall argument. Instead, we will control the conditioned second curve on all of $\R$ by obtaining similar control on the top curve under no conditioning and using Assumption~\ref{as.corr}(b). The following proposition obtains this control, showing that the top curve lies below a  linearly decaying function of any constant slope up to a random vertical shift.

\begin{proposition}\label{p.top curve below tent}
Suppose $\cL$ satisfies Assumptions~\ref{as.bg}--\ref{as.mono in cond}. Then $\sup_{x\in\R} (\cL_1(x) + K|x|)$ is almost surely finite for any $K>0$.
\end{proposition}

The basic idea of the proof is to find a sequence of deterministic points where $\cL_1(x) > -2K|x|$ only finitely often almost surely. But if there exist infinitely many (random) points where $\cL_1(x) > -|K|x$, it would be unlikely that a Brownian bridge between those points goes below $-2K|x|$ at some location, especially since the lower boundary condition pushes the bridge upwards. This causes a contradiction since the random points must contain points from the deterministic sequence between them. 

\begin{proof}[Proof of Proposition~\ref{p.top curve below tent}]
Let $K$ be given.
We start by observing that $\P(\cL_1(x) > -2K|x|) = \P(\cL_1(0) > x^2-2K|x|)$, so, since $\cL_1(0)$ is almost surely finite and $x^2-2K|x|\to\infty$ as $|x|\to\infty$, it follows that there exists a deterministic sequence of ordered points $\{x_n\}_{n\in\Z}$ with $x_n\to \infty, x_{-n}\to -\infty$ as $n\to\infty$ such that
\begin{equation}\label{e.x_n summability}
\sum_{n\in\Z}\P\left(\cL_1(x_n) > -2K|x_n|\right) < \infty.
\end{equation}
It will be convenient to assume (without loss of generality) that $|x_n-x_{n-1}| \geq 2$, and we do so.

Consider the collection of intervals $\mc I = \{[x_{n-1}, x_n]$: $n\in\Z$\}. Note that $\cup_{I\in\mc I} I = \R$. For any interval $I \in \mc I$, let $S_I = \sup_{x\in I} \cL_1(x) + K|x|$. Suppose that, for all $I \in \mc I$ with $\sup I \leq -1$ or $\inf I\geq 1$,
\begin{equation}\label{e.condition for pairs of intervals}
\sum_{\substack{J\in\mc I: \\0\to I\to J}} \P\left(S_I \geq 0, S_J \geq 0\right) < \infty;
\end{equation}
here the second line below the sum is to indicate that the sum is over all $J$ distinct from $I$ which are further from $0$ than $I$ is, i.e., if $I$ is to the right of zero then $J$ is to the right of $I$, and analogously for the left.

It is easy to see that \eqref{e.condition for pairs of intervals} implies that $\sup_{x\in\R} (\cL_1(x) + K|x|) < \infty$ almost surely. Indeed, given \eqref{e.condition for pairs of intervals} for all $I$ with $\inf I \geq 1$, it follows from the Borel-Cantelli lemma that almost surely, for each such $I$, either (i) $S_I < 0$ or (ii) $S_I \geq 0$ but there are only finitely many $J$ with $J$ to the right of $I$ such that $S_J \geq 0$. Taking an intersection over these events yields a probability 1 event on which either $S_I < 0$ for all $I\in \mc I$ with $\inf I \geq 1$, or there exists some $I$ with $S_I \geq 0$ and only finitely many $J$ to the right of zero such that $S_J \geq 0$. In both cases, there is a random compact set $\mc C\subseteq [0,\infty)$ such that $\cL_1(x) < -K|x|$ for all $x\in[0,\infty)\setminus \mc C$, and, by continuity, $\cL_1(x) + K|x|$ is bounded inside $\mc C$ almost surely. So we obtain that $\sup_{x\geq 0} (\cL_1(x) + K|x|) < \infty$ almost surely. Repeating the argument using \eqref{e.condition for pairs of intervals} for all $I\in\mc I$ such that $\sup I \leq -1$ gives that $\sup_{x\leq 0} (\cL_1(x) + K|x|) < \infty$ almost surely.

So it remains to prove \eqref{e.condition for pairs of intervals}. On the event that $S_I \geq 0$, let $\tau^\ell_I = \inf\{x\in I: \cL_1(x) \geq -K|x|\}$ and $\tau^r_I = \sup\{x\in I: \cL_1(x) \geq -K|x|\}$. If the corresponding sets are empty, i.e., when $S_I \leq 0$, let $\tau_I^\ell = \sup I + 1$ and $\tau_I^r = \inf I -1$ (this definition is purely to make some future statements cleaner). We will use the analogous definitions for $\tau_J^\ell$ and $\tau_J^r$ as well. 

We will now argue \eqref{e.condition for pairs of intervals} under the condition that $\inf I \geq 1$; the argument for the case that $\sup I \leq -1$ is analogous.
We assume without loss of generality that $\sup I + 2 \leq \inf J$ as we are ignoring at most one term in \eqref{e.condition for pairs of intervals} (recall $|x_n-x_{n-1}| \geq 2$). This is so that $\tau^\ell_I < \tau^r_J$.

We first observe that $[\tau_I^\ell, \tau_J^r]$ is a stopping domain (recall Definition~\ref{d.strong bg}). Let $\F_{\tau, I, J} = \Fext(1,[\tau^\ell_I, \tau^r_J])$ be the $\sigma$-algebra generated by the second and all lower curves on $\R$, and the top curve on $[\tau_I^\ell, \tau_J^r]$. Let $n$ be such that $x_n = \inf J$, so that $x_n\in[\tau^\ell_I, \tau^r_J]$ on the event that $S_I\geq 0, S_J\geq 0$. Consider the event 
$$A_{I,J}(\cL_1) = \Bigl\{\cL_1(x_n) \geq -2K|x_n|\Bigr\};$$
notice that the coefficient of $|x_n|$ is $-2K$ and not $-K$. Then we see that, since $\{S_I\geq 0, S_J\geq 0\} = \{\tau^\ell_I < \sup I +1, \tau^r_J > -\inf J-1\}$,
\begin{align*}
\P\bigl(S_I \geq 0, S_J \geq 0, A_{I,J}(\cL_1)\bigr)
&= \P\left(\tau^\ell_I < \sup I +1, \tau^r_J > -\inf J-1, A_{I,J}(\cL_1)\right)\\
&= \E\left[\P_{\F_{\tau,I,J}}\left(A_{I,J}(\cL_1)\right)\one_{\tau^\ell_I < \sup I +1, \tau^r_J > -\inf J-1}\right].
\end{align*}
Now, $A_{I,J}(\cL_1)$ is an increasing event. Further, on the event $\{\tau^\ell_I < \sup I +1, \tau^r_J > -\inf J-1\}$, $\cL_1$ is distributed as a Brownian bridge from $(\tau^\ell_I, \cL_1(\tau^\ell_I))$ to $(\tau^r_J, \cL_1(\tau^r_J))$, conditioned on avoiding the second curve, and $\cL_1(x) \geq -K|x|$ for $x\in\{\tau^\ell_I, \tau^r_J\}$. Let $B$ be a rate two Brownian bridge between the points $(\tau^\ell_I, -K|\tau^\ell_I|)$ and $(\tau^r_I, -K|\tau^r_I|)$, but with no lower boundary conditioning. So, by monotonicity (Lemma~\ref{l.monotonicity}),
\begin{align*}
\E\left[\P_{\F_{\tau,I,J}}\left(A_{I,J}(\cL_1)\right)\one_{\tau^\ell_I < \sup I +1, \tau^r_J > -\inf J-1}\right] \geq \E\left[\P_{\F_{\tau,I,J}}\left(A_{I,J}(B)\right)\one_{\tau^\ell_I < \sup I +1, \tau^r_J > -\inf J-1}\right].
\end{align*}
Notice that $\E[B(x_n)] = -K|x_n|$. Thus, by symmetry of the normal distribution,
$$\P_{\F_{\tau,I,J}}\bigl(A_{I,J}(B)\bigr) = \P_{\F_{\tau,I,J}}\Bigl(B(x_n) \geq -2K|x_n|\Bigr) \geq \frac{1}{2}.$$
This yields, again since $\{S_I \geq 0, S_J \geq 0\} = \{\tau^\ell_I < \sup I +1, \tau^r_J > -\inf J-1\}$,
$$\P\bigl(S_I \geq 0, S_J \geq 0\bigr) \leq 2\cdot\P\bigl(S_I \geq 0, S_J \geq 0, A_{I,J}(\cL_1)\bigr).$$
Now, clearly $\P(S_I \geq 0, S_J \geq 0, A_{I,J}(\cL_1)) \leq \P(A_{I,J}(\cL_1)) = \P(\cL_1(x_n) \geq -2K|x_n|)$. Note that, since $x_n=\inf J$, the $n$ corresponding to distinct $J$ are distinct. So, by our choice of $\{x_n\}_{n\in\Z}$ such that \eqref{e.x_n summability} holds, we obtain~\eqref{e.condition for pairs of intervals}.
\end{proof}

Recall from Section~\ref{s.intro.proof ideas.extremal}, and as in the proof of Theorem~\ref{t.upper tail upper bound}, that to prove an upper bound on the upper tail of $\cL_1(0)$, we first need to have control over the lower curve and know that the top curve is not too high at the boundary of an interval containing zero; then we will be able to resample on that interval and use these pieces of information to get a tail bound on the value at zero. In Theorem~\ref{t.upper tail upper bound} these two pieces of boundary information were provided by Theorem~\ref{t.limit shape}; here it is the next proposition.

\begin{proposition}\label{p.nearby pinning for extremal}
Suppose $\cL$ satisfies Assumptions~\ref{as.bg}--\ref{as.mono in cond}. Then there exist $R>0$ and $\theta_0$ such that, for $\theta>\theta_0$,
$$\P\Bigl(\cL_1(\pm \theta) \leq \tfrac{1}{2}\theta,\  \sup_{x\in \R} (\cL_2(x) + 2|x|) \leq R  \ \Big | \   \cL_1(0) = \theta\Bigr) \geq \frac{1}{2}.$$

\end{proposition}

\begin{remark}\label{r.why no uniformity in extremal argument}
The arguments of Propositions~\ref{p.top curve below tent} and \ref{p.nearby pinning for extremal} both apply essentially unchanged to the positive temperature analogue of extremal ensembles as well, if one has Assumption~\ref{as.corr}(b). However, the tail of the almost surely finite constant $K$ in Proposition~\ref{p.top curve below tent} may not be tight as $t$ varies, i.e. over $t\geq t_0$ for some $t_0>0$, and this would result in the constants $R$ and $\theta_0$ in Proposition~\ref{p.nearby pinning for extremal} depending on $t$ as well. So the overall argument we are giving for zero temperature extremal ensembles would not yield Theorems~\ref{mt.one point density asymptotics}--\ref{mt.two point tail} in the positive temperature case \emph{with the uniformity in $t$} mentioned in Remark~\ref{r.uniformity}---though they will yield those results if the constants are allowed to depend on $t$.
\end{remark}

As we outlined in Section~\ref{s.intro.proof ideas.extremal}, to prove Proposition~\ref{p.nearby pinning for extremal} we will use stationarity, parabolic curvature, and one-point tightness to first find a faraway point $x_\theta$ so that $\cL_1(x_\theta) \leq -|x_\theta|$, i.e., below a line of slope $-1$. 
Next, we know from Proposition~\ref{p.top curve below tent} that there exists a random constant $\tilde R$ such that $\cL_2$ a.s.\ lies below $-2|x| + \tilde R$ on \emph{all} of $\R$ (as usual, using the BK inequality from Assumption~\ref{as.corr}(b) to say that $\cL_2$ conditioned on any event of $\cL_1$ is dominated by an unconditioned copy of $\cL_1$). 

Using the Gibbs property and monotonicity, we can dominate $\cL_1$ on $[0,x_\theta]$ (and similarly for $[-x_\theta,0]$) with a Brownian bridge $B$ conditioned to stay above $-2|x|$, and so we have to understand the tail of $B(\theta)$. The next lemma says that the probability of this conditioning event is uniformly positive, and we prove it in Appendix~\ref{app.brownian estimates}. The setup has been affinely transformed so that the lower line has slope zero.

\begin{lemma}\label{l.bridge above line}
For $K, r>0$, let $B^K$ be a rate two Brownian bridge between $(0,0)$ and $(r, Kr)$. Let $\eta>0$. Then there exists $C, c>0$ (independent of $K$ and $r$) such that, for $K\geq \frac{1}{2}\max(1,\eta^{-1})$,
$$\P\left(\inf_{x\in[0,r]} B^K(x) < -\eta K \right) \leq C\exp\left(-cK^2\right).$$
\end{lemma}

The form of the assumed lower bound on $K$ is for technical convenience and as it suffices for the applications; but one could also assume $K\geq \delta\max(1, \eta^{-1})$ for some $\delta>0$.

With this overview, we turn to the details of the proof of Proposition~\ref{p.nearby pinning for extremal}. It may help to recall the depiction of the setting from Figure~\ref{f.stay below line}.

\begin{proof}[Proof of Proposition~\ref{p.nearby pinning for extremal}]
Let $\delta>0$ be a constant to be fixed later. We first observe that, by Assumption~\ref{as.corr}(b) (BK inequality) and Lemma~\ref{l.bound point conditioning by tail conditioning} (monotonicity in conditioning), for any $R>0$,
\begin{align*}
\P\left(\sup_{x\in\R} (\cL_2(x) + 2|x|) > R  \ \Big|\  \cL_1(0) = \theta\right) &\leq \P\left(\sup_{x\in\R} (\cL_2(x) + 2|x|) > R  \ \Big|\  \cL_1(0) \geq \theta\right)\\
&\leq \P\left(\sup_{x\in \R} (\cL_1(x) + 2|x|) > R\right).
\end{align*}
By Proposition~\ref{p.top curve below tent} with $K=2$, this probability is less than $\delta$ for all sufficiently large $R$. We fix such an $R$ for the remainder of the argument.

We now show that, for $x\in\{-\theta, \theta\}$,
$$\P\Bigl(\cL_1(x) > \tfrac{1}{2}\theta,\, \sup_{y\in \R} (\cL_2(y) + 2|y|) \leq R  \ \Big|\  \cL_1(0) = \theta\Bigr) \leq 2\delta.$$
This will clearly suffice as we may set $\delta$ a sufficiently small constant later. We focus on the $x=\theta$ case as the other is analogous.

We start by picking $x_\theta > \theta$ such that 
\begin{equation}\label{e.x_theta condition}
\P\left(\cL_1(x_\theta) > -x_\theta + \theta\right) \leq \delta \cdot\P(\cL_1(0) \geq \theta);
\end{equation}
this is possible as the left-hand side is $\P(\cL_1(0) > x_\theta^2-x_\theta + \theta)$ by stationarity and $x^2-x \to \infty$ as $x\to\infty$.

Next, by Lemma~\ref{l.bound point conditioning by tail conditioning} (monotonicity in conditioning),
$$\P\Bigl(\cL_1(x_\theta) > -x_\theta + \theta \ \Big|\  \cL_1(0) = \theta\Bigr) \leq \P\Bigl(\cL_1(x_\theta) > -x_\theta + \theta \ \Big|\  \cL_1(0) \geq \theta\Bigr) \leq \delta.$$
Using this yields that
\begin{align*}
\MoveEqLeft[6]
\P\Bigl(\cL_1(\theta) > \tfrac{1}{2}\theta,\  \sup_{x\in \R} (\cL_2(x) + 2|x|) \leq R \ \Big | \  \cL_1(0) = \theta\Bigr)\\
&\leq \P\Bigl(\cL_1(\theta) > \tfrac{1}{2}\theta,\  \cL_1(x_\theta) \leq -x_\theta+\theta,\  \sup_{x\in \R} (\cL_2(x) + 2|x|) \leq R  \ \Big |\  \cL_1(0) = \theta\Bigr) + \delta.
\end{align*}
Let $\F=\Fext(1,[0,x_\theta])$ be the $\sigma$-algebra generated by the top curve outside $[0,x_\theta]$ and all the lower curves on $\R$. The first term on the RHS in the previous display is equal to
$$\E\left[\PF\left(\cL_1(\theta) > \tfrac{1}{2}\theta \ \Big|\  \cL_1(0) = \theta\right)\one_{\cL_1(x_\theta) \leq -x_\theta + \theta, \sup_{x\in \R} (\cL_2(x) + 2|x|) \leq R}\right].$$
Under $\F$, $\cL_1$ is distributed as a Brownian bridge from $(0,\theta)$ to $(x_\theta, -\cL_1(x_\theta))$ conditioned to stay above $\cL_2$. Let $B$ be a Brownian bridge from $(0,\theta)$ to $(x_\theta, -x_\theta + \theta)$ without the lower boundary conditioning. Then, by monotonicity Lemma~\ref{l.monotonicity}, the previous display is upper bounded by
$$\P\Bigl(B(\theta) > \tfrac{1}{2}\theta \ \Big|\  B(x) > -2|x|+R \ \forall x\in[0,x_\theta]\Bigr).$$
We want to uniformly lower bound the probability of the conditioning event using Lemma~\ref{l.bridge above line}. For this we note that $B$ is a Brownian bridge whose endpoints lie on a line of slope $-1$, while the conditioning is to stay above a line of slope $-2$; so the difference of slopes is 1. Further, by setting $\theta_0$ large enough depending on $R$, the minimum distance between these lines is at least $\frac{1}{2}\theta$. Thus by Lemma~\ref{l.bridge above line} (with $K=1$ and $\eta=1/2$) the probability of the conditioning event is uniformly positive in $\theta$, so the previous display is upper bounded by
$$C\cdot \P\left(B(\theta) > \tfrac{1}{2}\theta\right)$$
for some $C<\infty$ independent of $\theta$. Now $B(\theta)$ is a normal random variable with mean $\theta - \theta = 0$ and variance $\smash{\sigma^2=2\cdot\frac{\theta\times (x_\theta-\theta)}{x_\theta} \leq 2\theta}$ (recall $B$ is of rate two). Thus using Gaussian tail bounds (Lemma~\ref{l.normal bounds}) the previous display is upper bounded by
$$C\cdot\exp\left(-\frac{\frac{1}{4}\theta^2}{2\sigma^2}\right) \leq C\cdot \exp\left(-\tfrac{1}{16}\theta\right).$$
We set $\theta_0$ such that this is less than $\delta$ for all $\theta>\theta_0$.
This completes the proof after setting $\delta<1/10$.
\end{proof}

With these statements in hand, we can now prove Theorem~\ref{t.extremal initial tail}.

\begin{proof}[Proof of Theorem~\ref{t.extremal initial tail}]
As in the proof of the optimal one-point upper bound Theorem~\ref{t.upper tail upper bound}, we will in fact first bound $\P(\cL_1(0)\in[\theta-1,\theta])$.

Let $R$ be as in Proposition~\ref{p.nearby pinning for extremal}, and let $A_{R,\theta}$ be defined by
$$A_{R,\theta} = \left\{\cL_1(\pm\theta)\leq \tfrac{1}{2}\theta, \sup_{x\in\R} \cL_2(x) + |x| \leq R\right\}.$$
 Note that the event in the statement of Proposition~\ref{p.nearby pinning for extremal} is the same with $2|x|$ in place of $|x|$, so the above event contains the event in Proposition~\ref{p.nearby pinning for extremal}. So, by Proposition~\ref{p.nearby pinning for extremal}, Lemma~\ref{l.bound point conditioning by tail conditioning} (monotonicity in conditioning), and since $A_{R,\theta}$ is a decreasing event,
 $$\P(A_{R,\theta} \mid \cL_1(0) \in [\theta-1,\theta]) \geq \P\left(A_{R,\theta}\mid\cL_1(0)= \theta\right)\geq \frac{1}{2},$$
 for $\theta>\theta_0$ with $\theta_0$ as given in Proposition~\ref{p.nearby pinning for extremal}. We raise the value of $\theta_0$ if needed so that $\theta_0> 4R$, which we will need shortly. We observe that
\begin{align*}
\tfrac{1}{2}\cdot\P\left(\cL_1(0) \in [\theta-1,\theta]\right) \leq \P\left(\cL_1(0) \in [\theta-1,\theta], A_{R,\theta}\right).
\end{align*}
Let $\F$ be the $\sigma$-algebra generated by the top curve outside $[-\theta,\theta]$ and all the lower curves on $\R$. Then the previous display is equal to
$$\E\left[\PF\left(\cL_1(0) \in [\theta-1,\theta]\right)\one_{A_{R,\theta}}\right] \leq \E\left[\PF\left(\cL_1(0) \geq \theta-1\right)\one_{A_{R,\theta}}\right];$$
we performed this bound so as to make use of the increasing nature of the event $\{\cL_1(0)>\theta\}$ on the right-hand side.

By the Brownian Gibbs property, conditionally on $\F$, $\cL_1$ is a Brownian bridge from $(-\theta,\cL_1(-\theta))$ to $(\theta,\cL_1(\theta))$ conditioned on staying above $\cL_2$. On $A_{R,\theta}$, by monotonicity Lemma~\ref{l.monotonicity}, this Brownian bridge is stochastically dominated by the Brownian bridge $B$ from $(-\theta, \frac{1}{2}\theta)$ to $(\theta,\frac{1}{2}\theta)$ conditioned to stay above $-|x| +R$. So, since $\{\cL_1(0)>\theta-1\}$ is an increasing event, the right-hand side of the previous display is upper bounded by
$$\frac{\P(B(0) \geq \theta-1)}{\P\left(B(x) > -|x|+R \ \forall x\in[-\theta,\theta]\right)}.$$
Now, $B$ has maximum deviation less than $\theta^{1/2}$ with positive probability independent of $\theta$. Recall that the endpoints of $B$ are at height $\frac{1}{2}\theta$. So, since $\theta>\theta_0>4R$ implies that $\frac{1}{2}\theta - \theta^{1/2} \geq R$, the denominator of the previous display is lower bounded by a constant. Now $B(0)$ is a normal random variable with mean $\theta/2$ and variance $2\times\frac{\theta\times\theta}{2\theta}= \theta$, so, by standard normal tail bounds from Lemma~\ref{l.normal bounds}, the numerator is upper bounded by
$$\exp\left(-\frac{(\theta-\frac{1}{2}\theta)^2}{2\theta}\right) = \exp\left(-\frac{1}{8}\theta\right).$$
Thus we have shown that $\P(\cL_1(0)\in[s-1,s])\leq\exp(-cs)$ for some $c>0$ and all $s>\theta_0$. Summing this $s=\theta+1$ to $\infty$ gives that $\P(\cL_1(0)>\theta)\leq\exp(-c\theta)$ for any $\theta>\theta_0$. This completes the proof.
\end{proof}

\section{General initial data}\label{s.general data}

In this section we prove the one-point upper tail bounds for general initial data that was stated in Section~\ref{s.intro} as Theorem~\ref{mt.general initial data}. We separate the upper and lower bounds into the following two theorems to clarify which hypotheses on the initial data are needed for each. Recall the definition of the class of initial data $\Hyp(K,L,M,\delta)$ from Definition~\ref{d.gen initial condition hypotheses}, that $\h$ refers to the KPZ line ensemble, and that $\hf$ (as defined in \eqref{e.hf definition}) is the scaled solution to the KPZ equation started from initial condition $f^{(t)}$.

In Theorem~\ref{t.general initial data upper bound} (upper bound), we assume that the initial condition $f^{(t)}$ satisfies certain growth conditions, but do not quantify any lower bounds. For Theorem~\ref{t.general initial data lower bound} (lower bound), we need to assume both growth conditions as well as a lower bound on $f^{(t)}$ on some non-trivial set.

\begin{theorem}\label{t.general initial data upper bound}
Let $T_0 \in (t_0,\infty]$ and $f^{(t)}\in \Hyp(K, L, M = \infty, \delta = 0)$ for some fixed $K$ and $L>0$ for all $t\in [t_0, T_0)$. There exist $\theta_0>0$ and $C<\infty$ such that, for $t\in [t_0, T_0)$ and $\theta>\theta_0$,
$$\P\left(\hf(0) \geq \theta\right) \leq \exp\left(-\frac{4}{3}\theta^{3/2} + C\theta^{3/4}\right).$$
\end{theorem}

\begin{theorem}\label{t.general initial data lower bound}
Let $T_0 \in (t_0,\infty]$ and $f^{(t)}\in \Hyp(K, L, M, \delta)$ for some fixed positive $K$, $L$, $M$, and $\delta$ for all $t\in [t_0, T_0)$.
 There exist $\theta_0>0$ and $C<\infty$ such that, for $t\in[t_0, T_0)$ and $\theta>\theta_0$,
$$\P\left(\hf(0) \geq \theta\right) \geq \exp\left(-\frac{4}{3}\theta^{3/2} - C\theta^{3/4}\right).$$

\end{theorem}

Both the upper and lower bounds' proofs relies on relating the upper tail for $\hf(0)$ to the upper tails of spatial maximums and minimums of $\h_1$. This is done via the following distributional identity.

\begin{lemma}[Lemma~4.3 of \cite{corwin2016kpz}]\label{l.kpz convolution}
Let $\mc H(t,x)$ be the solution to the KPZ equation started from general initial data $\mc H(0,\cdot)$. For fixed time $t>0$, the following distributional equality holds:
\begin{align*}
\mc H(t,0)\stackrel{d}{=} \log\left(\int_{-\infty}^\infty \exp\left\{t^{1/3}\bigl(\h_1(y) + t^{-1/3}\mc H(0, t^{2/3}y)\bigr)\right\}\, \dif y\right) -\frac{t}{12}+\frac{2}{3}\log t,
\end{align*}
so that, if $\mc H(0, y) = t^{1/3}f^{(t)}(t^{-2/3}y)$ for a function $f^{(t)}$,
$$\hf(0) \stackrel{d}{=} t^{-1/3}\log\left(\int_{-\infty}^\infty \exp\left\{t^{1/3}\bigl(\h_1(y) + f^{(t)}(y)\bigr)\right\}\, \dif y\right).$$

\end{lemma}

\begin{remark}\label{r.general density}
Using the above convolution formula and the sharp upper tail bounds, one can try to mimic the proof of Proposition~\ref{p.density control} to obtain a degree of regularity on the density of $\hf(0)$, and thereby obtain a sharp bound on this density as in the proof of Theorem~\ref{t.density}. 

To do this one would again condition on $\Fext([-\theta^{1/2}, \theta^{1/2}], 1)$ as well as on the bridges of $\h_1$ on $\smash{[-\theta^{1/2}, 0]}$ and $\smash{[0,\theta^{1/2}]}$. One would then consider the map $G$ (itself a function of the conditioned data) which, under the coupling of $\h_1$ and $\hf$ given by the convolution formula, takes $\h_1(0)$ to $\hf(0)$. To make use of this representation, one would then need to understand properties of this map: in particular, it would suffice to know that it and its inverse are Lipschitz, and that $G^{-1}(\theta)$ equals $\theta$ up to some error term. (It is easy to see that $G$ is increasing and so has an inverse, and in fact also that $G$ is 1-Lipschitz, from the convolution formula and the formula for $\h_1$ in terms of $\h_1(0)$ and the bridges on either side.)

However, to obtain that the inverse is Lipschitz and that $G^{-1}(\theta)\approx \theta$, one needs to control the entire profile of $\h_1$ conditional on $\hf(0)=\theta$, unlike in the narrow-wedge case, as the entire profile is involved in the convolution formula. This requires more work which we refrain from pursuing here, but we expect it should not be too difficult given the ideas and techniques developed in this paper.
\end{remark}

\subsection{Upper bound for general initial data}

For the upper bound we will need to control the supremum of $\h_1$ over $\R$. More precisely we need the following, formulated for a general line ensemble $\cL$.

\begin{proposition}\label{p.sup over real line tail}
Suppose $\cL$ satisfies Assumptions~\ref{as.bg}--\ref{as.tails}. Let $L>0$. There exist $\theta_0$ and $C<\infty$ such that, for all $\theta>\theta_0$,
$$\P\left(\sup_{x\in\R}\, (\cL_1(x) + x^2 - L|x|) > \theta\right) \leq \exp\left(-\frac{4}{3}\theta^{3/2} + C\theta^{3/4}\right).$$
\end{proposition}

We include the growth of $x^2$ in the supremum because we allow the initial data to grow like $x^2$; however, because $\h_1$ decays like $-x^2$, an extra lower order decay needs to be included with $+x^2$ to ensure that the solution does not immediately blow up, which is the role of $-L|x|$.

With Proposition~\ref{p.sup over real line tail} the proof of Theorem~\ref{t.general initial data upper bound} is straightforward.

\begin{proof}[Proof of Theorem~\ref{t.general initial data upper bound}]
Using the distributional identity from Lemma~\ref{l.kpz convolution} and the bound on $f^{(t)}$ for $t\in[t_0, T_0)$ coming from Definition~\ref{d.gen initial condition hypotheses} and the hypotheses,
\begin{align*}
\hf(0) &\stackrel{d}{=} t^{-1/3}\log\int_{-\infty}^\infty \exp\left\{t^{1/3} \left(\h_1(y) + f^{(t)}(y)\right)\right\}\, \dif y\\
&\leq t^{-1/3}\log\int_{-\infty}^\infty \exp\left\{t^{1/3} \left(\h_1(y) + y^2 - L|y| + K\right)\right\}\, \dif y\\
&= t^{-1/3}\log\int_{-\infty}^\infty \exp\left\{t^{1/3} \left(\h_1(y) + y^2 - \tfrac{1}{2}L|y|\right)\right\}\cdot e^{-t^{1/3}L |y|/2}\, \dif y + K\\
&\leq \sup_{x\in\R} \left[\h_1(x) + x^2 - \tfrac{1}{2}L|x|\right] + t^{-1/3}\log\int_{-\infty}^\infty e^{-t^{1/3}L |y|/2}\, \dif y + K\\
&= \sup_{x\in\R} \left[\h_1(x) + x^2 - \tfrac{1}{2}L|x|\right] + t^{-1/3} \log[4L^{-1}t^{-1/3}] + K.
\end{align*}
Since $t\geq t_0$, the second term in the final line is bounded uniformly in $t$. With this the proof is complete by invoking Proposition~\ref{p.sup over real line tail} (and using that $\h$ satisfies Assumptions~\ref{as.bg}, \ref{as.corr}(a), \ref{as.weak bk}(b\ensuremath{'}), \ref{as.mono in cond}, and \ref{as.tails} by Theorem~\ref{t.assumptions hold}).
\end{proof}

Next we turn to proving Proposition~\ref{p.sup over real line tail}. It relies on splitting up the supremum over $\R$ as a countable number of supremums over unit intervals and performing a union bound. For this we make use of Proposition~\ref{p.sharp sup over interval tail}, which gives a sharp upper tail estimate on the supremum of $\cL_1(x) + x^2$ over a unit interval; we will prove this latter statement in Section~\ref{s.gen data.no big max}.

\begin{proof}[Proof of Proposition~\ref{p.sup over real line tail}]
By a union bound,
$$\P\left(\sup_{x\in\R} (\cL_1(x) + x^2 - L|x|) > \theta\right) \leq \sum_{k=-\infty}^\infty \P\left(\sup_{x\in[k, k+1]} (\cL_1(x) + x^2 - L|x|) > \theta\right).$$
Next we see that, by stationarity of $\cL_1(x) + x^2$, for $k\in\Z$,
\begin{align*}
\sup_{x\in[k, k+1]} (\cL_1(x) + x^2 - L|x|) &\leq \sup_{x\in[k, k+1]} (\cL_1(x) + x^2) - L (|k|-1)\\
&\stackrel{d}{=} \sup_{x\in[0, 1]} (\cL_1(x) + x^2) - L(|k|-1).
\end{align*}
Thus, by Proposition~\ref{p.sharp sup over interval tail} and Theorem~\ref{t.upper tail upper bound}, for $k\in \Z$ and $\theta>\theta_0$
\begin{align*}
\P\left(\sup_{x\in[k, k+1]} (\cL_1(x) + x^2) - L|x| > \theta\right)
&\leq \P\left(\sup_{x\in[0, 1]} (\cL_1(x) + x^2)  > \theta + L(|k|-1)\right)\\
&\leq \exp\left(-\frac{4}{3}\theta^{3/2} -\frac{4}{3}L^{3/2}(|k|-1)^{3/2} + C\theta^{3/4}\right).
\end{align*}
This is clearly summable in $k$ to give $c(L)\exp(-\frac{4}{3}\theta^{3/2} + C\theta^{3/4})$, completing the proof by modifying the value of $C$ to absorb $c(L)$.
\end{proof}

\subsection{Lower bound for general initial data}

The estimate on $\h_1$ that we need to obtain a lower bound on the upper tail for $\hf(0)$ is the following, again formulated for a general line ensemble $\cL$.

\begin{proposition}\label{p.minimum over interval}
Suppose $\cL$ satisfies Assumptions~\ref{as.bg}--\ref{as.tails}. Let $M>0$. There exist $\theta_0$ and $C<\infty$ such that, for all $\theta>\theta_0$,
$$\P\left(\min_{x\in[-M,M]} \cL_1(x) \geq \theta\right) \geq \exp\left(-\frac{4}{3}\theta^{3/2} - C\theta^{1/2}\log\theta\right).$$
\end{proposition}

Similar to the upper bound, given Proposition~\ref{p.minimum over interval}, the proof of Theorem~\ref{t.general initial data lower bound} is almost immediate.

\begin{proof}[Proof of Theorem~\ref{t.general initial data lower bound}]
As in the upper bound, by the distributional identity from Lemma~\ref{l.kpz convolution},
\begin{align*}
\hf(0) &\stackrel{d}{=} t^{-1/3}\log\int_{-\infty}^\infty \exp\left\{t^{1/3} \left(\h_1(y) + f^{(t)}(y)\right)\right\}\, \dif y\\
&\geq \min_{x\in[-M,M]} \h_1(x) + t^{-1/3}\log\int_{\{y\in[-M,M]:f^{(t)}(y)\geq -K\}}\exp(t^{1/3} f^{(t)}(y)) \,\dif y\\
&\geq \min_{x\in[-M,M]} \h_1(x) - K - t^{-1/3}\log \delta^{-1},
\end{align*}
the last inequality since $f^{(t)}\in\Hyp(K,L,M,\delta)$ for all $t\in[t_0,T_0)$ implies that $\mrm{Leb}\{x\in[-M,M]: f^{(t)}(x)\geq -K\} \geq \delta$. With this, and noting that the last term is bounded since $t\geq t_0$, the proof of Theorem~\ref{t.general initial data lower bound} is complete by invoking Proposition~\ref{p.minimum over interval} and Theorem~\ref{t.assumptions hold}.
\end{proof}

It only remains to prove Proposition~\ref{p.minimum over interval}. One approach to proving it is to observe that, by Theorem~\ref{t.limit shape}, on the event that $\cL_1(0) >\theta +3M\theta^{1/2}$, it holds with probability at least $\frac{1}{2}$ that 
$$\min_{x\in[-M,M]} \cL_1(x) \geq (\theta+3M\theta^{1/2})-2M(\theta+3M\theta^{1/2})^{1/2}-K\theta^{1/4}\geq \theta$$
by choosing $K$ appropriately. The probability that $\cL_1(0) > \theta+3M\theta^{1/2}$ is at least $\exp(-\frac{4}{3}(\theta+3M\theta^{1/2})^{3/2}) \geq \exp(-\frac{4}{3}\theta^{3/2}-C\theta)$, which gives our claim with a slightly worse lower order error term.

We instead give a different argument making use of the explicit two-point upper tail asymptotic from Theorem~\ref{t.two point asymptotics} to lower bound the probability that $\cL_1(\pm M) \geq \theta + R$ for a large constant $R$; on this event the fluctuation of $\cL_1$ on $[-M,M]$ can be easily controlled via the Gibbs property. This approach gives a smaller error term as $R$ is a constant.

\begin{proof}[Proof of Proposition~\ref{p.minimum over interval}]
Let $R$ be a large constant to be fixed later. Trivially,
$$\P\left(\min_{x\in[-M,M]} \cL_1(x) \geq \theta\right) \geq \P\left(\cL_1(\pm M) \geq \theta+R, \min_{x\in[-M,M]} \cL_1(x) \geq \theta\right).$$
We apply the $H_t$-Brownian Gibbs property to $[-M,M]$, letting $\F$ be the $\sigma$-algebra generated by $\cL_1$ on $[-M,M]^c$ and the lower curves on $\R$, to see that
\begin{align*}
\P\left(\cL_1(\pm M) \geq \theta+R, \min_{x\in[-M,M]} \cL_1(x) \geq \theta\right) = \E\left[\PF\left(\min_{x\in[-M,M]} \cL_1(x) \geq \theta\right)\one_{\cL_1(\pm M) \geq \theta+R}\right]
\end{align*}
Let $B$ be a Brownian bridge from $(-M, \theta+R)$ to $(M, \theta+R)$. By monotonicity (Lemma~\ref{l.monotonicity}), the inner conditional probability in the previous display is lower bounded by the same probability with $B$ in place of $\cL_1$, so we obtain
\begin{align*}
\P\left(\min_{x\in[-M,M]} \cL_1(x) \geq \theta\right) \geq \P\left(\min_{x\in[-M,M]} B(x) \geq \theta\right)\cdot\P\left(\cL_1(-M) > \theta+R, \cL_1(M) > \theta+R\right).
\end{align*}
By Lemma~\ref{l.brownian bridge sup tail exact} (on the tail of the supremum of Brownian bridge), $\P\left(\min_{x\in[-M,M]} B(x) \geq \theta\right) = 1-\exp(-R^2/2)$, and we may set $R$ a large enough constant, depending only on $M$, so that this term is at least $\frac{1}{2}$.

So we have obtained that 
$$\P\left(\min_{x\in[-M,M]} \cL_1(x) \geq \theta\right) \geq \tfrac{1}{2}\P\Bigl(\cL_1(-M) > \theta+R, \cL_1(M) > \theta+R\Bigr)$$
We lower bound the final term using the second case of Theorem~\ref{t.two point asymptotics}. The parameters in that theorem are set as $\theta\mapsto M^2$, $a=b\mapsto (\theta+R)/M^2$. This yields that
\begin{align*}
\P\left(\cL_1(-M) > \theta+R, \cL_1(M) > \theta+R\right) \geq \exp\left(-\frac{4}{3}\theta^{3/2} - C\theta^{1/2}\log\theta\right),
\end{align*}
as can be checked by substituting the mentioned values of $\theta$, $a$, and $b$ into the asymptotic expression in Theorem~\ref{t.two point asymptotics}; in particular, the error term of $-((1 + a)^{1/2} + (1 + b)^{1/2})\theta^{1/2} \log[(1 + a)(1 + b)\theta]$ (in the notation of Theorem~\ref{t.two point asymptotics}) is of order $-((\theta+R)^{1/2}M^{-1})M \log[(\theta+R)^2M^{-4}\cdot M^2]$, which in turn is indeed of order $-\theta^{1/2}\log \theta$.
\end{proof}

\bibliographystyle{alpha}
\bibliography{../../airy-ld-limit-shape}

\begin{thebibliography}{CHHM21}

\bibitem[AB24]{aggarwal2024colored}
Amol Aggarwal and Alexei Borodin.
\newblock Colored line ensembles for stochastic vertex models.
\newblock {\em arXiv preprint arXiv:2402.06868}, 2024.

\bibitem[ACQ11]{amir2011probability}
Gideon Amir, Ivan Corwin, and Jeremy Quastel.
\newblock Probability distribution of the free energy of the continuum directed
  random polymer in 1+ 1 dimensions.
\newblock {\em Communications on pure and applied mathematics}, 64(4):466--537,
  2011.

\bibitem[Agg19]{aggarwal2019universality}
Amol Aggarwal.
\newblock Universality for lozenge tiling local statistics.
\newblock {\em arXiv preprint arXiv:1907.09991}, 2019.

\bibitem[Agg20]{aggarwal2020arctic}
Amol Aggarwal.
\newblock Arctic boundaries of the ice model on three-bundle domains.
\newblock {\em Inventiones mathematicae}, 220(2):611--671, 2020.

\bibitem[AH23]{aggarwal2023strong}
Amol Aggarwal and Jiaoyang Huang.
\newblock Strong characterization for the {A}iry line ensemble.
\newblock {\em arXiv preprint arXiv:2308.11908}, 2023.

\bibitem[AT07]{adler2007random}
Robert~J Adler and Jonathan~E Taylor.
\newblock {\em Random fields and geometry}, volume~80.
\newblock Springer, 2007.

\bibitem[Bar05]{barbato2005fkg}
David Barbato.
\newblock {FKG} inequality for {B}rownian motion and stochastic differential
  equations.
\newblock {\em Electronic Communications in Probability}, 10:7--16, 2005.

\bibitem[BC95]{bertini1995stochastic}
Lorenzo Bertini and Nicoletta Cancrini.
\newblock The stochastic heat equation: {F}eynman-{K}ac formula and
  intermittence.
\newblock {\em Journal of statistical Physics}, 78(5):1377--1401, 1995.

\bibitem[BG16]{borodin2016moments}
Alexei Borodin and Vadim Gorin.
\newblock Moments match between the {KPZ} equation and the {A}iry point
  process.
\newblock {\em SIGMA. Symmetry, Integrability and Geometry: Methods and
  Applications}, 12:102, 2016.

\bibitem[BG19]{basu2019connecting}
Riddhipratim Basu and Shirshendu Ganguly.
\newblock Connecting eigenvalue rigidity with polymer geometry: Diffusive
  transversal fluctuations under large deviation.
\newblock {\em arXiv preprint arXiv:1902.09510}, 2019.

\bibitem[BG20]{biroli2020large}
Giulio Biroli and Alice Guionnet.
\newblock Large deviations for the largest eigenvalues and eigenvectors of
  spiked random matrices.
\newblock {\em Electronic Communications in Probability}, 25, 2020.

\bibitem[BGS19]{basu2019delocalization}
Riddhipratim Basu, Shirshendu Ganguly, and Allan Sly.
\newblock Delocalization of polymers in lower tail large deviation.
\newblock {\em Communications in Mathematical Physics}, 370(3):781--806, 2019.

\bibitem[BQS16]{bufetov2016kernels}
Alexander~I Bufetov, Yanqi Qiu, and Alexander Shamov.
\newblock Kernels of conditional determinantal measures and the proof of the
  {L}yons-{P}eres conjecture.
\newblock {\em arXiv preprint arXiv:1612.06751}, 2016.

\bibitem[CC22]{cafasso2022riemann}
Mattia Cafasso and Tom Claeys.
\newblock A {Riemann-Hilbert} approach to the lower tail of the
  {Kardar-Parisi-Zhang} equation.
\newblock {\em Communications on Pure and Applied Mathematics}, 75(3):493--540,
  2022.

\bibitem[CCG23]{chowdhury2023characterizing}
Mriganka Basu~Roy Chowdhury, Pietro Caputo, and Shirshendu Ganguly.
\newblock Characterizing {G}ibbs states for area-tilted {B}rownian lines.
\newblock {\em arXiv preprint arXiv:2310.06817}, 2023.

\bibitem[CCR21]{charlier2021uniform}
Christophe Charlier, Tom Claeys, and Giulio Ruzza.
\newblock Uniform tail asymptotics for {A}iry kernel determinant solutions to
  {KdV} and for the narrow wedge solution to {KPZ}.
\newblock {\em arXiv preprint arXiv:2111.14569}, 2021.

\bibitem[CD15]{chen2015moments}
Le~Chen and Robert~C Dalang.
\newblock Moments and growth indices for the nonlinear stochastic heat equation
  with rough initial conditions.
\newblock {\em The Annals of Probability}, 43(6):3006--3051, 2015.

\bibitem[CG20a]{corwin2018kpz}
Ivan Corwin and Promit Ghosal.
\newblock {KPZ} equation tails for general initial data.
\newblock {\em Electronic Journal of Probability}, 25, 2020.

\bibitem[CG20b]{corwin2020lower}
Ivan Corwin and Promit Ghosal.
\newblock Lower tail of the {KPZ} equation.
\newblock {\em Duke Mathematical Journal}, 169(7):1329--1395, 2020.

\bibitem[CG23]{caputo2023uniqueness}
Pietro Caputo and Shirshendu Ganguly.
\newblock Uniqueness, mixing, and optimal tails for {B}rownian line ensembles
  with geometric area tilt.
\newblock {\em arXiv preprint arXiv:2305.18280}, 2023.

\bibitem[CGH21]{corwin2021kpz}
Ivan Corwin, Promit Ghosal, and Alan Hammond.
\newblock {KPZ} equation correlations in time.
\newblock {\em The Annals of Probability}, 49(2):832--876, 2021.

\bibitem[CH14]{corwin2014brownian}
Ivan Corwin and Alan Hammond.
\newblock Brownian {G}ibbs property for {A}iry line ensembles.
\newblock {\em Inventiones mathematicae}, 195(2):441--508, 2014.

\bibitem[CH16]{corwin2016kpz}
Ivan Corwin and Alan Hammond.
\newblock {KPZ} line ensemble.
\newblock {\em Probability Theory and Related Fields}, 166(1):67--185, 2016.

\bibitem[CHH19]{calvert2019brownian}
Jacob Calvert, Alan Hammond, and Milind Hegde.
\newblock Brownian structure in the {KPZ} fixed point.
\newblock {\em arXiv preprint arXiv:1912.00992}, 2019.

\bibitem[CHHM21]{KPZfixedptHD}
Ivan Corwin, Alan Hammond, Milind Hegde, and Konstantin Matetski.
\newblock Exceptional times when the {KPZ} fixed point violates {J}ohansson's
  conjecture on maximizer uniqueness.
\newblock {\em arXiv preprint arXiv:2101.04205}, 2021.

\bibitem[CJK13]{conus2013chaotic}
Daniel Conus, Mathew Joseph, and Davar Khoshnevisan.
\newblock On the chaotic character of the stochastic heat equation, before the
  onset of intermitttency.
\newblock {\em The Annals of Probability}, 41(3B):2225--2260, 2013.

\bibitem[CLDR10]{calabrese2010free}
Pasquale Calabrese, Pierre Le~Doussal, and Alberto Rosso.
\newblock Free-energy distribution of the directed polymer at high temperature.
\newblock {\em EPL (Europhysics Letters)}, 90(2):20002, 2010.

\bibitem[Cor12]{corwin2012kardar}
Ivan Corwin.
\newblock The {K}ardar--{P}arisi--{Z}hang equation and universality class.
\newblock {\em Random matrices: Theory and applications}, 1(01):1130001, 2012.

\bibitem[CQ13]{corwin2013crossover}
Ivan Corwin and Jeremy Quastel.
\newblock Crossover distributions at the edge of the rarefaction fan.
\newblock {\em The Annals of Probability}, 41(3A):1243--1314, 2013.

\bibitem[CS14]{corwin2014ergodicity}
Ivan Corwin and Xin Sun.
\newblock Ergodicity of the {A}iry line ensemble.
\newblock {\em Electronic Communications in Probability}, 19:1--11, 2014.

\bibitem[CS16]{colomo2016arctic}
Filippo Colomo and Andrea Sportiello.
\newblock Arctic curves of the six-vertex model on generic domains: the tangent
  method.
\newblock {\em Journal of Statistical Physics}, 164(6):1488--1523, 2016.

\bibitem[Dau24]{dauvergne2023wiener}
Duncan Dauvergne.
\newblock Wiener densities for the {A}iry line ensemble.
\newblock {\em Proceedings of the London Mathematical Society}, 129(4):e12638,
  2024.

\bibitem[DG21]{das2021law}
Sayan Das and Promit Ghosal.
\newblock Law of iterated logarithms and fractal properties of the {KPZ}
  equation.
\newblock {\em arXiv preprint arXiv:2101.00730}, 2021.

\bibitem[Dim21]{dimitrov2021characterizationH}
Evgeni Dimitrov.
\newblock Characterization of {$H$}-{B}rownian {G}ibbsian line ensembles.
\newblock {\em arXiv preprint arXiv:2103.01186}, 2021.

\bibitem[DM21]{dimitrov2021characterization}
Evgeni Dimitrov and Konstantin Matetski.
\newblock Characterization of {B}rownian {G}ibbsian line ensembles.
\newblock {\em The Annals of Probability}, 49(5):2477--2529, 2021.

\bibitem[Dot10]{dotsenko2010bethe}
Victor Dotsenko.
\newblock Bethe ansatz derivation of the {T}racy-{W}idom distribution for
  one-dimensional directed polymers.
\newblock {\em EPL (Europhysics Letters)}, 90(2):20003, 2010.

\bibitem[DOV18]{dauvergne2018directed}
Duncan Dauvergne, Janosch Ortmann, and B{\'a}lint Vir{\'a}g.
\newblock The directed landscape.
\newblock {\em arXiv preprint arXiv:1812.00309}, 2018.

\bibitem[DT21]{das2021fractional}
Sayan Das and Li-Cheng Tsai.
\newblock Fractional moments of the stochastic heat equation.
\newblock In {\em Annales de l'Institut Henri Poincar{\'e}, Probabilit{\'e}s et
  Statistiques}, volume~57, pages 778--799. Institut Henri Poincar{\'e}, 2021.

\bibitem[DV21]{dauvergne2018basic}
Duncan Dauvergne and B{\'a}lint Vir{\'a}g.
\newblock Bulk properties of the {A}iry line ensemble.
\newblock {\em Ann. Probab.}, 2021+.
\newblock To appear.

\bibitem[FKG71]{fortuin1971correlation}
Cees~M Fortuin, Pieter~W Kasteleyn, and Jean Ginibre.
\newblock Correlation inequalities on some partially ordered sets.
\newblock {\em Communications in Mathematical Physics}, 22(2):89--103, 1971.

\bibitem[Flo14]{flores2014strict}
Gregorio R~Moreno Flores.
\newblock On the (strict) positivity of solutions of the stochastic heat
  equation.
\newblock {\em The Annals of Probability}, pages 1635--1643, 2014.

\bibitem[GH20]{ganguly2020optimal}
Shirshendu Ganguly and Milind Hegde.
\newblock Optimal tail exponents in general last passage percolation via
  bootstrapping \& geodesic geometry.
\newblock {\em arXiv preprint arXiv:2007.03594}, 2020.

\bibitem[Gho18]{ghosal2018moments}
Promit Ghosal.
\newblock Moments of the {SHE} under delta initial measure.
\newblock {\em arXiv preprint arXiv:1808.04353}, 2018.

\bibitem[GHZ23]{ganguly2023brownian}
Shirshendu Ganguly, Milind Hegde, and Lingfu Zhang.
\newblock Brownian bridge limit of path measures in the upper tail of {KPZ}
  models.
\newblock {\em arXiv preprint arXiv:2311.12009}, 2023.

\bibitem[GHZ25]{ganguly2025van}
Shirshendu Ganguly, Milind Hegde, and Lingfu Zhang.
\newblock van den {B}erg-{K}esten--type correlation inequalities for disjoint
  polymers in the {KPZ} universality class.
\newblock {\em arXiv preprint arXiv:2512.17823}, 2025.

\bibitem[GIP15]{gubinelli2015paracontrolled}
Massimiliano Gubinelli, Peter Imkeller, and Nicolas Perkowski.
\newblock Paracontrolled distributions and singular {PDE}s.
\newblock In {\em Forum of Mathematics, Pi}, volume~3. Cambridge University
  Press, 2015.

\bibitem[GJ14]{gonccalves2014nonlinear}
Patr{\'\i}cia Gon{\c{c}}alves and Milton Jara.
\newblock Nonlinear fluctuations of weakly asymmetric interacting particle
  systems.
\newblock {\em Archive for Rational Mechanics and Analysis}, 212(2):597--644,
  2014.

\bibitem[GL20]{ghosal2020lyapunov}
Promit Ghosal and Yier Lin.
\newblock Lyapunov exponents of the {SHE} for general initial data.
\newblock {\em arXiv preprint arXiv:2007.06505}, 2020.

\bibitem[GP17]{gubinelli2017kpz}
Massimiliano Gubinelli and Nicolas Perkowski.
\newblock {KPZ} reloaded.
\newblock {\em Communications in Mathematical Physics}, 349(1):165--269, 2017.

\bibitem[Gri99]{grimmett1999percolation}
Geoffrey Grimmett.
\newblock {\em Percolation}.
\newblock Grundlehren der mathematischen Wissenschaften; 321. Springer, Berlin,
  2nd ed. edition, 1999.

\bibitem[Hai13]{hairer2013solving}
Martin Hairer.
\newblock Solving the {KPZ} equation.
\newblock {\em Annals of Mathematics}, pages 559--664, 2013.

\bibitem[Ham19a]{hammond2017modulus}
Alan Hammond.
\newblock Modulus of continuity of polymer weight profiles in {B}rownian last
  passage percolation.
\newblock {\em The Annals of Probability}, 47(6):3911--3962, 2019.

\bibitem[Ham19b]{hammond2017patchwork}
Alan Hammond.
\newblock A patchwork quilt sewn from {B}rownian fabric: regularity of polymer
  weight profiles in {B}rownian last passage percolation.
\newblock In {\em Forum of Mathematics, Pi}, volume~7. Cambridge University
  Press, 2019.

\bibitem[Ham20]{brownianLPPtransversal}
Alan Hammond.
\newblock Exponents governing the rarity of disjoint polymers in {B}rownian
  last passage percolation.
\newblock {\em Proceedings of the London Mathematical Society},
  120(3):370--433, 2020.

\bibitem[Ham22]{hammond2016brownian}
Alan Hammond.
\newblock Brownian regularity for the {A}iry line ensemble, and multi-polymer
  watermelons in {B}rownian last passage percolation.
\newblock {\em Memoirs of the American Mathematical Society}, 277(1363), 2022.

\bibitem[HKPV09]{manjunath}
John~Ben Hough, Manjunath Krishnapur, Yuval Peres, and B\'alint Vir\'ag.
\newblock {\em Zeros of {G}aussian analytic functions and determinantal point
  processes}, volume~51.
\newblock American Mathematical Soc., 2009.

\bibitem[Kal21]{Kallenberg}
Olav Kallenberg.
\newblock {\em Foundations of modern probability}, volume~99 of {\em
  Probability Theory and Stochastic Modelling}.
\newblock Springer, 2021.
\newblock Third edition.

\bibitem[KKX17]{khoshnevisan2017intermittency}
Davar Khoshnevisan, Kunwoo Kim, and Yimin Xiao.
\newblock Intermittency and multifractality: a case study via parabolic
  stochastic {PDE}s.
\newblock {\em The Annals of Probability}, 45(6A):3697--3751, 2017.

\bibitem[KMS16]{kamenev2016short}
Alex Kamenev, Baruch Meerson, and Pavel~V Sasorov.
\newblock Short-time height distribution in the one-dimensional
  {K}ardar-{P}arisi-{Z}hang equation: {S}tarting from a parabola.
\newblock {\em Physical Review E}, 94(3):032108, 2016.

\bibitem[KOS06]{kenyon2006dimers}
Richard Kenyon, Andrei Okounkov, and Scott Sheffield.
\newblock Dimers and amoebae.
\newblock {\em Annals of Mathematics}, pages 1019--1056, 2006.

\bibitem[LD20]{le2020large}
Pierre Le~Doussal.
\newblock Large deviations for the {K}ardar-{P}arisi-{Z}hang equation from the
  {K}adomtsev-{P}etviashvili equation.
\newblock {\em Journal of Statistical Mechanics: Theory and Experiment},
  2020(4):043201, 2020.

\bibitem[LLT21]{lamarre2021kpz}
Pierre Yves~Gaudreau Lamarre, Yier Lin, and Li-Cheng Tsai.
\newblock {KPZ} equation with a small noise, deep upper tail and limit shape.
\newblock {\em arXiv preprint arXiv:2106.13313}, 2021.

\bibitem[Lyo18]{lyons2018note}
Russell Lyons.
\newblock A note on tail triviality for determinantal point processes.
\newblock {\em Electronic Communications in Probability}, 23:1--3, 2018.

\bibitem[MQR21]{matetski2016kpz}
Konstantin Matetski, Jeremy Quastel, and Daniel Remenik.
\newblock The {KPZ} fixed point.
\newblock {\em Acta Mathematica}, 227(1):115--203, 2021.

\bibitem[Mue91]{mueller1991support}
Carl Mueller.
\newblock On the support of solutions to the heat equation with noise.
\newblock {\em Stochastics: An International Journal of Probability and
  Stochastic Processes}, 37(4):225--245, 1991.

\bibitem[Nic21]{nica2021intermediate}
Mihai Nica.
\newblock Intermediate disorder limits for multi-layer semi-discrete directed
  polymers.
\newblock {\em Electronic Journal of Probability}, 26:1--50, 2021.

\bibitem[NQR20]{nica2020one}
Mihai Nica, Jeremy Quastel, and Daniel Remenik.
\newblock One-sided reflected {B}rownian motions and the {KPZ} fixed point.
\newblock {\em arXiv preprint arXiv:2002.02922}, 2020.

\bibitem[O'C12]{o2012directed}
Neil O'Connell.
\newblock Directed polymers and the quantum {T}oda lattice.
\newblock {\em The Annals of Probability}, 40(2):437--458, 2012.

\bibitem[O'C14]{o2014whittaker}
Neil O'Connell.
\newblock Whittaker functions and related stochastic processes.
\newblock In {\em MSRI Publications: Random Matrix Theory, Interacting Particle
  Systems and Integrable Systems}, pages 385--409. Cambridge University Press,
  2014.

\bibitem[OO18]{osada2018discrete}
Hirofumi Osada and Shota Osada.
\newblock Discrete approximations of determinantal point processes on
  continuous spaces: tree representations and tail triviality.
\newblock {\em Journal of Statistical Physics}, 170(2):421--435, 2018.

\bibitem[OSZ14]{o2014geometric}
Neil O’Connell, Timo Sepp{\"a}l{\"a}inen, and Nikos Zygouras.
\newblock Geometric {RSK} correspondence, {W}hittaker functions and symmetrized
  random polymers.
\newblock {\em Inventiones mathematicae}, 197(2):361--416, 2014.

\bibitem[OT19]{olla2019exceedingly}
Stefano Olla and Li-Cheng Tsai.
\newblock Exceedingly large deviations of the totally asymmetric exclusion
  process.
\newblock {\em Electronic Journal of Probability}, 24:1--71, 2019.

\bibitem[OY01]{OY01}
Neil O'Connell and Marc Yor.
\newblock Brownian analogues of {B}urke's theorem.
\newblock {\em Stochastic Process. Appl.}, 96(2):285--304, 2001.

\bibitem[Pfa79]{pfanzagl1979conditional}
P~Pfanzagl.
\newblock Conditional distributions as derivatives.
\newblock {\em The Annals of Probability}, 7(6):1046--1050, 1979.

\bibitem[Pre06]{preston2006random}
Chris Preston.
\newblock {\em Random fields}, volume 534.
\newblock Springer, 2006.

\bibitem[PS02]{prahofer2002PNG}
Michael Pr\"{a}hofer and Herbert Spohn.
\newblock Scale invariance of the {PNG} droplet and the {A}iry process.
\newblock {\em J. Statist. Phys.}, 108(5-6):1071--1106, 2002.
\newblock Dedicated to David Ruelle and Yasha Sinai on the occasion of their
  65th birthdays.

\bibitem[QR19]{quastel2019kp}
Jeremy Quastel and Daniel Remenik.
\newblock {KP} governs random growth off a one dimensional substrate.
\newblock {\em arXiv preprint arXiv:1908.10353}, 2019.

\bibitem[QS22]{quastel2022convergence}
Jeremy Quastel and Sourav Sarkar.
\newblock Convergence of exclusion processes and the {KPZ} equation to the
  {KPZ} fixed point.
\newblock {\em Journal of the American Mathematical Society}, 2022.

\bibitem[QT21]{quastel2021hydrodynamic}
Jeremy Quastel and Li-Cheng Tsai.
\newblock Hydrodynamic large deviations of {TASEP}.
\newblock {\em arXiv preprint arXiv:2104.04444}, 2021.

\bibitem[RRV11]{ramirez2011beta}
Jose Ramirez, Brian Rider, and B{\'a}lint Vir{\'a}g.
\newblock Beta ensembles, stochastic {A}iry spectrum, and a diffusion.
\newblock {\em Journal of the American Mathematical Society}, 24(4):919--944,
  2011.

\bibitem[RY13]{revuz2013continuous}
Daniel Revuz and Marc Yor.
\newblock {\em Continuous martingales and {B}rownian motion}, volume 293.
\newblock Springer Science \& Business Media, 2013.

\bibitem[She05]{sheffield2005random}
Scott Sheffield.
\newblock {\em Random surfaces}.
\newblock Soci{\'e}t{\'e} math{\'e}matique de France Paris, 2005.

\bibitem[SS10a]{sasamoto2010crossover}
Tomohiro Sasamoto and Herbert Spohn.
\newblock The crossover regime for the weakly asymmetric simple exclusion
  process.
\newblock {\em Journal of Statistical Physics}, 140(2):209--231, 2010.

\bibitem[SS10b]{sasamoto2010exact}
Tomohiro Sasamoto and Herbert Spohn.
\newblock Exact height distributions for the {KPZ} equation with narrow wedge
  initial condition.
\newblock {\em Nuclear Physics B}, 834(3):523--542, 2010.

\bibitem[SS10c]{spohn2010one}
Herbert Spohn and Tomohiro Sasamoto.
\newblock The one-dimensional {KPZ} equation: an exact solution and its
  universality.
\newblock {\em Phys. Rev. Lett}, 104, 2010.

\bibitem[Tsa18]{tsai2018exact}
Li-Cheng Tsai.
\newblock Exact lower tail large deviations of the {KPZ} equation.
\newblock {\em arXiv preprint arXiv:1809.03410}, 2018.

\bibitem[TW94]{tracy1994level}
Craig~A Tracy and Harold Widom.
\newblock Level-spacing distributions and the {A}iry kernel.
\newblock {\em Communications in Mathematical Physics}, 159(1):151--174, 1994.

\bibitem[Vir20]{virag2020heat}
B{\'a}lint Vir{\'a}g.
\newblock The heat and the landscape {I}.
\newblock {\em arXiv preprint arXiv:2008.07241}, 2020.

\bibitem[Wu21a]{wu2021brownian}
Xuan Wu.
\newblock Brownian regularity for the {KPZ} line ensemble.
\newblock {\em arXiv preprint arXiv:2106.08052}, 2021.

\bibitem[Wu21b]{wu2021tightness}
Xuan Wu.
\newblock Convergence of the {KPZ} line ensemble.
\newblock {\em arXiv preprint arXiv:2106.08051}, 2021.

\end{thebibliography}

\appendix

\section{Monotonicity proofs \& assumption verification}\label{app.monotonicity proofs}

In this appendix we prove Theorem~\ref{t.assumptions hold}, i.e., we verify that that the KPZ, parabolic Airy, and extremal stationary ensembles all satisfy Assumptions~\ref{as.bg}, \ref{as.corr}(a), \ref{as.mono in cond}, and \ref{as.tails}; that the former satisfies Assumption~\ref{as.weak bk}(b\ensuremath{'}); and that the latter two satisfy Assumption~\ref{as.corr}(b). We also prove Lemma~\ref{l.monotonicity} in Section~\ref{app.mono.proof of mono in cond lemma}. 

In proving the validity of the assumptions, we will first prove an analogous statement for collections of Brownian bridges or motions which are reweighted by an appropriate Radon-Nikodym derivative, and then obtain the results by limiting arguments; 
 for the KPZ line ensemble, this goes through the O'Connell-Yor free energy line ensemble, which we introduce in Section~\ref{app.mono.o'connell-yor}. The Brownian monotonicity statements are given in Section~\ref{app.brownian mono statements}.
The assumptions are then formally verified in Section~\ref{app.hamiltonian.verifying assumptions}, and the Brownian statements are proved in Section~\ref{app.Brownian mono proofs}.

\subsection{Proof of Lemma~\ref{l.bound point conditioning by tail conditioning}}\label{app.mono.proof of mono in cond lemma}

\begin{proof}[Proof of Lemma~\ref{l.bound point conditioning by tail conditioning}]
We prove the first inequality, in which $\inf E_i = y_i$ and $F$ is increasing; the other is proved analogously.

Observe that
\begin{align}
\MoveEqLeft[9]
\E\left[F \mid \h_1(-\theta^{1/2}) \in E_1, \h_1(\theta^{1/2}) \in E_2\right]\nonumber\\
&= \frac{\E\left[F\cdot\one_{\h_1(-\theta^{1/2}) \in E_1, \h_1(\theta^{1/2}) \in E_2}\right]}{\P\left(\h_1(-\theta^{1/2}) \in E_1, \h_1(\theta^{1/2}) \in E_2\right)}\nonumber\\
&= \E\left[\E\Bigl[F\mid \h_1(-\theta^{1/2}), \h_1(\theta^{1/2})\Bigr]\frac{\one_{\h_1(-\theta^{1/2}) \in E_1, \h_1(\theta^{1/2}) \in E_2}}{\P\left(\h_1(-\theta^{1/2}) \in E_1, \h_1(\theta^{1/2}) \in E_2\right)}\right]. \label{e.the average}
\end{align}
The expression on the last line is an average of the function
\begin{align*}
(z_1,z_2)\mapsto\E\left[F \mid \h_1(-\theta^{1/2}) = z_1, \h_1(\theta^{1/2}) = z_2\right].
\end{align*}
against a certain probability measure supported on a subset of $[y_1,\infty)\times[y_2,\infty)$ (recall that $\inf E_i = y_i$). By Assumption~\ref{as.mono in cond}, this function is increasing in $(z_1,z_2)$. Thus the right-hand side of \eqref{e.the average} is lower bounded by $\E[F \mid \h_1(-\theta^{1/2}) = y_1, \h_1(\theta^{1/2}) = y_2]$, which completes the proof.
\end{proof}

\subsection{O'Connell-Yor diffusion}\label{app.mono.o'connell-yor}

The O'Connell-Yor free energy line ensemble is a diffusion $X^N:\{1, \ldots, N\}\times[0,\infty)\to\R$ whose infinitesimal generator is given by
\begin{equation}\label{e.oy diffusion generator}
\frac{1}{2}\Delta + \nabla\log\psi_0\cdot\nabla,
\end{equation}
where $\Delta$ is the Laplacian on $\R^N$, $\nabla$ is the gradient on $\R^N$, and $\psi_0$ is the class one $\mf{gl}_N$-Whittaker function, and the process is started according to a certain explicit entrance law $\mu_s$ for entrance at time $s>0$. See \cite[Theorem~3.1 and Corollary~4.1]{o2012directed} for the precise statement and specification of $\mu_s$.

One can consider the same diffusion started from a fixed point $z\in\R^N$, and for the next discussion let $X^N$ denote that diffusion; we will refer to this as the O'Connell-Yor diffusion, in contrast to the O'Connell-Yor free energy line ensemble. It follows from \cite[Section~8.3]{corwin2016kpz} that $X^N|_{[0,K]}$ has an equivalent description, for any $K>0$, as a collection of $N$ independent rate one Brownian motions on $[0,K]$ reweighted by the Radon-Nikodym derivative (where $H_1(x) = 2\exp(x)$ is as defined in Proposition~\ref{p.scaled gibbs})
\begin{equation}\label{e.brownian motion rn derivative}
\begin{split}
\MoveEqLeft[10]
\frac{\psi_0(X^N(K))}{\psi_0(z)}\exp\left\{-\int_0^K \sum_{i=0}^{N} \exp\left(X^N_{i+1}(u) - X^N_i(u)\right)\,\dif u\right\}\\
&= \frac{\psi_0(X^N(K))}{\psi_0(z)}\exp\left\{-\int_0^K \sum_{i=0}^{N} \tfrac{1}{2}H_1\left(X^N_{i+1}(u) - X^N_i(u)\right)\,\dif u\right\},
\end{split}
\end{equation}
where $X_0$ and $X_{N+1}$ are interpreted as $+\infty$ and $-\infty$ respectively. Using this expression, one can also define more general processes with upper and lower boundary data functions. One can also give the Brownian motions an ordered drift $\lambda = (\lambda_1 > \lambda_2 > \ldots  > \lambda_N)$: in this situation, the diffusion has generator as in \eqref{e.oy diffusion generator} and Radon-Nikodym derivative as in \eqref{e.brownian motion rn derivative} with $\psi_\lambda$ , the class one $\mf{gl}_N$-WHittaker function of index $\lambda$, replacing $\psi_0$ (see \cite[Sections 3.2 and 8.3]{corwin2016kpz}). This process has no direct relation with the O'Connell-Yor free energy line ensemble. However, in the case of ordered drift, one has a Feynman-Kac representation for $\psi_\lambda$ \cite[Corollary 1]{o2014whittaker} (here under $\E^\lambda_z$, $X^N$ is distributed as a Brownian motion in $\R^N$ with drift $\lambda$ and starting point $z$):
\begin{align}\label{e.feynman-kac whittaker}
\psi_\lambda(z) = \prod_{i<j}\Gamma(\lambda_i-\lambda_j)e^{\lambda\cdot z}\E^\lambda_z\left[\exp\left\{-\int_0^\infty \sum_{i=0}^N \exp\left(X^N_{i+1}(u) - X^N_i(u)\right)\,\dif u\right\}\right];
\end{align}
the ordered drift condition ensures that the integral converges almost surely. For technical reasons that will be apparent soon (see Lemma~\ref{l.OY local limit}), this probabilistic representation (which is not available for $\psi_0$) will be useful to us.

Now we return to the case where $\lambda = 0^N$ and the diffusion has entrance law $\mu_s$. It is proven in \cite[Corollary~1.7]{nica2021intermediate} that, after appropriate scaling, $X^N$ converges in the sense of finite dimensional distributions to an ensemble whose lowest-indexed line is the narrow-wedge solution to the KPZ equation at time $t$, i.e., the KPZ line ensemble; the KPZ equation considered in \cite{nica2021intermediate} has slightly different coefficients than in this paper, but one can go between the two by simple scaling relations, as we discuss in Section~\ref{app.hamiltonian} ahead. The arguments for \cite[Theorem~3.10]{corwin2016kpz} then upgrades this convergence to hold on the level of processes. This yields the following theorem.

\begin{theorem}\label{t.OY to KPZ}
Fix $t>0$. Let $C(N,t,x) = \exp(N+\frac{\sqrt{Nt}+x}{2}+xt^{-1/2}N^{1/2})(t^{1/2}N^{-1/2})^N$. Then, in the topology of uniform convergence on compact sets, as $N\to\infty$ and for each $k\in\N$,
\begin{align*}
X^N_k(\sqrt{Nt}+x)- \log C(N,t,x) \stackrel{d}{\to} \widetilde{\mc H}_k(t,x),
\end{align*}
where $\widetilde{\mc H}_1(t,x) = \mc H(t/2, x/2)$ is the narrow-wedge solution to $\partial_t \widetilde{\mc H}_1 = \frac{1}{2}\partial_{x}^2 \widetilde{\mc H}_1+\frac{1}{2}(\partial_x \widetilde{\mc H}_1)^2+\xi$ (see Appendix~\ref{app.hamiltonian}) and $\widetilde{\mc H}_k$ is defined by the above for $k\geq 2$.
\end{theorem}	

In \cite[Corollary 1.9]{nica2021intermediate}, there is an additional term of $\log(t^{-(k-1)}(k-1)!)$ on the lefthand side of the convergence, yielding a slightly shifted definition for $\widetilde{\mc H}_k$ for $k\geq 2$ compared to the above. We adopt the above definition as it yields a line ensemble with a Gibbs property which is the same for all curves (i.e., the Hamiltonian does not depend on the index of the curve), unlike the case for the shifted definition. Note that the definitions coincide for $k=1$, so $\widetilde{\mc H}_1$ (and thus ultimately $\h_1$) is unaffected by the choice.

As a result of Theorem~\ref{t.OY to KPZ}, for the convergence of $X^N$ to the KPZ line ensemble, it is enough to consider $X^N$ on the interval $[0,N]$, and by \eqref{e.brownian motion rn derivative}, this has a description as a reweighted collection of independent Brownian motions. In the next section we will first state these properties for such collections without the endpoint reweighting factor $\psi_0(X^N(N))/\psi_0(z)$, and then extend them to the correctly reweighted process (see Lemma~\ref{l.OY local limit} ahead).

\subsection{Reweighted Brownian bridge/motion ensembles}\label{app.brownian mono statements}
Next we state the results we will prove. Introducing some notation will be helpful. For $N\in\N$, let $\smash{\R^N_>}$ be the set of decreasing $N$-vectors, i.e., $\smash{\R^N_> = \{\vec x\in\R^N: x_1 > x_2 > \ldots > x_N\}}$. 

For $T>0$, $N, m\in\N$, vectors $\smash{\vec w, \vec z \in \R^N}$, $\vec x, \vec y \in \R^m$, and $f, g:[0,T]\to\R\cup \{-\infty\}$, let $\smash{\mu_{N,H, \inf}^{\vec x, \vec y, \vec w, \vec z, f, g}}$ be the law of $N$ independent Brownian bridges $B_1, \ldots, B_N$ on $[0,T]$ reweighted by the Radon-Nikodym derivative given by \eqref{e.gibbs rn derivative} with lower boundary data $g$, where $B_j(a) = w_j$ and $B_i(b) = z_i$ for $i\in\intint{1,N}$, and additionally we condition on $B_1(x_j)$ equaling $y_j$ for $j\in\intint{1,m}$ and on $\inf_{[0,T]} (B_1-f) \geq 0$. Let $\smash{\mu_{N,H, \sup}^{\vec x, \vec y, \vec w, \vec z, f, g}}$ be the same with $\inf$ in the final condition replaced by $\sup$.

Similarly, $\smash{\mu_{N,H}^{\vec w, \vec z, g}}$ will be the same without the final mentioned conditionings in the definition of $\smash{\mu_{N,H, \inf}^{\vec x, \vec y, \vec w, \vec z, f, g}}$.  

The next statement gives monotonicity in conditioning for Brownian bridge ensembles. It will be proved in Section~\ref{app.Brownian mono proofs}.

\begin{theorem}[Monotonicity in conditioning of reweighted Brownian bridges]\label{t.monotonicity under extra conditioning bb}
Let $H:\R\to[0,\infty)$ be convex. Let $[a,b]\subseteq \R$, $\smash{y_{\shortuparrow, j}} > y_{\shortdownarrow,j}$ for $j\in\intint{1,m}$ and $f, g:[a,b]\to\R\cup\{-\infty\}$ be upper semicontinuous. There is a coupling of $\mu^{\vec x,\vec y_{\shortuparrow}, \vec w,\vec z, f, g}_{N, H, \inf}$ and $\mu^{\vec x,\vec y_{\shortdownarrow}, \vec w,\vec z, f, g}_{N,H, \sup}$ under which $\smash{B^\shortuparrow_i(x) > B^\shortdownarrow_i(x)}$ almost surely for all $x\in[a,b]$ and $i\in\intint{1,N}$, where $\smash{B^\shortuparrow = (B^\shortuparrow_1, \ldots, B^\shortuparrow_N)}$ and $\smash{B^\shortdownarrow = (B^\shortdownarrow_1, \ldots, B^\shortdownarrow_N)}$ are distributed according to the respective measures.

Suppose $\lambda_N \leq  \ldots \leq \lambda_1$. Then the same holds for Brownian motions $(B_1, \ldots, B_N)$ with drift vector $\lambda$ on compact intervals $[a,b]$ reweighted by \eqref{e.gibbs rn derivative}, where initial data consists of ordered starting values at $a$, i.e., $B_i(a) = w_i$. 

The same also holds for the zero temperature (i.e., non-intersecting) cases of drifted Brownian motion and Brownian bridge if additionally $\vec w, \vec z \in \R^N_>$ and $f(a), g(a) < w_N$ and $f(b), g(b)< z_N$.
\end{theorem}

Theorem~\ref{t.monotonicity under extra conditioning bb} does not give an analogous result for the O'Connell-Yor diffusion as defined in \eqref{e.brownian motion rn derivative} with the entrance law $\mu_s$, since the Brownian motion processes in the statement do not have the endpoint reweighting factor $\psi_0(X^N(K))/\psi_0(z)$. To go from Theorem~\ref{t.monotonicity under extra conditioning bb} to a version for the O'Connell-Yor diffusion, an important step is the following.

\begin{lemma}[Local limit of O'Connell-Yor diffusion]\label{l.OY local limit}
Fix $K>0$ and $z\in\R^N$. Consider the process $X^N = (X^N_1, \ldots, X^N_N)$ started at $z$ defined by \eqref{e.brownian motion rn derivative}. For $M>K$ and $\lambda = (\lambda_1> \ldots >\lambda^N)$, let $Y^{N,M,\lambda} = (Y^{N,M,\lambda}_1, \ldots, Y^{N,M,\lambda}_N)$ be the restriction to $[0,K]$ of $N$ independent Brownian motions on $[0,M]$ with drift vector $\lambda$, started at $z$, and reweighted by the Radon-Nikodym derivative proportional to $\exp(-\int_0^M\sum_{i=0}^N \exp(Y^{N,M,\lambda}_{i+1}(y) - Y^{N,M,\lambda}_i(u))\, \dif u)$.

Then, as $M\to\infty$ and $\lambda\to 0^N$ in that order, $Y^{N,M,\lambda}$ converges weakly to $X^N|_{[0,K]}$.
\end{lemma}

\begin{corollary}\label{c.mono in cond OY}
Theorem~\ref{t.monotonicity under extra conditioning bb} holds with $B^{\shortuparrow}$ and $B^{\shortdownarrow}$ replaced by O'Connell-Yor diffusions with entrance law $\mu_s$, respectively conditioned on equaling $y_{\shortuparrow, j}$ and $y_{\shortdownarrow, j}$ at $x_j$ for $j\in\intint{1,m}$ and satisfying $\inf_{[a,b]} (B^{\shortuparrow} -f) \geq 0$ and $\sup_{[a,b]} (B^{\shortuparrow} -f) \geq 0$, respectively.
\end{corollary}

\begin{proof}
For $*\in\{\shortuparrow,\shortdownarrow\}$ and $\varepsilon>0$, let $Y^{N,M,\lambda, f, \varepsilon}_{*}$ be the process in Lemma~\ref{l.OY local limit} conditioned on $Y^{N,M,\lambda, f, \varepsilon}_{*, 1}(x_j)\in[y_{*,j}, y_{*,j}+\varepsilon]$ for each $j\in\intint{1,m}$ and $\inf_{[a,b]}(Y^{N,M,\lambda, f, \varepsilon}_{\shortuparrow, 1} - f) \geq 0$ and $\sup_{[a,b]}(Y^{N,M,\lambda, f, \varepsilon}_{\shortdownarrow, 1} - f) \geq 0$. Theorem~\ref{t.monotonicity under extra conditioning bb} yields that $Y^{N,M,\lambda, \varepsilon}_{\shortuparrow}$ is stochastically larger than $Y^{N,M,\lambda, \varepsilon}_{\shortdownarrow}$. Since $\varepsilon>0$, the conditioning is a positive probability event, so Lemma~\ref{l.OY local limit} yields that the same stochastic monotonicity holds for the O'Connell-Yor diffusion started at $z$. Taking $z=R\cdot(-\frac{1}{2}(N-1), \frac{1}{2}(N-1)-1, \ldots, \frac{1}{2}(N-1))$ and taking $R\to\infty$ yields the stochastic monotonicity for the O'Connell-Yor free energy line ensemble (see the discussion following \cite[Proposition 8.3]{o2012directed}). Taking $\varepsilon\to 0$ then completes the proof.
\end{proof}

\begin{proof}[Proof of Lemma~\ref{l.OY local limit}]
Let $B^N=(B^N_1, \ldots, B^N_N)$ be $N$ independent Brownian motions with drift vector $\lambda$ on $[0,M]$. We observe that the law of $Y^{N,M,\lambda}$ is given by that of $B^N$ reweighted by the Radon-Nikodym derivative
\begin{align*}
\frac{\psi^{M-K}_\lambda(B^N(K))}{\psi^M_\lambda(z)} \exp\left\{-\int_0^K \sum_{i=0}^N \exp\left(B^N_{i+1}(u) - B^N_i(u)\right)\,\dif u\right\},
\end{align*}
where, with $\E^\lambda_w$ being the expectation associated to the law $\P^\lambda_z$ of independent Brownian motions with drift $\lambda$ started at $w\in\R^N$,
\begin{align*}
\psi^{M}_\lambda(w) := \E^\lambda_w\left[\exp\left\{-\int_0^{M} \sum_{i=0}^N \exp\left(B^N_{i+1}(u) - B^N_i(u)\right)\,\dif u\right\}\right].
\end{align*}
We claim that, as $M\to\infty$, $\psi^M_\lambda(z)\to \psi^\infty_\lambda(z)$; in particular, this implies $\psi^{M-K}_\lambda(B^N(K))/\psi^M_\lambda(z) \to \psi_\lambda(B^N(K))/\psi^\lambda(z)$, where $\psi_\lambda$ is as defined in \eqref{e.feynman-kac whittaker}. The claim follows if 
$$\int_0^{\infty} \sum_{i=0}^N \exp\left(B^N_{i+1}(u) - B^N_i(u)\right)\,\dif u$$
converges almost surely under $\P^\lambda_z$. This in turn is an immediate consequence of the fact that $B^N_{i+1} - B^N_i$ has negative drift for each $i$ under $\P^\lambda_z$ since $\lambda_1> \ldots >\lambda_N$.

Thus we have shown that the weak limit of $Y^{N,M,\lambda}$ as $M\to\infty$ is given by $N$ independent $\lambda$-drifted Brownian motions on $[0,K]$ tilted by the Radon-Nikodym derivative given in \eqref{e.brownian motion rn derivative} with $\psi_\lambda$ replacing $\psi_0$. Now, it is known that $\psi_\lambda(x)\to \psi_0(x)$ as $\lambda\to 0^N$ for each $x$ (see, e.g., \cite[page 4]{o2012directed}). This completes the proof.
\end{proof}

For the next result, we let $H_{\infty}$ be the Hamiltonian associated to non-intersection. Recall the notation $\smash{\mu_{N, H_\infty}^{\vec w, \vec z, g}}$ from before Theorem~\ref{t.monotonicity under extra conditioning bb}.

\begin{theorem}[BK inequality for reweighted Brownian bridge ensemble in zero temperature]\label{t.BK inequality bb}
Let $[a,b]\subseteq \R$, $\vec w, \vec z\in\R^N_{>}$ and $g:[a,b]\to\R\cup \{-\infty\}$ upper semicontinuous with $w_N>g(a)$ and $z_N>g(b)$. Let $B=(B_1, \ldots, B_N)$ be a collection of $N$  Brownian bridges on $[a,b]$ distributed according to $\smash{\mu^{\vec w, \vec z, -\infty}_{N, H_\infty}}$. Let $A, E\subseteq \mc C([a,b], \R)$ with $A$ an increasing event and $\P(B_1\in E)>0$. Then
$$\P(B_2\in A \mid B_1\in E)\leq \P(B_1\in A).$$

\end{theorem}

This theorem is stated for the zero-temperature case of a Brownian bridge ensemble, and note that the boundary values $\vec w, \vec z$ are ordered. The same proof also works for Brownian bridges or motions reweighted by the Radon-Nikodym derivative \eqref{e.gibbs rn derivative}, i.e., the positive temperature interaction. But in the positive temperature case, i.e., the KPZ line ensemble and O'Connell-Yor diffusion, points need not be ordered and so the positive temperature version of Theorem~\ref{t.BK inequality bb} cannot be applied directly to obtain the BK inequality for the O'Connell-Yor diffusion.

Next we prove Theorem~\ref{t.BK inequality bb}.

\begin{proof}[Proof of Theorem~\ref{t.BK inequality bb}]
Let $B = (B_1, \ldots,  B_N)$ be distributed according to $\smash{\mu_{N, H_\infty}^{\vec w, \vec z, g}}$. We have to prove that, for $A,E\subseteq \mc C([a,b], \R)$ with $A$ an increasing event,
\begin{equation}\label{e.prelimit bk}
\P\left(B_2\in A \mid B_1\in E\right) \leq \P\left(B_1\in A\right),
\end{equation}
To prove \eqref{e.prelimit bk}, we make use of the Brownian Gibbs property and monotonicity (Lemma~\ref{l.monotonicity}). We extend the notation to allow for an upper boundary condition. For $N\geq 1$, $\vec w, \vec z\in\R^k_>$, and $f,g:[a,b]\to \R$ measurable, let $\smash{\mu^{\vec w,\vec z, f, g}_{N, H_{\infty}}}$ be the law of $N$ Brownian bridges $(W_1, \ldots, W_N)$ on $[a, b]$, with $W_i(a) = w_i$ and $W_i(b) = z_i$ for $i\in\intint{1,N}$, conditioned to not intersect each other, the lower boundary curve $f$, or the upper boundary curve $g$. Thus $(B_1, \ldots, B_N)$ is distributed as $\smash{\mu_{N, H_\infty}^{\vec w, \vec z, \infty, g}}$.

Let $\F$ be the $\sigma$-algebra generated by $B_1$ on $[a,b]$.
Then we see that
\begin{align*}
\P\left(B_2 \in A, B_1 \in E\right) = \E\left[\P_{\F}\left(B_2 \in A\right)\one_{B_1\in E}\right]
&= \E\left[\mu^{\vec w', \vec z\,', B_1, g}_{N-1, H_{\infty}}\left(B'_1 \in A\right)\one_{B_1\in E}\right],
\end{align*}
where $\vec w'=(w'_1, \ldots, w'_{N-1})$ and $\vec z\,'=(z'_1, \ldots, z'_{N-1})$ are given by $w'_i = w_{i+1}$ and $z'_i = z_{i+1}$ for $i\in\intint{1,k-1}$, and $B'=(B_1', \ldots, B_{N-1}')$ is distributed as $\smash{\mu^{\vec w', \vec z\,', B_1, g}_{N-1, H_{\infty}}}$ conditionally on $\F$. By monotonicity (Lemma~\ref{l.monotonicity}) and since $A$ is an increasing event, by raising the upper boundary $B_1$ to $+\infty$, 
\begin{align*}
\mu^{\vec w', \vec z\,', B_1, g}_{N-1, H_{\infty}}\left(B'_1 \in A\right) \leq \mu^{\vec w', \vec z\,', \infty, g}_{N-1, H_{\infty}}\left(B'_1\in A\right).
\end{align*}
Thus
\begin{align}\label{e.bk proof monotonicity}
\P\left(B_2\in A, B_1 \in E\right)
&\leq \P\left(B_1 \in E\right)\cdot \mu^{\vec w', \vec z\,', \infty, g}_{N-1, H_{\infty}}\left(B'_1 \in A\right).
\end{align}
Now it is easy to see that $B_1'$ (the top curve of an ensemble distributed as $\mu^{\vec w', \vec z\,', \infty, g}_{N-1, H_{\infty}}$) is dominated by the top curve of an ensemble distributed as $\smash{\mu^{\vec w, \vec z, \infty, g}_{N, H_{\infty}}}$, as the latter ensemble is obtained from the first by increasing all the endpoint values and including an extra curve at the bottom (which is pointwise larger than $g$ and can be regarded as a lower boundary that, by Lemma~\ref{l.monotonicity}, pushes the top curve up). Thus, again since $A$ is an increasing event, and since the dominating ensemble is the same in law as the original ensemble $(B_1, \ldots, B_N)$,
$$\mu^{\vec w', \vec z\,', \infty, g}_{N-1, H_{\infty}}\left(B'_1 \in A\right) \leq \P\left(B_1\in A\right),$$
thus completing the proof of Theorem~\ref{t.BK inequality bb}.
\end{proof}

In the next section we give the proof of Theorem~\ref{t.assumptions hold}, using Theorems~\ref{t.monotonicity under extra conditioning bb} and \ref{t.BK inequality bb} and Corollary~\ref{c.mono in cond OY} for part of it. After that, the rest of Appendix~\ref{app.monotonicity proofs} will be devoted to proving Theorem~\ref{t.monotonicity under extra conditioning bb}.

\subsection{Verification of the  assumptions for the KPZ, parabolic Airy, and extremal stationary line ensembles}\label{app.hamiltonian.verifying assumptions}

To establish Theorem~\ref{t.assumptions hold} we start by verifying Assumptions~\ref{as.bg} (Brownian Gibbs and stationarity) and \ref{as.tails} (a priori upper tail bounds) for the KPZ and parabolic Airy line ensembles using references to the literature.

We begin with Assumption~\ref{as.bg}.

\begin{proposition}[Stationarity and Brownian Gibbs]\label{p.stationarity and bg}
For each $t>0$, there exists a line ensemble $\h$ with the $H_t$-Brownian Gibbs property such that $\h_1$ is distributed as the narrow-wedge solution to the KPZ equation \eqref{e.KPZ}. Additionally, for each $t>0$ and $k\in\N$, adding $x^2$ to $\h_k$ gives a stationary process. 

Similarly, the parabolic Airy line ensemble possesses the Brownian Gibbs property and each curve is stationary after the addition of $x^2$.
\end{proposition}

Note that in fact we need stationarity only of the top curve to meet Assumption~\ref{as.bg}.

\begin{proof}[Proof of Proposition~\ref{p.stationarity and bg}]
For $t>0$, the existence of the line ensemble and that it possesses the $H_t$-Brownian Gibbs property is stated in \cite[Theorem~2.15]{corwin2016kpz}. \cite{amir2011probability} gives that $x\mapsto\h_1(x)+x^2$ is stationary, but not the lower curves (or the whole ensemble). The latter assertion is proved in \cite[Corollary~1.13]{nica2021intermediate}.

For the parabolic Airy line ensemble $\cP$, that it possesses the Brownian Gibbs property is \cite[Theorem~3.1]{corwin2014brownian}, and the stationarity of $\cP_1(x)+x^2$ is \cite[Theorem~4.3]{prahofer2002PNG}. Stationarity of the entire ensemble $(i,x)\mapsto\cP_i(x)+x^2$ follows from the fact that it is determinantal and its correlation kernel as given in \eqref{e.extended airy kernel} is stationary, i.e., depends on only $t-s$; this suffices since the finite dimensional distributions form a separating class for continuous line ensembles (see e.g., \cite[Lemma 3.1]{dimitrov2021characterization}).
\end{proof}

Next we verify that the KPZ line ensemble and the parabolic Airy line ensemble satisfy Assumption~\ref{as.tails} with $\beta=\frac{3}{2}$.

\begin{theorem}[Theorem~1.11 of \cite{corwin2018kpz} and Theorem~1.3 of \cite{ramirez2011beta}]\label{t.one point asymptotics}
For any $t_0>0$, there exist $c_1 \geq c_2 > 0$ and $\theta_0$ such that, for any $t\in [t_0, \infty]$, $x\in\R$, and $\theta \geq \theta_0$,
\begin{align*}
\exp\bigl(-c_1\theta^{3/2}\bigr) \leq \P\bigl(\h_1(x) + x^2 > \theta\bigr) &\leq \exp\bigl(-c_2\theta^{3/2}\bigr).
\end{align*}
\end{theorem}

With this we have verified Assumptions~\ref{as.bg} and \ref{as.tails} for the KPZ and parabolic Airy line ensembles $\h$ and $\cP$. Note that Assumption~\ref{as.bg} holds for extremal stationary line ensembles by definition, and Assumption~\ref{as.tails} is verified for them via Theorems~\ref{t.upper tail lower bound} and \ref{t.extremal initial tail} assuming Assumptions~\ref{as.corr} and \ref{as.mono in cond}. So it remains only to prove Assumptions~\ref{as.corr} and \ref{as.mono in cond stronger} hold for the three ensembles (replacing Assumption~\ref{as.corr}(b) by Assumption~\ref{as.weak bk}(b\ensuremath{'}) for $\h$). %

We first do Assumptions \ref{as.corr}(b) (for zero temperature) and \ref{as.mono in cond stronger} and then turn to Assumptions~\ref{as.weak bk}(b\ensuremath{'}) and \ref{as.corr}(a).

\begin{proof}[Proof of Theorem~\ref{t.assumptions hold}: Assumption~\ref{as.mono in cond stronger}, and Assumption~\ref{as.corr}(b) in zero temperature]
We must handle three cases: the KPZ line ensemble, the parabolic Airy line ensemble, and extremal stationary zero-temperature ensembles. For the first, unlike the others, we do not know that the $\sigma$-algebra at infinity is trivial, though we expect this to be the case; showing this would be an interesting result.

For this reason for the KPZ line ensemble we prove the needed statements for the prelimiting model of the O'Connell-Yor free energy line ensemble introduced in Appendix~\ref{app.mono.o'connell-yor}, and take the edge scaling limit to obtain the result for the KPZ line ensemble. For general extremal stationary ensembles, in contrast, we do not have access to any such prelimiting model, and so we must work with the ensemble directly; in place of the prelimiting model we use the extremality. This approach also works for the parabolic Airy line ensemble using its extremality as noted in Section~\ref{s.notions of extremality}.

We will first prove Assumption~\ref{as.mono in cond stronger} here; the proofs for Assumption~\ref{as.corr}(b) for the zero temperature cases are analogous with Theorem~\ref{t.BK inequality bb} used in place of Theorem~\ref{t.monotonicity under extra conditioning bb}.

We have to show that, for any interval $[a,b]\subseteq \R$ containing $x_1, \ldots, x_m$, any upper semicontinuous $f:[a,b]\to\R\cup\{-\infty\}$, any $k\in\N$ and any bounded increasing function $F:\mc C([a,b], \R)^k\to\R$, and any $y_j^\ast\in\R$ with $j=1, \ldots, m$ and $\ast\in\{\shortuparrow,\shortdownarrow\}$ with $y_j^{\shortuparrow} > y_j^{\shortdownarrow}$,
\begin{align*}
\E\left[F \midd \h_1(x_j) = y^\shortuparrow_j, j\in \intint{1, m}, \inf_{[a,b]}(\h_1-f)\geq 0\right]
&\geq \E\biggl[F \midd \h_1(x_j) = y^\shortdownarrow_j, j\in \intint{1, m}, \sup_{[a,b]}(\h_1-f)\geq 0\biggr],
\end{align*}
where $F$ is shorthand for $F(\cL_1|_{[a,b]}, \ldots, \cL_k|_{[a,b]})$. 
For the O'Connell-Yor line ensemble, this statement is provided by Corollary~\ref{c.mono in cond OY}. However, we cannot take $n\to\infty$ directly to obtain the previous display because we are conditioning on a zero probability event. 

\emph{KPZ line ensemble:} For the O'Connell-Yor free energy line ensemble, scaled such that it converges to $\h$, which we denote by $\mf h^{t,n} = (\mf h^{t,n}_1, \ldots, \mf h^{t,n}_n)$, it follows from Corollary~\ref{c.mono in cond OY} along with an averaging argument as in the proof of Lemma~\ref{l.bound point conditioning by tail conditioning} above that
\begin{equation}\label{e.monotonicity to show}
\begin{split}
\MoveEqLeft[10]
\E\left[F \midd \mf h^{t,n}_1(x_j) \in [y_j^\shortuparrow-\varepsilon, y_j^\shortuparrow], j\in \intint{1, m}, \inf_{[a,b]}(\mf h^{t,n}_1-f)\geq 0\right]\\
&\geq \E\left[F \midd \mf h^{t,n}_1(x_j) \in [y_j^\shortdownarrow-\varepsilon, y_j^\shortdownarrow], j\in \intint{1, m}, \inf_{[a,b]}(\mf h^{t,n}_1-f)\geq 0\right].
\end{split}
\end{equation}

We take $n\to\infty$ and use that the conditioning events are positive probability to conclude that \eqref{e.monotonicity to show} holds with $\h$ in place of $\mf h^{t,n}$. Taking $\varepsilon\to 0$ completes the proof in the cases of the KPZ line ensemble.

\emph{Extremal zero-temperature and parabolic Airy ensembles:} 
We need to prove \eqref{e.monotonicity to show} with $\cP$ or an extremal line ensemble $\cL$ in place of $\mf h^{t,n}_1$; we will stick with the $\cP$ notation for both cases. For an event $A$ and $\sigma$-algebra $\F$, we will use the notation $\P(\cdot\mid A, \F) = \P(\cdot\cap A\mid \F)/\P(A\mid \F)$. 

We condition on the $\sigma$-algebra $\F_n = \Fext(n, [-n,n])$. By commutativity of conditionings (see \cite[Theorem 8.15]{Kallenberg}), we have a collection of $n$ non-intersecting Brownian bridges with the top one conditioned to pass through $[y^\ast_j-\varepsilon,y^\ast_j]$ at $x_j$ for $j=1, \ldots, m$, one collection for each of $\ast\in\{\shortuparrow,\shortdownarrow\}$, and conditioned to stay above $\cP_{n+1}$. By the Brownian bridge and zero-temperature case of Theorem~\ref{t.monotonicity under extra conditioning bb}, we obtain that the $\shortuparrow$-ensemble stochastically dominates the $\shortdownarrow$ one, so that
\begin{align*}
\MoveEqLeft[14]
\E\left[F \midd \cP_1(x_j) \in [y^\shortuparrow_j-\varepsilon, y^\shortuparrow_j], j\in\intint{1,m}, \inf_{[a,b]}(\cP_1-f)\geq 0, \F_n\right]\\
&\geq \E\biggl[F \midd \cP_1(x_j) \in [y^\shortdownarrow_j-\varepsilon, y^\shortdownarrow_j], j\in\intint{1,m}, \sup_{[a,b]}(\cP_1-f)\geq 0, \F_n\biggr].
\end{align*}
Now we take $n\to\infty$ and use that the limiting $\sigma$-algebra is trivial. Since $\cP_1(x_j) \in [y^\ast_j-\varepsilon, y^\ast_j]$ for $j=1, \ldots, m$ is a positive probability event, we see that
\begin{align*}
\MoveEqLeft[12]
\E\left[F \midd \cP_1(x_j) \in [y^\shortuparrow_j-\varepsilon, y^\shortuparrow_j], j\in\intint{1,m}, \inf_{[a,b]}(\cP_1-f)\geq 0\right]\\
&=\lim_{n\to\infty}\E\left[F \midd \cP_1(x_j) \in [y^\shortuparrow_j-\varepsilon, y^\shortuparrow_j], j\in\intint{1,m}, \inf_{[a,b]}(\cP_1-f)\geq 0, \F_n\right]\\
&\geq \lim_{n\to\infty}\E\biggl[F \midd \cP_1(x_j) \in [y^\shortdownarrow_j-\varepsilon, y^\shortdownarrow_j], j\in\intint{1,m}, \sup_{[a,b]}(\cP_1-f)\geq 0, \F_n\biggr]\\
&=\E\biggl[F \midd \h_1(x_j) \in [y^\shortdownarrow_j-\varepsilon, y^\shortdownarrow_j], j\in\intint{1,m}, \sup_{[a,b]}(\cP_1-f)\geq 0\biggr].
\end{align*}
Taking $\varepsilon\to0$ completes the proof.

\emph{Alternate proof for $\cP$:} One can also use that $\cP$ is the edge scaling limit of $n$ non-intersecting Brownian bridges \cite[Theorem 3.1]{corwin2014brownian} and mimic the KPZ line ensemble proof above with the zero temperature Brownian bridge case of Theorem~\ref{t.monotonicity under extra conditioning bb}; this avoids knowledge of extremality of $\cP$.
\end{proof}

To prove that Assumption~\ref{as.weak bk}(b\ensuremath{'}) holds for the KPZ line ensemble we will need to invoke a result from the recent work \cite{ganguly2025van}, which we record here.

\begin{proposition}[{\cite[Theorem 1.3]{ganguly2025van}}]\label{p.bk proved for kpz}
There exist $\tilde C, c, \theta_0>0$ such that the following holds. Let $x_0\in\R$, $K\geq 0$, and $A\subseteq \mc C([0,K],\R)$ be an increasing Borel measurable set. For any $t>0$, $\theta\geq \theta_0(t^{-1/6}\vee 1)$, and $M > \tilde C(\theta+x_0^2)^{3/4}$,
\begin{align*}
\P\Bigl(\h_{2}|_{[x_0,x_0+K]} - \tilde Ct^{-1/3}\log M \in A \midd \h_{1}(x_0) \geq \theta \Bigr)
&\leq \P\left(\h_{1}|_{[x_0,x_0+K]} \in A\right) + 3(K+1)t^{2/3}\exp(-cM^2).
\end{align*}
The same also holds under both conditionings when the processes are considered on $[x_0-K, x_0]$ rather than $[x_0,x_0+K]$.

\end{proposition}

\begin{proof}[Proof of Theorem~\ref{t.assumptions hold}: Assumption~\ref{as.weak bk}(b\ensuremath{'}) in zero and positive temperature]
That $\cP$ and extremal ensembles satisfy Assumption~\ref{as.weak bk}(b\ensuremath{'}) follows from Lemma~\ref{l.derive weak bk} since we have already established that it satisfies Assumptions~\ref{as.bg}, \ref{as.corr}(b), and \ref{as.tails}.

Since $\h$ satisfies Assumption~\ref{as.bg}, it suffices to prove Assumption~\ref{as.weak bk}(b\ensuremath{'}) for $x_0=0$. We may also assume without loss of generality that $I = [-K,K]$ for some $K>0$. Fix $t_0>0$ and recall we are considering $t>t_0$.

Take 
$$A = \left\{f\in \mc C([0,K],\R): \sup_{x\in[0,K]}(f(x) + x^2) > (\log M)^C\right\}.$$
Then by Proposition~\ref{p.bk proved for kpz} (replacing $A$ by $A-\tilde C t^{-1/3}\log(Mt)$), if $M>\tilde C\theta^{3/4}$ and $\theta\geq \theta_0(t_0^{-1/6}\wedge 1)$,
\begin{align*}
\MoveEqLeft[3]
\P\Bigl(\sup_{x\in[0,K]}(\h_2(x) + x^2) > (\log M)^C \midd \h_{1}(0) \geq \theta \Bigr)\\
&\leq \P\Bigl(\sup_{x\in[0,K]}(\h_1(x) + x^2) > (\log M)^C - \tilde Ct^{-1/3}\log (Mt)\Bigr) + 3(K+1)t^{2/3}\exp(-cM^2t^2).
\end{align*}
Now since we are considering $t>t_0$, it follows that $t^{-1/3}$ and $t^{-1/3}\log t$ are upper bounded by an absolute constant. As a result, and since $M>C$ is a condition in Assumption~\ref{as.weak bk}(b\ensuremath{'}), for a large enough constant $C$, we may upper bound the probability on the righthand side of the previous display with the same with $(\log M)^C - \tilde Ct^{-1/3}\log (Mt)$ replaced by $(\log M)^{C/2}$ as long as $C>1$. 

Next recall that $\log |I| \leq (\log M)^{2}$ has been assumed, which implies $\log K \leq (\log M)^{2}$ as well. Thus for all large enough $M$ (depending on $t_0$), it holds that $3(K+1)\exp(-\frac{1}{2}cM^2t^2) \leq 1/4$. It is also immediate that, for all $t>t_0$, $t^{2/3}\exp(-\frac{1}{2}cM^2t^2)$ can be made smaller than $1/4$ for all large enough $M$.

Overall, we have thus far obtained that
\begin{align*}
\P\Bigl(\sup_{x\in[0,K]}(\h_2(x) + x^2) > (\log M)^C \midd \h_{1}(0) \geq \theta \Bigr) \leq \P\Bigl(\sup_{x\in[0,K]}(\h_1(x) + x^2) > (\log M)^{C/2}\Bigr) + \tfrac{1}{16}.
\end{align*}
By Proposition~\ref{p.sharp sup over interval tail}, stationarity from Assumption~\ref{as.bg}, the upper bound on the upper tail from Assumption~\ref{as.tails} (verified above with the tail exponent $3/2$), and a union bound, we obtain that the first term on the righthand side in the previous display is upper bounded by
\begin{align*}
(K+1) \exp(-c(\log M)^{3C/4}) = \exp(-c(\log M)^{3C/4} + \log(K+1)).
\end{align*}
Since $\log K \leq (\log M)^{2}$, it follow that the previous display is upper bounded by $\frac{1}{16}$ for all large enough $M$ as long as $C>8/3$, thus
\begin{align*}
\P\Bigl(\sup_{x\in[0,K]}(\h_2(x) + x^2) > (\log M)^C \midd \h_{1}(0) \geq \theta \Bigr) \leq \tfrac{1}{8}.
\end{align*}
Repeating the same argument for the interval $[-K,0]$ and doing a union bound over the decomposition of $I$ into $[-K,0]$ and $[0,K]$ yields that
\begin{align*}
\P\Bigl(\sup_{x\in I}\,(\h_2(x) + x^2) > C'\log M \midd \h_{1}(0) \geq \theta \Bigr) \leq \tfrac{1}{2},
\end{align*}
which completes the proof.
\end{proof}

Next we turn to Assumption~\ref{as.corr}(a) for $\h_1$, $\cP_1$, and extremal stationary ensembles. We in fact show that it is implied by the multi-point version of Assumption~\ref{as.mono in cond} which we just proved.

\begin{proof}[Proof of Theorem~\ref{t.assumptions hold}: Assumption~\ref{as.corr}(a)]
We give the proof for any $\cL$ satisfying Assumption~\ref{as.mono in cond}. It suffices to prove that, for any $N\in\N$, increasing events $A, B\subseteq \R^N$, and $x_1, \ldots, x_N\in [a,b]$, $\P((\cL_1(x_i))_{i=1}^N\in A\cap B) \geq \P((\cL_1(x_i))_{i=1}^N\in A)\cdot \P((\cL_1(x_i))_{i=1}^N\in B)$. Then the FKG inequality for general increasing events of $\mc C([a,b],\R)$ follows by an approximation argument (see \cite[Lemma 6]{barbato2005fkg}).

Now we argue by induction on $N$ using the tower property of conditional expectations and monotonicity in conditioning of $\cL_1$ (Assumption~\ref{as.mono in cond}). First, the FKG inequality for finite dimensional distributions is equivalent to $\E[F((\cL_1(x_i))_{i=1}^N)G((\cL_1(x_i))_{i=1}^N)] \geq \E[F((\cL_1(x_i))_{i=1}^N)]\E[G((\cL_1(x_i))_{i=1}^N)]$ for all increasing square-integrable functions $F,G:\R^N\to\R$. Letting $\F_{N-1} = \sigma(\cL_1(x_i):i\in\intint{1,N-1})$,
\begin{align*}
\P((\cL_1(x_i))_{i=1}^N\in A\cap B) &= \E\left[\P\left((\cL_1(x_i))_{i=1}^N\in A\cap B \mid \F_{N-1}\right)\right]\\
&\geq \E\left[\P\left((\cL_1(x_i))_{i=1}^N\in A \mid \F_{N-1}\right)\cdot \P\left((\cL_1(x_i))_{i=1}^N\in B \mid \F_{N-1}\right)\right]\\
&\geq
\E\left[\P\left((\cL_1(x_i))_{i=1}^N\in A \mid \F_{N-1}\right)\right]\cdot \E\left[\P\left((\cL_1(x_i))_{i=1}^N\in B \mid \F_{N-1}\right)\right]\\
&= \P((\cL_1(x_i))_{i=1}^N\in A)\cdot \P((\cL_1(x_i))_{i=1}^N\in B).
\end{align*}
Here in the second line we used that, conditionally on $\F_{N-1}$, $\one_{(\cL_1(x_i))_{i=1}^N\in A}$ is an increasing function of $\cL_1(x_N)$ (and similarly for $A$ replaced by $B$), and then invoked the Harris inequality for measures on $\R$ (this also establishes the $N=1$ case). In the third line we used the induction hypothesis and the fact that, by monotonicity in conditioning, $\P((\cL_1(x_i))_{i=1}^N\in A \mid \F_{N-1})$ is an increasing function of $(\cL_1(x_i))_{i=1}^{N-1}$ (and similarly for $B$ in place of $A$).
\end{proof}

In the rest of the appendix we will prove Theorem~\ref{t.monotonicity under extra conditioning bb}.

\subsection{Proof of monotonicity in conditioning}\label{app.Brownian mono proofs}

Here we prove Theorem~\ref{t.monotonicity under extra conditioning bb}. As in earlier works (e.g., \cite{corwin2014brownian,corwin2016kpz}) the proof proceeds by establishing an analogous statement in a discrete setting by a Markov chain Monte Carlo argument, and takes a diffusive scaling limit to obtain the result in the Brownian setting.

\begin{proof}[Proof of Theorem~\ref{t.monotonicity under extra conditioning bb}]
First we fix some parameters. Let $(p_1, \ldots, p_N)\in(0,1)^N$, $m,n,N\in\N$, $1\leq x_1 <\ldots < x_m\leq n-1$, $0\leq y_1\leq \ldots\leq y_m\leq n$, and $\vec w, \vec z\in\intint{0,n}^N$, subject to $w_i \leq z_i$ for $i\in\intint{1,N}$, $w_1\leq y_1$ and $z_1\geq y_m$. Let $f:\intint{1,n}\to\R\cup\{-\infty\}$ be any function.

Given these parameters, we will define a Markov chain on collections of $N$ paths $(\mc X_1, \ldots, \mc X_N):\intint{1,n}^N\to \Z$ which have increments lying in $\{0,1\}$; have $\mc X_1(x_j) = y_j$ for each $j\in\intint{1, m}$; and have $\mc X_i(1) = w_i, \mc X_i(n) = z_i$ for each $i\in\intint{1,N}$. We call the set of all such collections of paths by $\Omega_{N,n}^{\vec x, \vec y, \vec w, \vec z}$, and we will sometimes refer to paths $\mc X:\intint{1,n}\to\Z$ with increments in $\{0,1\}$ as Bernoulli paths. 

We call $\smash{\Omega_{N,n}^{\vec x, \vec y, \vec w}}$ the same set without the constraint that $\mc X_i(n) = z_i$ for each $i$. The sets $\Omega_{N,n, \geq}^{\vec x, \vec y, \vec w, \vec z}$ and $\Omega_{N,n, \geq}^{\vec x, \vec y, \vec w}$ will mean the sets $\smash{\Omega_{N,n}^{\vec x, \vec y, \vec w, \vec z}}$ and $\smash{\Omega_{N,n}^{\vec x, \vec y, \vec w}}$ respectively under the additional non-crossing constraint $\mc X_1(i) \geq \mc X_2(i)\geq \ldots \geq\mc X_N(i)$ for each $i\in\intint{1,n}$ (which require $\vec w$ and $\vec z$ to be ordered to be non-empty). 

Finally, define $\Omega_{N,n, \sup}^{\vec x, \vec y, \vec w, \vec z, f}$ and $\Omega_{N,n, \inf}^{\vec x, \vec y, \vec w, \vec z, f}$ by additionally imposing that $\max_{\intint{1,n}}(\mc X_1 - f) \geq 0$ and $\min_{\intint{1,n}}(\mc X_1 - f) \geq 0$, respectively. Analogously define $\Omega_{N,n, *}^{\vec x, \vec y, \vec w,  f}$, $\Omega_{N,n, \geq, *}^{\vec x, \vec y, \vec w, \vec z,  f}$, and $\Omega_{N,n, \geq, *}^{\vec x, \vec y, \vec w,  f}$ for $*\in\{\sup, \inf\}$.

For each $*\in\{\sup,\inf\}$, we will define a Markov chain on each of $\Omega_{N,n, \geq, *}^{\vec x, \vec y, \vec w, \vec z, f}$, $\Omega_{N,n, \geq, *}^{\vec x, \vec y, \vec w, f}$, $\Omega_{N,n, *}^{\vec x, \vec y, \vec w, \vec z, f}$, and $\Omega_{N,n, *}^{\vec x, \vec y, \vec w,f}$. In each case, for $*\in\{\sup, \inf\}$, we assume that $\Omega_{N,n, *}^{\vec x, \vec y, \vec w, f}$, $\Omega_{N,n, *}^{\vec x, \vec y, \vec w, \vec z, f}$, $\Omega_{N,n, \geq, *}^{\vec x, \vec y, \vec w, f}$, or $\Omega_{N,n, \geq, *}^{\vec x, \vec y, \vec w, \vec z, f}$ as applicable are non-empty.

Recall the Boltzmann weight $W_H$ from \eqref{e.gibbs rn derivative} (with lower boundary condition $g$ and upper boundary condition $f\equiv \infty$); this can also take as argument a collection of paths from $\Omega_{N,n, *}^{\vec x, \vec y, \vec w, \vec z, f}$ or $\Omega_{N,n, *}^{\vec x, \vec y, \vec w, f}$ for $*\in\{\sup,\inf\}$ by interpreting the path as a continuous function by linear interpolation. In the zero temperature case of non-intersection, by $W_H$ we will simply mean the indicator function of non-crossing.

Fix one state space from $\Omega_{N,n, \geq, *}^{\vec x, \vec y, \vec w, \vec z, f}$, $\Omega_{N,n, \geq, *}^{\vec x, \vec y, \vec w, f}$, $\Omega_{N,n, *}^{\vec x, \vec y, \vec w, \vec z, f}$, and $\Omega_{N,n, *}^{\vec x, \vec y, \vec w, f}$ where $*\in\{\sup,\inf\}$; the first two will be called the zero temperature case and the last two the positive temperature case. We denote the state of the Markov chain at time $t\in\N$ by $\mc X^t = (\mc X^t_1, \ldots, \mc X^t_N)$. The dynamics are as follows. At time $t+1$, a curve index $I\in\intint{1,N}$, location $X\in\intint{2,n}$, and random variables $U\in[0,1]$ and $\sigma\in\{0,1\}$ are chosen uniformly at random, independent of each other and of all earlier times. For each $i\in\intint{1,N}$, we also choose $\rho_{i}\in\{0,1\}$, with $\rho_{i} = 1$ with probability $p_i$ and $\rho_{i}=0$ with probability $1-p_i$, independently at each time.

\begin{enumerate}
	\item If $X\neq n$, $\mc X^t_I(X-1)= \mc X^t_I(X)$ and $\mc X^t_I(X+1)= \mc X^t_I(X)+1$, and either (a) $I\in \intint{2,N}$ or (b) $I=1$ and $X\not\in\{x_1, \ldots, x_m\}$, then the following change is attempted. Define $(\tilde{\mc X}^{t+1}_1, \ldots, \tilde{\mc X}^{t+1}_N)$ as follows. If $\sigma=1$, define $\tilde{\mc X}^{t+1}_I(X)=\mc X^t_I(X)+1$, and if $\sigma=0$, $\tilde{\mc X}^{t+1}_I(X)=\mc X^t_I(X)$. Also define $\tilde{\mc X}^{t+1}_i(x)=\tilde{\mc X}^{t}_i(x)$ for all $(i,x)\neq (I,X)$. We set $({\mc X}^{t+1}_1, \ldots, {\mc X}^{t+1}_N) = (\tilde{\mc X}^{t+1}_1, \ldots, \tilde{\mc X}^{t+1}_N)$ if $(\tilde{\mc X}^{t+1}_1, \ldots, \tilde{\mc X}^{t+1}_N)$ lies in the selected state space and if
	\begin{align}\label{e.glauber ratio}
	\frac{W_H(\tilde{\mc X}^{t+1}_1, \ldots, \tilde{\mc X}^{t+1}_N)}{W_H({\mc X}^{t}_1, \ldots, {\mc X}^{t}_N)} \geq U
	\end{align}
	 (in the zero temperature case, in fact, the condition that \eqref{e.glauber ratio} holds is implied by membership in the relevant space). Otherwise, $\mc X^{t+1}_i = \mc X^t_i$ and for all $i\in\intint{1,N}$. This corresponds to flipping a local minimum to a local maximum.

	\item If $X\neq n$, $\mc X^t_I(X)= \mc X^t_I(X-1)+1$ and $\mc X^t_I(X+1)= \mc X^t_I(X)$, and either (a) $I\in \intint{2,N}$ or (b) $I=1$ and $X\not\in\{x_1, \ldots, x_m\}$, then the following change is attempted. Define $(\tilde{\mc X}^{t+1}_1, \ldots, \tilde{\mc X}^{t+1}_N)$ as follows. If $\sigma=1$, define $\tilde{\mc X}^{t+1}_I(X)=\mc X^t_I(X)-1$, and if $\sigma=0$, $\tilde{\mc X}^{t+1}_I(X)=\mc X^t_I(X)$. Also define $\tilde{\mc X}^{t+1}_i(x)=\tilde{\mc X}^{t}_i(x)$ for all $(i,x)\neq (I,X)$. We set $({\mc X}^{t+1}_1, \ldots, {\mc X}^{t+1}_N) = (\tilde{\mc X}^{t+1}_1, \ldots, \tilde{\mc X}^{t+1}_N)$ if $(\tilde{\mc X}^{t+1}_1, \ldots, \tilde{\mc X}^{t+1}_N)$ lies in the selected state space and if \eqref{e.glauber ratio} holds. Otherwise, $\mc X^{t+1}_i = \mc X^t_i$ and for all $i\in\intint{1,N}$.  This corresponds to flipping a local maximum to a local minimum.

	\item If $X=n$ and the state space is $\Omega_{N,n, \geq}^{\vec x, \vec y, \vec w}$ or $\Omega_{N,n}^{\vec x, \vec y, \vec w}$ (i.e., no right boundary values), define $(\tilde{\mc X}^{t+1}_1, \ldots, \tilde{\mc X}^{t+1}_N)$ by setting $\tilde{\mc X}^{t+1}_I(n) = \mc X^{t}_I(n-1) + \rho_I$, and $\tilde{\mc X}^{t+1}_i(x) = \mc X^t_i(x)$ for all $(i,x)\neq (I,n)$. Again, we set $({\mc X}^{t+1}_1, \ldots, {\mc X}^{t+1}_N) = (\tilde{\mc X}^{t+1}_1, \ldots, \tilde{\mc X}^{t+1}_N)$ if $(\tilde{\mc X}^{t+1}_1, \ldots, \tilde{\mc X}^{t+1}_N)$ lies in the selected state space and if \eqref{e.glauber ratio} holds, and, if not, $\mc X^{t+1}_i = \mc X^t_i$ and for all $i\in\intint{1,N}$.
\end{enumerate}

It is easy to check that the Markov chain so defined is irreducible in all cases of the state space. When the state space is $\Omega_{N,n, \geq, *}^{\vec x, \vec y, \vec w, \vec z, f}$ or $\Omega_{N,n, *}^{\vec x, \vec y, \vec w, \vec z, f}$ the chain has stationary distribution given by the uniform measure tilted by the Radon-Nikodym derivative proportional to $W_H$, while when the state space is $\Omega_{N,n, \geq, *}^{\vec x, \vec y, \vec w, f}$ or $\Omega_{N,n, *}^{\vec x, \vec y, \vec w, f}$, the stationary distribution is that of Bernoulli random walks with drift vector $(p_1, \ldots, p_N)$, conditioned on lying in the state space and tilted by the Radon-Nikodym derivative $W_H$. The chains are also aperiodic since with positive probability they do not change their state in a single step. Thus, since the state spaces are all finite, the law of the chains at time $t$ converge weakly to the their respective stationary distributions as $t\to\infty$.

\medskip

\emph{The monotone coupling. }
We define a monotone coupling between the above Markov chains associated to two pairs of values $\vec y$ that the top path is conditioned to pass through, such that one set of values is pointwise higher than the other; for the higher values, we also impose the condition $\inf(\mc X_1 - f)\geq 0$, and for the lower values, the condition $\sup(\mc X_1 - f)\geq 0$. More precisely, fix $m,n,N\in\N$, $1\leq x_1 <\ldots < x_m\leq n$, $0\leq y_1\leq \ldots\leq y_m\leq n-1$, $0\leq y'_1\leq \ldots\leq y'_m\leq n-1$ with $y_i'\geq y_i$ for each $i\in\intint{1,m}$, and $\vec w, \vec z\in\intint{0,n}^N$ as above, subject to the same conditions (so that, for example, $w_1\leq \min(y_1, y'_1) = y_1$  and $z_1 \geq \max(y_m, y_m') = y_m'$).

We will use the same boundary values $\vec w, \vec z$ as well as conditioned locations $\vec x$ for both chains. The chains are coupled by using the same random variables $I$, $X$, $U$, and $\sigma$ at each time step $t$, and initial conditions which are ordered across the two chains. That it is possible to choose such initial conditions follows from the ordering of $\vec y$ and $\vec y\,'$ and the assumption that respective state spaces are non-empty. 

It is straightforward to check, using an analysis as in \cite[Section 6]{corwin2014brownian} or \cite[Appendix 8.2]{corwin2016kpz} respectively for the zero and positive temperature cases (see also \cite[Section 5.3]{dimitrov2021characterization,dimitrov2021characterizationH} respectively for more detailed treatments), that, under this coupling, the pointwise ordering of the initial conditions are preserved by the dynamics forever; in the positive temperature case this also uses the convexity of $H$. Compared to the earlier analyses, here we have additionally imposed the condition of $\mc X_1^{\shortuparrow}$ satisfying $\inf(\mc X^{\shortuparrow}_1 -f)\geq 0$ ($\shortuparrow$ indicating the larger chain) and $\sup(\mc X_1^{\shortdownarrow} -f)\geq 0$ ($\shortdownarrow$ indicating the smaller chain). The earlier arguments nevertheless apply without any change after nothing that, by induction on $t$, if $\mc X^{t,\shortuparrow}_{1} \geq \mc X^{t,\shortdownarrow}_1$ and $\mc X^{t,\shortuparrow}_1(x)$ attempts a downward flip which succeeds, then the same will occur for $\mc X^{t,\shortdownarrow}_1$; in other words, the additional $\sup$/$\inf$ condition cannot make $\mc X^{t,\shortuparrow}_1$ go down while keeping $\mc X^{t, \shortdownarrow}_1$ unchanged. The case where an upward flip is attempted has no difference with the earlier analyses as the $\sup$/$\inf$ conditions are not relevant then.

By using the marginal convergence of each chain to their respective stationary distributions, we obtain a monotone coupling of these stationary distributions. 

\medskip

\emph{The diffusive scaling limit. }
Next we wish to take a diffusive scaling limit of the monotone coupling of the discrete Markov chains from above and obtain a monotone coupling of the analogous Brownian measures. First we note that, by an averaging argument as in the proof of Lemma~\ref{l.bound point conditioning by tail conditioning} in Section~\ref{app.mono.proof of mono in cond lemma}, the following holds. Let $E^\shortuparrow_1, E^\shortdownarrow_1, \ldots,  E^{\shortuparrow}_m, E^{\shortdownarrow}_m\subseteq \Z$ be finite intervals such that $\max E^{\shortdownarrow}_i \leq \min E^{\shortuparrow}_i$ for each $i\in\intint{1,m}$, $1\leq x_1 < \ldots  < x_m \leq n-1$, $\vec w\in\intint{0,n}^N$, and $0 < p_N \leq  \ldots \leq p_1 < 1$. Let $E^{\shortuparrow} = \prod_{i=1}^m E^{\shortuparrow}_i$ and $E^{\shortdownarrow} = \prod_{i=1}^m E^{\shortdownarrow}_i$. For $*\in\{\shortuparrow,\shortdownarrow\}$, consider the measure $\smash{\mu^{\vec x, E^{*}, \vec w,  \vec p, f}_{N, n, H, *}}$ which is the law of $N$ Bernoulli random walks $(\mc X^*_1, \ldots, \mc X^*_N)$ with drift vector $\vec p$ on $\intint{1,n}$ with starting values $\vec w$ which is tilted according to $W_H$ and conditioned on $\mc X^*_1(x_j)\in E_j^{*}$ for each $j\in\intint{1,m}$ and on $\inf(\mc X_1-f)\geq 0$ if $*=\shortuparrow$ and on $\sup(\mc X_1-f) \geq 0$ if $*=\shortdownarrow$; in the zero temperature case, tilting by $W_H$ simply means conditioning on non-crossing. Then, since these are the stationary distributions described above, the coupling produced above yields the two laws can be coupled such that $\mc X^{\shortdownarrow}_i(x) \leq \mc X^{\shortuparrow}_i(x)$ for all $i\in\intint{1,N}$ and $x\in\intint{0,n}$.

Next for given $T>0$, $\varepsilon>0$, $0< x_1 < \ldots <x_m < T$ and $\vec w, \vec y\in \R^m$ (note that the latter are not necessarily ordered), for each $j\in\intint{1,m}$ we define $x^{n}_j = nx_j$, $w^n_j = n^{1/2} w_j$, $E^{n,\varepsilon}_j = [p_1x^n_j + n^{1/2}y_j, p_1x^n_j + n^{1/2}(y_j+\varepsilon)]$, and $E^n = \prod_{j=1}^m E^{n,\varepsilon}_j$. For given real numbers $\lambda_1 > \ldots >\lambda_N$, define $p_i^n = \frac{1}{2}+ n^{-1/2}\lambda_i$. Finally define $f^n:[0,nT]\to\R\cup\{-\infty\}$ by $f^n(x) = p^n_1 x + n^{1/2}f(x/n)$.

If $(\mc X_1, \ldots ,\mc X_N)$ is distributed according to $\smash{\mu^{\vec x^n, \Z^m, \vec w^n, \vec p^n, -\infty}_{N,nT, H, *}}$, where $* = \sup$ or $\inf$, then it follows by an invariance principle (see \cite[Lemma 5.5]{dimitrov2021characterization} and \cite[Lemma 5.6]{dimitrov2021characterizationH} respectively for the zero and positive temperature settings for Brownian bridges; the Brownian motion case follows from an analogous argument) that $(\mc X_1, \ldots ,\mc X_N)$ suitably rescaled converges weakly to $N$ Brownian motions $(B_1, \ldots, B_N)$ on $[0,T]$ started from $\vec w$, with drift vector $\vec\lambda$, and tilted by $W_H$. The latter has positive probability of satisfying $B_1(x_j)\in [y_j, y_j + \varepsilon]$ for each $j\in\intint{1,m}$. Thus it follows that if $(\mc X_1, \ldots ,\mc X_N)$ is distributed according to $\mu^{\vec x^n, E^{n,\varepsilon}, \vec w^n, \vec p, f^n}_{N,nT, H, *}$, then after rescaling it converges weakly to $N$ Brownian motions $(B_1, \ldots, B_N)$ on $[0,T]$ started from $\vec w$, with drift vector $\vec\lambda$, tilted by $W_H$, and conditioned on satisfying $B_1(x_j) \in [y_j, y_j+\varepsilon]$ for each $j\in\intint{1,m}$ and on $\inf(B_1-f)\geq 0$ if $*=\shortuparrow$ and on $\sup(B_1-f)\geq 0$ if $*=\shortdownarrow$; this used that the latter conditioning events have positive limiting probability.

By taking the weak limit as $\varepsilon\to 0$, one obtains the law of $N$ Brownian motions $(B_1, \ldots, B_N)$ on $[0,T]$ started from $\vec w$, with drift vector $\vec\lambda$, tilted by $W_H$, and conditioned on $B_1(x_j) = y_j$ for each $j\in\intint{1,m}$ and on $\inf(B_1-f)\geq 0$ if $*=\shortuparrow$ and on $\sup(B_1-f)\geq 0$ if $*=\shortdownarrow$. This follows for Lebesgue almost every $\vec y\in\R^m$ by \cite{pfanzagl1979conditional}, and can be upgraded to all $\vec y\in\R^m$ by invoking the continuity in $\vec y$ of the conditional distributions. By taking the same weak limits described above for the monotone coupling of the stationary distributions of the above Markov chains, we complete the proof of Theorem~\ref{t.monotonicity under extra conditioning bb}.
\end{proof}

\section{Calculations involving $\h_1$}\label{app.hamiltonian}

In this Section~\ref{app.hamiltonian 2} we give the proofs of Proposition~\ref{p.scaled gibbs} (the form of the Hamiltonian $H_t$ for the Brownian Gibbs property enjoyed by the scaled narrow-wedge solution $\h$), and in Section~\ref{s.gen data.no big max} we give the proof of of Proposition~\ref{p.sharp sup over interval tail} (upper bound on the tail of $\sup_{[-1,1]}\h_1(x)+x^2$ in terms of the one-point tail).

\subsection{Calculation of the Hamiltonian}\label{app.hamiltonian 2}

Let the height function $\mc{\widetilde H} = \mc{\widetilde H}(t,x)$ be the Cole-Hopf solution to the KPZ equation
$$\partial_t \mc{\widetilde H} = \frac{1}{2}\partial_x^2 \mc{\widetilde H} + \frac{1}{2}(\partial_x \mc{\widetilde H})^2 + \xi,$$
where $\xi:(0,\infty)\times\R\to\R$ is rate 1 space-time white noise. It is known \cite[Theorem~2.15(ii)]{corwin2016kpz} that $\smash{\mc{\widetilde H}(t,\cdot)}$ can be embedded as the lowest-indexed curve of a line ensemble satisfies the $H$-Brownian-Gibbs property with respect to rate \emph{one} Brownian bridges where $H(x) = \exp(x)$. This is a consequence of the (non-trivial) fact that the O'Connell-Yor polymer free energy line ensemble has the same resampling property and converges to the KPZ line ensemble corresponding to $\smash{\mc{\widetilde H}}(t,\cdot)$ \cite{corwin2016kpz,nica2021intermediate}.

The process $\mc H$ we consider from \eqref{e.KPZ} can be related to $\mc{\widetilde H}$ by
$$\mc H(t,x) = \mc{\widetilde H}(2t,2x).$$
Next recall that the rescaled height function $\h_1(x)=\mf h_1(t,x)$ is given by
$\h_1(x) = t^{-1/3} \mc H(t, t^{2/3}x).$
We can now provide the proof of Proposition~\ref{p.scaled gibbs} that $\h$ satisfies the $H_t$-Brownian Gibbs property with $H_t(x) = 2t^{2/3}\exp(t^{1/3}x)$.

\begin{proof}[Proof of Proposition~\ref{p.scaled gibbs}]
Our starting point is the Gibbs property enjoyed by the line ensemble $\mc{\widetilde H}(t,\cdot)$ recalled above embeds into. 
We consider the following more abstract setup. Suppose $(\Omega_1,\F_1)$ and $(\Omega_2, \F_2)$ are measure spaces, and let $\mu,\nu$ be probability measures on $\Omega_1$ with $\mu\ll\nu$ (i.e., $\mu$ is absolutely continuous to $\nu$). Suppose $T:\Omega_1\to\Omega_2$ is measurable, and let $\mu^*$, $\nu^*$ be the pushforwards under $T$ of the respective measures. Then it is easy to show that $\mu^*\ll\nu^*$ and, for all $x\in\Omega_2$,
$$\frac{\dif\mu^*}{\dif\nu^*}(x) = \frac{\dif\mu}{\dif\nu}(T^{-1}(x)).$$
In our setting, $\Omega_1 = \mc C(\intint{1,k}\times [\ell ,r])$; $\Omega_2 = \mc C(\intint{1,k}\times [2^{-1}t^{-2/3}\ell, 2^{-1}t^{-2/3}r])$; $(T(g_i))(x) = t^{-1/3}g_i(2t^{2/3}x)$ for all $i\in\intint{1,k}$, and $x\in[\ell,r]$ and for all $(g_1, \ldots, g_k)\in\Omega_1$; $\mu = \P_H^{k,a,b,\vec x,\vec y, f}$ with $H(x) = e^x$ is the conditional law of $(\mc{\widetilde H}_1(t,\cdot), \ldots, \mc{\widetilde H}_k(t,\cdot))$ given $\Fext(k,\ell,r)$, where $t>0$ is fixed; and $\nu = \Pfree^{k,a,b,\vec x,\vec y}$ is the law of $k$ independent rate one Brownian bridges on $[\ell,r]$.

By Brownian scaling, $\nu^*$ is the law of $k$ independent rate \emph{two} Brownian bridges with appropriately transformed boundary values. By definition $\mu^*$ is the law of $(\h_1(\cdot), \ldots, \h_k(\cdot))$ for the same $t$ and conditionally on the analogously $T$-transformed boundary data.

We have to show that $\dif\mu^*/\dif\nu^*$ is given by the Radon-Nikodym derivative described by the $H_t$-Brownian Gibbs property. By Definition~\ref{d.bg} and the above fact, for any $(g_1, \ldots, g_k)\in\Omega_2$,
\begin{align*}
\frac{\dif \mu^*}{\dif \nu^*}(g_1, \ldots, g_k) &\propto \exp\left\{-\sum_{i=0}^k \int_\ell^r H\Bigl(\bigl(T^{-1}(g_{i+1})\bigr)(x) - \bigl(T^{-1}(g_i)\bigr)(x)\Bigr)\,\dif x\right\}\\
&= \exp\left\{-\sum_{i=0}^k \int_\ell^r H\Bigl(t^{1/3}g_{i+1}(2^{-1}t^{-2/3}x) - t^{1/3}g_i(2^{-1}t^{-2/3}x)\Bigr)\,\dif x\right\}\\
&= \exp\left\{-\sum_{i=0}^k \int_{2^{-1}t^{-2/3}\ell}^{2^{-1}t^{-2/3}r} 2t^{2/3}\cdot H\Bigl(t^{1/3}(g_{i+1}(y) - g_i(y))\Bigr)\,\dif y\right\}
\end{align*}
by making the transformation $y=2^{-1}t^{-2/3}x$ in the last line. Now $2t^{2/3}H(t^{1/3}x)$ is exactly $H_t(x)$ and the final line is exactly the form of the $H_t$-Brownian Gibbs property on $\Omega_2$ so the proof is complete.
\end{proof}

\subsection{The proof of Proposition~\ref{p.sharp sup over interval tail} via the no big max argument}\label{s.gen data.no big max}

Here we prove Proposition~\ref{p.sharp sup over interval tail}. As mentioned earlier, the proof is a refinement of that of the ``no big max'' argument given in \cite[Proposition~2.27]{hammond2016brownian} and \cite[Proposition~4.4]{corwin2014brownian}.

\begin{proof}[Proof of Proposition~\ref{p.sharp sup over interval tail}]
Since $x^2\leq 1$ on $[-1,1]$, we may bound $\P(\sup_{x\in[-1,1]}\cL_1(x) \geq \theta)$ and replace $\theta$ by $\theta-1$ at the end. 

Let $\chi = \sup\{x\in[-1,1]: \cL_1(x) \geq \theta\}$, with $\sup\emptyset = -\infty$. Then we are trying to bound $\P(\chi \in [-1,1])$; this formulation will help in applying the strong Gibbs property.

Let $x_0, x_2, \ldots, x_N$ be given by $x_i = -1 + \frac{i}{2}\lfloor\theta^{-1}\rfloor$ with $N = 4(\lfloor \theta^{-1} \rfloor)^{-1}$. Consider the event
\begin{equation}\label{e.no big max A defn}
A=\bigcup_{i=0}^{N-1} \bigl\{\cL_1(x_i) \geq \theta - 1\bigr\}
\end{equation}
and let $E = \{\cL_1(-2), \cL_1(2) \geq -\theta/2\}$. We set $\theta_0$ depending on $K$ from the statement of Proposition~\ref{p.sharp sup over interval tail} such that $\P(E^c)\leq \frac{1}{2}$ for all $\theta>\theta_0$.

We can bound $\P(A)$ in terms of the one-point tail of $\cL_1(0)$ using stationarity; so if we can show that $\P(A\mid \chi\in[-1,1])$ is not small, we can use this to show that $\P(\chi\in[-1,1])$ must be small. So we start by lower bounding the probability (for the final term we write $[x_{N-1}, x_N)$ for notational ease, but understand it to be $[x_{N-1}, x_N]$)
\begin{align*}
\P\left(A \ \Big|\  \chi\in[-1,1], E\right)
&\geq \sum_{i=0}^{N-1}\P\left(\cL_1(x_i) \geq \theta-1, \chi\in[x_i, x_{i+1}) \ \Big|\   \chi\in[-1,1], E\right)\\
&= \sum_{i=0}^{N-1}\P\left(\cL_1(x_i) \geq \theta-1 \ \Big|\   \chi\in[x_i, x_{i+1}), E\right)\\
&\qquad\qquad\times \P\left(\chi\in[x_i, x_{i+1}) \ \big|\   \chi\in[-1,1], E\right).
\end{align*}

We claim that the right-hand side is lower bounded by the constant $\frac{1}{2}$ for large enough $\theta$. Since $\sum_{i=0}^{N-1}\P\left(\chi\in[x_i, x_{i+1}] \ \big|\   \chi\in[-1,1], E\right) = 1$, it is sufficient to show that, for all $\theta>\theta_0$ and $i=0, \ldots, N$,
\begin{equation}\label{e.likely to be very high}
\P\left(\cL_1(x_i) < \theta-1 \ \Big|\   \chi\in[x_i, x_{i+1}), E\right) \leq \frac{1}{2}.
\end{equation}
To prove this, we first note that $[-2,\chi]$ is a stopping domain (recall from Definition~\ref{d.strong bg}), so we may apply the strong $H_t$-Brownian Gibbs property to it. By the strong $H_t$-Brownian Gibbs and monotonicity (Lemma~\ref{l.monotonicity}), we see that the previous probability is upper bounded by
\begin{align*}
\P\left(B(x_i) < \theta-1\right),
\end{align*}
where $B$ is a Brownian bridge on $[-2,\chi]$ with boundary values $-\theta/2$ and $\theta$. (Note that by monotonicity we have ignored the interaction with the lower curve $\cL_2$, and brought down the boundary values as far as possible.) Since $\chi-x_i \leq \frac{1}{2}\theta^{-1}$ on the event that $\chi\in[x_i, x_{i+1}]$ and $\chi\geq -1$ on the same event, we see that
\begin{align*}
\E[B(x_i)] &= \frac{\chi-x_i}{\chi+2}\left(-\theta/2\right) + \frac{x_i+2}{\chi+2}\theta
= \theta - \frac{\chi-x_i}{\chi+2}\cdot\frac{3\theta}{2}
\geq \theta - \frac{3}{4}.
\end{align*}
Letting $\sigma^2 = \mathrm{Var}(B(x_i))$, we see that
\begin{align*}
\P\left(B(x_i) < \theta-1\right) \leq \P\left(N\bigl(\theta - \tfrac{3}{4}, \sigma^2\bigr)< \theta-1\right) \leq \frac{1}{2}.
\end{align*}
What we have shown above is that, with $A$ as in \eqref{e.no big max A defn},
\begin{align*}
\P\Bigl(A\mid \chi\in[-1,1], E\Bigr) \geq \frac{1}{2}.
\end{align*}
Recall our ultimate goal is to show $\P(\chi\in[-1,1])$ is small. We will break this probability up based on whether $E$ occurs or not. We have an a priori bound on $\P(A)$ in terms of the one-point upper tail (along with a union bound), and the previous display says $\P(A\mid \chi\in[-1,1], E)$ is large, which can be combined to control $\P(\chi\in[-1,1], E)$ using the inequality
$$\P(A) \geq \P(A,\chi\in[-1,1],E) = \P(\chi\in[-1,1], E)\cdot \P(A\mid\chi\in[-1,1], E) \geq \frac{1}{2}\P(\chi\in[-1,1], E).$$
Thus we see that
\begin{align*}
\P\left(\chi\in[-1,1]\right) \leq \P(\chi\in[-1,1], E) + \P(\chi\in[-1,1], E^c) \leq 2\cdot\P(A) + \P(\chi\in[-1,1])\cdot\P(E^c),
\end{align*}
the second term bounded in the last inequality by the FKG inequality from Assumption~\ref{as.corr}(a) since both $E$ and $\{\chi\in[-1,1]\} = \{\sup_{x\in[-1,1]}\cL_1(x) \geq \theta\}$ are increasing events; note that this is an important step, since we do not have access to quantitative upper bounds on the lower tail to usefully upper bound $\P(\chi\in[-1,1], E^c)$ by $\P(E^c)$ directly. The above argument via FKG allows us to rely on tightness instead.

Recall $\theta_0$ was chosen so that $\P(E^c)\leq \frac{1}{2}$ for all $\theta>\theta_0$. Substituting this into the previous display shows that, for $\theta>\theta_0$,
$$\P(\chi\in[-1,1]) \leq 4\cdot \P(A) \leq 4\theta\cdot\P(\cL_1(0)\geq \theta-1).$$
the final inequality by a union bound. Replacing $\theta$ by $\theta-1$ as mentioned in the beginning of the proof completes the argument.
\end{proof}

\section{Proofs of Brownian estimates}\label{app.brownian estimates}

In this appendix we provide the proofs of a number of Brownian estimates from the main paper. 
We prove them in the following sections:
\begin{itemize}
	\item In Section~\ref{app.Br.restricted supremum tail proof}, Lemma~\ref{l.brownian bridge restricted sup tail} (the tail of the supremum of Brownian bridge over a subinterval)

	\item In Section~\ref{app.Br.parabola avoidance proof}, Proposition~\ref{p.parabola avoidance probability} (lower bound on the probability of a Brownian bridge avoiding a parabola with endpoints near the parabola)

	\item In Section~\ref{app.Br.parabola avoidance on tangent}, Proposition~\ref{p.brownian versatile tangent estimate} (same as previous with endpoints along a tangent) and

	\item In Section~\ref{app.Br.avoiding steeper line}, Lemma~\ref{l.bridge above line} (the probability of a Brownian bridge avoiding a line of lower slope, uniformly in the length of the interval).
\end{itemize}

\subsection{The tail of the supremum of Brownian bridge over a subinterval}\label{app.Br.restricted supremum tail proof}

\begin{proof}[Proof of Lemma~\ref{l.brownian bridge restricted sup tail}]
We condition upon $B(\inf J)$ and $B(\sup J)$ to see that
$$\P\left(\sup_{x\in J} B(x) \geq M\sigma_J\right) = \E\left[\P\left(\sup_{x\in J} B(x) \geq M\sigma_J \ \Big|\  B(\inf J), B(\sup J)\right)\right].$$
For $x\in J$, let $X(x) = \E[B(x)\mid B(\inf J), B(\sup J)]$ and $\overline X = \max(B(\inf J), B(\sup J))$; note that $X(x) \leq \overline X$ for all $x\in J$. Then the right-hand side of the previous display is upper bounded by
\begin{align*}
\E\left[\P\left(\sup_{x\in J} B(x)-X(x) \geq M\sigma_J - \overline X \ \Big|\  B(\inf J), B(\sup J)\right)\right].
\end{align*}
We note that, conditionally on $B(\inf J)$ and $B(\sup J)$, $B(x) - X(x)$ is a Brownian bridge on $J$ with boundary values zero. Thus by Lemma~\ref{l.brownian bridge sup tail exact}, we have the bound
\begin{align}
\P\left(\sup_{x\in J} B(x) \geq M\sigma_J\right)
&\leq \E\left[\exp\left\{-\frac{(M\sigma_J - \overline X)^2}{2\cdot|J|/4}\right\}\one_{\overline X \leq \tfrac{1}{2}M\sigma_J}\right] + \P(\overline X\geq \tfrac{1}{2}M\sigma_J)\nonumber\\
&\leq  \exp\left\{-\frac{2(\frac{1}{2}M\sigma_J)^2}{|J|}\right\} + \P(\overline X\geq \tfrac{1}{2}M\sigma_J) \label{e.restricted bridge sup bound}
\end{align}
For the second term, at least one of $B(\inf J)$ and $B(\sup J)$ must exceed $\frac{1}{2}M\sigma_J$; these are mean zero normal variables, and each have variance at most $\sigma_J^2$ by definition of $\sigma_J$. Thus by a union bound and Lemma~\ref{l.normal bounds} the second term is at most $2\exp(-M^2/8)$. Next we bound the first term, which amounts to lower bounding $\sigma_J^2/|J|$. We break into two cases for this.

In the first case, the midpoint $m_I$ of $I$ lies inside $J$. In this case a computation shows that $\sigma_J^2 = \Var(B(m_I)) = |I|/4 \geq |J|/4$, so we obtain a lower bound on $\sigma_J^2/|J|$ of $1/4$.

The second case is when $m_I\not\in J$. We may assume $\sup J < m_I$, as the other case of $\inf J > m_I$ is symmetric. Here a computation shows that $\sigma_J^2 = \Var B(\sup J) = (\sup J-\inf I)(\sup I - \sup J)/|I| \geq |J|\times (\sup I - m_I)/|I| = \frac{1}{2}|J|$.

Substituting $\sigma_J^2/|J|\geq \frac{1}{4}$ into the right-hand side of \eqref{e.restricted bridge sup bound} yields that it is upper bounded by $3\exp(-M^2/8)$. This completes the proof of Lemma~\ref{l.brownian bridge restricted sup tail}.
\end{proof}

\subsection{Lower bound on parabolic avoidance probability}\label{app.Br.parabola avoidance proof}
Here we prove Proposition~\ref{p.parabola avoidance probability}. The proof is straightforward but somewhat long. Essentially, we define a fine mesh and consider the probability that the Brownian bridge $B$ remains at least distance $1$ above the parabola at all these points; this is calculated using the covariance formulas for Brownian bridge and Gaussian tail bounds. The mesh is chosen fine enough that with high probability $B$ has fluctuations less than $\frac{1}{2}$ on all the intervals between mesh points, and thus avoids the parabola throughout.

\begin{proof}[Proof of Proposition~\ref{p.parabola avoidance probability}]
First we note that, by monotonicity (Lemma~\ref{l.monotonicity}) it is enough to prove the statement for the case that endpoints are equal to $(z_1,-z_1^2+1)$ and $(z_2,-z_2^2+1)$ instead of higher than them.

For an interval $[a,b]\subseteq [z_1, z_2]$, we define $B^{[a,b]}:[a,b]\to \R$ by
\begin{equation}\label{e.bridgified formula}
B^{[a,b]}(x) = B(x) - \frac{x-a}{b-a}B(b) - \frac{b-x}{b-a}B(a),
\end{equation}
which is distributed as a rate two standard Brownian bridge on $[a,b]$.

For $\varepsilon>0$ to be specified later (for simplicity we assume $(z_2-z_1)\varepsilon^{-1}$ is an integer), and for $j=0, \ldots, (z_2-z_1)\varepsilon^{-1}$, let $x_j = z_1+\varepsilon j$. Consider the event
\begin{equation}\label{e.easier event}
\bigcap_{j=1}^{(z_2-z_1)\varepsilon^{-1}-1}\left\{B(x_j) > -x_j^2 + 1\right\}\cap \left\{\inf B^{[x_j, x_{j+1}]} \geq -\frac{1}{2}\right\}.
\end{equation}
By Lemma~\ref{l.inclusion} ahead, we see that this event is contained in $\{B(x) > -x^2 \ \text{ for all } x\in[z_1, z_2]\}$ if $\varepsilon < 2^{1/2}$, and we will indeed ultimately set $\varepsilon$ to satisfy this constraint.

Thus we need to lower bound the probability of \eqref{e.easier event}. By properties of Brownian bridge, $B^{[x_j, x_{j+1}]}$ are independent across $j$ and are independent of $B(x_j)$ for all $j$, and are also identically distributed as rate two standard Brownian bridges on an interval of size $\varepsilon$. Letting
$$p := \P\left(\inf B^{[x_j, x_{j+1}]} \geq -\frac{1}{2}\right) > 1-\exp\left(-c\varepsilon^{-1}\right)$$
(using Lemma~\ref{l.brownian bridge sup tail exact} with $\sigma_I^2 = \varepsilon/2$), we see that
\begin{align}
\MoveEqLeft[14]
\P\left(\bigcap_{j=1}^{(z_2-z_1)\varepsilon^{-1}-1}\left\{B(x_j) > -x_j^2 + 1\right\}\cap \left\{\inf B^{[x_j, x_{j+1}]} \geq -\frac{1}{2}\right\}\right)\\
&= p^{(z_2-z_1)\varepsilon^{-1}-1}\cdot\P\left(\bigcap_{j=1}^{(z_2-z_1)\varepsilon^{-1}}\left\{B(x_j) > -x_j^2 + 1\right\}\right).\label{e.independence break up}
\end{align}
Next we lower bound the second factor. We use the property of Brownian bridges that, conditional on $B(x_0)$ for any given $x_0$, the distribution of $B$ on $[x_0, z_2]$ is a Brownian bridge from $(x_0, B(x_0))$ to $(z_2, B(z_2))$. This, along with monotonicity of the probabilities of increasing events in the endpoint values, implies that
\begin{align}
\MoveEqLeft[8]
\P\left(\bigcap_{j=1}^{(z_2-z_1)\varepsilon^{-1}-1}\left\{B(x_j) > -x_j^2 + 1\right\}\right)\nonumber\\
&\geq \prod_{j=1}^{(z_2-z_1)\varepsilon^{-1}-1} \P\left(B(x_j) > -x_j^2 + 1 \ \Big|\  B(x_{j-1}) = -x_{j-1}^2 + 1\right).\label{e.product lower bound}
\end{align}
Now, the earlier mentioned property of Brownian bridges implies that the distribution of $B(x_j)$ conditional on $B(x_{j-1})$ is a normal distribution with mean $\mu$ and variance $\sigma^2$, whose values are given by
\begin{align*}
\mu &:= \frac{z_2-x_{j}}{z_2-x_{j-1}} B(x_{j-1}) + \frac{x_j-x_{j-1}}{z_2-x_{j-1}} B(z_2) = (1-\lambda_j)B(x_{j-1}) + \lambda_jB(z_2).\\
\sigma^2 &:= 2\times\frac{(x_j-x_{j-1})\times (z_2-x_j)}{z_2-x_{j-1}} = 2\varepsilon\cdot\left(1-\lambda_j\right) \leq 2\varepsilon,
\end{align*}
where $\lambda_j = \frac{x_j-x_{j-1}}{z_2-x_{j-1}} = \frac{\varepsilon}{z_2-z_1-\varepsilon(j-1)}$, and recalling that we are working with rate two Brownian bridges.

Let $\tilde \mu$ be $\mu$ with $-x_{j-1}^2+1$ in place of $B(x_{j-1})$. Thus we see
\begin{align}
\MoveEqLeft[8]
\P\left(B(x_j) > -x_j^2 + 1 \ \Big|\  B(x_{j-1}) = -x_{j-1}^2 + 1\right)\nonumber\\
&\geq \P\left(\mc N(\tilde \mu, \sigma^2) \geq -x_j^2 + 1\right)\nonumber\\
&\geq \frac{1}{4\sqrt{\pi}}\cdot\frac{\varepsilon^{1/2}(1-\lambda_j)^{1/2}}{|-x_j^2+1-\tilde\mu|}\exp\left(-\frac{1}{4\varepsilon(1-\lambda_j)}(-x_j^2+1 - \tilde\mu)^2\right)\label{e.single point lower bound},
\end{align}
where we used Lemma~\ref{l.normal bounds} for the final inequality, assuming $-x_j^2+1 - \tilde \mu > (4/3)^{1/2} \sigma$ for now, which we will soon verify. First, we can simplify $\smash{-x_j^2+1-\tilde\mu}$ as
\begin{align*}
-x_j^2+1-\tilde\mu &= -\left[(x_{j-1}+\varepsilon)^2 - x_{j-1}^2 + \lambda_j(x_{j-1}^2-z_2^2)\right]\\
&= -\left[2\varepsilon x_{j-1} + \varepsilon^2 + \lambda_j(x_{j-1}^2-z_2^2)\right]\\
&= -\left[2\varepsilon(z_1 + \varepsilon(j-1)) -  \varepsilon(z_2+z_1) - \varepsilon^2(j-1) + \varepsilon^2\right]\\
&=\varepsilon\left[z_2-z_1 - \varepsilon j\right].
\end{align*}
Substituting this final expression for $-x_j^2+1-\tilde \mu$ into \eqref{e.single point lower bound} shows that
\begin{align*}
\MoveEqLeft[10]
\P\left(B(x_j) > -x_j^2 + 1 \ \Big|\  B(x_{j-1}) = -x_{j-1}^2 + 1\right)\\
&\geq \frac{1}{4\sqrt{\pi}}\cdot \frac{\varepsilon^{-1/2}(1-\lambda_j)^{1/2}}{z_2-z_1-\varepsilon j}\exp\left(-\frac{\varepsilon}{4(1-\lambda_j)}(z_2-z_1-\varepsilon j)^2\right)\\
&= \frac{1}{4\sqrt{\pi}}\cdot \frac{\varepsilon^{-1/2}(1-\lambda_j)^{1/2}}{z_2-z_1-\varepsilon j}\exp\left(-\frac{\varepsilon}{4}(z_2-z_1-\varepsilon j)(z_2-z_1-\varepsilon (j-1))\right),
\end{align*}
where for the first inequality we assumed that $\varepsilon(z_2-z_1-\varepsilon j) \geq (4/3)^{1/2}\sigma = (8/3)^{1/2}\varepsilon^{1/2}(1-\lambda_j)^{1/2}$.

We have to verify this inequality holds. Recalling that $1-\lambda_j = (z_2-z_1-\varepsilon j)/(z_2-z_1-\varepsilon (j-1))$ and squaring both sides reduces it to showing that $\varepsilon(z_2-z_1-\varepsilon j)(z_2-z_1-\varepsilon(j-1))^2 \geq 8/3$. Substituting $j\leq (z_2-z_1)\varepsilon^{-1}-1$ on the left side, it reduces to $2\varepsilon^3 \geq 8/3$, which is equivalent to $\varepsilon\geq(4/3)^{1/3}$. We will now set $\varepsilon$ to satisfy this inequality.
Recall we previously also assumed $\varepsilon < 2^{1/2}$, and note that $(4/3)^{1/3}  < 2^{1/2}$. We now set $\varepsilon$ to be such that both the upper and lower bounds are satisfied (e.g., $\varepsilon=6/5$).

By substituting in the same expression for $1-\lambda_j$ and using $j\geq 1$ we also see that the factor in front of the exponential in the previous display is bounded below by $c\varepsilon^{-1/2}(z_2-z_1)^{-1}$ some $c>0$.

Substituting the previous bound into \eqref{e.product lower bound} and then using \eqref{e.independence break up} shows that
\begin{align*}
\MoveEqLeft[0]
\P\Bigl(B(x) > - x^2\quad \forall x\in[z_1, z_2]\Bigr)\\
&\geq p^{(z_2-z_1)\varepsilon^{-1}}\\
&\quad\times\exp\left\{-\frac{\varepsilon}{4}\sum_{j=1}^{(z_2-z_1)\varepsilon^{-1}-1}(z_2-z_1-\varepsilon j)(z_2-z_1-\varepsilon (j-1)) - 2(z_2-z_1)\varepsilon^{-1}\log[\varepsilon(z_2-z_1)]\right\},
\end{align*}
implicitly absorbing the term $-(z_2-z_1)\varepsilon^{-1}\log c^{-1}$ into the $(z_2-z_1)\varepsilon^{-1}\log(\varepsilon(z_2-z_1))$ term for all large enough $z_2-z_1$ by raising the latter's coefficient. Expanding yields sums of powers of $j$, and using standard formulas shows that 
$$\varepsilon\sum_{j=1}^{(z_2-z_1)\varepsilon^{-1}-1}(z_2-z_1-\varepsilon j)(z_2-z_1-\varepsilon (j-1)) = \frac{(z_2-z_1)^3}{3} - \varepsilon^2\frac{(z_2-z_1)}{3} \leq \frac{(z_2-z_1)^3}{3}.$$
Thus, since $\varepsilon = \frac{6}{5}$ and $p>0$ uniformly, we obtain
\begin{align*}
\P\left(B(x) > -x^2\quad \forall x\in[z_1, z_2]\right)
&\geq \exp\left\{-\frac{(z_2-z_1)^3}{12}  - (z_2-z_1)\varepsilon^{-1}\left(2\log[\varepsilon(z_2-z_1)]+\log p^{-1}\right)\right\}\\
&\geq \exp\left\{-\frac{(z_2-z_1)^3}{12} - 2(z_2-z_1)\log(z_2-z_1)\right\}.
\end{align*}
for $z_2-z_1$ sufficiently large. This completes the proof.
\end{proof}

\begin{lemma}\label{l.inclusion}
If $\varepsilon < 2^{1/2}$, then \eqref{e.easier event} is contained in $\{B(x) > -x^2 \ \text{ for all } x\in[z_1, z_2]\}$.
\end{lemma}

\begin{proof}
From \eqref{e.bridgified formula}, it is sufficient to verify that, for every $j=1, \ldots, (z_2-z_1)\varepsilon^{-1}-1$ and $x\in[x_j,x_{j+1}]$,
$$-\frac{1}{2} + \frac{x-x_j}{x_{j+1}-x_j}(-x_{j+1}^2) + \frac{x_{j+1}-x}{x_{j+1}-x_j}(-x_{j}^2) + 1> -x^2,$$
which we do now. Letting $x=x_j+ \varepsilon y$ (so that $y\in[0,1]$), using that $x_{j+1}-x_j = \varepsilon$, and multiplying throughout by $-1$, the inequality can be simplified to
$$-\frac{1}{2} + yx_{j+1}^2 + (1-y)x_j^2 < x_j^2 + \varepsilon^2y^2 + 2\varepsilon x_jy.$$
Cancelling $x_j^2$ gives
$$-\frac{1}{2} +y(x_{j+1}^2 - x_j^2) < \varepsilon^2y^2 + 2\varepsilon x_jy.$$
Since $x_{j+1}^2-x_j^2 = (x_{j+1}-x_j)(x_{j+1}+x_j) = \varepsilon(2x_j+\varepsilon)$, we can further simplify to
$\varepsilon^2y(1-y) < \frac{1}{2}.$
Noting that $y(1-y) \leq \frac{1}{4}$, the previous inequality is satisfied under our hypothesis on~$\varepsilon$.
\end{proof}

\subsection{Parabola avoidance probability along a tangent}\label{app.Br.parabola avoidance on tangent}

\begin{proof}[Proof of Proposition~\ref{p.brownian versatile tangent estimate}]
Recall $I=[a,b]$. We assume without loss of generality that $x_\tan \leq \frac{a +b}{2}$.
Let $p(x) = -x^2$ and let 
$$\nonint = \Bigl\{B(x) > p(x) + \varepsilon_0 M\sigma_\tan\ \forall x\in I\Bigr\}.$$
Observe that the probability in \eqref{e.brownian versatile tangent to bound} is lower bounded by $\P(\nonint)$, and that $p(x)\leq \ellt(x)$ for all $x\in\R$. 

We see that
$
\P\bigl(\nonint\bigr) = \P\left(\inf_{x\in I}B(x) - p(x) \geq \varepsilon_0 M\sigma_\tan\right).
$
Now, if we define $\tilde B:[0,1]\to\R$ by
$$\tilde B(x) = |I|^{-1/2}\Bigl[B(a + |I|x) - \ellt(a + |I|x)\Bigr],$$
then $\tilde B$ is a rate two Brownian bridge on $[0,1]$ with $\tilde B(0) = \tilde B(1) = 0$. This means that
\begin{align*}
\P\bigl(\nonint\bigr)
&= \P\left(\inf_{x\in[0,1]}|I|^{1/2}\tilde B(x) + \ellt(a+|I|x) - p(a+|I|x) \geq \varepsilon_0 M\sigma_\tan\right).%
\end{align*}
For notational convenience, define $\tilde x_\tan$ via $x_\tan= a+ |I|\tilde x_\tan$ and $\tilde\sigma_\tan= |I|^{-1/2}\sigma_\tan$; note that $\tilde x_\tan \in[0,1/2]$ by the assumption on $x_\tan$ we made at the beginning of the proof. 

Since $\ellt$ is tangent to $p(x)$ at $x_\tan$, it is an easy calculation that $\ellt(w) - p(w)  = (w-x_\tan)^2$ for any $w$, so the previous displayed probability is equal to
\begin{align*}
\MoveEqLeft[10]
\P\left(\inf_{x\in[0,1]}\tilde B(x) + |I|^{-1/2}(a +|I|x-x_\tan)^2 \geq \varepsilon_0 M\tilde\sigma_\tan\right)\\
&= \P\left(\inf_{x\in[0,1]}\tilde B(x) + |I|^{3/2}(x-\tilde x_\tan)^2 \geq \varepsilon_0 M\tilde\sigma_\tan\right)
\end{align*}
To lower bound this probability we break up $[0,1]$ into three intervals given by $I_1 = [0,\frac{1}{2}\tilde x_\tan]$, $I_2 = [\frac{1}{2}\tilde x_\tan, 2\tilde x_\tan]$, and $I_3 = [2\tilde x_\tan, 1]$,  and use the positive association property of Brownian bridges. The previous display is lower bounded by
\begin{align}\label{e.tilde B interval breakup}
\MoveEqLeft[25]
\prod_{i=1}^3\P\left(\inf_{x\in I_i}\tilde B(x) + |I|^{3/2}(x-\tilde x_\tan)^2 \geq \varepsilon_0 M\tilde\sigma_\tan\right).
\end{align}
We will show that the $i=2$ factor is lower bounded by $\exp(-c\varepsilon_0^2M^2)$, and the $i=1$ and $3$ factors by a constant factor independent of $M$ and $|I|$.

For $i=1$, we will reduce the calculation to that of lower bounding the probability that a Brownian bridge on an interval of size $\delta$, with starting and ending points at height at least a constant times $\delta^{1/2}$, stays above $-c\delta^{1/2}$, where $c>0$ is a constant and $\delta = \smash{\frac{1}{2}\tilde x_\tan}$; this probability is of course uniformly positive. Indeed, we can lower bound $(x-\tilde x_\tan)^2$ on $I_1$ by $\tilde x_\tan^2/4$, and we note that
\begin{equation}\label{e.sigma tilde bound by x_tan tilde}
\varepsilon_0 M\tilde \sigma_\tan < \frac{1}{4}|I|^{3/2}\tilde x_\tan^2
\end{equation}
by our assumption that $M\leq C^{-1}(x_\tan-a)^{3/2}\leq C^{-1}|I|^{3/2}$ for a $C$ to be specified, since $\smash{\tilde \sigma_\tan \leq \tilde x_\tan^{1/2}}$ and
$$\frac{|I|^{3/2}\tilde x_\tan^{2}}{\tilde x_\tan^{1/2}} = |I|^{3/2} \frac{(x_\tan-a)^{3/2}}{|I|^{3/2}} = (x_\tan-a)^{3/2},$$
using the definition of $\tilde x_\tan$ (to be explicit, we take $C^{-1}=\smash{\frac{1}{4}}\varepsilon_0^{-1}$, where $\varepsilon_0>0$ is a constant still to be set). So by Lemma~\ref{l.brownian bridge restricted sup tail} for the second inequality,
\begin{align*}
\P\left(\inf_{x\in I_1}\tilde B(x) + |I|^{3/2}(x-\tilde x_\tan)^2 \geq \varepsilon_0 M\tilde\sigma_\tan\right)
&\geq \P\left(\inf_{x\in I_1}\tilde B(x) \geq -\tfrac{1}{4}|I|^{3/2}\tilde x_\tan^2\right)\\
&\geq 1-\exp\left(-c'\frac{|I|^3\tilde x_\tan^4}{\tilde x_\tan}\right),
\end{align*}
the second inequality also using that $\Var(\tilde B(x)) \leq x$ for all $x\in[0,1]$. The final expression is strictly positive independent of $|I|$ as by hypothesis $\tilde x_\tan \geq |I|^{-1}$.

For $i=2$ in \eqref{e.tilde B interval breakup} we lower bound the parabolic term by 0 and consider the event that the values of $\tilde B$ at the endpoints of $I_2$ are at height $2\varepsilon_0 M\tilde\sigma_\tan$; so the $i=2$ term is lower bounded by
$$\P\left(\inf_{I_2} \tilde B > \varepsilon_0 M\tilde\sigma_\tan \ \Big|\  \tilde B(\tfrac{1}{2}\tilde x_\tan), \tilde B(2\tilde x_\tan)\geq 2\varepsilon_0 M\tilde\sigma_\tan\right)\cdot \P\left(\tilde B(\tfrac{1}{2}\tilde x_\tan),\tilde B(2\tilde x_\tan)\geq 2\varepsilon_0 M\tilde\sigma_\tan\right).$$
The first probability is at least that of a Brownian bridge on an interval of size $|I_2| = \tfrac{3}{2}\tilde x_\tan$ with zero endpoints staying above $-\varepsilon_0 M\tilde\sigma_\tan$, which is $1-\exp(-2\times\varepsilon_0^2 M^2\tilde\sigma_\tan^2/(\tfrac{3}{2}\tilde x_\tan))$ by Lemma~\ref{l.brownian bridge sup tail exact}. Since $\tilde\sigma_\tan^2 \geq \tilde x_\tan/2$ (by our assumption that $\tilde x_\tan \leq\frac{1}{2}$), the expression in the previous sentence is at least $1-\exp(-2\varepsilon_0^2M^2/3)$. 

Turning to the second probability in the previous display, using the positive association of $\tilde B(\frac{1}{2}\tilde x_\tan)$ and $\tilde B(2\tilde x_\tan)$, the fact that they have variances $2\times\frac{1}{2}\tilde x_\tan(1-\frac{1}{2}\tilde x_\tan)\leq \tilde x_\tan$ and $2\times2\tilde x_\tan(1-2\tilde x_\tan)\leq 4\tilde x_\tan$ respectively, and the lower bound on Gaussian tails from Lemma~\ref{l.normal bounds}, we obtain that
\begin{align*}
\P\left(\tilde B\left(\tfrac{1}{2}\tilde x_\tan\right), \tilde B\left(2\tilde x_\tan\right)\geq 2\varepsilon_0 M\tilde\sigma_\tan\right)
&\geq \left[\frac{1}{\sqrt{8\pi\tilde x_\tan}}\exp\left(-\frac{(2\varepsilon_0 M\tilde\sigma_\tan)^2}{2\tilde x_\tan(1-\frac{1}{2}\tilde x_\tan)}\right)\right]^2\\
&\geq (8\pi)^{-1} \exp\left(-8\varepsilon_0^2 M^2\right),
\end{align*}
the last inequality since $\tilde x_\tan\leq 1$ and $\tilde \sigma_\tan^2 \leq \tilde x_\tan$. 

Finally we handle the $i=3$ term in the product in \eqref{e.tilde B interval breakup}. Here we will analyse the fluctuations on dyadic scales, which is needed because the interval is of unit order and the increasing fluctuations need to be balanced against the increasing decay of the parabola.

Let $k_0$ be the smallest $k$ such that $2^{k+1}\tilde x_\tan\geq 1$. For simplicity of notation let us interpret $2^{k_0+1}\tilde x_\tan$ as 1 in the following. Making use of the positive association of $\tilde B$, we see that
\begin{align}
\MoveEqLeft[8]
\P\left(\inf_{x\in I_3}\tilde B(x)+|I|^{3/2} (x-\tilde x_\tan)^2 \geq \varepsilon_0 M\tilde\sigma_\tan\right)\nonumber\\
&\geq \prod_{k=1}^{k_0}\P\left(\inf_{x\in [2^k\tilde x_\tan, 2^{k+1}\tilde x_\tan]} \tilde B(x) + |I|^{3/2}(x-\tilde x_\tan)^2 > \varepsilon_0 M\tilde \sigma_\tan\right)\nonumber\\
&\geq \prod_{k=1}^{k_0}\P\left(\inf_{x\in [2^k\tilde x_\tan, 2^{k+1}\tilde x_\tan]} \tilde B(x) > -|I|^{3/2}(2^k-1)^2\tilde x_\tan^2 + \varepsilon_0 M\tilde \sigma_\tan\right).\label{e.B_3 inf prob}
\end{align}
Now since $\varepsilon_0 M\tilde \sigma_\tan < \tfrac{1}{4}|I|^{3/2}\tilde x_\tan^2$ by \eqref{e.sigma tilde bound by x_tan tilde}, 
$$\varepsilon_0 M\tilde \sigma_\tan - |I|^{3/2}(2^k-1)^2\tilde x_\tan^2 \leq -\frac{1}{2}|I|^{3/2}(2^k-1)^2\tilde x_\tan^2,$$
so the $k$\textsuperscript{th} term in the product in \eqref{e.B_3 inf prob} is lower bounded by
\begin{align*}
\P\left(\inf_{x\in [2^k\tilde x_\tan, 2^{k+1}\tilde x_\tan]} \tilde B(x) > - \tfrac{1}{2}|I|^{3/2}(2^k-1)^2\tilde x_\tan^2\right)
&\geq 1-\exp\left(-c'\frac{|I|^3(2^k-1)^4\tilde x_\tan^4}{2^k\tilde x_\tan}\right)\\
&\geq 1-\exp\left(-c'2^{3k-4}\right)
\end{align*}
for an absolute constant $c'>0$ by Lemma~\ref{l.brownian bridge restricted sup tail}, using that $\Var(\tilde B(x)) \leq x$ for all $x\in[0,1]$. Substituting the previous display into \eqref{e.B_3 inf prob} shows that the latter expression is lower bounded by a constant.

Overall we have shown that, for some $c>0$,
$$\P(\nonint) \geq c\cdot\exp(-8\varepsilon_0^2 M^2).$$
Thus, setting $\varepsilon_0= 1$, we obtain that, for $0 < M < C^{-1}(x_\tan-a)^{3/2}$, 
$$\P\left(\nonint\right) \geq c \exp(-8M^2).$$
This completes the proof of the first (zero-temperature) part of Proposition~\ref{p.brownian versatile tangent estimate}.

To obtain a similar lower bound on $Z_{H_t}$, we make use of Lemma~\ref{l.pos temp Z  lower bound via non-avoid prob} with $[z_1,z_2]= I$, $x=y_{a}$, $y=y_{b}$, $p(x) = -x^2 + \varepsilon_0 M\sigma_\tan$, and $g(x) = \ellt(x) + x^2 = (x-x_\tan)^2$, which is non-negative. This yields, with the above lower bound on $\P(\nonint)$,
\begin{align*}
Z_{H_t}
&\geq \exp\left[-2t^{2/3}e^{-t^{1/6}}\int_{I}\exp\left(-\frac{t^{1/3}}{2}(u-x_\tan)^2\right)\,\mathrm du\right]\cdot e^{-c'\varepsilon_0^2 M^2}\\
&\geq \exp\left[-2t^{2/3}e^{-t^{1/6}}\int_{-\infty}^{\infty}\exp\left(-\frac{t^{1/3}}{2}(u-x_\tan)^2\right)\,\mathrm du\right]\cdot e^{-c'\varepsilon_0^2 M^2}\\
&=\exp\left[-2t^{2/3}e^{-t^{1/6}}\cdot \sqrt{2\pi}t^{-1/6}\right]\cdot e^{-c'\varepsilon_0^2 M^2}.
\end{align*}
The proof is complete by noting that $t^{1/2}\exp(-t^{1/6})$ is upper bounded by a uniform constant for all $t>0$, and by observing that we can set $\varepsilon_0$ independent of $t$ so that $c'\varepsilon_0^2=1$.
\end{proof}

\subsection{Avoiding a line of more extreme slope}\label{app.Br.avoiding steeper line}
\begin{proof}[Proof of Lemma~\ref{l.bridge above line}]
We may assume $\eta=1$ without loss of generality. Indeed, if $\eta>1$, then 
$$\P\left(\inf_{x\in[0,r]} B^K(x) < -\eta K\right) \leq \P\left(\inf_{x\in[0,r]} B^K(x) < -K\right).$$
If $\eta<1$, then $B^K$ stochastically dominates $B^{\eta K}$, and so 
$$\P\left(\inf_{x\in[0,r]} B^K(x) < -\eta K\right) \leq \P\left(\inf_{x\in[0,r]} B^{\eta K}(x) < -\eta K\right).$$

So we are in the case $\eta = 1$ and $K\geq \frac{1}{2}$. Let $B$ be a rate two Brownian bridge between $(0,0)$ and $(r,0)$, where we will assume $r$ is an integer for notational convenience. We also assume $r\geq 2$ as the claim is trivial for $r\in[0,2]$ (indeed, for any fixed bounded interval). We may define $B^K$ as $B^K(x) = B(x) + Kx$. Now, the right-hand side of the previous display can be upper bounded by
\begin{align}
\sum_{j=1}^r \P\left(\inf_{x\in[j-1,j]} B(x) + Kx < -K \right) \nonumber
&\leq \sum_{j=1}^r \P\left(\inf_{x\in[j-1,j]} B(x) < -Kj \right)\nonumber\\
&\leq 2\cdot\sum_{j=1}^{\lceil r/2\rceil} \P\left(\inf_{x\in[j-1,j]} B(x) < -Kj \right), \label{e.bridge above line breakup}
\end{align}
the last line by the equality in distribution between $x\mapsto B(x)$ and $x\mapsto B(r-x)$. Let $\sigma_j = \max_{x\in[j-1,j]}\Var(B(x)) = \Var(B(j)) = 2j(r-j)/r$, the second equality as we have assumed $j\leq \lceil r/2\rceil$ (this may fail for $j=\lceil r/2\rceil$, but it is easy to see that this case can be handled separately, as we will have $\sigma_{\lceil r/2\rceil} = O(r)$). Now by Lemma~\ref{l.brownian bridge restricted sup tail},
\begin{align*}
\P\left(\inf_{x\in[j-1,j]} B(x) < -Kj \right) \leq 3\exp\left(-\frac{K^2j^2}{8\sigma_j^2}\right) \leq 3\exp\left(-\frac{K^2 jr}{8(r-j)}\right) \leq 3\exp\left(-\frac{1}{8}K^2 j\right),
\end{align*}
as $r/(r-j) \geq r/(r-1)\geq 1$ since $j\geq 1$. Putting this bound back in \eqref{e.bridge above line breakup} yields that, for $K\geq \frac{1}{2}$,
\begin{equation*}
\P\left(\inf_{x\in[0,r]} B^K(x) < -K \right) \leq C\exp\left(-\frac{1}{8}K^2\right).\qedhere
\end{equation*}
\end{proof}

\end{document}